\documentclass[12pt]{amsart}
\usepackage{amsfonts, amssymb, latexsym, hyperref, mathrsfs}
\usepackage[usenames, dvipsnames]{xcolor}
\usepackage{ytableau}
\usepackage{tikz}
\usetikzlibrary{calc, shapes, backgrounds,arrows,positioning}
\tikzset{>=stealth',
  head/.style = {fill = white, text=black}, 
  pil/.style={->,thick},
  und/.style={thick},
  junct/.style = {draw,circle,inner sep=0.5pt,outer sep=0pt, fill=black}
  }
\usepackage{enumerate}

\newtheorem{theorem}{Theorem}[section]
\newtheorem{proposition}[theorem]{Proposition}
\newtheorem{lemma}[theorem]{Lemma}

\newtheorem{claim}[theorem]{Claim}
\newtheorem*{claim*}{Claim}
\newtheorem{corollary}[theorem]{Corollary}

\newtheorem{definition-lemma}[theorem]{Definition-Lemma}
\theoremstyle{remark}
\newtheorem{example}[theorem]{Example}

\numberwithin{equation}{section}

\newcommand\complexes{{\mathbb C}}

\DeclareMathOperator{\GL}{\sf GL}

\newcommand{\blank}{\phantom{2}}
\newcommand{\DD}{\mathcal{D}}
\newcommand{\EE}{\mathcal{E}}
\newcommand{\FF}{\mathcal{F}}
\newcommand{\GG}{\mathcal{G}}
\newcommand{\HH}{\mathcal{H}}
\newcommand{\QQ}{\mathcal{Q}}

\newcommand{\w}{{\sf w}}
\newcommand\x{{\sf x}}
\newcommand\y{{\sf y}}
\newcommand\z{{\sf z}}
\newcommand\uuu{{\sf u}}
\newcommand\aaa{{\sf a}}
\newcommand\bbb{{\sf b}}
\newcommand{\m}{{\sf mswap}}
\newcommand{\revm}{{\sf revmswap}}
\DeclareMathOperator{\wt}{{\tt wt}}
\DeclareMathOperator{\lab}{{\tt label}}
\DeclareMathOperator{\family}{{\tt family}}
\DeclareMathOperator{\swap}{{\tt swap}}

\DeclareMathOperator{\swapset}{{\tt swapset}}
\DeclareMathOperator{\revswap}{{\tt revswap}}
\DeclareMathOperator{\head}{{\tt head}}
\DeclareMathOperator{\body}{{\tt body}}
\DeclareMathOperator{\tail}{{\tt tail}}
\DeclareMathOperator{\im}{{\rm im}}
\newcommand{\gap}{\hspace{1in} \\ \vspace{-.2in}}
\newcommand{\dnu}{\overline{\nu}}
\newcommand{\dlambda}{\overline{\lambda}}
\newcommand{\dalpha}{\overline{\alpha}}
\newcommand{\drho}{\overline{\rho}}
\newcommand{\upper}[1]{#1_\star}

\newcommand*{\Scale}[2][4]{\scalebox{#1}{$#2$}}%

\newcommand*\circled[1]{\tikz[baseline=(char.base)]{\node[shape=circle,draw,inner sep=1pt] (char) {$#1$};}}

\begin{document}
\pagestyle{plain}

\mbox{}
\title{Equivariant $K$-theory of Grassmannians}
\author{Oliver Pechenik}
\author{Alexander Yong}
\address{Dept.~of Mathematics, U.~Illinois at
Urbana-Champaign, Urbana, IL 61801, USA}
\email{pecheni2@illinois.edu, ayong@uiuc.edu}

\date{June 5, 2015}

\begin{abstract}
We address a unification of the Schubert calculus 
problems solved by [A.~Buch '02] and [A.~Knutson-T.~Tao '03]. That is, we prove a combinatorial
rule for the structure coefficients in the torus-equivariant $K$-theory of Grassmannians with respect to the basis of Schubert structure sheaves.
We thereby deduce the conjectural rule of [H.~Thomas-A.~Yong '13] for the same coefficients. Both rules
are positive in the sense of [D.~Anderson-S.~Griffeth-E.~Miller '11] (and moreover in a stronger form).
Our work is based on the combinatorics of \emph{genomic tableaux} and a
generalization of [M.-P.~Sch\"{u}tzenberger '77]'s \emph{jeu de taquin}. 
\end{abstract}

\maketitle

\vspace{-.1in}
\tableofcontents

\vspace{-.4in}

\section{Introduction}\label{sec:introduction}
\subsection{Overview}
Let $X={\rm Gr}_k(\complexes^n)$ denote the Grassmannian of $k$-dimensional subspaces of $\complexes^n$. The action of ${\sf GL}_n(\complexes)$ on $X$ restricts to an action
of the Borel subgroup ${\sf B}$ of invertible upper triangular matrices and its subgroup ${\sf T}$ of invertible diagonal matrices.
The ${\sf T}$-fixed points $e_\lambda\in X$ are naturally indexed by Young diagrams
$\lambda$ contained in the rectangle $k\times (n-k)$. The {\bf Schubert varieties} are the
orbit closures $X_\lambda =\overline{{\sf B}_-e_\lambda}$. The Poincar\'e duals $[X_\lambda]$ of their classes form a ${\mathbb Z}$-basis of the cohomology ring $H^\star(X,{\mathbb Z})$.

The (classical) {\bf Schubert structure constants} $c_{\lambda,\mu}^\nu$ are defined by
$[X_\lambda]\cdot [X_\mu]=\sum_\nu c_{\lambda,\mu}^\nu [X_\nu]$.
In Schubert calculus, one interprets
$c_{\lambda,\mu}^\nu\in {\mathbb Z}_{\geq 0}$ as the number of points (when finite) in a generic triple intersection of Schubert varieties. Combinatorially, $c_{\lambda,\mu}^\nu$ is computed, in a manifestly nonnegative manner, by 
Littlewood-Richardson rules. The first such rule was stated by D.~E.~Littlewood-A.~R.~Richardson in the 1930s \cite{LRrule} in their study of the representation theory of the symmetric group.  The first rigorous proof of a rule was given by M.-P.~Sch\"{u}tzenberger \cite{Schutzenberger} in the 1970s. These rules describe $c_{\lambda,\mu}^\nu$ as a count of certain Young tableaux.

In the modern Schubert calculus, there is significant attention on the problem of generalizing the above work to richer
cohomology theories. Early last decade, two problems of this type were solved. A.~Buch \cite{Buch} found the first rule for the multiplication of the Schubert structure sheaves in $K$-theory. His rule
is positive after accounting for a predictable alternation of sign. Separately, A.~Knutson-T.~Tao \cite{Knutson.Tao} introduced \emph{puzzles} to give the first rule for equivariant Schubert calculus that is positive in the sense of \cite{Graham}.

We turn to a unification of these problems. Let $K_{\sf T}(X)$ denote the Grothendieck ring of ${\sf T}$-equivariant vector bundles over $X$. This ring has a natural $K_{\sf T}(\rm{pt})$-module structure and an additive basis given by the classes of Schubert structure sheaves; 
for background, we refer the reader to, e.g., \cite{Kostant.Kumar, Anderson.Griffeth.Miller} and the references therein.
The analogues of Littlewood-Richardson coefficients are the Laurent polynomials
$K_{\lambda,\mu}^{\nu}\in {\mathbb Z}[t_1^{\pm 1},\ldots, t_n^{\pm 1}] \cong K_{\sf T}(\rm{pt})$ defined by \[[{\mathcal O}_{X_\lambda}]\cdot [{\mathcal O}_{X_\mu}]=\sum_{\nu\subseteq k\times (n-k)}K_{\lambda,\mu}^\nu[{\mathcal O}_{X_{\nu}}],\] where
$[{\mathcal O}_{X_\lambda}]$ is the class of the structure sheaf of $X_\lambda$. 
These coefficients may be algebraically computed using 
\emph{double Grothendieck polynomials}; see \cite{lascoux.schuetzenberger, fulton.lascoux}.
The problem addressed by this paper is to prove a combinatorial rule for $K_{\lambda,\mu}^{\nu}$.

We give a  summary of
earlier contributions to the problem: A.~Knutson-R.~Vakil
conjectured a (still-open) formula for $K_{\lambda,\mu}^{\nu}$ in terms of puzzles 
(reported in \cite[$\mathsection 5$]{Coskun.Vakil}). V.~Kreiman \cite{Kreiman} proved a rule for the case
$\lambda=\nu$, corresponding to a certain localization (cf.~Section~\ref{sec:basecase}). C.~Lenart-A.~Postnikov \cite{Lenart.Postnikov}
determined a rule for the case $\lambda=(1)$ (in a broader context applicable
to any generalized flag variety); we use this result. Later, W.~Graham-S.~Kumar \cite{Graham.Kumar} 
determined the coefficients in the case 
$X={\mathbb P}^{n-1}$. ``Positivity'' of $K_{\lambda,\mu}^\nu$ (in a more general context) was geometrically
established by D.~Anderson-S.~Griffeth-E.~Miller \cite{Anderson.Griffeth.Miller}. More recently, A.~Knutson \cite{Knutson:positroid} obtained a puzzle rule in $K_{\sf T}(X)$ for the different problem of multiplying the class of a Schubert structure sheaf by that of an \emph{opposite} Schubert structure sheaf. Finally, H.~Thomas and the second author conjectured the first Young tableau rule for $K_{\lambda,\mu}^\nu$
\cite[Conjecture~4.7]{Thomas.Yong:H_T}; they showed their conjectural rule is positive in the
sense of \cite{Anderson.Griffeth.Miller};
see \cite[$\mathsection 4.1$]{Thomas.Yong:H_T}.

This paper introduces and proves an \cite{Anderson.Griffeth.Miller}-positive 
rule for $K_{\lambda,\mu}^{\nu}$ (Theorem~\ref{thm:main}); in fact our rule exhibits a further property of the coefficients which seems at present not to have a geometric explanation. Our rule
allows us to deduce the aforementioned conjecture of \cite{Thomas.Yong:H_T}. Indeed, our approach completes the strategy set out in \emph{loc.~cit.}\ and our Theorem~\ref{thm:main} is a generalization of the rule of \cite{Thomas.Yong:H_T} for {\sf T}-equivariant cohomology. The general approach of our proof is to relate our combinatorial rule to a $K$-theoretic generalization of a recurrence proven by A.~Molev-B.~Sagan \cite{Molev.Sagan} and A.~Knutson-T.~Tao \cite{Knutson.Tao} (who also credit A.~Okounkov). A similar approach was employed by A.~Buch \cite{Buch:quantum} who gave a rule for the equivariant quantum cohomology of Grassmannians, cf.~\cite{Buch.Mihalcea}. (The case of non-equivariant quantum cohomology had been previously handled geometrically by \cite{Coskun} and combinatorially by \cite{Buch.Kresch.Purbhoo.Tamvakis}, cf.~\cite{Buch.Kresch.Tamvakis}.)

We introduce \emph{genomic tableaux} and a generalization of M.-P.~Sch\"utzenberger's \emph{jeu de taquin} \cite{Schutzenberger}. C.~Monical has reported applications of genomic tableaux in the study of \emph{Lascoux polynomials} (see, e.g., \cite{ross} and references therein) and $K$-theoretic analogues of \emph{Demazure atoms}, extending results of \cite{haglund.etal}.  
These tableaux also give a new
rule for (non-equivariant) $K$-theory of Grassmannians;
the announcement \cite{Pechenik.Yong:announce} outlines applications to analogous problems when $X$ is replaced by Lagrangian or maximal orthogonal Grassmannians. 
In addition, one hopes to use our results to shed light on the A.~Knutson-R.~Vakil's puzzle conjecture for $K_{\lambda,\mu}^{\nu}$.
Note the puzzle conjecture does not directly recover an earlier ordinary $K$-theory puzzle rule; this is a qualitative difference with the rules of this text. Moreover, closely related to the equivariant Schubert calculus of $X$, the combinatorial rule of A.~Molev-B.~Sagan \cite{Molev.Sagan} solves a \emph{triple Schubert calculus} problem in $H^\star(\GL_{n}({\mathbb C})/{\sf B}  \times X \times \GL_{n}({\mathbb C})/{\sf B})$ (see \cite[$\mathsection$6]{Knutson.Tao}). Our methods should extend to give a $K$-theoretic analogue, cf.~\cite[$\mathsection$6.2]{Knutson.Tao}. 

\subsection{Genomic tableaux} A
{\bf genomic} tableau is a Young diagram filled with (subscripted) labels
$i_j$ where $i\in {\mathbb Z}_{>0}$ and the $j$'s that appear for each $i$ form
an initial segment of ${\mathbb Z}_{>0}$. It is {\bf edge-labeled}
of shape $\nu / \lambda$ if each
horizontal edge of a box weakly below the southern border of
$\lambda$ (viewed as a lattice path from $(0,0)$ to $(k,n-k)$) is filled with a subset of $\{i_j\}$.

Let $\x^\rightarrow$ be the box immediately east of
$\x$, $\x^\uparrow$ the box immediately north of
$\x$, etc. For a box $\x$, let $\overline{\x}$ denote the upper horizontal edge of $\x$ and
$\underline{\x}$ denote the lower horizontal edge. We write $\family(i_j)=i$. 
We distinguish two orders on subscripted labels. 
Say $i_j < k_\ell$ if $i < k$. Write $i_j \prec k_\ell$ if $i < k$ or 
$i=k$ with $j < \ell$. Note that $\prec$ is a total order, while $<$ is not.

A genomic tableau $T$ is {\bf semistandard} if the following
four conditions hold: 
\begin{itemize}
\item[(S.1)] ${\tt label}({\sf x}) \prec {\tt label}({\sf x}^{\rightarrow})$;
\item[(S.2)] every label is $<$-strictly smaller than any label South\footnote{Throughout, we write ``West'', ``west'' and ``NorthWest'' to mean ``strictly west'', ``weakly west'' and ``strictly north and strictly west'' respectively, etc.} in its column;
\item[(S.3)] if $i_j, k_\ell$ appear on the same edge then $i\neq k$;
\item[(S.4)] if $i_j$ is West of $i_k$, then $j \leq k$.
\end{itemize}
Refer to the multiset $\{i_j\}$ (for fixed $i$ and $j$) 
collectively as a {\bf gene}.
The {\bf content} of $T$ is $(c_1,c_2,c_3,\ldots)$ where
$c_i$ is the number of genes of family $i$. Suppose $\x$ is in row $r$. A label $i_j$ is {\bf too high} if  $i\geq r$ and $i_j\in \overline\x$, or alternatively if $i>r$ and $i_j\in \x$ or $i_j\in \underline\x$. 

\begin{example}
\label{exa:1.1}
For $\lambda=(4,2,2,1)$ and $\nu=(6,5,4,3,2)$ consider the genomic tableau $T$:
\[\begin{picture}(240,113)
\ytableausetup{boxsize=2.0em}
\put(-60,83){$\ytableaushort{ {*(lightgray)\blank} {*(lightgray)\blank} {*(lightgray)\blank} {*(lightgray)\blank} {1_2} {1_3}, {*(lightgray)\blank} {*(lightgray)\blank} {1_2} {2_1} {2_2},
{*(lightgray)\blank} {*(lightgray)\blank} {2_1} {3_2},
{*(lightgray)\blank} {1_1} {3_2},
{2_1} {3_2}}$}
\put(-34,6){$2_1 \ 3_2$}
\put(-4,6){$4_2$}
\put(-53,-20){$3_1$}
\put(-28,-20){$4_1$}

\put(90,60){The content of $T$ is $(3,2,2,2)$. The tableau $T$ is not}
\put(90,47){semistandard, since the second column from the}
\put(90, 34){left fails (S.2). If we deleted the $3_2$ from the edge,}
\put(90, 21){the result would be semistandard. No label is too}
\put(90, 8){high.}
\end{picture}
\]
\qed
\end{example}

\subsection{The ballot property}\label{sec:ballot_property}
A {\bf genotype} $G$ of $T$ is a choice of one label from each gene of $T$.
Let ${\tt word}(G)$ be obtained by reading $G$ down columns from right to left. (If there are multiple labels on an edge, read them from smallest to largest in $\prec$-order.)
Then $G$ is {\bf ballot} if in every initial segment of ${\tt word}(G)$, there are at least as many labels of family $i$ as of family $i+1$, for each $i\geq 1$. We say $T$ is {\bf ballot} if all of its genotypes are ballot. Let ${\tt BallotGen}(\nu/\lambda)$ be the set of ballot, semistandard, edge-labeled genomic 
tableaux of shape $\nu/\lambda$ where no label is too high. 

\begin{example}
Let
$\ytableausetup{boxsize=1.2em}
T=\Scale[0.8]{\ytableaushort{ {*(lightgray)\blank} {1_2}, {1_1} {2_1}}} \mbox{ \  and \  }
U=\Scale[0.8]{\ytableaushort{ {*(lightgray)\blank} {1_1}, {1_1} {2_1}}}.
$
Then $T$ is ballot: the one genotype (itself) has reading word is $1_2 2_1 1_1$, which is a
ballot sequence. $U$ is not ballot: it has two genotypes
$\Scale[0.8]{\ytableausetup{boxsize=1.2em}
\ytableaushort{ {*(lightgray)\blank} \ , {1_1} {2_1}}}$ and $\Scale[0.8]{\ytableaushort{ {*(lightgray)\blank} {1_1}, \ {2_1}}}$
and the word for the former is $2_1 1_1$, which is not ballot.\qed
\end{example}

\subsection{Tableau weights and the main theorem}\label{sec:tableau_weights}
Let $T \in {\tt BallotGen}(\nu/\lambda)$.
For a box $\x$, ${\sf Man}(\x)$ is
the ``Manhattan distance'' from the southwest corner (point) of $k\times (n-k)$
to the northwest corner (point)
of $\x$ (the length of any north-east lattice path between the
corners).

For a gene $\GG$, let $N_\GG$ be the number of genes $\GG'$ with $\family(\GG')=\family(\GG)$ and $\GG' \succ \GG$.
For instance, in Example~\ref{exa:1.1}, $N_{1_1}=2$ since the genes $1_2$ and $1_3$ are of the same family
as $1_1$ (namely family $1$) but $1_1\prec 1_2,1_3$.

If $\ell = i_j \in \underline{\x}$ and $\x$ is in row $r$, then
\begin{equation}
\label{eqn:edgefactordef}
\mathtt{edgefactor}(\ell) := \mathtt{edgefactor}_{\underline{\x}}(i_j) := 1 - \frac{t_{\sf Man}(\x)}{t_{r - i + N_{i_j} + 1 + {\sf Man}(\x)}}.
\end{equation}
The {\bf edge weight} $\mathtt{edgewt}(T)$ is
$\prod_\ell \mathtt{edgefactor}(\ell)$; the product is over edge labels of $T$.

A nonempty box $\x$ in row $r$ is {\bf productive} if
$\lab(\x) <
\lab(\x^\rightarrow)$.
If $i_j\in \x$, set
\begin{equation}
\label{eqn:boxfactordef}
\mathtt{boxfactor}(\x):=\frac{t_{{\sf Man}(\x)+1}}{t_{r - i + N_{i_j} + 1 + {\sf Man}(\x)}}.
\end{equation}
The {\bf box weight} of a tableau $T$ is
$\mathtt{boxwt}(T):=\prod_\x \mathtt{boxfactor}(\x)$, where the product is over all productive boxes of $T$.
The {\bf weight} of $T$ is
$ \wt T := (-1)^{d(T)} \times {\tt boxwt}(T)\times {\tt edgewt}(T)$.
Here $d(T)=\sum_{\GG} (|\GG|-1)$, 
where the sum is over all genes $\GG$
and $|\GG|$ is the (multiset) cardinality of $\GG$. Set
\[L_{\lambda,\mu}^\nu:=\sum_T \wt T,
\]
where the sum is over all $T\in {\tt BallotGen}(\nu/\lambda)$ that have content $\mu$.

\begin{theorem}[Main Theorem]
\label{thm:main}
$K_{\lambda,\mu}^\nu=L_{\lambda,\mu}^\nu$.
\end{theorem}

This provides the first proved rule for $K_{\lambda,\mu}^\nu$ that is manifestly \cite{Anderson.Griffeth.Miller}-positive. That is,
let $z_i :=\frac{t_i}{t_{i+1}}-1$. For $j>i$, we have
\begin{equation}
\label{eqn:thetzsubs}
\frac{t_i}{t_j}=\prod_{k=i}^{j-1}(z_k+1) \mbox{\ \ and  \ \ }
1-\frac{t_i}{t_j}=-\left(\prod_{k=i}^{j-1}(z_k+1)-1\right).
\end{equation}
Therefore, 
$(-1)^{\#\text{edge labels}}\times {\tt boxwt}(T)\times {\tt edgewt}(T)$ is $z$-positive. Since
clearly $d(T)=|\nu|-|\lambda|-|\mu|+\#\text{edge labels}$,
we have that $(-1)^{|\nu|-|\lambda|-|\mu|}L_{\lambda,\mu}^\nu
=\sum_T (-1)^{|\nu|-|\lambda|-|\mu|}\wt T$
is $z$-positive. This positivity is the same as that of
\cite[Corollary~5.3]{Anderson.Griffeth.Miller} after the substitution $z_i\mapsto e^{\alpha_i}-1$
where $\alpha_i$ is the $i$-th simple root for the root system $A_{n-1}$. 

\begin{example}
\label{exa:(2),(2,1),(2,2)}
To compute $K_{(2),(2,1)}^{(2,2)}$ for ${\rm Gr}_2({\mathbb C}^4)$, the required
tableaux are
\[\begin{picture}(400,25)
\ytableausetup{boxsize=1.2em}
\put(10,10){$T_1=\ytableaushort{ {*(lightgray)\blank} {*(lightgray)\blank} , {1_1} {1_2}}, \
T_2=\ytableaushort{ {*(lightgray)\blank} {*(lightgray)\blank} , {1_1} {1_2}}, \
T_3=\ytableaushort{ {*(lightgray)\blank}  {*(lightgray)\blank} , {1_1} {1_2}}, \
T_4=\ytableaushort{ {*(lightgray)\blank} {*(lightgray)\blank} , {1_1} {2_1}}, \
T_5=\ytableaushort{ {*(lightgray)\blank}   {*(lightgray)\blank} , {1_1} {2_1}}$}
\put(40,-10){$2_1$}
\put(121,-10){$2_1$}
\put(172,-10){$2_1$}
\put(186,-10){$2_1$}
\put(252,8){$1_2$}
\put(318,8){$1_2$}
\put(304,-10){$2_1$}
\end{picture}
\]
Then
\begin{itemize}
\item ${\tt edgewt}(T_1)=1-\frac{t_1}{t_2}$, ${\tt boxwt}(T_1)=\frac{t_3}{t_4}$ and $d(T_1)=0$;
\item ${\tt edgewt}(T_2)=1-\frac{t_2}{t_3}$, ${\tt boxwt}(T_2)=\frac{t_3}{t_4}$ and $d(T_2)=0$;
\item ${\tt edgewt}(T_3)=(1-\frac{t_1}{t_2})(1-\frac{t_2}{t_3})$, ${\tt boxwt}(T_3)=\frac{t_3}{t_4}$ and $d(T_3)=1$;
\item  ${\tt edgewt}(T_4)=(1-\frac{t_3}{t_4})$, ${\tt boxwt}(T_4)=\frac{t_2}{t_4}$ and $d(T_4)=0$; and
\item  ${\tt edgewt}(T_5)=(1-\frac{t_1}{t_2})(1-\frac{t_3}{t_4})$,
${\tt boxwt}(T_5)=\frac{t_2}{t_4}$ and $d(T_5)=1$.
\end{itemize}
Hence {\small
\[K_{(2),(2,1)}^{(2,2)} \!=\! \left(\!1-\frac{t_1}{t_2}\!\right)\frac{t_3}{t_4} + \left(1-\frac{t_2}{t_3}\right)\frac{t_3}{t_4}
-\left(1-\frac{t_1}{t_2}\right)\left(1-\frac{t_2}{t_3}\right)\frac{t_3}{t_4}
+\left(1-\frac{t_3}{t_4}\right)\frac{t_2}{t_4}
-\left(1-\frac{t_1}{t_2}\right)\left(1-\frac{t_3}{t_4}\right)\frac{t_2}{t_4}.\]}
Observe that, after rewriting using (\ref{eqn:thetzsubs}), each term is $z$-negative, in agreement with the discussion above; that is,
\begin{eqnarray}\nonumber
(-1)^{|(2,2)|-|(2)|-|(2,1)|}K_{(2),(2,1)}^{(2,2)} & = & -(-z_1)(z_3+1)-(-z_2)(z_3+1)+(-z_1)(-z_2)(z_3+1)\\ \nonumber
& & -(-z_3)(z_2+1)(z_3+1)+(-z_1)(-z_3)(z_2+1)(z_3+1)\\ \nonumber
& = & z_1(z_3+1)+z_2(z_3+1)+z_1 z_2 (z_3+1)\\ \nonumber
& & + z_3(z_2+1)(z_3+1)+z_1 z_3(z_2+1)(z_3+1) \nonumber
\end{eqnarray} is $z$-positive (without any cancellation needed).
\qed
\end{example}

There is a stronger positivity property exhibited by the rule
of Theorem~\ref{thm:main}. The work of
\cite{Anderson.Griffeth.Miller} generalizes the positivity of W.~Graham
\cite{Graham}: the equivariant Schubert structure coefficients are polynomials
with nonnegative integer coefficients in the simple roots $\alpha_i$. In 
\cite{Knutson:positroid}, A.~Knutson observes W.~Graham's geometric argument further implies the coefficients can be expressed as polynomials
with nonnegative integer coefficients in the positive roots \emph{such that each
monomial is square-free}. Moreover, A.~Knutson raises the issue of finding a ``proper analogue'' in equivariant $K$-theory for this square-free property. For $X$, we offer:
\begin{corollary}[Strengthened \cite{Anderson.Griffeth.Miller}-positivity]
\label{cor:strengthedAGM}
Let $z_{ij}:=\frac{t_i}{t_j}-1$. Then 
$(-1)^{|\nu|-|\lambda|-|\mu|}K_{\lambda,\mu}^{\nu}$ is expressible
as a polynomial with nonnegative integer coefficients in the $z_{ij}$'s 
such that each monomial is square-free.
\end{corollary}
\begin{proof}
The nonnegativity of the coefficients is immediate from each $z_{ij}$ being positive in the $z_i$'s. It remains to show each monomial in our expression $L_{\lambda,\mu}^\nu$ is square-free.

Consider a $T \in {\tt BallotGen}(\nu / \lambda)$. Every ${\tt edgefactor}(\ell)$ is of the form $-z_{ij}$, while every ${\tt boxfactor}(\x)$ is of the form $z_{ij} + 1$. Define an $(i,j)$-label to be either an edge label with ${\tt edgefactor}(\ell) = -z_{ij}$ or a label $\ell$ in a productive box $\x$ with ${\tt boxfactor}(\x) = z_{ij}+1$. 

Suppose $\ell, \ell'$ are $(i,j)$-labels of $T$. Say $\ell \in \x$ or $\overline{\x}$ and $\ell' \in \y$ or $\overline{\y}$. Since both are $(i, -)$-labels, ${\sf Man}(\x) = {\sf Man}(\y)$. Hence $\x$ and $\y$ are boxes of the same diagonal. We may assume $\x$ northwest of $\y$.
Let $\ell$ be an instance of $m_n$ and $\ell'$ and instance of $p_q$.
Since both are $(-, j)$-labels, ${\rm row}(\x) - m + N_{m_n} = {\rm row}(\y) - p + N_{p_q}$.
By (S.1) and (S.2), $m + r(\y) - r(\x) \leq p$, so $N_{m_n} = r(\y) - r(\x) + m - p + N_{p_q} \leq p - p + N_{p_q} = N_{p_q}$. Hence by ballotness of $T$, $\x = \y$ and moreover $m=p$. Therefore by (S.2) and (S.3), $\ell = \ell'$, and thus $T$ contains at most one $(i,j)$-label and each monomial in our expression is square-free.
\end{proof}

We do not know a geometric explanation for Corollary~\ref{cor:strengthedAGM}.
However, based on this result, one speculates that for any $G/P$,
if for each positive root
$\alpha$ we set $z_{\alpha}:=e^\alpha -1$, then the corresponding Schubert 
structure
coefficients for $K_T(G/P)$ may be expressed in a square-free manner with nonnegative coefficients in the $z_{\alpha}$'s.

\subsection{Organization}
The first key to the proof is to reformulate Theorem~\ref{thm:main} in terms of the more technical \emph{bundled} tableaux that are appropriate for the
inductive argument; this is presented in
Section~\ref{sec:bundled_tableaux}.
In Section~\ref{sec:structure_of_proof}, we outline this inductive argument that the rule of Theorem~\ref{thm:main} satisfies the key recurrence alluded to above. The base case is in Section~\ref{sec:basecase}. Both the plan of induction
and the base case may be considered routine. 

The core of the argument lies
in Sections~\ref{sec:good_tableaux}--\ref{sec:weight_preservation}. There we construct local swapping rules, defining a genomic generalization of M.-P.~Sch\"utzenberger's \emph{jeu de taquin}. This permits us
to establish a combinatorial map of formal sums of tableaux. 
This part of the argument is developed as a sequence of 
four main ideas:
\begin{itemize}
\item[(1)] To show well-definedness of the map, we identify and characterize the class of \emph{good tableaux} that arise via genomic \emph{jeu de taquin} (Sections~\ref{sec:good_tableaux}, \ref{sec:snakes} and~\ref{sec:swaps}).
\item[(2)] To establish surjectivity, we define reverse swaps and slides and prove these operations keep one in the class of good tableaux (Sections~\ref{sec:ladders} and~\ref{sec:reversal}).
\item[(3)] To prove that the map respects the coefficients of the key 
recurrence, we define and prove properties of a \emph{reversal tree} (Sections~\ref{sec:reversal_tree} and~\ref{sec:recurrence_proof}).
\item[(4)] The map is weight-preserving. However,
a significant subtlety is that it is not generally weight-preserving on \emph{individual} tableaux. To establish this property of the map, we need involutions that pair tableaux (Section~\ref{sec:weight_preservation}).
\end{itemize}
In Section~\ref{sec:proof_of_conjecture}, we recall the conjecture of \cite{Thomas.Yong:H_T} and prove it from Theorem~\ref{thm:main}; this argument is essentially independent of the rest of the paper. The three appendices isolate essentially straightforward but long technical proofs of important propositions.

\section{Bundled tableaux and a reformulation of Theorem~\ref{thm:main}}\label{sec:bundled_tableaux}
A tableau $T\in {\tt BallotGen}(\nu/\lambda)$ is
{\bf bundled} if every edge label is the westmost label of its gene. 
For example, 
in Example~\ref{exa:(2),(2,1),(2,2)}, only $T_3$ is \emph{not} bundled (the eastmost $2_1$ is to blame).
We denote the set of bundled tableaux of shape $\nu/\lambda$ by ${\tt Bundled}(\nu/\lambda)$.

Define a surjection
${\tt Bun}:{\tt BallotGen}(\nu/\lambda)\twoheadrightarrow {\tt Bundled}(\nu/\lambda)$.
This sends $T$ to ${\tt Bun}(T)$ by
deleting each edge label of $T$ that is not maximally west in its gene.
If $B\in {\tt Bundled}(\nu/\lambda)$, then any $T\in {\tt Bun}^{-1}(B)$
differs from $B$ by having (possibly $0$) additional edge labels. Let $E_{i_j}$ be the edges where $i_j$ appears in some
$T\in {\tt Bun}^{-1}(B)$ but not in $B$, i.e., the set of edges of $B$ where adding an $i_j$ would yield an element of ${\tt BallotGen}(\nu/\lambda)$.
We say $B$ has a {\bf virtual label} $i_j$ on each edge of $E_{i_j}$. We denote a virtual label $i_j$ by $\circled{i_j}$.

\begin{example}\label{ex:bundled_tableau}
All virtual labels are depicted below:
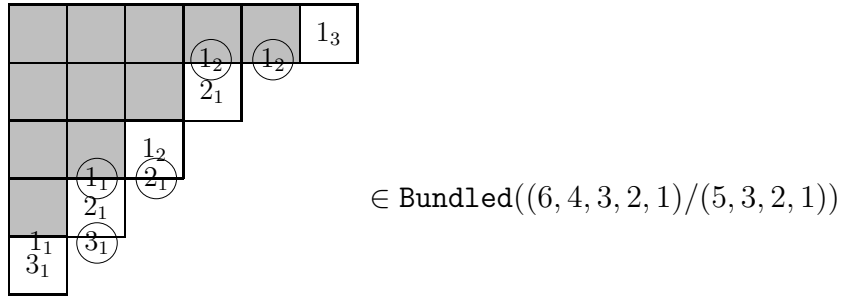
\[\Scale[0.9]{\begin{picture}(100,110)
\ytableausetup{boxsize=2em}
\put(-40,90){$\ytableaushort{ {*(lightgray)\blank} {*(lightgray)\blank} {*(lightgray)\blank} {*(lightgray)\blank} {*(lightgray)\blank} {1_3}, {*(lightgray)\blank} {*(lightgray)\blank} {*(lightgray)\blank}  {2_1}, {*(lightgray)\blank} {*(lightgray)\blank} {1_2}, {*(lightgray)\blank} {2_1},{3_1}}$}
\put(-32, 11){$1_1$}
\put(-12,11){\circled{3_1}}
\put(-12,38){\circled{1_1}}
\put(13,38){\circled{2_1}}
\put(36,88){\circled{1_2}}
\put(62,88){\circled{1_2}}
\end{picture}}
\begin{picture}(100,120)
\put(10,26){$\in {\tt Bundled}((6,4,3,2,1)/(5,3,2,1))$}
\end{picture}
\]
\qed
\end{example}

For $B\in {\tt Bundled}(\nu/\lambda)$, let
\begin{equation}
\label{eqn:buninveqn}
\wt(B)=\sum_{T\in {\tt Bun}^{-1}(B)} \wt(T).
\end{equation}
Let $B_{\lambda, \mu}^\nu$ denote the set of tableaux in ${\tt Bundled}(\nu/\lambda)$ with content $\mu$.
\begin{proposition}\label{lem:bundled_main_theorem}
$L_{\lambda,\mu}^{\nu} = \sum_{B \in B_{\lambda, \mu}^\nu} \wt(B)$.
\end{proposition}
\begin{proof}
Immediate from (\ref{eqn:buninveqn}) and the definition of $L_{\lambda,\mu}^{\nu}$.
\end{proof}

Compute ${\widetilde\wt}(B)$
as a product: an edge label $\ell$ contributes a factor of ${\tt edgefactor}(\ell)$ 
and each productive box $\x$ contributes a factor of ${\tt boxfactor}(\x)$. Each virtual label $\circled{\ell}\in \underline{\x}$ contributes $1 - {\tt edgefactor}_{\underline{\x}}(\ell)$ (where the latter is calculated as if $\circled{\ell}$ were instead $\ell$). 
Multiply by $(-1)^{d(T)}$ where $d(T)=\sum_{\GG} (|\GG|-1)$ and here $|\GG|$ is interpreted to be
the multiset cardinality of non-virtual $\GG$ in $T$.

\begin{example}
For $B$ from Example~\ref{ex:bundled_tableau},
${\widetilde\wt}(B)=(-1)^1 \cdot \left(1 - \frac{t_2}{t_8}\right) \cdot \frac{t_2}{t_4} \frac{t_4}{t_6} \frac{t_6}{t_9}\frac{t_8}{t_8}\frac{t_{11}}{t_{11}} \cdot \frac{t_3}{t_5} \frac{t_4}{t_9} \frac{t_5}{t_7}  \frac{t_8}{t_{10}} \frac{t_9}{t_{11}}$.\qed
\end{example}

\begin{lemma}\label{lem:wt=tildewt}
$\wt(B)={\widetilde \wt}(B)$.
\end{lemma}
\begin{proof}
Let $m$ be the number of virtual labels in $B$ and $a_i$ be the non-virtual weight of the
$i$-th virtual label (listed in some given order).
By the weights' definitions, the lemma follows from 
the ``inclusion-exclusion'' identity $\prod_{i\in[m]}a_i=\sum_{S\subseteq [m]}(-1)^{|S|}\prod_{i\in S}(1-a_i)$.
\end{proof}

\section{Structure of the proof of Theorem~\ref{thm:main}}\label{sec:structure_of_proof}

Let $\lambda^+ := \{ \rho \supsetneq \lambda : \text{$\rho/\lambda$ has no two boxes in the same row or column} \}$ and $\nu^- := \{ \delta \subsetneq \nu : \text{$\nu / \delta$ has no two boxes in the same row or column}  \}$. For a set $D$ of boxes, let $\wt D := \prod_{\x\in D} \frac{t_{{\sf Man}(\x)}}{t_{{\sf Man}(\x)+1}}$.
\begin{proposition}[Key recurrence]\label{prop:recurrence}
\begin{equation}
\label{eqn:therec}
 \sum _{\rho \in \lambda^+} (-1)^{|\rho/\lambda|+1}K^\nu_{\rho,\mu}
= K^\nu_{\lambda,\mu}(1-\wt \nu/\lambda )+ \sum_{\delta \in \nu^-}(-1)^{|\nu/\delta|+1} K^\delta_{\lambda,
\mu} \wt \delta/\lambda.
\end{equation}
\end{proposition}
\begin{proof}
The Chevalley formula in equivariant $K$-theory 
\cite[Corollary~8.2]{Lenart.Postnikov} implies:
\[ [\mathcal{O}_{X_\lambda}] [\mathcal{O}_{X_\square}] = [\mathcal{O}_{X_\lambda}] (1-\wt\lambda) + \sum_{\rho
\in \lambda^+} (-1)^{|\rho/\lambda|+1}[\mathcal{O}_{X_\rho}] \wt\lambda.\]

Thus, the coefficient of $[\mathcal{O}_{X_\nu}]$ in $\left([\mathcal{O}_{X_\lambda}][\mathcal{O}_{X_\square}]\right)[\mathcal{O}_{X_\mu}]$
is
\[ K^\nu_{\lambda,\mu} (1-\wt\lambda) + \sum _{\rho \in \lambda^+}
(-1)^{|\rho/\lambda|+1}
K^\nu_{\rho,\mu} \wt\lambda.\]
On the other hand, the coefficient of $[\mathcal{O}_{X_\nu}]$ in $\left([\mathcal{O}_{X_\lambda}][\mathcal{O}_{X_\mu}]\right)[\mathcal{O}_{X_\square}]$ is
\[
K^\nu_{\lambda,\mu}(1-\wt\nu) +
\sum_{\delta \in \nu^-} (-1)^{|\nu/\delta|+1}K^\delta_{\lambda,\mu} \wt\delta.\]
The proposition then follows from associativity and commutativity: \[ \left([\mathcal{O}_{X_\lambda}][\mathcal{O}_{X_\square}]\right)[\mathcal{O}_{X_\mu}] = \left([\mathcal{O}_{X_\lambda}][\mathcal{O}_{X_\mu}]\right)[\mathcal{O}_{X_\square}]. \qedhere\] 
\end{proof}

To prove $K_{\lambda,\mu}^{\nu}=L_{\lambda,\mu}^{\nu}$,
we induct on
$|\nu/\lambda|$. Proposition~\ref{prop:basecase} is the base case: $K_{\lambda,\mu}^{\lambda} = L_{\lambda,\mu}^{\lambda}$; this is proved using the description of
$L_{\lambda,\mu}^{\lambda}$ from Section~\ref{sec:introduction}.

The remaining cases use the description of $L_{\lambda,\mu}^{\nu}$ from Proposition~\ref{lem:bundled_main_theorem}. Assume
$K_{\theta,\mu}^{\tau}=L_{\theta,\mu}^{\tau}$ when $|\tau/\theta|\leq h$.
Suppose we are given $\lambda,\nu$ with $|\nu/\lambda|=h+1$.
We show that $L_{\lambda,\mu}^{\nu}$ satisfies (\ref{eqn:therec}).
Since Proposition~\ref{prop:recurrence} asserts $K_{\lambda,\mu}^\nu$ also satisfies (\ref{eqn:therec}) we will be done by induction.
 
Fix $\lambda$, $\mu$, $\nu$ with $\lambda \subsetneq \nu$. Define the formal sum
\[
\Lambda^+ := \sum_{\rho\in \lambda^+} (-1)^{|\rho/\lambda|+1} \sum_{T \in B_{\rho, \mu}^\nu} T.
\] Similarly define
\[ \Lambda := (1-\wt\nu/\lambda) \sum_{T \in B_{\lambda, \mu}^\nu} T
\mbox{\ \ \  and \ \ \ }
\Lambda^- := \sum_{\delta \in \nu^-} (-1)^{|\nu/\delta|+1} (\wt\delta/\lambda) \sum_{T \in B_{\lambda, \mu}^\delta} T.\]

In Section~\ref{subsection:swapsandslides},
we define an operation ${\tt slide}_{\rho/\lambda}$ that takes as input $T\in \Lambda^+$ and returns a
formal sum of tableaux with coefficients from ${\mathbb Z}[t_1^{\pm 1},\ldots,t_n^{\pm 1}]$.
The construction of ${\tt slide}_{\rho/\lambda}$ and proof of its correctness are found in Sections~\ref{sec:good_tableaux}--\ref{sec:swaps}.
Specifically, Corollary~\ref{cor:codomain} shows the tableaux in the formal sum are from $B_{\lambda, \mu}^\nu \cup \left(\bigcup_{\delta \in \nu^-} B_{\lambda, \mu}^\delta\right)$.

In Section~\ref{sec:recurrence_proof} we prove that
\[{\tt slide}(\Lambda^+):=\sum_{\rho\in \lambda^+} (-1)^{|\rho/\lambda|+1} \sum_{T \in B_{\rho, \mu}^\nu} {\tt slide}_{\rho/\lambda}(T)=\Lambda+\Lambda^{-};
\]
see Proposition~\ref{prop:slideLambda+equals}  for the precise statement. Finally Proposition~\ref{thm:weight_preservation} shows that $\wt \Lambda^+ = \wt {\tt slide}(\Lambda^+)$, so  $\sum _{\rho \in \lambda^+} (-1)^{|\rho/\lambda|+1}L^\nu_{\rho,\mu}
= L^\nu_{\lambda,\mu}(1-\wt\nu/\lambda)+ \sum_{\delta \in \nu^-}(-1)^{|\nu/\delta|+1} L^\delta_{\lambda,\mu} \wt\delta/\lambda$. This completes the proof that the Laurent polynomials $L_{\lambda,\mu}^\nu$ defined by the rule of 
Proposition~\ref{lem:bundled_main_theorem} equal $K_{\lambda,\mu}^\nu$. Hence we have
completed our proof of Theorem~\ref{thm:main}. \qed

\section{The base case of the recurrence}
\label{sec:basecase}

A different rule for the case $K_{\lambda,\mu}^{\lambda}$ was given by V.~Kreiman~\cite{Kreiman}. We give an independent proof of the following:

\begin{proposition}[Base case of the recurrence]
\label{prop:basecase}
$K_{\lambda,\mu}^{\lambda} = L_{\lambda,\mu}^{\lambda}$.
\end{proposition}
\begin{proof}
We use the original (unbundled) definition of
$L_{\lambda,\mu}^{\lambda}$ from Section~\ref{sec:introduction}.

One says that $\pi\in S_n$ is a {\bf Grassmannian permutation}
if there is at most one $k$ such that $\pi(k)>\pi(k+1)$.
The Grassmannian permutation for $\lambda \subseteq k \times (n-k)$ 
is the (unique) Grassmannian permutation $\pi_\lambda \in S_n$ 
defined by $\pi_\lambda(i)=i+\lambda_{k-i+1}$
for $1\leq i\leq k$ and $\pi(i)<\pi(i+1)$ for $i\neq k$.

Let $w',v'\in S_n$ be the Grassmannian permutations for the conjugate diagrams $\lambda',\mu'\subseteq (n-k)\times k$.
The following
identity relates $K_{\lambda,\mu}^{\lambda}$ to the localization of the class $[\mathcal{O}_{X_\lambda}]$ at the ${\sf T}$-fixed point $e_\mu$,
expressed as a specialization of a double Grothendieck polynomial:
\begin{lemma}
\label{lem:specialization}
$K_{\lambda,\mu}^{\lambda}=\overline{{\mathfrak G}_{v'}(t_{w'(1)},\ldots,t_{w'(n)};t_1,\ldots,t_n)}$, where $\overline{f(t_1,\ldots,t_n)}$ is obtained by applying the substitution $t_j\mapsto t_{n-j+1}$ to $f(t_1,\ldots,t_n)$.
\end{lemma}
\begin{proof}
This lemma is known to experts, but for completeness we give details and references. 
Suppose $X_w$ is a Schubert variety in $\GL_{n}({\mathbb C})/{\sf B}$. 
We have in $K_{\sf T}(\GL_{n}({\mathbb C})/{\sf B})$,
\begin{equation}
\label{eqn:expforloc}
[{\mathcal O}_{X_{v}}][{\mathcal O}_{X_w}]=K_{v,w}^w [{\mathcal O}_{X_w}]+\sum_{\theta\neq w}K_{v,w}^\theta [{\mathcal O}_{X_\theta}].
\end{equation}
It is known that $K_{v,w}^\theta=0$ unless $v\leq \theta$
in Bruhat order; this follows for instance from the equivariant $K$-theory
localization formula of M.~Willems \cite{Willems} together with the \emph{mutatis mutandis} modification of the proof of \cite[Proposition~1]{Knutson.Tao}. 

Now, let $[{\mathcal O}_{X_v}]|_{e_w}$ denote the localization of the class $[{\mathcal O}_{X_v}]$ at the ${\sf T}$-fixed point $e_w := w{\sf B}/{\sf B}$. Localization is a ${\mathbb Z}[t_1^{\pm 1},\ldots, t_n^{\pm 1}]$-module homomorphism from
$K_{\sf T}(\GL_n({\mathbb C})/{\sf B})$ to $K_{\sf T}(e_w)\cong {\mathbb Z}[t_1^{\pm 1},\ldots, t_n^{\pm 1}]$. Applying the localization map to (\ref{eqn:expforloc}) gives
\[[{\mathcal O}_{X_{v}}]|_{e_w}[{\mathcal O}_{X_w}]|_{e_w}=K_{v,w}^w [{\mathcal O}_{X_w}]|_{e_w}.\]
All terms in the summation vanish because $[{\mathcal O}_{X_\pi}]|_{e_\rho}=0$ unless $\rho\leq \pi$ in Bruhat order. This vanishing condition appears
in \cite{Willems} for generalized flag varieties; it also 
follows in the case at hand from, e.g., from the later work \cite[Theorem 4.5]{Woo.Yong:AJM} (see specifically the proof). For similar reasons, $[{\mathcal O}_{X_w}]|_{e_w}\neq 0$. Hence dividing by this shows $K_{v,w}^w=[{\mathcal O}_{X_v}]|_{e_w}$.

Consider the natural projection $\GL_n({\mathbb C})/{\sf B}\twoheadrightarrow 
X$. The pullback of of each Schubert variety in $X$ is a distinct Schubert 
variety in $\GL_n({\mathbb C})/{\sf B}$ (see, e.g., \cite[Example~1.2.3(6)]{Brion}). Thus 
the Schubert basis of $X$ is sent into the Schubert basis of $\GL_n({\mathbb C})/{\sf B}$. Hence we obtain an injection
$K_{\sf T}(X)\hookrightarrow K_{\sf T}(\GL_n({\mathbb C})/{\sf B})$.
 Thus, if $\lambda,\mu\subseteq k\times (n-k)$ and $w,v\in S_n$ are 
respectively their Grassmannian permutations, then 
$K_{\lambda,\mu}^{\lambda}=K_{w,v}^w$.  The lemma now follows
from \cite[Theorem 4.5]{Woo.Yong:AJM} (after chasing conventions).
\end{proof}

Since $v'$ is Grassmannian, by \cite[Theorem~5.8]{KMY}
${\mathfrak G}_{v'}(X;Y)=\sum_{T}{\tt SVSSYTwt}(T)$,
where the sum is over all set-valued semistandard Young tableaux $T$ of shape $\mu'$ with entries bounded above by $n-k$.
Here ${\tt SVSSYTwt}(T)=(-1)^{|L(T)| - |\mu'|} \prod_{\ell \in L(T)}(1 - \frac{x_{\ell}}{y_{\ell+{\rm col}(\x)-{\rm row}(\x)}})$, where $L(T)$ is the set of labels in $T$ and $\x$ is the box containing $\ell$.

Let ${\tt SVSSYTeqwt}(T)$ be the result of the substitution $x_j\mapsto t_{w'(j)}, y_j\mapsto t_j$.
Define $\mathcal{A}$ to be the set of $T\in {\tt BallotGen}(\lambda/\lambda)$
that have content $\mu$. Define $\mathcal{B}$ to be the set of set-valued semistandard tableaux $U$ of shape $\mu'$ where
${\tt SVSSYTeqwt}(U)\neq 0$.

\begin{lemma}
\label{lem:base_bijection}
There is a bijection $\xi:\mathcal{A} \to \mathcal{B}$, with ${\tt wt}(T)=\overline{{\tt SVSSYTeqwt}(\xi(T))}$ for all $T\in \mathcal{A}$.
\end{lemma}
\begin{proof}
  Index columns of $k\times (n-k)$ by $1,2,\dots,n-k$ from \emph{right to left}.
To construct $\xi(T)$, begin with a Young diagram of shape $\mu'$. For each label in $T$, we add a label to $\xi(T)$ as follows: If $i_j$ appears in column $c$ in $T$, place a label $c$ in position $(\mu_i + 1 - j, i)$ in $\xi(T)$.

We have a candidate inverse map $\xi^{-1} : \mathcal{B} \to {\mathcal A}$: For each label $c$ in (matrix) position $(r,i)$ 
in $U \in \mathcal{B}$, we place an $i_{\mu_i + 1 -r}$ at the bottom of column $c$ of $\lambda/\lambda$.

\begin{example}
Let $n=7$, $k=3$, $\lambda=(4,2,1)$ and $\mu=(3,2,0)$. Then $T$, together with the column labels
$1,2, 3, 4$, and $\xi(T)$ are depicted below:
\[\Scale[0.8]{\begin{picture}(240,90)
\ytableausetup{boxsize=2em}
\put(-20,50){$T=$}
\put(20,75){$\ytableaushort{{\none[4]} {\none[3]} {\none[2]} {\none[1]}, {*(lightgray)\blank}{*(lightgray)\blank}{*(lightgray)\blank}{*(lightgray)\blank},{*(lightgray)\blank}{*(lightgray)\blank},{*(lightgray)\blank}}$}
\put(19, -2){$1_1, 2_1$}
\put(44, 23){$1_2, 2_2$}
\put(76, 48){$1_2$}
\put(100, 48){$1_3$}
\put(140,50){$\mapsto$}
\put(160,50){$\xi(T)=$}
\put(200,50){$\ytableaushort{13,{2,3} 4,4}.$}
\end{picture}}
\]

We compute that
$\wt(T)=(-1)^1 \left( 1-\frac{t_1}{t_6}\right) \left(1-\frac{t_3}{t_6}\right) \left(1-\frac{t_5}{t_7}\right) \left(1-\frac{t_6}{t_7}\right) \left(1-\frac{t_1}{t_4}\right) \left(1-\frac{t_3}{t_4}\right)$,
where the first four factors correspond to the labels $1_j$ of $T$ from left to right and
the last two factors correspond to the labels $2_j$ of $T$ from left to right. Now,
\[{\tt SVSSYTwt}(\xi(T))=(-1)^1 \left(1 - \frac{x_4}{y_2}\right) \left(1 - \frac{x_3}{y_2}\right)\left(1 - \frac{x_2}{y_1}\right)\left(1 - \frac{x_1}{y_1}\right)\left(1 - \frac{x_4}{y_4}\right)\left(1 - \frac{x_3}{y_4}\right),\]
where the factors correspond to the entries of $\xi(T)$ as read up columns from left
to right (i.e., consistent with the order of factors of $\wt(T)$ above).

Since $\lambda'=(3,2,1,1)$ we have $w'=2357146$ 
(one-line notation). So substituting, we get
\[{\tt SVSSYTeqwt}(\xi(T))=(-1)^1 \left(1 - \frac{t_7}{t_2}\right) \left(1 - \frac{t_5}{t_2}\right)\left(1 - \frac{t_3}{t_1}\right)\left(1 - \frac{t_2}{t_1}\right)\left(1 - \frac{t_7}{t_4}\right)\left(1 - \frac{t_5}{t_4}\right).\]
The reader can check ${\overline{{\tt SVSSYTeqwt}(\xi(T))}}=\wt(T)$, in agreement
with the lemma.\qed
\end{example}

($\xi^{-1}$ is well-defined and is weight-preserving): Let $U\in \mathcal{B}$.
That $\xi^{-1}(U)$ is an edge-labeled genomic tableau is immediate from
the column strictness of $U$. Ballotness follows from the row increasingness
of $U$.

We now check that no label of $\xi^{-1}(U)$ is too high. Suppose $c$ is a
{\bf bad} label in $U$ in (matrix) position $(r,i)$, i.e., one  
such that the label $i_{\mu_i+1-r}$ placed in column $c$ of $\lambda/\lambda$ is too high. Observe that every label $c'$ North of $c$ and in the same column of $U$ is also bad: this is since $c'$ corresponds to placing another label
of family $i$ in the weakly shorter column $c'$ East of column $c$ (since $c'<c$).  Thus we may assume $c$ is in the northmost row
of $U$, i.e., $r=1$.  Now if $i=1$, then since $c$ is bad, it must be that
$\lambda_{n-k-c+1}'=0$, so $w'(c) = c + 0$. Now $c$ contributes a factor of $1-\frac{x_c}{y_c}$ to 
${\tt SVSSYTwt}(U)$ and hence a factor of $1-\frac{t_{c+0}}{t_{c}}=0$ to 
${\tt SVSSYTeqwt}(U)$. That is, ${\tt SVSSYTeqwt}(U)=0$, so $U\notin {\mathcal B}$, a contradiction. Otherwise, we may also assume $i>1$ is smallest such that a label in $(r=1,i)$ is bad. Since no label $c'$ 
in $(r=1,i-1)$ of $U$ is bad, it must be that $c$ is ``barely'' bad, i.e.,
\begin{equation}
\label{eqn:barelybad}
\lambda_{n-k-c+1}'=i-1
\end{equation} 
(column $c$ is one box too short). However, 
$c$ contributes a factor of $1-\frac{x_c}{y_{c+i-1}}$ to  ${\tt SVSSYTwt}(U)$
and hence a factor of $1-\frac{t_{c+\lambda_{n-k-c+1}'}}{t_{c+i-1}}$ to ${\tt SVSSYTeqwt}(U)$.
This latter factor is $0$ precisely by (\ref{eqn:barelybad}). Hence again
$U\not\in {\mathcal B}$, a contradiction. Thus $U$ has no bad labels and thus
no label of $\xi^{-1}(U)$ is too high, as desired.

The sign appearing in $\wt \xi^{-1}(U)$ records the difference between  $|\mu|$ and the number of labels in $\xi^{-1}(U)$, while the sign in $\overline{{\tt SVSSYTeqwt}(U)}$ records the difference between $|\mu|$ and number of labels in $U$. Since the number of labels in $U$ is clearly the same as the number of labels in $\xi^{-1}(U)$, these signs are equal.

We check that the weight assigned to a label $c$ of $U$ in position $(r,i)$ is the same as the ${\tt edgefactor}$ assigned to the corresponding label $i_{\mu_i + 1 - r}$ at the bottom of column $c$ in $\xi^{-1}(U)$. The label $c$ is assigned the weight
\[
{\tt SSYTeqfactor}_{(r,i)}(c) :=1 - \frac{x_c}{y_{c + i - r}} = 1 - \frac{t_{c + \lambda_{n-k + 1 -c}'}}{t_{c +i- r}}.
\]
Hence we must show the equality of these two quantities:
\begin{align*}
\overline{{\tt SSYTeqfactor}_{(r,i)}(c)} &= 1 - \frac{t_{n + 1 -c - \lambda_{n-k + 1 -c}'}}{t_{n + 1 -c + r - i}} \mbox{ and}\\
{\tt edgefactor}_{\underline{\x}}(i_{\mu_i + 1 - r})&= 1 - \frac{t_{{\tt Man}(\x)}}{t_{\lambda_{n-k+1-c}' - i + r + {\tt Man}(\x)}},
\end{align*}
where $\underline{\x}$ is the southern edge of $\lambda$ in column $c$.

Now, counting the rows and columns separating $\x$ from the
southwest corner of $k\times (n-k)$, we have
\[{\tt Man}(\x)=(n-k-c)+(k-\lambda_{n-k +1 -c}'+1)=n + 1 - c -\lambda_{n-k+1-c}'.\]
Thus, the numerators of the quotients of ${\overline{{\tt SSYTeqfactor}(c)}}$ and ${\tt edgefactor}(c)$
are equal. To see that the denominators are also equal, observe 
\begin{align*}
{\tt Man}(\x) + \lambda_{n-k+1-c}' - i + r &= \left(n + 1 - c -\lambda_{n-k+1-c}'\right) + \lambda_{n-k+1-c}'- i + r \\
&= n + 1 -c-i+r.
\end{align*}

($\xi$ is well-defined and weight-preserving): Let $T\in \mathcal A$.
We must show $\xi(T)$ is strictly increasing along columns. This is clear
since $T$ satisfies (S.3) and (S.4).

Now we show that $\xi(T)$ is weakly increasing along rows. Suppose we have $a$ in position $(r, i)$ and $b$ in position $(r, i+1)$. This $a$ comes from an $i_{\mu_i + 1 - r}$ in column $a$ in $T$, while this $b$ comes from an $(i+1)_{\mu_{i+1} + 1 - r}$ in column $b$. By ballotness of $T$, each $i_{\mu_i + 1 - r}$ must be weakly right of each $(i+1)_{\mu_{i+1} + 1 - r}$. Thus $a \leq b$.

Hence $\xi(T)$ is a set-valued semistandard tableau of shape $\mu'$. The same computations showing $\xi^{-1}$
is weight preserving shows $0\neq {\tt wt}(T)=\overline{{\tt SSYTeqwt}(\xi(T))}$ and so the desired
conclusions hold.
\end{proof}
The proposition now follows immediately from 
Lemmas~\ref{lem:specialization} and~\ref{lem:base_bijection}.
\end{proof}

\section{Good tableaux}\label{sec:good_tableaux}
In this section, we give an intrinsic description of the tableaux that will appear during our generalized {\it jeu de taquin} process (defined in Section~\ref{sec:swaps}).
Since we will use \emph{box} labels $\bullet_\GG$, we distinguish
labels $i_j$ as {\bf genetic labels}. As a visual aid, we
mark genetic labels $\FF$ southeast of a $\bullet_\GG$ with $\FF \prec \GG$ 
as $\FF^{!}$.
For a gene $\GG$, let $\GG^+$ (respectively, $\GG^-$) denote the successor (respectively, predecessor) of $\GG$ in the total order $\prec$ on genes. For example, $1_1^+ = 2_1$ if $\mu_1 = 1$, and $1_1^+ = 1_2$ if $\mu_1 > 1$. Let $\GG_{\rm max}$ be the maximum
gene that can appear, namely $\ell(\mu)_{\mu_{\ell(\mu)}}$ where $\ell(\mu)$ is the number of nonzero rows of $\mu$.
Declare $\GG_{\rm max}^+ := (\ell(\mu)+1)_1$.

A {\bf $\GG$-good tableau} is an edge-labeled filling $T$ of $\nu/\lambda$ by genetic labels $i_j$ (such that $i \in \mathbb{Z}_{>0}$ and the $j$'s that appear for each $i$ form an initial segment of $\mathbb{Z}_{>0}$) and \emph{box} labels $\bullet_\GG$, satisfying the conditions (G.1)--(G.13) below:
\begin{itemize}
\item[(G.1)] no genetic label is too high;
\item[(G.2)] no $\bullet_{\GG}$ is southeast of another (in particular, $\bullet_\GG$'s are in distinct rows and columns);
\item[(G.3)] the labels $\prec$-increase along rows (ignoring any $\bullet_\GG$'s), except for possibly three consecutive labels $\Scale[0.9]{\ytableausetup{boxsize=1.5em}\ytableaushort{\HH  {\bullet_\GG} {\FF^!}}}$ with $\HH > \FF$;
\item[(G.4)] the labels $<$-increase down columns (ignoring any $\bullet_\GG$'s), except that unmarked $\FF$ may 
appear adjacent and above $\FF^!$ when both are box labels; 
\item[(G.5)] if $i_j, k_\ell$ appear on the same edge, then $i\neq k$;
\item[(G.6)] if $i_j$ is West
of $i_k$, then $j\leq k$;
\item[(G.7)] each edge label is maximally west in its gene;
\item[(G.8)] each genotype $G$ obtained by choosing
one label of each gene of $T$ is ballot in the sense defined in Section~\ref{sec:ballot_property}.
\item[(G.9)] if $\FF$ appears northwest of $\bullet_\GG$, then $\FF \prec \GG$;
\item[(G.10)] if $\FF^!\in \x$ or $\FF^!\in\underline{\x}$, then $\bullet_\GG$ appears in $\x$'s row;
\item[(G.11)] $\bullet_{\GG}$ does not appear in a column containing a marked label;
\item[(G.12)] if $\ell$ and $\ell'$ are genetic labels of the same family with $\ell$ NorthWest of $\ell'$, then there are boxes $\x ,\z$ in row $r$ with $\x$ West of $\z$, $\ell \in \x$ or $\overline{\x}$, and $\ell' \in \z$ or $\underline{\z}$; further, $\bullet_{\GG}$ appears in some box $\y$ of $r$ that is East of $\x$ and west of $\z$. Pictorially,
the scenarios are:
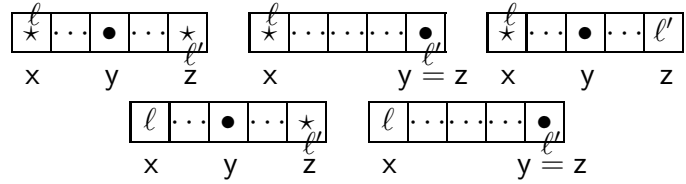
\[
\begin{picture}(90,25)
\ytableausetup{boxsize=1.2em}
\put(0,9){$\ytableaushort{{\star} {\cdots} {\bullet} {\cdots} {\star}}$}
\put(6,19){$\ell$}
\put(65,4){$\ell'$}
\put(5,-3){$\x$}
\put(35,-3){$\y$}
\put(65,-3){$\z$}
\end{picture}
\begin{picture}(90,25)
\ytableausetup{boxsize=1.2em}
\put(0,9){$\ytableaushort{{\star} {\cdots} {\cdots} {\cdots} {\bullet}}$}
\put(6,19){$\ell$}
\put(65,4){$\ell'$}
\put(5,-3){$\x$}
\put(56,-3){$\y=\z$}
\end{picture}
\begin{picture}(90,25)
\ytableausetup{boxsize=1.2em}
\put(0,9){$\ytableaushort{{\star} {\cdots} {\bullet} {\cdots} {\ell'}}$}
\put(6,19){$\ell$}
\put(5,-3){$\x$}
\put(35,-3){$\y$}
\put(65,-3){$\z$}
\end{picture}
\]
\[
\begin{picture}(90,25)
\ytableausetup{boxsize=1.2em}
\put(0,9){$\ytableaushort{{\ell} {\cdots} {\bullet} {\cdots} {\star}}$}
\put(65,4){$\ell'$}
\put(5,-3){$\x$}
\put(35,-3){$\y$}
\put(65,-3){$\z$}
\end{picture}
\begin{picture}(90,25)
\ytableausetup{boxsize=1.2em}
\put(0,9){$\ytableaushort{{\ell} {\cdots} {\cdots} {\cdots} {\bullet}}$}
\put(65,4){$\ell'$}
\put(5,-3){$\x$}
\put(56,-3){$\y=\z$}
\end{picture}
\]
Furthermore, if $\y=\z=\x^\rightarrow$ in the last scenario, then
$\y^\rightarrow$ does not contain a marked label nor another instance of the gene of $\ell'$. 
\end{itemize}

We place a {\bf virtual label} $\circled{\HH}$ on each edge $\underline{\x}$ where $\HH \in \underline{\x}$ would
\begin{itemize}
\item[(V.1)] \emph{not} be marked (hence if $\circled{\HH}$ appears southeast of a $\bullet_\GG$, then $\HH \succeq \GG$);
\item[(V.2)] \emph{not} be maximally west in its gene (hence violating condition (G.7)); and
\item[(V.3)] satisfy the conditions (G.1), (G.4), (G.5), (G.6), (G.8), (G.9) and (G.12).
\end{itemize}

\begin{itemize}
\item[(G.13)] If $\EE^! \in \x$ or $\EE^! \in \underline{\x}$, then there is $\FF$ or $\circled{\FF}$ on $\underline{\x}$ with $N_\EE = N_\FF$ and $\family(\FF) = \family(\EE) + 1$.
\end{itemize}

A tableau is {\bf good} if it is $\GG$-good for some $\GG$.

\begin{example}
The tableau
\ytableausetup{boxsize=1.5em}
$\Scale[0.8]{\begin{picture}(66,40)
\put(0,19){$\ytableaushort{{*(lightgray)\blank} {*(lightgray)\blank} {*(lightgray)\blank}, {2_1} {\bullet_{2_2}} {1_2^!}}$}
\put(5,15){$1_1$}
\put(41,-4){$2_2$}
\end{picture}}
$ is $2_2$-good. Although the labels in the second row do not increase left to right, they satisfy (G.3). Furthermore, notice the $1_1$ and $1_2^!$ satisfy (G.12), as do the $2_1$ and $2_2$.

The tableau \ytableausetup{boxsize=1.5em}
$\Scale[0.8]{\begin{picture}(60,37)
\put(0,19){$\ytableaushort{{*(lightgray)\blank} {1_1} {\bullet_{2_1}}, {\bullet_{2_1}} {1_1^!}}$}
\put(23,-4){$2_1$}
\end{picture}}
$ is also good. Although the label $1_1$ appears twice in the same column, the lower instance is marked in accordance with (G.4). \qed
\end{example}

\begin{example}
The following tableaux are \emph{not} good:
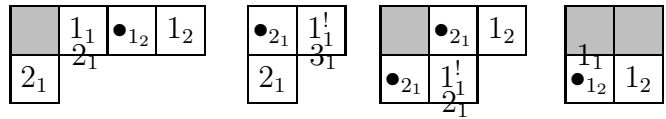
\[\begin{picture}(90,40)
\put(0,19){$\ytableaushort{{*(lightgray)\blank} {1_1} {\bullet_{1_2}} {1_2}, {2_1}}$}
\put(23,15){$2_1$}
\end{picture}
\begin{picture}(50,40)
\put(0,19){$\ytableaushort{{\bullet_{2_1}} {1_1^!}, {2_1}}$}
\put(23,15){$3_1$}
\end{picture}
\begin{picture}(70,40)
\put(0,19){$\ytableaushort{{*(lightgray)\blank} {\bullet_{2_1}} {1_2}, {\bullet_{2_1}} {1_1^!}}$}
\put(23,-3){$2_1$}
\end{picture}
\begin{picture}(55,40)
\put(0,19){$\ytableaushort{{*(lightgray)\blank} {*(lightgray)\blank}, {\bullet_{1_2}} {1_2}}$}
\put(4,15){$1_1$}
\end{picture}
\]
The first fails conditions (G.1) and (G.7) because of the edge label $2_1$. The second fails (G.8), as the  unique genotype is not ballot. Although the marked $1_1^!$ in the third tableau has a label of family $2$ on the lower edge of its box, the tableau fails (G.13) as $1 = N_{1_1} \neq N_{2_1} = 0$. It also fails (G.11) by having both a $\bullet_{2_1}$ and a marked label in the second column. The fourth tableau fails (G.12). 
\qed
\end{example}

\begin{lemma}\label{lem:bundled_tableaux_are_good}
If $T\in {\tt Bundled}(\nu/\lambda)$, then $T$ is $\GG$-good for every $\GG$. Moreover the virtual labels of the $\GG$-good tableau $T$ (as defined by {\normalfont (V.1)}--{\normalfont (V.3)}) are the same as the virtual labels of the bundled tableau $T$ (as defined in Section~\ref{sec:bundled_tableaux}). 
\end{lemma}
\begin{proof}
Since $T$ is bundled, (S.1), (S.2), (S.3) and (S.4) hold. These conditions respectively imply
(G.3), (G.4), (G.5) and (G.6).
(G.1), (G.7) and (G.8) are part of the definition of a bundled tableau. For (G.12), if $\ell$ is NorthWest of $\ell'$ and both are from the same family, (S.1) or (S.2) is violated. The remaining conditions are vacuous since $T$ has no $\bullet_\GG$'s. Hence $T$ is $\GG$-good.

The claim about virtual labels is then clear from the definitions.
\end{proof}

\begin{lemma}[Strong form of (G.10)]\label{lem:strong_form_of_G10}
Assume $T$ is $\GG$-good. Let $\x$ be a box of $T$ in row $r$.
\begin{itemize}
\item[(I)] If $\FF^! \in\underline{\x}$, then $\lab(\x)$ is marked.
\item[(II)] If $\FF^! \in\x$, then there is a $\y$ West of $\x$ in $r$ such that $\bullet_\GG \in \y$.
Every box label of $r$ between $\x$ and $\y$ is marked.
\end{itemize}
\end{lemma}
\begin{proof}
(I): Since $\FF^!\in \underline{\x}$,
 $\underline{\x}$ (and hence also $\x$) is southeast of a $\bullet_\GG$. 
By (G.11), $\bullet_\GG \notin \x$. Hence some $\EE\in \x$. By (G.4), $\EE < \FF$. Therefore the $\EE \in \x$ is marked.

(II): Since $\FF^! \in \x$, there is a $\bullet_\GG$ northwest of $\x$. By (G.10), there is a $\bullet_\GG$ in $\x$'s row. If this latter $\bullet_\GG$ is East of $\x$, these two $\bullet_\GG$'s are distinct and violate (G.2). Hence the $\bullet_\GG$ in $\x$'s row is in some box $\y$ West of $\x$. If $\EE$ is a box label between $\x$ and $\y$ (and in the same row), it is southeast of the $\lab(\y)=\bullet_\GG$.
By (G.3) $\EE \prec \FF$. Hence this $\EE$ is also marked.
\end{proof}

\begin{lemma}[Strong form of (G.13)]\label{lem:strong_form_of_G13}
Let $T$ be $\GG$-good. 
Suppose $\EE^!\in \x$ or $\EE^!\in \underline{\x}$ with $\family(\GG) - \family(\EE) = k > 0$. For each $0 < h < k$, 
there is $\HH^! \in \underline{\x}$ with $N_\HH = N_\EE$ and $\family(\HH) = \family(\EE) + h$. 
Also, there is a $\GG'$ or $\circled{\GG'} \in \underline{\x}$ with $N_{\GG'} = N_\EE$ and $\family(\GG') = \family(\GG)$.
\end{lemma}
\begin{proof} This follows by repeated application of (G.13). Note that none of the $\HH$'s of the statement can be virtual
since they must be marked.
\end{proof}

\begin{lemma}
\label{lem:how_to_check_ballotness}
If $\EE < \FF$ are genes of a good tableau $T$ with $N_\EE = N_\FF$, then no $\FF$ or $\circled{\FF}$ is East of any $\EE$.
\end{lemma}
\begin{proof}
First suppose that some $\FF$ is East of some $\EE$. Let $G$ be a genotype of $T$ with $\FF\in G $ that is East of some $\EE\in G$. 
Then $\FF$ appears before $\EE$ in ${\tt word}(G)$. By (G.6), the initial segment $W$ of ${\tt word}(G)$ ending at $\FF$ 
contains $N_\FF+1$ labels of $\family(\FF)$ and at most $N_\EE$ labels of $\family(\EE)$. Thus $T$'s (G.8) is violated for some
$\family(\EE) \leq i < \family(\FF)$, a contradiction. Finally, 
if some $\circled{\FF}$ is East of some $\EE$, then by (V.3)
the tableau $T'$ obtained by replacing that $\circled{\FF}$ by $\FF$ satisfies (G.6) and (G.8). Now we derive the same contradiction
as before, using $T'$ in place of $T$.
\end{proof}

\begin{lemma}
\label{lemma:markedHiswestmost}
If $\EE^!$ appears in a good tableau $T$, then it is maximally west in its gene.
\end{lemma}
\begin{proof}
Suppose $\EE^!\in \x$ or $\EE^!\in \underline{\x}$.
By (G.13), there is an $\FF$ or $\circled{\FF}\in\underline{\x}$ with $N_\EE=N_\FF$ and $\EE<\FF$. Thus we are done
by Lemma~\ref{lem:how_to_check_ballotness}.
\end{proof}

\begin{lemma}
\label{lem:draggable_labels_in_complete_sets}
Suppose column $c$ of good tableau $T$ contains labels $\HH$ and ${\mathcal J}$ with $\HH < {\mathcal J}$ and $N_\HH = N_{\mathcal J}$. Then 
for every $i$ such that $\family(\HH)<i<\family({\mathcal J})$,
there is a label ${\mathcal I}$ of family $i$ in column $c$ such that
$N_\HH = N_{\mathcal I}$.
\end{lemma}
\begin{proof}
Suppose not. By (G.8), there is some ${\mathcal I}\in T$ of family $i$  
such that
$N_\HH = N_{\mathcal{J}}=N_{\mathcal I}$. If this ${\mathcal I}$ is not in column $c$, we contradict
Lemma~\ref{lem:how_to_check_ballotness}.
\end{proof}

\begin{lemma}
\label{lem:virtualG13label}
Suppose $\EE$ and $\FF$ satisfy $N_\EE=N_\FF$ and $\family(\FF)=\family(\EE)+1$.
Let $T$ be a $\GG$-good tableau with $\circled{\FF}\in {\underline{\x}}$ and 
either $\EE^!\in \x$ or $\EE^!\in \underline{\x}$. Then $\bullet_{\GG}\in \x^\leftarrow$ and $\family(\FF)=\family(\GG)$. 
\end{lemma}
\begin{proof}
If $\bullet_{\GG}\not\in \x^\leftarrow$, then by Lemma~\ref{lem:strong_form_of_G10}, ${\mathcal D}^!\in \x^\leftarrow$. By 
(G.3) and (G.4), ${\mathcal D}\prec \EE$. Also $\EE\prec \GG$ since $\EE^!\in T$. Thus by (G.6) and Lemma~\ref{lem:strong_form_of_G13}, there is a
$\widetilde{\EE}^!\in \underline{\x^\leftarrow}$ or $\widetilde{\EE}^!\in \x^\leftarrow$ with $\family(\EE)=\family(\widetilde{\EE})$ and
$N_{\widetilde{\EE}}=N_{{\mathcal D}}$. By (G.13), there is ${\widetilde \FF}$
or $\circled{\widetilde{\FF}} \in \underline{\x^\leftarrow}$ with 
$\family(\FF)=\family(\widetilde{\FF})$ and
$N_{\widetilde{\EE}}=N_{\widetilde{\FF}}$. Thus,
by Lemma~\ref{lem:how_to_check_ballotness}, $\FF\neq {\widetilde{\FF}}$, contradicting $\circled{\FF}\in \underline{\x}$. Finally, $\family(\FF)=\family(\GG)$ by Lemma~\ref{lem:strong_form_of_G13}.
\end{proof}

\begin{lemma}
\label{lemma:Gsoutheast}
If $T$ is $\GG$-good, then no $\HH$ is southEast of another.
\end{lemma}
\begin{proof}
If some $\HH$ is SouthEast of another $\HH$,  
by (G.12) there is a $\bullet_{\GG}$ in between the two $\HH$'s. 
If two $\HH$'s are box labels of the same row, then by (G.3) we reach the same conclusion that there is a $\bullet_{\GG}$ in between the two $\HH$'s.
In either case, since this $\bullet_{\GG}$ is southeast
of the western $\HH$ we have $\HH\prec \GG$ by (G.9). Since this $\bullet_{\GG}$ is northwest of the eastern $\HH$, this
eastern $\HH$ is marked. This contradicts Lemma~\ref{lemma:markedHiswestmost}.
Finally, suppose two $\HH$'s are edge labels on the bottom of the same row. This contradicts (G.7). 
\end{proof}

\begin{lemma}
\label{lemma:Gsoutheastofbullet}
Let $T$ be a $\GG$-good tableau. Suppose $\family(\FF)\leq \family(\GG)$, $\bullet_{\GG} \in \y$ and $\FF \in \z$ or $\underline{\z}$. Then $\z$ is not SouthEast of $\y$.
\end{lemma}
\begin{proof}
Suppose $\z$ is SouthEast of $\y$.
First assume $\FF < \GG$. Consider the box $\aaa$ that is in $\y$'s column and $\z$'s row. By Lemma~\ref{lem:strong_form_of_G10}, either $\aaa$ contains a marked label (contradicting (G.11)) or $\bullet_\GG\in \aaa$ Southeast of $\y$ (contradicting (G.2)).

Now assume $\family(\FF) = \family(\GG)$. (We do not assume $\FF \preceq \GG$.) Consider the box $\bbb$ of $T$ that is in $\y$'s row and $\z$'s column. By (G.2), $\bbb$ contains a genetic label. By (G.4), $\lab(\bbb) < \FF$. Hence $\lab(\bbb)$ is marked in $T$. By Lemma~\ref{lem:strong_form_of_G13}, $\underline{\bbb}$ then contains a (possibly virtual) label of the same family as $\FF$ and $\GG$. This contradicts (G.4).
\end{proof}

\begin{lemma}
\label{lem:GandbulletG+}
Let $U$ be a $\GG^+$-good tableau. Suppose that $\bullet_{\GG^+} \in \x$ and that either  $\GG \in \y$ or $\GG \in \underline{\y}$. Then $\y$ is not NorthWest of $\x$.
\end{lemma}
\begin{proof}
Suppose otherwise. Consider the box $\bbb$ that is in $\y$'s column and $\x$'s row. By (G.2) it contains a genetic label. By (G.4) either $\lab(\bbb) > \GG$ or else $\GG^! \in \bbb$. If $\GG^! \in \bbb$, then $\bbb$ is southeast of a $\bullet_{\GG^+}$ by definition. This contradicts (G.2). If $\GG < \lab(\bbb)$, we contradict (G.9). 
\end{proof}

\begin{lemma}
\label{lem:same999}
Let $c$ be a column of a $\GG$-good tableau $T$. Suppose $\bullet_\GG \in c$ and either $\GG \in c$ or $\circled{\GG} \in c$. Further suppose that $\EE^! \in \y$, where $\y$ is a box of column $c^\rightarrow$.
Then $\circled{\GG} \in \underline{\y}$. 
\end{lemma}
\begin{proof}
Since $\EE^!$ appears in $T$, $\EE \prec \GG$. Since $\EE$ appears East of some $\GG$, by (G.6) this implies $\EE < \GG$.

Hence by Lemma~\ref{lem:strong_form_of_G13}, there is either $\GG' \in \underline{\y}$ or $\circled{\GG'} \in \underline{\y}$ with $\family(\GG') = \family(\GG)$. It remains to show $\GG' = \GG$, for then by (G.7), $\circled{\GG} \in \underline{\y}$.

Suppose $\GG' \neq \GG$. Then by (G.4), (G.5) and (G.6), $\GG' = \GG^+$. By Lemma~\ref{lem:strong_form_of_G13}, $N_\EE = N_{\GG^+}$; thus $\family(\EE^-) = \family(\EE)$ by (G.8). Also by (G.8), every instance of $\EE^-$ must be read before any $\GG$ or $\circled{\GG}$. By (G.4), $\EE^- \notin c^\rightarrow$. By (G.6), $\EE^-$ does not appear East of $c^\rightarrow$. But by assumption either $\GG \in c$ or $\circled{\GG} \in c$, so $\EE^-$ must appear in $c$.

Consider the box $\y^\leftarrow$. By Lemma~\ref{lem:strong_form_of_G10}, either $\bullet_\GG \in \y^\leftarrow$ or some $\DD^! \in \y^\leftarrow$. The latter is impossible by (G.11), since $\bullet_\GG \in c$. Hence $\bullet_\GG \in \y^\leftarrow$. 

Now $\EE^-$ cannot appear South of $\y^\leftarrow$ in $c$, for then it would be marked, in violation of (G.11). We have $\EE^- \notin \y^\leftarrow$, since $\bullet_\GG \in \y^\leftarrow$. By (G.12), $\EE^-$ cannot appear North of $\y^\leftarrow$ in $c$. This contradicts that $\EE^-$ 
must appear in $c$, and therefore the assumption $\GG'\neq \GG$.
\end{proof}

\section{Snakes of good tableaux}\label{sec:snakes}
In this section, we give structural results about certain subsets of a good tableau; these will play a critical role in the definition of our generalized \textit{jeu de taquin} (given in Section~\ref{sec:swaps}).
\subsection{Snakes} Let $T$ be a $\GG$-good tableau. Let $\GG = g_k$ and consider the set 
of \emph{boxes} in $T$ that contain either
$\bullet_\GG$ or $\GG$. This set decomposes into edge-connected components $R$ that we call {\bf presnakes}.
A {\bf short ribbon} is a connected skew shape without a
$2\times 2$ subshape and where each row and column contains at most two
boxes. 
\begin{lemma}
\label{lem:presnake_description}
Each presnake $R$ is a short ribbon. Any row of $R$ with two boxes is \ytableausetup{boxsize=1.2em} 
$\ytableaushort{{\bullet_{\GG}} {\GG}}$. Any column of $R$ with two boxes is 
$\ytableaushort{{\bullet_{\GG}}, {\GG}}$.
\end{lemma}
\begin{proof}
Since $T$ is $\GG$-good, there is no $\GG^!$. So any column of $R$ has at most one $\GG$ by (G.4) and at most one
$\bullet_{\GG}$ by (G.2). Hence any column of $R$ has at most two boxes. By (G.9) if $\bullet_{\GG}$ and $\GG$ are in the same
column, the $\bullet_{\GG}$ is to the north. The description of rows of $R$ holds by (G.2), (G.3) and (G.9).
That $R$ is a skew shape with no $2\times 2$ subshape then follows immediately.
\end{proof}

A {\bf snake} $S$ is a presnake $R$ extended by (R.1)--(R.3):
\begin{itemize}
\item[(R.1)] If the box immediately right of the northmost
$\bullet_\GG$ in $R$ contains $\GG^+$ with $\family(\GG^+) = \family(\GG)$, then adjoin this box to $R$.
\item[(R.2)] If the box immediately left of the southmost $\GG$ in $R$
contains a marked label, adjoin this box to $R$.
\item[(R.3)] If $\x$ in the northmost row of $R$ contains $\bullet_\GG$,
$\lab(\x^\rightarrow)$ is marked and either $\GG$ or $\circled{\GG}
\in \underline{\x^\rightarrow}$, then adjoin $\x^\rightarrow$ to $R$.
\end{itemize}
 \emph{The entries of
$S$ are its box labels and labels appearing on the bottom edges of its boxes}.

\begin{example}
\ytableausetup{boxsize=1.5em}
Below are snakes for $\GG = 2_2$: 
\[\Scale[1]{$\begin{picture}(80,34)
\put(0,19){$\ytableaushort{\none {\bullet_{2_2}} {2_3}, {\bullet_{2_2}} {2_2}}$,}
\put(6,-4){$2_2$}
\end{picture}$}
\Scale[1]{$\begin{picture}(30,30)
\put(0,19){$\ytableaushort{{2_2}}$,}
\end{picture}$}
\Scale[1]{$\begin{picture}(80,30)
\put(0,19){$\ytableaushort{\none {\bullet_{2_2}} {2_2}, {2_1^!} {2_2}}$,}
\put(5,-5){$3_1$}
\put(24,-5){$3_2$}
\end{picture}$}
\Scale[1]{$\begin{picture}(80,30)
\put(0,19){$\ytableaushort{{\bullet_{2_2}} {1_3^!}}$.}
\put(21,16){$2_2$}
\end{picture}$}\]

On the other hand,
\Scale[.8]{$\begin{picture}(58,30)
\put(0,19){$\ytableaushort{\none {\bullet_{2_2}} {3_1}, {\bullet_{2_2}} {2_2}}$}
\put(6,-4){$2_2$}
\end{picture}$} is \emph{not} a snake, even if $3_1 = 2_2^+$ ((R.1) does not apply).
\qed
\end{example}

\begin{example}[Snakes can share a row]
\label{exa:cansharerow}
$\Scale[0.8]{\ytableausetup{boxsize=1.6em}
\begin{picture}(75,37)
\put(0,20){$\ytableaushort{{*(lightgray)\blank} {*(lightgray)\blank} {*(lightgray)\blank} {1_2}, {*(Dandelion) \bullet_{2_2}} {*(SkyBlue) 1_1^!} {*(SkyBlue) 2_2}}$}
\put(22,-3){$\Scale[.7]{2_1^!, 3_1}$}
\put(44,-3){$3_2$}
\end{picture}}$ 
contains two snakes as colored.\qed
\end{example}

\begin{example}[Snakes can share a column]
\label{exa:cansharecolumn}
$\ytableausetup{boxsize=1.3em}\Scale[0.8]{\ytableaushort{{*(lightgray)\blank} 
{*(SkyBlue) \bullet_{1_1}} {*(SkyBlue) 1_2}, {*(Dandelion) \bullet_{1_1}} {*(Dandelion) 1_2}, {*(Dandelion) 1_1}}}$ has two snakes as colored. \qed
\end{example}

\begin{lemma}\label{lem:snakes_are_short_ribbons}
Every snake $S$ is a short ribbon.
\end{lemma}
\begin{proof}
$S$ is built by adjoining boxes to a presnake $R$.
By Lemma~\ref{lem:presnake_description}, $R$ is a short ribbon.
In view of Lemma~\ref{lem:presnake_description}, (R.1) and (R.3) only apply if the northmost row of $R$ is a single box with $\bullet_\GG$. 
So adjoining a box to the right maintains shortness.
Similarly, (R.2) maintains shortness.
\end{proof}

\begin{lemma}[Disjointness and relative positioning of snakes]\label{lem:snakes_arranged_SW-NE}
Suppose $S, S'$ are snakes obtained from distinct
presnakes $R, R'$ respectively. Up to relabeling of the snakes, one of the following holds:
\begin{itemize}
\item[(I)]  $S$ is entirely SouthWest of the $S'$ (that is, if $\bbb, \bbb'$ are
respectively boxes of these snakes, then $\bbb$ is SouthWest of $\bbb'$).
\item[(II)] $S$ consists of a single box containing $\bullet_\GG$ with neither $\GG$ nor $\circled{\GG}$ on its lower edge; further, this box appears West of and in the same row as the southmost row of $S'$, and all intervening box labels are marked; cf.~Example~\ref{exa:cansharerow}.
\item[(III)] $S$ involves an (R.1) extension, adjoining a $\GG^+$ in some box $\w$, while $S'=\{\bullet_\GG\in \w^\uparrow\}$
or $S'=\{\bullet_\GG \in \w^\uparrow, \GG^+\in \w^{\uparrow\rightarrow}\}$; cf.~ Example~\ref{exa:cansharecolumn}.
\end{itemize}
In particular, $S$ and $S'$ are box disjoint.
\end{lemma}
\begin{proof} 
By Lemma~\ref{lem:presnake_description}, (G.2) and/or (G.4), $R$ and $R'$ share at most one row and do not share
a column. Moreover, one sees that $R$ is southWest of $R'$ (say). By (R.1)--(R.3), $S$ and $S'$ share a row if and only if 
$R$ and $R'$ do. 

\noindent
{\sf Case 1: ($R$ and $R'$ share a row $r$):} The northmost row of $R$ and the southmost row of $R'$ are in row $r$. We must show that (II) holds and that $S, S'$ are box disjoint. 

By (G.2), (G.9) and Lemma~\ref{lemma:Gsoutheast}, $R$ has in row $r$ only a $\bullet_\GG\in \x$ while $R'$ has in $r$ only $\GG\in \y$. Since $S\neq S'$, $\y\neq \x^\rightarrow$. 
By (G.3), $\lab(\y^\leftarrow) \prec \GG$, so we have some marked label $\FF^!\in \y^\leftarrow$. Therefore $R'$ extends to $S'$ by (R.2). 

\begin{claim}\label{claim:snake_claim}
No $\GG$ or $\circled{\GG}$ appears in columns west of $\y^\leftarrow$.
\end{claim}
\begin{proof}
Since $\FF\prec\GG$, we are done by (G.4) and (G.6)
if $\family(\FF) = \family(\GG)$. 
Thus assume
$\FF < \GG$. By Lemma~\ref{lem:strong_form_of_G13}, there is either $\GG' \in \underline{\y^\leftarrow}$ or $\circled{\GG'} \in \underline{\y^\leftarrow}$
such that $\family(\GG')=\family(\GG)$ and $N_\FF=N_{\GG'}$. By (G.6), $\GG' \preceq \GG$ 
because $\GG\in \y$. If $\GG' = \GG$, then since $N_\FF=N_{\GG'(=\GG)}$, the $\GG\in \y$ and $\FF^!\in \y^\leftarrow$ combine to contradict Lemma~\ref{lem:how_to_check_ballotness}. Thus $\GG'\prec \GG$ and we
are done by (G.6) and (G.4).
\end{proof}

By Claim~\ref{claim:snake_claim}, $R = \{\bullet_{\GG}\in \x\}$ without
$\GG$ or $\circled{\GG} \in \underline{\x}$. Observe that $R$ cannot extend to $S$ by (R.1), since (R.1) requires $\GG^+\in \x^\rightarrow$, which contradicts (G.3) in view of $\GG\in \y$. It cannot be extended by (R.2) since $\GG\not\in \x$. If $R$ 
were extended by (R.3), there would be a $\GG$ or $\circled{\GG}$ in $\underline{\x^\rightarrow}$ in violation of Claim~\ref{claim:snake_claim}. Thus $R = S = \{ \x\}$.
 
By Lemma~\ref{lem:strong_form_of_G10}(II), all labels
strictly between $\x$ and $\y$ are marked. Hence (II) holds. Since $\y^\leftarrow\not\in S$, we
see by (R.1)--(R.3) that $S$ and $S'$ are box disjoint.

\noindent
{\sf Case 2: ($R$ and $R'$ do \emph{not} share a row)}: We may assume $S$ and $S'$ share a column, for if they do not, then clearly (I) and box-disjointness both hold. Since $R$ and $R'$ do not share a column, $S$ and $S'$ can only 
share a column if $R$ is extended East by (R.1) or (R.3) or if $R'$ is
extended West by (R.2). Let $\x$ be the northeastmost box of $R$ and 
$\y$ be the southwestmost box of $R'$.

\noindent
{\sf Subcase 2.1: ($R$ is extended by {\normalfont (R.1)}):} Since 
$\lab(\x^\rightarrow)=\GG^+$ and $\family(\GG^+) = \family(\GG)$, by (G.6) $R'$ cannot contain any $\GG$'s 
and therefore $R'=\{\bullet_\GG \in \y\}$. Hence (R.2)
does not extend $R'$. 
By assumption, $\x^\rightarrow$ and $\y$ are in the
same column. Hence by (G.4) and (G.11), $\y = \x^{\rightarrow\uparrow}$. 
By (G.6), $R'$ is not extended by (R.3), since $\GG^+ \in \x^\rightarrow$ and (R.3) requires $\GG \in \underline{\y^\rightarrow}$ or $\circled{\GG} \in \underline{\y^\rightarrow}$. If $R'$ is extended by 
(R.1), we obtain the second scenario described by (III) (and $S$, $S'$ are box disjoint).
If $R'$ is not extended by any of (R.1)--(R.3), then we have the first scenario described by (III) (and $S$, $S'$ are box disjoint).   

\noindent
{\sf Subcase 2.2: ($R$ is extended by {\normalfont (R.3)}):} Let $c$ be $\x^\rightarrow$'s column. We have $\FF^!\in \x^\rightarrow$ and either $\GG \in \underline{\x^{\rightarrow}}$ or $\circled{\GG} \in \underline{\x^{\rightarrow}}$. Moreover $N_\FF = N_\GG$. Hence by Lemma~\ref{lem:how_to_check_ballotness}, no $\GG$ appears East of $c$. Thus $R'=\{\bullet_\GG \in \y\}$. By (G.11), $\y \notin c$.
Thus $S$ and $R'$ do not share a column. Since $\bullet_\GG \in \y$, $R'$ is not extended by
(R.2). Thus $S$ and $S'$ do not share a column.

\noindent
{\sf Subcase 2.3: ($R'$ is extended by {\normalfont (R.2)}; $R$ is not extended by either {\normalfont (R.1)} or {\normalfont (R.3)}):} Here $\GG\in \y$ and $\FF^!\in \y^\leftarrow$. By Lemma~\ref{lem:strong_form_of_G13}, either $\family(\FF) = \family(\GG)$ or else we have $\GG' \in \underline{\y^\leftarrow}$ or $\circled{\GG'} \in \underline{\y^\leftarrow}$ such that $\family(\GG')=\family(\GG)$. Hence by (G.4) and (G.11),
$R$ cannot contain a box in the column of $\y^\leftarrow$. Hence $R, S'$ do not share a column. Hence by the assumption of the subcase, $S$ and $S'$ do not share a column.
\end{proof}

\subsection{Snake sections}
We decompose each snake $S$ into three {\bf snake sections}
denoted ${\tt head}(S)$, ${\tt body}(S)$ and ${\tt tail}(S)$
as follows:

\begin{definition-lemma}\label{def-lem:snake_classification}
\gap
\begin{itemize}
\item[(I)] If a snake $S$ has at least two rows and its southmost row has two boxes, then $\head(S)$ is the southmost row of $S$, $\tail(S)$ is the northmost row and $\body(S)$ is the remaining rows.

\item[(II)] If a snake $S$ has at least two rows and its southmost row has exactly one box, then $\head(S)$ is empty, $\tail(S)$ is the northmost row and $\body(S)$ is the other rows.

\item[(III)] If $S$ has exactly one row, then $S$ is one of the following (edge labels not depicted):
\ytableausetup{boxsize=1.2em}
\[\mbox{{\rm{(i)}} $S=\ytableaushort{\GG}=\body(S)$; \ \
{\rm{(ii)}} $S=\ytableaushort{{\bullet_\GG}}=\head(S)$; \ \ {\rm{(iii)}} $S=\ytableaushort{{\bullet_\GG} \GG}=\head(S)$;}\]
\[\mbox{{\rm{(iv)}}
$R=\ytableaushort{{\bullet_\GG} {\GG^+}}=\head(S)$; \ \
{\rm{(v)}} $S=\ytableaushort{{\FF^!} \GG}=\head(S)$;}\]
\[\mbox{{\rm{(vi)}} $S=\ytableaushort{{\bullet_\GG} {\FF^!}}=\tail(S)$ (with 
$\GG$ or $\circled{\GG}$ on the lower right edge).}\]
\end{itemize}
\end{definition-lemma}
\begin{proof}
It is only required to verify that in (III) all possible one-row snakes are shown. This is done by
combining Lemma~\ref{lem:presnake_description} and (R.1)--(R.3).
\end{proof}

\begin{lemma}[Properties of $\head,\body,\tail$]
\label{lem:piecesobservations}
\gap
\begin{itemize}
\item[(I)] If $\head(S) = \{\x\}$, then $\bullet_{\GG}\in \x$.
\item[(II)] If $\head(S)=\{\x,\x^\rightarrow\}$, then
$\ytableausetup{boxsize=1.2em} \head(S)=\ytableaushort{{\FF^!} \GG}$,
$\ytableaushort{{\bullet_{\GG}} \GG}$ or
$\ytableaushort{{\bullet_{\GG}} {\GG^+}}$.
\item[(III)] $\body(S)$ is a short ribbon consisting only of $\bullet_{\GG}$'s and $\GG$'s (with no edge label $\GG$'s or $\circled{\GG}$'s).
\item[(IV)] If $\tail(S) = \{\x\}$, then  $\tail(S)= \ytableaushort{{\bullet_{\GG}}}$ and $S$ has at least two rows.
\item[(V)] If $\tail(S)=\{\x,\x^\rightarrow\}=\ytableaushort{{\bullet_{\GG}} \GG}$ or $\ytableaushort{{\bullet_{\GG}} {\GG^+}}$, then $S$ has at least two rows,
$\GG \notin \underline{\x}$ and $\circled{\GG} \notin \underline{\x}$.
\item[(VI)] If $\tail(S)=\{\x, \x^\rightarrow\}$ and $\GG$ or $\circled{\GG} \in \underline{\x^\rightarrow}$, then
$\tail(S)= 
\begin{picture}(31,17)
\put(0,2){$\ytableaushort{{\bullet_{\GG}} {\FF^!}}$}
\put(18,-2){$\GG$} 
\end{picture}$
or
$\begin{picture}(45,17)
\put(0,2){$\ytableaushort{{\bullet_{\GG}} {\FF^!}}$.}
\put(17,-2){$\Scale[0.7]{\circled{\GG}}$} 
\end{picture}$
\item[(VII)] If $S$ has at least two rows, then $\GG \in \x^\downarrow$ where $\x$ is the westmost box of $\tail(S)$.
\end{itemize}
\end{lemma}
\begin{proof}
If $S$ has one row, then by Definition-Lemma~\ref{def-lem:snake_classification}(III) these claims are clear (or irrelevant).
Thus assume $S$ has at least two rows.

\noindent
(I): Under the assumption that $S$ has at least two rows, the claim is vacuous since
by Definition-Lemma~\ref{def-lem:snake_classification}(I,II) we know $|\head(S)|\neq 1$.

\noindent
(II): Either the southmost row of $S$ is
$\ytableaushort{{\FF^!} \GG}$ if (R.2) was used, or it is
$\ytableaushort{{\bullet} \GG}$ if (R.2) was not used; cf. Lemma~\ref{lem:presnake_description}.

\noindent
(III): That $\body(S)$ is a short ribbon is clear, since $S$ is a short ribbon by Lemma~\ref{lem:snakes_are_short_ribbons}. 
Boxes of $\body(S)$ only contain $\GG$ or $\bullet_{\GG}$ because (R.1)--(R.3) adjoin boxes only to the northmost
or southmost row (and if the southmost row of $S$ has two boxes, then by definition that row is not
part of $\body(S)$). By (G.12), an edge label $\GG$ or ${\circled\GG}$ can only appear in the northmost or southmost row of $S$.
In those cases, the row is not part of $\body(S)$ by Definition-Lemma~\ref{def-lem:snake_classification}(I,II).

\noindent
(IV): $\tail(S)$ is the northmost
row of $S$ and, since $|\tail(S)|=1$, it is the northmost row of the presnake of $S$. Thus we are done by Lemma~\ref{lem:presnake_description}.

\noindent
(V): $\tail(S)$ is the northmost row and by Lemma~\ref{lem:presnake_description}, $\GG\in \x^\downarrow$ ($\x^\downarrow$
is in the presnake of $S$) so $\GG,\circled{\GG}\not\in \underline{\x}$ by (G.4).

\noindent
(VI): $\x$ is in the presnake of $S$ and so by Lemma~\ref{lem:presnake_description}, $\bullet_{\GG}\in \x$. 
By (G.2), $\bullet_{\GG}\not\in \x^\rightarrow$.
By
(G.4), $\lab(\x^\rightarrow)<\GG$ and so $\lab(\x^\rightarrow)$ is marked, since it is southeast of the $\bullet_{\GG}\in \x$.

\noindent
(VII): $\x$ and $\x^\downarrow$ are part of the presnake of $S$. Now apply Lemma~\ref{lem:presnake_description}.
\end{proof}

\section{Genomic jeu de taquin}\label{sec:swaps}
\subsection{Miniswaps}
We first define {\bf miniswaps} on snake sections of a $\GG$-good tableau. The output is a formal sum of tableaux.
Below, interpret $\bullet=\bullet_\GG$ before the miniswap and $\bullet=\bullet_{\GG^+}$ after the miniswap. We depict $\circled{\GG}$ whenever it exists. Labels and virtual labels from other genes 
are not depicted unless relevant to the miniswap's definition. For a box $\x$, define
\[\beta(\x):=1-\frac{t_{{\sf Man}(\x)}}{t_{{\sf Man}(\x)+1}}
\text{ \ \ \ and  \ \ \ }
\hat{\beta}(\x) :=  1 - \beta(\x) =\frac{t_{{\sf Man}(\x)}}{t_{{\sf Man}(\x)+1}}.\] Note that if $\x = \alpha /\beta$, then $\hat{\beta}(\x) = \wt \alpha/\beta$, as defined in Section~\ref{sec:structure_of_proof}.
If a snake section is empty, then $\m$ acts trivially, so below
we assume otherwise.

\subsubsection{Miniswaps on $\head(S)$}
\gap

\noindent
\ytableausetup{boxsize=1em}
{\sf (Case H1: $\head(S)=\{\x\}$ and $\GG \in \underline{\x}$):} 
\[\begin{picture}(350,12)
\put(0,0){$\head(S) = \ytableaushort{\bullet}\mapsto \m(\head(S)) = \mbox{$\beta(\x)$} \cdot \ytableaushort{\GG} + \gamma \cdot \ytableaushort{\bullet}$}
\put(60,-5){$\GG$}
\put(266,10){$\GG$}
\end{picture}
\]
Set $\gamma := 0$ if ${\rm row}(\x) = \family(\GG)$ (that is, if $\GG \in \overline{\x}$ would be too high); otherwise set $\gamma := 1$.

\noindent
\ytableausetup{boxsize=1em}
{\sf (Case H2: $\head(S)=\{\x\}$ and $\circled{\GG} \in \underline{\x}$):} 
\[\begin{picture}(350,18)
\put(0,0){$\head(S)=\ytableaushort{{\bullet}} \mapsto
\m(\head(S)) = \ytableaushort{{\bullet}} + \beta(\x) \cdot \ytableaushort{\GG}$}
\put(59,-4){\Scale[.7]{\circled{\GG}}}
\end{picture}
\]

\noindent
\ytableausetup{boxsize=1em}
{\sf (Case H3: $\head(S)=\{\x\}$ and Cases H1/H2 do not apply):} 
\[\begin{picture}(350,18)
\put(0,0){$\head(S)=\ytableaushort{\bullet} \mapsto  \m(\head(S)) = \ytableaushort{\bullet}$}
\end{picture}
\]

\noindent
\ytableausetup{boxsize=1em}
{\sf (Case H4: $\head(S)=\{\x,\x^\rightarrow\}$, $\GG \in \x^\rightarrow$, and $\GG \in \underline{\sf x}$):} 
\[\begin{picture}(350,14)
\put(0,0){$\head(S) = \ytableaushort{\bullet \GG} \mapsto \m(\head(S)) = 0$}
\put(60,-4){$\GG$}
\end{picture}
\]

\noindent
{\sf (Case H5: $\head(S)=\{\x,\x^\rightarrow\}$, $\GG \in \x^\rightarrow$, and $\GG\not\in {\underline{{\sf x}}}$):}

\ytableausetup{boxsize=1em}
{\sf (Subcase {\sf  H5.1:} $\HH\in \underline{{\sf x}^\rightarrow}$, $\family(\mathcal{H}) = \family(\GG) + 1$ and $N_\mathcal{H}  = N_\GG$):}
\[\begin{picture}(350,14)
\put(0,0){$\head(S) = \ytableaushort{\bullet \GG} \mapsto \m(\head(S)) = \ytableaushort{\bullet {\GG^!}}$}
\put(71,-6){$\HH$}
\put(215,-6){$\HH$}
\end{picture}
\]

\ytableausetup{boxsize=1em}
{\sf (Subcase {\sf H5.2}: $\circled{\HH}\in \underline{{\sf x}^\rightarrow}$, $\family(\mathcal{H}) = \family(\GG) + 1$ and $N_\mathcal{H}  = N_\GG$):}
\[\begin{picture}(350,18)
\put(0,0){$\head(S) = \ytableaushort{\bullet \GG} \mapsto \m(\head(S)) = \ytableaushort{\bullet {\GG^!}} + \hat{\beta}(\x) \cdot \ytableaushort{\GG \bullet}$}
\put(71,-6){\Scale[.7]{$\circled{\HH}$}}
\put(215,-6){\Scale[.7]{$\circled{\HH}$}}
\end{picture}
\]

\ytableausetup{boxsize=1em}
{\sf (Subcase {\sf H5.3}: Subcases {\sf H5.1}/{\sf H5.2} do not apply):}
\[\begin{picture}(350,14)
\put(0,0){$\head(S) = \ytableaushort{\bullet \GG} \text{ or }  \ytableaushort{\bullet \GG} \mapsto \m(\head(S)) = \hat{\beta}(\x) \cdot \ytableaushort{\GG \bullet}$}
\put(59,-5){$\Scale[0.7]{\circled{\GG}}$}
\end{picture}
\]

\noindent
\ytableausetup{boxsize=1.3em}
{\sf (Case H6: $\head(S)=\{\x,\x^\rightarrow\}$, $\GG^+ \in \x^\rightarrow$, and $\GG \in \underline{\x}$):} 
\[\begin{picture}(350,14)
\put(0,0){$\head(S) = \ytableaushort{\bullet {\GG^+}} \mapsto \m(\head(S)) = \mbox{$\beta(\x)$} \cdot\ytableaushort{\GG {\GG^+}} + \alpha \cdot \ytableaushort{\GG \bullet}$}
\put(62,-5){$\GG$}
\put(321,-5){$\GG^+$}
\end{picture}
\]
Set $\alpha:=0$ if the second tableau has two $\bullet_{\GG^+}$'s in the same column; otherwise set $\alpha:=\hat{\beta}(\x)$.

\noindent
\ytableausetup{boxsize=1.3em}
{\sf (Case H7: $\head(S)=\{\x,\x^\rightarrow\}$, $\GG^+ \in \x^\rightarrow$, and $\circled{\GG} \in \underline{\x}$):}
\[\begin{picture}(110,14)
\put(-120,0){$\head(S) = \ytableaushort{{\bullet}{\GG^+}} \mapsto \m(\head(S)) =
\ytableaushort{{\bullet} {\GG^+}} + \beta(\x) \cdot \ytableaushort{\GG {\GG^+}} + \alpha \cdot \ytableaushort{\GG {\bullet}}$}
\put(-60,-5){\Scale[.8]{\circled{\GG}}}
\put(247,-5){$\GG^+$}
\end{picture}
\]
Set $\alpha:=0$ if the third tableau has two $\bullet_{\GG^+}$'s in the same column; otherwise set $\alpha:=\hat{\beta}(\x)$.

\noindent
\ytableausetup{boxsize=1.3em}
{\sf (Case H8: $\head(S)=\{\x,\x^\rightarrow\}$, $\GG^+ \in \x^\rightarrow$, and Cases H6 and H7 do not apply):}
\[\begin{picture}(110,14)
\put(-120,0){$\head(S) = \ytableaushort{{\bullet}{\GG^+}} \mapsto \m(\head(S)) =
\ytableaushort{{\bullet} {\GG^+}}$}
\end{picture}
\]

\noindent
\ytableausetup{boxsize=1.3em}
{\sf (Case H9: $\head(S)=\{\x,\x^\rightarrow\}$, $\FF^! \in \x$, and $\GG \in \x^\rightarrow$):}
\[\begin{picture}(110,14)
\put(-120,0){$\head(S) = \ytableaushort{{\FF^!} \GG} \mapsto \m(\head(S)) = \ytableaushort{{\FF^!} {\GG^!}}$}
\end{picture}
\]

\begin{lemma}\label{lem:head_wd}
Every nonempty $\head(S)$ falls into exactly one
of {\sf H1}--{\sf H9}.
\end{lemma}
\begin{proof}
Since $\head(S)\neq \emptyset$, $|\head(S)|\in \{1,2\}$
by Lemma~\ref{lem:snakes_are_short_ribbons}. If $\head(S)=\{\x\}$,
then by Lemma~\ref{lem:piecesobservations}(I), $\bullet_\GG\in\x$. 
Then $\underline{\x}$ contains exactly one of $\GG, \circled{\GG}$ or neither; these are respectively Cases
{\sf H1}, {\sf H2} and {\sf H3}. If $\head(S)=\{\x,\x^\rightarrow\}$, 
see Lemma~\ref{lem:piecesobservations}(II):
one possibility is $\FF^! \in \x$ and $\GG \in \x^\rightarrow$; this is {\sf H9}. Otherwise, $\bullet_\GG\in \x$ and $\x^{\rightarrow}$ contains $\GG$ or  $\GG^+$. The cases where $\GG\in\x^\rightarrow$ are covered by {\sf H4}--{\sf H5}. 
The cases where $\GG^+\in\x^\rightarrow$ are covered by {\sf H6}--{\sf H8}.
\end{proof}

\subsubsection{Miniswaps on $\body(S)$}
Let ${\tt body_{\bullet_{\GG}}}(S)=\{\x\in \body(S)\colon \bullet_{\GG}\in \x\}$.

\noindent
{\sf (Case B1: $\body(S) = S$):}
By Definition-Lemma~\ref{def-lem:snake_classification}, $S= \ytableausetup{boxsize=1em} \ytableaushort{\GG}$.  Define
\[\begin{picture}(250,14)
\put(-60,0){
$\body(S)= \ytableausetup{boxsize=1em} \ytableaushort{\GG}\mapsto\m( \body(S)) = \ytableaushort{\GG}$}
\end{picture}.\]

\noindent
{\sf (Case B2: The southmost row of $\body(S)$ contains two boxes):}
Replace each $\GG$ in $\body(S)$ with $\bullet_{\GG^+}$ and each $\bullet_{\GG}$ with $\GG$, emitting a weight $\prod_{\x \in {\tt body_\bullet}(S)} \hat{\beta}(\x)$.
That is (cf. Lemma~\ref{lem:piecesobservations}(III)),
\ytableausetup{boxsize=1em}
\[\begin{picture}(250,25)
\put(-60,18){$\body(S) = \ytableaushort{\none \none \bullet \GG, \none \bullet \GG, \bullet \GG}\mapsto \m(\body(S)) = \mbox{$\prod_{\x \in {\tt body_{\bullet_{\GG}}}(S)} \hat{\beta}(\x)$}\cdot\ytableaushort{\none \none \GG \bullet, \none \GG \bullet, \GG \bullet}$}
\end{picture}
\]

\noindent
{\sf (Case B3: Cases B1/B2 do not apply):}
Replace each $\GG$ in $\body(S)$ with $\bullet_{\GG^+}$ and each $\bullet_{\GG}$ with $\GG$, 
emitting $-\prod_{\x \in {\tt body_\bullet}(S)} \hat{\beta}(\x)$.
That is (cf. Lemma~\ref{lem:piecesobservations}(III)),
\ytableausetup{boxsize=1em}
\[\begin{picture}(250,25)
\put(-60,18){$\body(S) = \ytableaushort{ \none \bullet \GG,  \bullet \GG,  \GG}\mapsto \m(\body(S)) = \mbox{$-\prod_{\x \in {\tt body_{\bullet_{\GG}}}(S)} \hat{\beta}(\x)$}\cdot\ytableaushort{ \none \GG \bullet,  \GG \bullet, \bullet}$}
\end{picture}
\]

\begin{lemma}\label{lem:body_wd}
Every nonempty $\body(S)$ falls into exactly one of {\sf B1}--{\sf B3}.
\end{lemma}
\begin{proof}
If {\sf B1} applies, then by Definition-Lemma~\ref{def-lem:snake_classification}, $S= \ytableausetup{boxsize=1em} \ytableaushort{\GG}$. The lemma follows.
\end{proof}

\subsubsection{Miniswaps on $\tail(S)$}
\gap

\noindent
\ytableausetup{boxsize=1.3em}
{\sf (Case T1: $\tail(S)=\{\x\}$):}
\[\begin{picture}(350,17)
\put(0,0){$\tail(S)=\ytableaushort{\bullet} \mapsto  \m(\tail(S)) = -\hat{\beta}(\x) \cdot \ytableaushort{\GG}$}
\end{picture}
\]

\noindent
{\sf (Case T2: $\tail(S) = \{\x,\x^\rightarrow\}$ and $\GG\in \x^\rightarrow$):}
\[\begin{picture}(350,17)
\put(0,0){$\tail(S)=\ytableaushort{\bullet \GG} \mapsto \m(\tail(S)) = \hat{\beta}(\x) \cdot \ytableaushort{\GG \bullet}$}
\end{picture}\]

\noindent
\ytableausetup{boxsize=1.3em}
{\sf (Case T3: $\tail(S) = \{\x,\x^\rightarrow\}$ and $\GG^+ \in \x^\rightarrow$):} 
\[\begin{picture}(350,17)
\put(0,4){$\tail(S) = \ytableaushort{\bullet {\GG^+}} \mapsto \m(\tail(S)) = \mbox{$-\hat{\beta}(\x)$} \cdot\ytableaushort{\GG {\GG^+}} + \alpha \cdot \ytableaushort{\GG \bullet}$}
\put(329,0){$\GG^+$}
\end{picture}
\]
Set $\alpha:=0$ if the second tableau has two $\bullet_{\GG^+}$'s in the same column; otherwise set $\alpha:=\hat{\beta}(\x)$.

\noindent
{\sf (Case T4: $\tail(S) = \{\x,\x^\rightarrow\}$, $\GG \in \underline{\x^\rightarrow}$):}
Let $Z=\{\ell\in \underline{\x^\rightarrow}\colon\FF \prec \ell \prec \GG\}$.

\ytableausetup{boxsize=2em}
{\sf (Subcase {\sf T4.1}: $\mathcal{H}\in\underline{\x^\rightarrow}$, $\family(\mathcal{H}) = \family(\GG) + 1$ and $N_\mathcal{H}  = N_\GG$):}
\[\begin{picture}(370,20)
\put(0,0){$\tail(S) = \ytableaushort{\bullet {\mathcal{F}^!}} \mapsto \m(\tail(S)) = \ytableaushort{\bullet {\mathcal{F}^!}}$}
\put(84,-2){\Scale[.6]{Z, \GG, \HH}}
\put(251,-2){\Scale[.6]{Z, \GG^!, \HH}}
\end{picture}
\]

\ytableausetup{boxsize=2.2em}
{\sf (Subcase {\sf T4.2}: $\circled{\mathcal{H}}\in
\underline{\x^\rightarrow}$, $\family(\mathcal{H}) = \family(\GG) + 1$ and $N_\mathcal{H}  = N_\GG$):}
\[\begin{picture}(370,20)
\put(0,0){$\tail(S) = \ytableaushort{\bullet {\mathcal{F}^!}} \mapsto \m(\tail(S)) = \ytableaushort{\bullet {\mathcal{F}^!}} + \hat{\beta}(\x) \cdot \ytableaushort{\GG \bullet}$}
\put(84,-3){\Scale[.6]{Z, \GG, \circled{\HH}}}
\put(256,-3){\Scale[.6]{Z, \GG^!, \circled{\HH}}}
\put(334,25){\Scale[.6]{\mathcal{F}, Z}}
\end{picture}
\]

\ytableausetup{boxsize=1.4em}
{\sf (Subcase {\sf T4.3}: Subcases {\sf T4.1}/{\sf T4.2} do not apply):}
\[\begin{picture}(370,18)
\put(0,0){$\tail(S) = \ytableaushort{\bullet {\mathcal{F}^!}} \mapsto \m(\tail(S)) = \hat{\beta}(\x) \cdot \ytableaushort{\GG \bullet}$}
\put(78,-2){\Scale[.6]{Z, \GG}}
\put(244,16){\Scale[.6]{\mathcal{F}, Z}}
\end{picture}
\]

\noindent
\ytableausetup{boxsize=1.6em}
{\sf (Case T5: $\tail(S) = \{\x,\x^\rightarrow\}$, $\circled{\GG} \in \underline{\x^\rightarrow}$, $\GG \notin \underline{\x}$):}
Let $Z=\{\ell\in \underline{\x^\rightarrow}\colon\FF \prec \ell \prec \GG\}$.
\[\begin{picture}(370,18)
\put(0,0){$\tail(S) = \ytableaushort{\bullet {\FF^!}} \mapsto \m(\tail(S)) = \hat{\beta}(\x) \cdot \ytableaushort{\GG \bullet}$}
\put(79,-2){\Scale[.6]{Z, \circled{\GG}}}
\put(250,18){\Scale[.6]{\FF, Z}}
\end{picture}
\]

\noindent
\ytableausetup{boxsize=1.2em}
{\sf (Case T6: $\tail(S) = \{ \x, \x^\rightarrow \}$, $\circled{\GG} \in \underline{\x^\rightarrow}$, $\GG \in \underline{\x}$):}
\[\begin{picture}(370,10)
\put(0,0){$\tail(S) = \ytableaushort{\bullet {\mathcal{F}^!}} \mapsto \m(\tail(S)) = 0$}
\put(76,-3){\Scale[.6]{\circled{\GG}}}
\put(63,-3){\Scale[.8]{\GG}}
\end{picture}
\]

\begin{lemma}\label{lem:tail_wd}
Every nonempty $\tail(S)$ falls into exactly one of {\sf T1}--{\sf T6}.
\end{lemma}
\begin{proof}
Since $\tail(S)\neq\emptyset$, $|\tail(S)|\in\{1,2\}$
by Lemma~\ref{lem:snakes_are_short_ribbons}.
If $|\tail(S)|=1$, then by Lemma~\ref{lem:piecesobservations}(IV),
$\tail(S)=\ytableausetup{boxsize=1.1em}\ytableaushort{{\bullet_{\GG}}}$; this is covered
by {\sf T1}. Suppose $\tail(S)=\{\x,\x^\rightarrow\}$. By Lemma~\ref{lem:presnake_description},
(R.1)--(R.3) and Definition--Lemma~\ref{def-lem:snake_classification},
$\tail(S)=\ytableaushort{{\bullet_{\GG}} \GG}$ (handled by {\sf T2}), $\ytableausetup{boxsize=1.2em}\tail(S)=\ytableaushort{{\bullet_{\GG}} {\GG^+}}$ (handled by {\sf T3})
or $\tail(S)=\ytableaushort{{\bullet_{\GG}} {\FF^!}}$ with $\GG$ or $\circled{\GG} \in \underline{\x^\rightarrow}$ (handled by {\sf T4}, {\sf T5} or {\sf T6}).
\end{proof}

\subsection{Swaps and slides}
\label{subsection:swapsandslides}
We define $\swap_{\GG}(T)$ and
${\tt slide}_{\{\x_i\}}(T)$ for a good tableau $T$. Define ${\tt swap}_\GG$ on a single snake $S$ by applying $\m$ simultaneously to $\head(S)$, $\body(S)$, and $\tail(S)$ (where the conditions on each $\m$ refer to the original $S$). 

\begin{lemma}
On the edges shared by two adjacent snake sections, the modifications to the labels given by the two miniswaps are compatible. 
\end{lemma}
\begin{proof}
Suppose the lower of the two adjacent sections is $\head(S)$. The only miniswap that introduces a label to  the northeast edge (i.e., $\overline \x$ if
$\head(S)=\{\x\}$ or $\overline{\x^\rightarrow}$ if $\head(S)=\{\x,\x^\rightarrow\}$) is {\sf H1}. However in that case 
$\head(S)=S$ and the compatibility issue is moot. 
Since $\body$ miniswaps do not affect edge labels, the remaining check is when a $\tail$ miniswap involves $\underline{\x}$ where
$\x$ is the left box of $\tail(S)$. This only occurs in {\sf T6}. In this case $\tail(S)=S$, so compatibility is again moot.
\end{proof}

\begin{lemma}[Swap commutation]\label{lem:miniswap_commutation}
If $S_1, S_2$ are distinct snakes in a $\GG$-good tableau $T$, then applying ${\tt swap}_\GG$ to $S_1$ commutes with applying ${\tt swap}_\GG$ to $S_2$.
\end{lemma}
\begin{proof}
By definition, the locations of virtual labels in one snake are unaffected by swapping another snake.
Hence if the snakes do not share a horizontal edge, there is no concern.
If they do, this is the situation of
Lemma~\ref{lem:snakes_arranged_SW-NE}(III). The northmost row $r$ of the lower snake (say $S_1$) is $\{\x,\x^\rightarrow\}$ with $\GG^{+}\in\x^\rightarrow$. Hence by (G.4), $\GG,\circled{\GG}\not\in \overline{\x^\rightarrow}$. 
By inspection, no miniswap involving $r$ affects $\overline{\x^\rightarrow}$.
Now, the upper snake $S_2$ has a single row, which by the previous sentence is either an {\sf H3} or {\sf H8} $\head$, irregardless of whether
we have acted on $r$ already. Therefore, $\swap_{\GG}$
acts trivially on $S_2$ whether we act on $S_1$ or $S_2$ first. 
\end{proof}

Lemma~\ref{lem:miniswap_commutation} permits us to define the {\bf swap} operation $\swap_\GG$ on a $\GG$-good tableau as
the result of applying $\swap_\GG$ to all snakes (in arbitrary order). Extend
$\swap_\GG$
to a ${\mathbb Z}[t_1^{\pm 1},\ldots, t_n^{\pm 1}]$-linear
operator.

An {\bf inner corner} of $\nu/\lambda$ is a maximally southeast box of $\lambda$. An {\bf outer corner} of $\nu/\lambda$ is a
maximally southeast box of $\nu/\lambda$.

Let $T\in {\tt Bundled}(\nu/\lambda)$ and $\{\x_i\}$ be a subset of the inner corners of $\nu/\lambda$.
Define $T^{(1_1)}$ to be $T$ with $\bullet_{1_1}$ placed in each $\x_i$.

\begin{lemma}\label{lem:adding_bullets_ok}
Each $T^{(1_1)}$ is $1_1$-good.
\end{lemma}
\begin{proof}
(G.2) is clear.
By Lemma~\ref{lem:bundled_tableaux_are_good}, $T$ is good; (G.1), (G.3)--(G.8) and (G.12) are unaffected by adding
$\bullet_{1_1}$'s to inner corners. (G.9)--(G.11) and (G.13) hold vacuously.
\end{proof}

The {\bf slide} of $T$ at $\{\x_i\}$ is
\begin{equation}
\label{eqn:defjdt}
{\tt slide}_{\{\x_i\}}(T) := \swap_{\GG_{\rm max}} \circ \swap_{\GG_{\rm max}^-}\circ \cdots \circ \swap_{1_1}(T^{(1_1)}),
\end{equation}
with all
$\bullet_{\GG_{\rm max}^+}$'s deleted. 
If $\Sigma$ is a formal ${\mathbb Z}[t_1^{\pm 1},\ldots, t_n^{\pm 1}]$-linear sum of tableaux we write $V\in \Sigma$ to  mean
$V$ occurs in $\Sigma$ with nonzero coefficient. The following proposition, proved in Appendix~\ref{sec:forward_goodness_proof}, 
shows (\ref{eqn:defjdt}) is well-defined.

\begin{proposition}[Swaps preserve goodness]
\label{prop:goodness_preservation}
If $T$ is a $\GG$-good tableau, then each $U\in\swap_\GG(T)$ is $\GG^+$-good.
\end{proposition}

\begin{lemma}[Swaps preserve content]\label{lem:content_preservation}
If $T$ is a $\GG$-good tableau of content $\mu$, then each $U\in\swap_\GG(T)$ has content $\mu$.
\end{lemma}
\begin{proof}
No miniswap eliminates genes in a section. We consider each 
miniswap that introduces a new gene to a section; this gene must be $\GG$. We
show that $\GG$ appears elsewhere in $T$.
The first case is {\sf H2}, which produces a $\GG$ in its section, where there was only a $\circled{\GG}$ previously. $\circled{\GG}$ only appears if some $\GG$ is west of it in $T$. The same analysis applies \emph{verbatim} to {\sf H7} and {\sf T5}.
The remaining cases are {\sf T1} and {\sf T3}. By Lemma~\ref{lem:piecesobservations}(IV, V), the snake on which these miniswaps act has at 
least two rows. Moreover, there is a $\GG$ directly below the $\bullet_{\GG}$ under consideration. In particular, $\GG$ already appeared in $T$.
\end{proof}

\begin{lemma}\label{lem:bullets_migrate_out}
No label is strictly southeast of a 
$\bullet_{\GG_{\rm max}^+}$ in any $U\in\swap_{\GG_{\rm max}} \circ \swap_{\GG_{\rm max}^-}\circ \cdots \circ \swap_{1_1}(T^{(1_1)})$. 
In particular, all $\bullet_{\GG_{\rm max}^+}$'s are at outer corners of $\nu/\lambda$.
\end{lemma}
\begin{proof}
By Proposition~\ref{prop:goodness_preservation}, 
$U$ is $\GG_{\rm max}^+$-good. 
Let $\x$ be a box of $U$ and $\bullet_{\GG_{\rm max}^+}\in \x$. There is no $\bullet_{\GG_{\rm max}^+}$
strictly southeast of $\x$ by (G.2). By definition, there is no label $\QQ$ in $T^{(1_1)}$ with $\family(\QQ) \geq \family(\GG_{\rm max}^+)$. Hence by Lemma~\ref{lem:content_preservation}, there are no such labels in $U$. Therefore, any genetic label $\ell$ southeast of $\x$ is marked. Clearly, we may assume $\ell$ is in
$\x$'s row or column. If $\ell$ is in $\x$'s column, we contradict (G.11). 
If $\ell$ is in $\x$'s row, we contradict Lemma~\ref{lem:strong_form_of_G13}.
\end{proof}

Clearly,
\begin{lemma}\label{lem:deleting_outer_bullets_ok}
If $T$ is a good tableau with no genetic label southeast of a $\bullet$, then deleting all $\bullet$'s gives a bundled tableau.\qed
\end{lemma}
 
\begin{corollary}
\label{cor:codomain}
Given $\rho\in\lambda^+$ and a tableau $T\in B_{\rho,\mu}^\nu$, any tableau $U \in {\tt slide}_{\rho/\lambda}(T)$ is in either $B_{\lambda,\mu}^\nu$ or $B_{\lambda,\mu}^\delta$ for some $\delta\in\nu^-$.
\end{corollary}
\begin{proof}
By Lemma~\ref{lem:bundled_tableaux_are_good}, $T$ is a good tableau. By Lemma~\ref{lem:adding_bullets_ok}, adding $\bullet_{1_1}$ to each box of $\rho/\lambda$ gives a good tableau $T^{(1_1)}$. By Proposition~\ref{prop:goodness_preservation}, each swap gives a formal sum of good tableaux. By Lemma~\ref{lem:bullets_migrate_out}, after all swaps, $\bullet_{\GG_{\rm max}^+}$'s are at outer corners with no labels strictly southeast. By Lemma~\ref{lem:deleting_outer_bullets_ok}, deleting these $\bullet_{\GG_{\rm max}^+}$'s gives a bundled tableau (namely $U$). $U$ has shape $\nu/\lambda$ or $\delta/\lambda$ for $\delta \in \nu^-$, since there is at most one $\bullet_{\GG_{\rm max}^+}$ deleted in any row or column by (G.2). Content preservation is Lemma~\ref{lem:content_preservation}.
\end{proof}

\subsection{Examples}
We give a number of examples of computing ${\tt slide}_{\{\x_i\}}(T)$. It is convenient to encode the computations
in a diagram. Each non-terminal tableau has its snakes differentiated by color. The notation above each arrow indicates the types
of the snakes from southwest to northeast, for example {\sf H5.3}/$\emptyset$/{\sf T2} means the head of the snake is {\sf H5.3}, the body is empty and the tail is {\sf T2}. The notation below arrows indicates the product of the coefficients coming from each miniswap (we will assume for this purpose that the lower left corner of $T$ coincides with the lower left
corner of $k\times (n-k)$). 
Each $U\in {\tt slide}_{\{\x_i\}}(T)$ is a terminal tableau of the diagram. Moreover, $[U]{\tt slide}_{\{\x_i\}}(T)$
is the sum of the products of the coefficients over all directed paths from $T$ to $U$.

\noindent
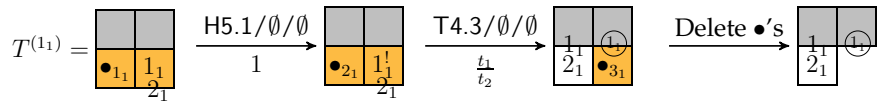
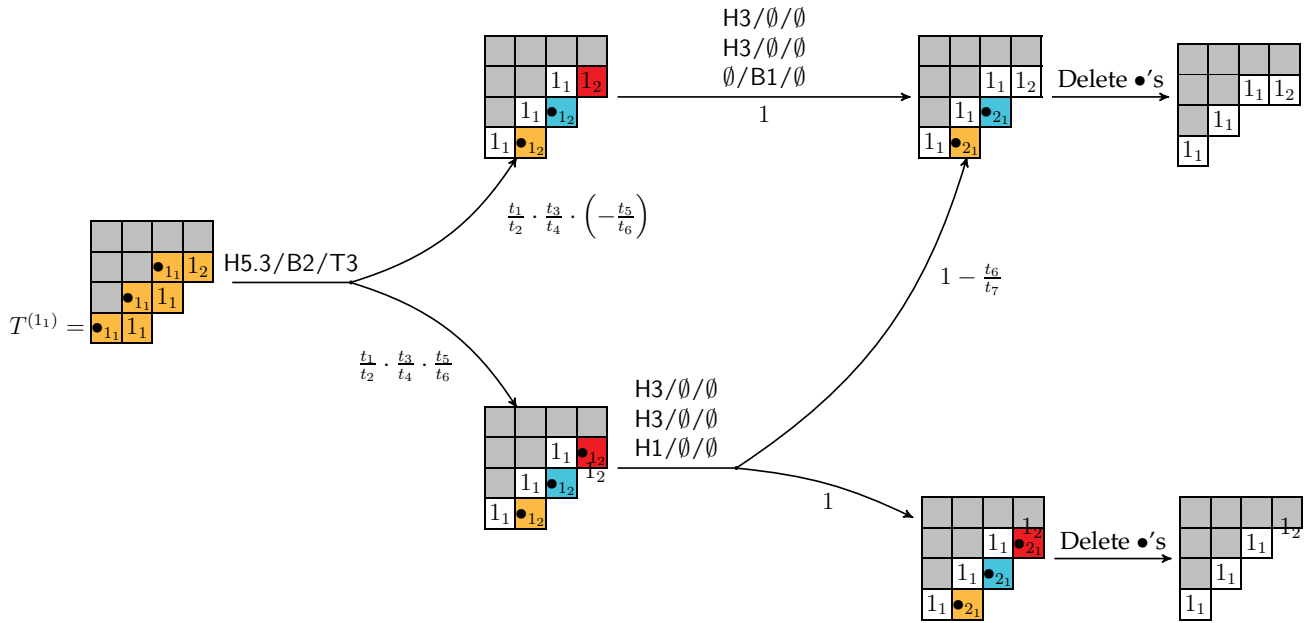
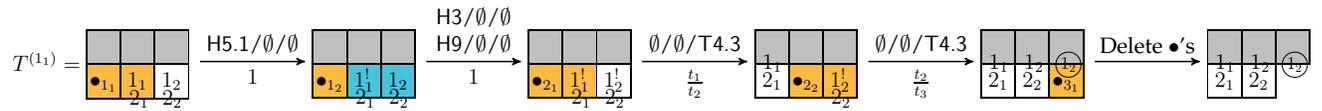
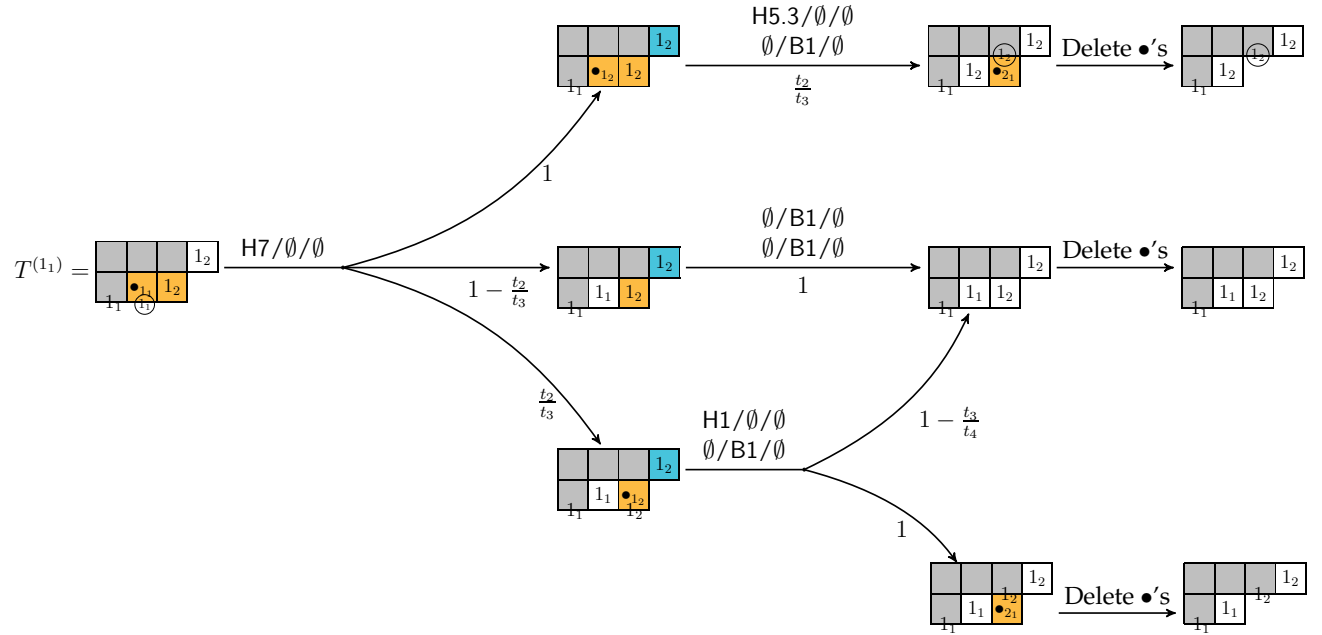
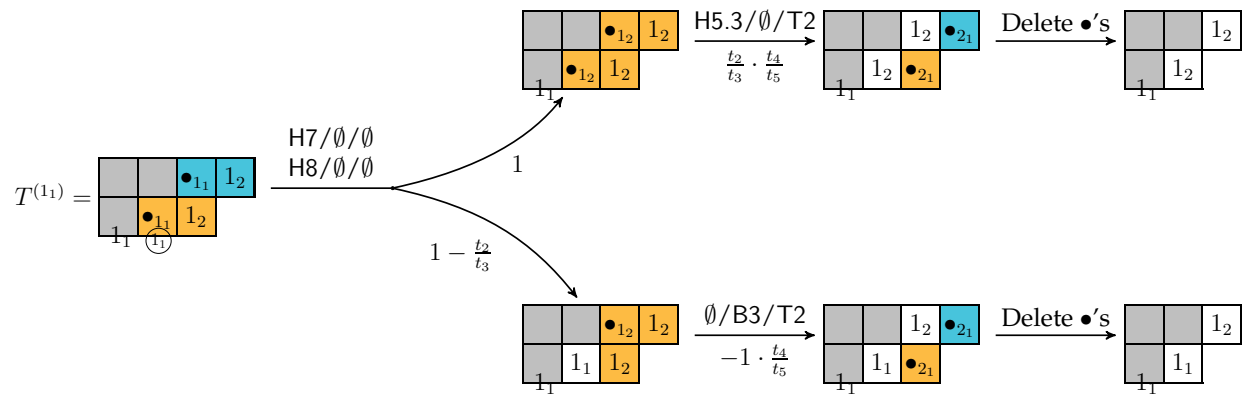
\begin{example} 
\ytableausetup{boxsize=1.5em}
\[\Scale[0.8]{\begin{tikzpicture}[node distance=1cm, auto,]
 \node (A) {$T^{(1_1)}=\begin{picture}(43,8)
\put(0,4){$\ytableaushort{{*(lightgray)\blank} {*(lightgray)\blank}, {*(Dandelion)\bullet_{1_1}} {*(Dandelion)1_1}}$}
\put(24,-19){$2_1$}
\end{picture}$};
 \node[right=2cm of A] (B) {$\begin{picture}(43,8)
\put(0,4){$\ytableaushort{{*(lightgray)\blank} {*(lightgray)\blank}, {*(Dandelion)\bullet_{2_1}} {*(Dandelion)1^!_1}}$}
\put(24,-19){$2_1$}
\end{picture}
$}
   edge[pil, <-] node[above]{{\sf H5.1}/$\emptyset$/$\emptyset$} node[below]{$1$} (A);
 \node[right=2cm of B] (C) {$\begin{picture}(50,8)
\put(0,4){$\ytableaushort{{*(lightgray)\blank} {*(lightgray)\blank}, {2_1} {*(Dandelion)\bullet_{3_1}}}$}
\put(4,1){$1_1$}
\put(22,3){\Scale[.7]{$\circled{1_1}$}}
\end{picture}$}
   edge[pil, <-] node[above]{{\sf T4.3}/$\emptyset$/$\emptyset$} node[below]{$\frac{t_1}{t_2}$} (B.east);
 \node[right=2cm of C] (D) {$\begin{picture}(50,8)
\put(0,4){$\ytableaushort{{*(lightgray)\blank} {*(lightgray)\blank}, {2_1}}$}
\put(4,1){$1_1$}
\put(22,3){\Scale[.7]{$\circled{1_1}$}}
\end{picture}$}
   edge[pil,<-] node[above]{Delete $\bullet$'s} (C);
\end{tikzpicture}}
\]
\qed
\end{example}

\begin{example}
\ytableausetup{boxsize=1.5em}
\[\Scale[0.8]{\begin{tikzpicture}[node distance=1cm, auto]
\node (A) {$T^{(1_1)} = \begin{picture}(40,12)
\put(0,4){$\ytableaushort{{*(lightgray)\blank} {*(lightgray)\blank}, {*(Dandelion) \bullet_{1_1}} {*(Dandelion) 1_1},{2_1}}$}
\put(22,-19){\Scale[.7]{$\circled{2_1}$}}
\end{picture}$};
\node[junct, right=2cm of A] (dummy1) {};
\node[above right=3cm of dummy1] (B) {$\begin{picture}(43,8)
\put(0,4){$\ytableaushort{{*(lightgray)\blank} {*(lightgray)\blank}, {*(Dandelion) \bullet_{2_1}} {*(Dandelion) 1^!_1},{*(Dandelion) 2_1}}$}
\put(22,-19){\Scale[.7]{$\circled{2_1}$}}
\end{picture}$};
\node[below right=3cm of dummy1] (C) {$\begin{picture}(43,8)
\put(0,4){$\ytableaushort{{*(lightgray)\blank} {*(lightgray)\blank}, {1_1} {*(SkyBlue) \bullet_{2_1}},{*(Dandelion) 2_1}}$}
\put(22,2){\Scale[.7]{$\circled{1_1}$}}
\end{picture}$};
\node[right=2cm of C] (E) {\begin{picture}(43,8)
\put(0,4){$\ytableaushort{{*(lightgray)\blank} {*(lightgray)\blank}, {1_1} {*(Dandelion) \bullet_{3_1}},{2_1}}$}
\put(22,2){\Scale[.7]{$\circled{1_1}$}}
\end{picture}};
\node[right=2cm of E] (F) {\begin{picture}(43,8)
\put(0,4){$\ytableaushort{{*(lightgray)\blank} {*(lightgray)\blank}, {1_1},{2_1}}$}
\put(22,2){\Scale[.7]{$\circled{1_1}$}}
\end{picture}};
\node[right=2cm of B] (D) {\begin{picture}(43,8)
\put(0,4){$\ytableaushort{{*(lightgray)\blank} {*(lightgray)\blank}, {2_1} {*(SkyBlue)\bullet_{3_1}},{*(Dandelion)\bullet_{3_1}}}$}
\put(4,2){$1_1$}
\put(22,2){\Scale[.7]{$\circled{1_1}$}}
\end{picture}};
\node[right=2cm of D] (G) {\begin{picture}(43,8)
\put(0,4){$\ytableaushort{{*(lightgray)\blank} {*(lightgray)\blank}, {2_1}}$}
\put(4,2){$1_1$}
\put(22,2){\Scale[.7]{$\circled{1_1}$}}
\end{picture}};
\path  (A) edge[und] node[above]{{\sf H5.2}/$\emptyset$/$\emptyset$} (dummy1);
\path (dummy1) edge[pil, bend left=20] node[below]{$1$} (B);
\path (dummy1) edge[pil, bend right=20] node[below]{$\frac{t_2}{t_3}$} (C);
\path (C) edge[pil] node[above](lab1){{\sf H3}/$\emptyset$/$\emptyset$} node[above=0.5 of lab1]{$\emptyset$/{\sf B1}/$\emptyset$} node[below]{1} (E);
\path (E) edge[pil] node[above]{Delete $\bullet$'s} (F);
\path (B) edge[pil] node[above]{$\emptyset$/{\sf B3}/{\sf T5}} node[below]{$-1\cdot\frac{t_2}{t_3}$} (D);
\path (D) edge[pil] node[above]{Delete $\bullet$'s} (G);  
\end{tikzpicture}}
\]
\qed
\end{example}

\begin{example}
\ytableausetup{boxsize=1.5em}
\[\Scale[0.8]{$\begin{tikzpicture}
\node (A) {$T^{(1_1)}=\begin{picture}(55,20)
\put(0,0){$\ytableaushort{{*(Dandelion) \bullet_{1_1}} {*(Dandelion) 1_2} {1_3}}$}
\put(5,-6){$1_1$}
\end{picture}$};
\node[junct, right=2cm of A] (dummy1) {};
\node[above right=3cm of dummy1] (B) {$\begin{picture}(55,20)
\put(0,0){$\ytableaushort{{1_1} {*(Dandelion) 1_2} {1_3}}$}
\end{picture}$};
\node[below right=3cm of dummy1] (C) {$\begin{picture}(55,20)
\put(0,0){$\ytableaushort{{1_1} {*(Dandelion) \bullet_{1_2}} {*(Dandelion) 1_3}}$}
\put(25,-6){$1_2$}
\end{picture}$};
\node[right=3cm of B] (D) {$\begin{picture}(55,20)
\put(0,0){$\ytableaushort{{1_1} {1_2} {*(Dandelion) 1_3}}$}
\end{picture}$};
\node[right=3cm of C] (E) {$\begin{picture}(55,20)
\put(0,0){$\ytableaushort{{1_1} {1_2} {*(Dandelion) \bullet_{1_3}}}$}
\put(42,-6){${1_3}$}
\end{picture}$};
\node[right=3cm of D] (F) {$\begin{picture}(55,20)
\put(0,0){$\ytableaushort{{1_1} {1_2} {1_3}}$}
\end{picture}$};
\node[junct, right=2cm of C] (dummy2) {};
\node[junct, right=2cm of E] (dummy3) {};
\path (A) edge[und] node[above]{${\sf H6}/\emptyset/\emptyset$} (dummy1);
\path (dummy1) edge[pil, bend left=20] node[below right]{$1 - \frac{t_1}{t_2}$} (B);
\path (dummy1) edge[pil, bend right=20] node[below]{$\frac{t_1}{t_2}$} (C);
\path (B) edge[pil] node[below]{$1$} node[above]{$\emptyset$/{\sf B1}/$\emptyset$} (D);
\path (dummy2) edge[pil, bend right=15] node[below]{$\frac{t_2}{t_3}$} (E);
\path (D) edge[pil] node[below]{$1$} node[above]{$\emptyset$/{\sf B1}/$\emptyset$} (F);
\path (dummy2) edge[pil, bend left=20] node[below right]{$1-\frac{t_2}{t_3}$} (D);
\path (C) edge[und] node[above]{${\sf H6}/\emptyset/\emptyset$} (dummy2);
\path (dummy3) edge[pil, bend left=20] node[below right]{$1-\frac{t_3}{t_4}$} (F);
\path (E) edge[und] node[above]{${\sf H1}/\emptyset/\emptyset$} (dummy3);
\end{tikzpicture}$}
\]
\qed
\end{example}

\begin{example}
\ytableausetup{boxsize=1.2em}
\[\Scale[.78]{$
\begin{tikzpicture}
\node (A) {$T^{(1_1)}=\begin{picture}(64,50)
\put(0,40){$\ytableaushort{{*(lightgray)\blank} {*(lightgray)\blank} {*(lightgray)\blank} {*(lightgray)\blank}, {*(lightgray)\blank} {*(lightgray)\blank} {*(Dandelion) \bullet_{1_1}} {*(Dandelion)1_2}, {*(lightgray)\blank} {*(Dandelion)\bullet_{1_1}} {*(Dandelion)1_1},
{*(Dandelion)\bullet_{1_1}} {*(Dandelion)1_1}}$}
\end{picture}$};
\node[junct, right=2cm of A] (dummy1) {};
\node[above right=3cm of dummy1] (B) {\begin{picture}(60,50)
\put(0,40){$\ytableaushort{{*(lightgray)\blank} {*(lightgray)\blank} {*(lightgray)\blank} {*(lightgray)\blank}, {*(lightgray)\blank} {*(lightgray)\blank} {1_1} {*(Red) 1_2}, {*(lightgray)\blank} {1_1} {*(SkyBlue) \bullet_{1_2}}, {1_1} {*(Dandelion) \bullet_{1_2}}}$}
\end{picture}};
\node[below right=3cm of dummy1] (C) {\begin{picture}(60,50)
\put(0,40){$\ytableaushort{{*(lightgray)\blank} {*(lightgray)\blank} {*(lightgray)\blank} {*(lightgray)\blank}, {*(lightgray)\blank} {*(lightgray)\blank}
{1_1} {*(Red)\bullet_{1_2}}, {*(lightgray)\blank} {1_1} {*(SkyBlue) \bullet_{1_2}}, {1_1} {*(Dandelion) \bullet_{1_2}}}$}
\put(48,20){$1_2$}
\end{picture}};
\node[right=5cm of B] (D) {\begin{picture}(60,50)
\put(0,40){$\ytableaushort{{*(lightgray)\blank} {*(lightgray)\blank} {*(lightgray)\blank} {*(lightgray)\blank}, {*(lightgray)\blank} {*(lightgray)\blank} {1_1} {1_2}, {*(lightgray)\blank} {1_1} {*(SkyBlue) \bullet_{2_1}}, {1_1} {*(Dandelion) \bullet_{2_1}}}$}
\end{picture}};
\node[junct, right=2cm of C] (dummy2) {};
\node[below right=0.5 and 3 of dummy2] (E) {$\begin{picture}(60,50)
\put(0,40){$\ytableaushort{{*(lightgray)\blank} {*(lightgray)\blank} {*(lightgray)\blank} {*(lightgray)\blank}, {*(lightgray)\blank} {*(lightgray)\blank} {1_1} {*(Red) \bullet_{2_1}}, {*(lightgray)\blank} {1_1} {*(SkyBlue) \bullet_{2_1}}, {1_1} 
{*(Dandelion)\bullet_{2_1}}}$}
\put(48,37){${1_2}$}
\end{picture}$};
\node[right=2 of D] (F) {$\begin{picture}(60,50)
\put(0,36){$\ytableaushort{{*(lightgray)\blank} {*(lightgray)\blank} {*(lightgray)\blank} {*(lightgray)\blank}, {*(lightgray)\blank} {*(lightgray)\blank} {1_1} {1_2}, {*(lightgray)\blank} {1_1}, {1_1}}$}
\end{picture}$};
\node[right=2 of E] (G) {$\begin{picture}(60,50)
\put(0,40){$\ytableaushort{{*(lightgray)\blank} {*(lightgray)\blank} {*(lightgray)\blank} {*(lightgray)\blank}, {*(lightgray)\blank} {*(lightgray)\blank} {1_1}, {*(lightgray)\blank} {1_1}, {1_1}}$}
\put(48,37){${1_2}$}
\end{picture}$};
\path (A) edge[und] node[above]{{\sf H5.3}/{\sf B2}/{\sf T3}} (dummy1);
\path (dummy1) edge[pil, bend right=20] node[below right, pos=.85]{$\frac{t_1}{t_2} \cdot \frac{t_3}{t_4} \cdot \left(-\frac{t_5}{t_6}\right)$} (B);
\path (dummy1) edge[pil, bend left=20] node[below left, pos=.6]{$\frac{t_1}{t_2} \cdot \frac{t_3}{t_4} \cdot \frac{t_5}{t_6}$} (C);
\path (B) edge[pil] node[below]{$1$} node[above](lab3){$\emptyset$/{\sf B1}/$\emptyset$} node[above=0.5 of lab3]{{\sf H3}/$\emptyset$/$\emptyset$} node[above=1 of lab3]{{\sf H3}/$\emptyset$/$\emptyset$} (D);
\path (C) edge[und] node[above](sec3){{\sf H1}/$\emptyset$/$\emptyset$} node[above=0.5 of sec3]{{\sf H3}/$\emptyset$/$\emptyset$} node[above=1 of sec3]{{\sf H3}/$\emptyset$/$\emptyset$} (dummy2);
\path (dummy2) edge[pil, bend right=20] node[below right, near end]{$1 - \frac{t_6}{t_7}$} (D);
\path (dummy2) edge[pil, bend left=10] node[below]{$1$} (E);
\path (D) edge[pil] node[above]{Delete $\bullet$'s} (F);
\path (E) edge[pil] node[above]{Delete $\bullet$'s} (G);
\end{tikzpicture}
$}\]
\qed
\end{example}

\begin{example} 
\ytableausetup{boxsize=1.5em}
\[\Scale[0.7]{$
\begin{tikzpicture}
\node (A) {$T^{(1_1)}=\begin{picture}(57,10)
\put(0,4){$\ytableaushort{{*(lightgray)\blank} {*(lightgray)\blank} {*(lightgray)\blank}, {*(Dandelion)\bullet_{1_1}} {*(Dandelion) 1_1} {1_2}}$}
\put(23,-18){$2_1$}
\put(41,-18){$2_2$}
\end{picture}$};
\node[right=2 of A] (B) {$\begin{picture}(57,10)
\put(0,4){$\ytableaushort{{*(lightgray)\blank} {*(lightgray)\blank} {*(lightgray)\blank}, {*(Dandelion) \bullet_{1_2}} {*(SkyBlue) 1_1^!} {*(SkyBlue) 1_2}}$}
\put(23,-18){$2_1$}
\put(41,-18){$2_2$}
\end{picture}$};
\node[right=1.8 of B] (C) {$\begin{picture}(57,10)
\put(0,4){$\ytableaushort{{*(lightgray)\blank} {*(lightgray)\blank} {*(lightgray)\blank}, {*(Dandelion) \bullet_{2_1}} {*(Dandelion) 1_1^!} {1_2^!}}$}
\put(23,-18){$2_1$}
\put(41,-18){$2_2$}
\end{picture}$};
\node[right=2 of C] (D) {$\begin{picture}(57,10)
\put(0,4){$\ytableaushort{{*(lightgray)\blank} {*(lightgray)\blank} {*(lightgray)\blank}, {2_1} {*(Dandelion) \bullet_{2_2}} {*(Dandelion) 1_2^!}}$}
\put(4,1){$1_1$}
\put(41,-18){$2_2$}
\end{picture}$};
\node[right=2 of D] (E) {$\begin{picture}(57,10)
\put(0,4){$\ytableaushort{{*(lightgray)\blank} {*(lightgray)\blank} {*(lightgray)\blank}, {2_1} {2_2} {*(Dandelion) \bullet_{3_1}}}$}
\put(4,1){$1_1$}
\put(23,1){$1_2$}
\put(40,1){$\Scale[.8]{\circled{1_2}}$}
\end{picture}$};
\node[right=2 of E] (F) {$\begin{picture}(57,10)
\put(0,4){$\ytableaushort{{*(lightgray)\blank} {*(lightgray)\blank} {*(lightgray)\blank}, {2_1} {2_2}}$}
\put(4,1){$1_1$}
\put(23,1){$1_2$}
\put(40,1){$\Scale[0.8]{\circled{1_2}}$}
\end{picture}$};
\path (A) edge[pil] node[above]{{\sf H5.1}/$\emptyset$/$\emptyset$} node[below]{$1$} (B);
\path (B) edge[pil] node[above](l2){{\sf H9}/$\emptyset$/$\emptyset$} node[above=.5 of l2]{{\sf H3}/$\emptyset$/$\emptyset$} node[below]{$1$} (C);
\path (C) edge[pil] node[above]{$\emptyset$/$\emptyset$/{\sf T4.3}} node[below]{$\frac{t_1}{t_2}$} (D);
\path (D) edge[pil] node[above]{$\emptyset$/$\emptyset$/{\sf T4.3}} node[below]{$\frac{t_2}{t_3}$} (E);
\path (E) edge[pil] node[above]{Delete $\bullet$'s} (F);
\end{tikzpicture}
$}\]
\qed
\end{example}

\begin{example} 
\ytableausetup{boxsize=1.5em}
\[\Scale[0.78]{$
\begin{tikzpicture}
\node (A) {$T^{(1_1)}=\Scale[0.8]{\begin{picture}(72,10)
\put(0,4){$\ytableaushort{{*(lightgray)\blank} {*(lightgray)\blank} {*(lightgray)\blank} {1_2}, {*(lightgray)\blank} {*(Dandelion) \bullet_{1_1}} {*(Dandelion) 1_2}}$}
\put(5,-18){$1_1$}
\put(23,-18){\Scale[.7]{$\circled{1_1}$}}
\end{picture}}$};
\node[junct, right=2 of A] (dummy1) {};
\node[above right=3 and 3.5 of dummy1] (B) {$\Scale[0.8]{\begin{picture}(72,20)
\put(0,16){$\ytableaushort{{*(lightgray)\blank} {*(lightgray)\blank} {*(lightgray)\blank} {*(SkyBlue) 1_2}, {*(lightgray)\blank} {*(Dandelion) \bullet_{1_2}} {*(Dandelion) 1_2}}$}
\put(5,-6){$1_1$}
\end{picture}}$};
\node[right=3.5 of dummy1] (C) {$\Scale[0.8]{\begin{picture}(72,20)
\put(0,4){$\ytableaushort{{*(lightgray)\blank} {*(lightgray)\blank} {*(lightgray)\blank} {*(SkyBlue) 1_2}, {*(lightgray)\blank} {1_1} {*(Dandelion) 1_2}}$}
\put(5,-18){$1_1$}
\end{picture}}$};
\node[below right=3 and 3.5 of dummy1] (D) {$\Scale[0.8]{\begin{picture}(72,20)
\put(0,4){$\ytableaushort{{*(lightgray)\blank} {*(lightgray)\blank} {*(lightgray)\blank} {*(SkyBlue) 1_2}, {*(lightgray)\blank} {1_1} {*(Dandelion) \bullet_{1_2}}}$}
\put(5,-18){$1_1$}
\put(41,-18){$1_2$}
\end{picture}}$};
\node[right=4 of B] (E) {$\Scale[0.8]{\begin{picture}(72,20)
\put(0,16){$\ytableaushort{{*(lightgray)\blank} {*(lightgray)\blank} {*(lightgray)\blank} {1_2}, {*(lightgray)\blank} {1_2} {*(Dandelion) \bullet_{2_1}}}$}
\put(5,-6){$1_1$}
\put(39,14){\Scale[.8]{\circled{1_2}}}
\end{picture}}$};	
\node[right=4 of C] (F) {$\Scale[0.8]{\begin{picture}(72,20)
\put(0,4){$\ytableaushort{{*(lightgray)\blank} {*(lightgray)\blank} {*(lightgray)\blank} {1_2}, {*(lightgray)\blank} {1_1} {1_2}}$}
\put(5,-18){$1_1$}
\end{picture}}$};
\node[junct, right=2 of D] (dummy2) {};
\node[below right=2 and 2 of dummy2] (G) {$\Scale[0.8]{\begin{picture}(72,20)
\put(0,22){$\ytableaushort{{*(lightgray)\blank} {*(lightgray)\blank} {*(lightgray)\blank} {1_2}, {*(lightgray)\blank} {1_1} {*(Dandelion) \bullet_{2_1}}}$}
\put(5,0){$1_1$}
\put(42,18){$1_2$}
\end{picture}}$};
\node[right=2 of E] (H) {$\Scale[0.8]{\begin{picture}(72,20)
\put(0,16){$\ytableaushort{{*(lightgray)\blank} {*(lightgray)\blank} {*(lightgray)\blank} {1_2}, {*(lightgray)\blank} {1_2} }$}
\put(5,-6){$1_1$}
\put(39,14){\Scale[.8]{\circled{1_2}}}
\end{picture}}$};
\node[right=2 of F] (I) {$\Scale[0.8]{\begin{picture}(72,20)
\put(0,4){$\ytableaushort{{*(lightgray)\blank} {*(lightgray)\blank} {*(lightgray)\blank} {1_2}, {*(lightgray)\blank} {1_1} {1_2}}$}
\put(5,-18){$1_1$}
\end{picture}}$};
\node[right=2 of G] (J) {$\Scale[0.8]{\begin{picture}(72,20)
\put(0,22){$\ytableaushort{{*(lightgray)\blank} {*(lightgray)\blank} {*(lightgray)\blank} {1_2}, {*(lightgray)\blank} {1_1}}$}
\put(5,0){$1_1$}
\put(42,18){$1_2$}
\end{picture}}$};
\path (A) edge[und] node[above]{{\sf H7}/$\emptyset$/$\emptyset$} (dummy1);
\path (dummy1) edge[pil, bend right=20] node[below, near end]{$1$} (B);
\path (dummy1) edge[pil] node[below, near end]{$1 -\frac{t_2}{t_3}$} (C);
\path (dummy1) edge[pil, bend left=20] node[below, near end]{$\frac{t_2}{t_3}$} (D);
\path (B) edge[pil] node[below]{$\frac{t_2}{t_3}$} node[above](l){$\emptyset$/{\sf B1}/$\emptyset$} node[above=0.5 of l]{{\sf H5.3}/$\emptyset$/$\emptyset$} (E);
\path (C) edge[pil] node[below]{$1$} node[above](ll){$\emptyset$/{\sf B1}/$\emptyset$} node[above=0.5 of l]{$\emptyset$/{\sf B1}/$\emptyset$} (F);
\path (D) edge[und] node[above](lll){$\emptyset$/{\sf B1}/$\emptyset$} node[above=0.5 of l]{{\sf H1}/$\emptyset$/$\emptyset$} (dummy2);
\path (dummy2) edge[pil, bend right=20, shorten >=11pt] node[below right]{$1 - \frac{t_3}{t_4}$} (F);
\path (dummy2) edge[pil, bend left=20, shorten >=15pt] node[below]{$1$} (G);
\path (E) edge[pil] node[above]{Delete $\bullet$'s} (H);
\path (F) edge[pil] node[above]{Delete $\bullet$'s} (I);
\path (G) edge[pil] node[above]{Delete $\bullet$'s} (J);
\end{tikzpicture}
$}\]
\qed
\end{example}

\begin{example}
\ytableausetup{boxsize=1.5em}
\[\Scale[0.8]{$
\begin{tikzpicture}
\node (A) {$T^{(1_1)} = \begin{picture}(77,16)
\put(0,4){$\ytableaushort{{*(lightgray)\blank} {*(lightgray)\blank} {*(SkyBlue) \bullet_{1_1}} {*(SkyBlue) 1_2}, {*(lightgray)\blank} {*(Dandelion)\bullet_{1_1}} {*(Dandelion) 1_2}}$}
\put(5,-18){$1_1$}
\put(22,-19){\Scale[.7]{$\circled{1_1}$}}
\end{picture}$};
\node[junct, right=2 of A] (dummy1) {};
\node[above right=2 and 2 of dummy1] (B) {$\begin{picture}(77,16)
\put(0,4){$\ytableaushort{{*(lightgray)\blank} {*(lightgray)\blank} {*(Dandelion)\bullet_{1_2}} {*(Dandelion)1_2}, {*(lightgray)\blank} {*(Dandelion)\bullet_{1_2}} {*(Dandelion)1_2}}$}
\put(5,-18){$1_1$}
\end{picture}$};
\node[below right=2 and 2 of dummy1] (C) {$\begin{picture}(77,16)
\put(0,4){$\ytableaushort{{*(lightgray)\blank} {*(lightgray)\blank} {*(Dandelion)\bullet_{1_2}} {*(Dandelion)1_2}, {*(lightgray)\blank} {1_1} {*(Dandelion)1_2}}$}
\put(5,-18){$1_1$}
\end{picture}$};
\node[right=2 of B] (D) {$\begin{picture}(77,16)
\put(0,4){$\ytableaushort{{*(lightgray)\blank} {*(lightgray)\blank} {1_2} {*(SkyBlue)\bullet_{2_1}}, {*(lightgray)\blank} {1_2} {*(Dandelion)\bullet_{2_1}}}$}
\put(5,-18){$1_1$}
\end{picture}$};
\node[right=2 of C] (E) {$\begin{picture}(77,16)
\put(0,4){$\ytableaushort{{*(lightgray)\blank} {*(lightgray)\blank} {1_2} {*(SkyBlue)\bullet_{2_1}}, {*(lightgray)\blank} {1_1} {*(Dandelion)\bullet_{2_1}}}$}
\put(5,-18){$1_1$}
\end{picture}$};
\node[right=2 of D] (F) {$\begin{picture}(77,16)
\put(0,4){$\ytableaushort{{*(lightgray)\blank} {*(lightgray)\blank} {1_2}, {*(lightgray)\blank} {1_2}}$}
\put(5,-18){$1_1$}
\end{picture}$};
\node[right=2 of E] (G) {$\begin{picture}(77,16)
\put(0,4){$\ytableaushort{{*(lightgray)\blank} {*(lightgray)\blank} {1_2}, {*(lightgray)\blank} {1_1}}$}
\put(5,-18){$1_1$}
\end{picture}$};
\path (A) edge[und] node[above](l){{\sf H8}/$\emptyset$/$\emptyset$} node[above=0.5 of l]{{\sf H7}/$\emptyset$/$\emptyset$} (dummy1);
\path (dummy1) edge[pil, bend right=20, shorten >=18pt] node[below right]{$1$} (B);
\path (dummy1) edge[pil, bend left=20, shorten >=7pt] node[below left]{$1 -\frac{t_2}{t_3}$} (C);
\path (B) edge[pil] node[above]{{\sf H5.3}/$\emptyset$/{\sf T2}} node[below]{$\frac{t_2}{t_3} \cdot \frac{t_4}{t_5}$} (D);
\path (C) edge[pil] node[above]{$\emptyset$/{\sf B3}/{\sf T2}} node[below]{$-1 \cdot \frac{t_4}{t_5}$} (E);
\path (D) edge[pil] node[above]{Delete $\bullet$'s} (F);
\path (E) edge[pil] node[above]{Delete $\bullet$'s} (G);
\end{tikzpicture}
$}\] \qed
\end{example}

\section{Ladders}\label{sec:ladders}
Let $U$ be a $\GG^+$-good tableau. Consider the boxes of $U$ containing $\bullet_{\GG^+}$ or unmarked $\GG$. This set decomposes into maximal edge-connected components, which we call {\bf ladders}.

\begin{example}
\[\begin{picture}(200,105)
\put(-30,95){$\ytableaushort{
{*(lightgray)\blank} {*(lightgray)\blank} {*(lightgray)\blank} {*(lightgray)\blank} {1_2},
{*(lightgray)\blank} {*(lightgray)\blank} {*(lightgray)\blank} {1_1} {2_2},
{*(lightgray)\blank} {*(lightgray)\blank} {*(lightgray)\blank} {*(Red) 2_1} {3_2},
{*(lightgray)\blank} {*(lightgray)\blank} {*(Red) 2_1} {*(Red)\bullet_{2_2}},
{*(lightgray)\blank} {*(SkyBlue) \bullet_{2_2}} {2_1^!},
{*(Dandelion) \bullet_{2_2}}
}$}
\put(90,70){This $2_2$-good tableau has three ladders; we}
\put(90,57){have given each ladder a separate color.}
\put(90, 44){(All virtual labels are depicted.)}
\put(313,0){$\qed$}
\put(-24,-2){$3_1$}
\put(-8,15){$\Scale[0.7]{\circled{3_1}}$}
\put(10,15){$\Scale[0.7]{\circled{3_1}}$}
\end{picture}\]
\end{example}

\begin{lemma}\label{lem:ladder_row_classification}
A row $r$ of a ladder $L$ is one of the following (edge labels other than $\GG$ and virtual labels are not shown):
\ytableausetup{boxsize=1.3em}
\[{\sf{(L1)} \ \ } \ytableaushort{\bullet} \  \ \
{\sf{(L2)}} \ \  \ytableaushort{\GG}  \ \ \
{\sf{(L3)}} \ \  \begin{picture}(15,13)
\put(0,0){$\ytableaushort{\bullet}$}
\put(5,12){$\GG$}
\end{picture} \ \ \
{{\sf{(L4)}}} \ \ \ytableaushort{\GG \bullet}
\]
\end{lemma}
\begin{proof}
By (G.2), at most one $\bullet_{\GG^+}$ occurs in each row. By 
Lemma~\ref{lemma:Gsoutheast}, at most one $\GG$ appears in each row. Thus
$r$ has at most two boxes. If it has one box, $r$ is clearly {\sf L1}, {\sf L2} or {\sf L3}.
If $r$ has two boxes, then it has one box label $\GG$ and one box label $\bullet_{\GG^+}$. Since the $\GG$ is not marked, it is West of the $\bullet_{\GG^+}$.
By (G.4) and (G.7), no edge label $\GG$ is possible in this two-box scenario. Thus {\sf L4} is the only two box possibility.
\end{proof}

\begin{lemma}
\label{lemma:laddersareshort}
A ladder $L$ is a short ribbon where each column with $2$ boxes is $\ytableaushort{\GG, {\bullet}}$.
\end{lemma}
\begin{proof}
In each column, there is at most one $\bullet_{\GG^+}$ by (G.2) and at most one $\GG$ by (G.4).  If the column consists of $\bullet_{\GG+}$ and $\GG$, then the $\GG$ is North of the $\bullet_{\GG^+}$, since otherwise the $\GG$ is marked. Therefore the columns are as described.

If $L$ has a $2\times 2$ subsquare the North box of each column must contain $\GG$, violating (G.3).
Each row has at most two boxes by Lemma~\ref{lem:ladder_row_classification}.
That $L$ is a skew shape is now immediate from the descriptions of $L$'s rows and columns.
\end{proof}

\begin{lemma}[Relative positioning of ladders]
\label{lem:ladders.sw.ne}
Suppose $U$ is $\GG^+$-good, and that $L, M$ are distinct ladders of $U$. Then, up to 
relabeling of the ladders, $L$ is entirely SouthWest of $M$ (that is, if $\bbb, \bbb'$ are boxes of $L, M$ respectively, then $\bbb$ is SouthWest of $\bbb'$).
\end{lemma}
\begin{proof}
Suppose not. There are three cases to consider:

\noindent
{\sf Case 1: ($\bbb\in L$ is NorthWest of $\bbb'\in M$):}
By definition, $\bbb$ and $\bbb'$ contain either $\bullet_{\GG^+}$ or $\GG$. 
By (G.2) and Lemmas~\ref{lemma:Gsoutheast}, \ref{lemma:Gsoutheastofbullet} and~\ref{lem:GandbulletG+}, we see that no combination of these choices is possible.

\noindent
{\sf Case 2: ($\bbb$ is North and in the same column as $\bbb'$):}
If $\bullet_{\GG^+}\in \bbb$ and  $\bullet_{\GG^+}\in \bbb'$,
we violate (G.2). If  $\bullet_{\GG^+}\in \bbb$ and  $\GG\in \bbb'$, then the latter would be marked.
Hence $\GG\in \bbb$. Since $\GG\in \bbb'$ or $\bullet_{\GG^+} \in \bbb'$, we have by (G.4) and (G.9) that $\bbb^\downarrow=\bbb'$ and so $\bbb,\bbb'$ are in the same ladder,
contradicting $L\neq M$. 

\noindent
{\sf Case 3: ($\bbb$ is West and in the same row as $\bbb'$):} 
By (G.2), at least one of $\bbb, \bbb'$ contains
$\GG$. By Lemma~\ref{lemma:Gsoutheast}, at least one of $\bbb, \bbb'$ contains $\bullet_{\GG^+}$. If $\GG \in \bbb$ and $\bullet_{\GG^+}\in \bbb'$, then by (G.3) and (G.9), $\bbb'=\bbb^\rightarrow$, contradicting $L\neq M$. If $\bullet_{\GG^+} \in \bbb$ and $\GG \in \bbb'$, then the latter is marked. 
\end{proof}

\section{Reverse genomic jeu de taquin}\label{sec:reversal}
Let $r$ be a ladder row in a $\GG^+$-good tableau $U$ and let $\x$ be the westmost box in $r$. We define
the {\bf reverse miniswap} operation $\revm$ on $r$. The cases below are labeled in accordance with the classification of Lemma~\ref{lem:ladder_row_classification}. Below, each $\bullet$ on the left
 of the ``$\mapsto$'' is a $\bullet_{\GG^+}$, while
on the right it is a $\bullet_\GG$.

\ytableausetup{boxsize=1.2em}
\noindent
{\sf (Case L1):} 

{\sf (Subcase L1.1: $\GG \in \x^\uparrow$):}
\[\begin{picture}(110,14)
\put(-120,0){$r = \ytableaushort{\bullet} \mapsto \revm(r) = \ytableaushort{\GG}$}
\end{picture}
\]

{\sf (Subcase L1.2: $\GG \notin \x^\uparrow$):} 
\[\begin{picture}(110,14)
\put(-120,0){$r = \ytableaushort{\bullet} \mapsto \revm(r) = \ytableaushort{\bullet}$}
\end{picture}
\]

\noindent
{\sf (Case L2):}  

{\sf (Subcase L2.1: $\bullet_{\GG^+} \in \x^\downarrow$ or $\GG^! \in \x^\downarrow$):} 
\[\begin{picture}(110,14)
\put(-120,0){$r = \ytableaushort{\GG} \mapsto \revm(r) = \ytableaushort{\bullet}$}
\end{picture}
\]

{\sf (Subcase L2.2: $\bullet_{\GG^+} \notin \x^\downarrow$, $\GG^! \notin \x^\downarrow$, $\GG^! \notin \x$, $\x$ contains the westmost $\GG$):}
\[\begin{picture}(110,14)
\put(-120,0){$r = \ytableaushort{\GG} \mapsto \revm(r) = \ytableaushort{\GG} + \begin{picture}(25,12)
\put(0,0){$\ytableaushort{\bullet}$.}
\put(4,-4){$\GG$}
\end{picture}$}
\end{picture}
\]

{\sf (Subcase L2.3: $\bullet_{\GG^+} \notin \x^\downarrow$, $\GG^! \notin \x^\downarrow$, $\GG^! \notin \x$, $\x$ does \emph{not} contain the westmost $\GG$):}
\[\begin{picture}(110,14)
\put(-120,0){$r = \ytableaushort{\GG} \mapsto \revm(r) = \ytableaushort{\GG} + \begin{picture}(25,12)
\put(0,0){$\ytableaushort{\bullet}$.}
\put(2,-4){$\Scale[.8]{\circled{\GG}}$}
\end{picture}$}
\end{picture}
\]

\noindent
{\sf (Case L3):}
\[\begin{picture}(110,14)
\put(-120,0){$r = \begin{picture}(18,12)
\put(0,0){$\ytableaushort{\bullet}$}
\put(4,12){$\GG$}
\end{picture} \mapsto \revm(r) = \begin{picture}(23,18)
\put(0,0){$\ytableaushort{\bullet}$.}
\put(3,-5){$\GG$}
\end{picture}$}
\end{picture}
\]

\noindent
{\sf (Case L4):} 

{\sf (Subcase L4.1:  $\GG^+ \in \underline{\x^\rightarrow}$ with $\family(\GG^+) = \family(\GG)$, and either $\bullet_{\GG^+} \in \x^\downarrow$ or $\GG^! \in \x^\downarrow$):}
\[\begin{picture}(110,14)
\put(-120,0){$r = \begin{picture}(32,12)
\put(0,0){$\ytableaushort{\GG \bullet}$}
\put(16,-4){$\GG^+$}
\end{picture} \mapsto \revm(r) = \begin{picture}(40,18)
\put(0,0){$\ytableaushort{\bullet {\GG^+}}$}
\end{picture}$}
\end{picture}
\]

{\sf (Subcase L4.2: $\GG^+ \in \underline{\x^\rightarrow}$ with $\family(\GG^+) = \family(\GG)$, $\bullet_{\GG^+} \notin \x^\downarrow$, $\GG^! \notin \x^\downarrow$ and $\x$ contains the westmost $\GG$):}
\[\begin{picture}(110,14)
\put(-120,0){$r = \begin{picture}(32,12)
\put(0,0){$\ytableaushort{\GG \bullet}$}
\put(16,-4){$\GG^+$}
\end{picture} \mapsto \revm(r) = \begin{picture}(40,18)
\put(0,0){$\ytableaushort{\bullet {\GG^+}}$}
\put(3,-4){$\GG$}
\end{picture}$}
\end{picture}
\]

{\sf (Subcase L4.3: $\GG^+ \!\in\! \underline{\x^\rightarrow}$ with $\family(\GG^+) = \family(\GG)$, $\bullet_{\GG^+} \!\notin\! \x^\downarrow$, $\GG^! \notin \x^\downarrow$ and $\x$ does \emph{not} contain the westmost $\GG$):}
\[\begin{picture}(110,14)
\put(-120,0){$r = \begin{picture}(32,12)
\put(0,0){$\ytableaushort{\GG \bullet}$}
\put(16,-4){$\GG^+$}
\end{picture} \mapsto \revm(r) = \begin{picture}(40,18)
\put(0,0){$\ytableaushort{\bullet {\GG^+}}$}
\put(2,-4){$\Scale[.8]{\circled{\GG}}$}
\end{picture}$}
\end{picture}
\]

\ytableausetup{boxsize=1.8em}
{\sf (Subcase L4.4: there is \emph{no} $\GG^+ \in \underline{\x^\rightarrow}$ with $\family(\GG^+) = \family(\GG)$,
and $\x$ contains the westmost $\GG$):}
Let $A$ be the labels in $\overline{\x}$, $Z = \{ \EE \in A \colon N_\GG = N_\EE\}$, $Z^\sharp = Z \cup \{\GG\}$, $\FF = \min Z^\sharp$, $A'' = Z^\sharp \backslash \{\FF\}$, and $A' = A \backslash Z$.
\[\begin{picture}(110,24)
\put(-120,0){$r = \begin{picture}(50,24)
\put(5,0){$\ytableaushort{\GG \bullet}$}
\put(11,17){$A$}
\end{picture} \mapsto \revm(r) =
 \begin{picture}(100,24)
\put(5,0){$\ytableaushort{\bullet \FF}$.}
\put(11,17){$A'$}
\put(33,-5){$A''$}
\end{picture}$}
\end{picture}
\]

{\sf (Subcase L4.5: there is \emph{no} $\GG^+ \in \underline{\x^\rightarrow}$ with $\family(\GG^+) = \family(\GG)$, and
 $\x$ does \emph{not} contain the westmost $\GG$):}
Let $A,Z, Z^\sharp, \FF$ and $A'$ be as in {\sf L4.4}; also let
$A''' = Z \backslash \{\FF\}$.
\[\begin{picture}(110,24)
\put(-120,0){$r = \begin{picture}(52,24)
\put(5,0){$\ytableaushort{\GG \bullet}$}
\put(11,17){$A$}
\end{picture} \mapsto \revm(r) =
 \begin{picture}(100,24)
\put(5,0){$\ytableaushort{\bullet \FF}$.}
\put(11,17){$A'$}
\put(28,-5){$\Scale[.6]{A''', \circled{\GG}}$}
\end{picture}$}
\end{picture}
\]

\begin{lemma}
Every ladder row falls into exactly one of the above cases. 
\end{lemma}
\begin{proof}
This is tautological, given Lemma~\ref{lem:ladder_row_classification}.
\end{proof}

\begin{lemma}
No $\revm$ affects an edge that is shared by two rows of the same ladder $L$.
\end{lemma}
\begin{proof}
No $\revm$ affects the upper (virtual) edge labels of the right box of a ladder row. Hence it suffices to analyze
those cases that affect the lower (virtual)  edge labels of the left box of a ladder row. These are {\sf L2.2}, {\sf L2.3},
{\sf L3}, {\sf L4.2} and {\sf L4.3}. In each case there can be no ladder row of $L$ below, by Lemma~\ref{lemma:laddersareshort}.
Hence that edge is not shared.
\end{proof}

Thus it makes sense to define $\revswap_{\GG^+}$ on a ladder $L$, by applying $\revm$ to each row of $L$ simultaneously (where the conditions on each $\revm$ refer to the original ladder $L$). 

\begin{lemma}
\label{lemma:laddercommute}
If $L_1, L_2$ are distinct ladders in a $\GG^+$-good tableau $U$, then applying $\revswap_{\GG^+}$ to $L_1$ commutes with applying $\revswap_{\GG^+}$ to $L_2$.
\end{lemma}
\begin{proof}
This follows, since by definition $L_1$ and $L_2$ do not share any edges.
\end{proof}

Lemma~\ref{lemma:laddercommute} permits us to define the {\bf reverse swap} $\revswap_{\GG^+}$ on a $\GG^+$-good tableau by applying $\revswap_{\GG^+}$ to all ladders (in arbitrary order). We extend this to a $\mathbb{Z}[t_1^{\pm 1}, \dots, t_n^{\pm 1}]$-linear operator. 

\begin{lemma}[Reverse swaps preserve content] 
\label{lemma:reversecontentpres}
If $U$ is $\GG^+$-good and of content $\mu$, then each $T \in \revswap_{\GG^+}(U)$
has content $\mu$. 
\end{lemma}
\begin{proof}
Let $\HH$ be a gene in $U$. We must show $\HH \in T$. Let $\ell$ be the westmost instance of $\HH$ in $U$. If $\ell$ is not part of a ladder,
$\HH$ appears in the same location in $T$ and we are done. Thus suppose $\ell$ is in a ladder row $r$. Consider the reverse miniswap applied to $r$. If it is anything but {\sf L2.1} or {\sf L4.1}, then there is an $\HH$ in that row of $T$. If it is {\sf L2.1} or {\sf L4.1}, let $\x$ be the box containing $\ell$. By definition, $U$ has $\bullet_{\HH^+} \in \x^\downarrow$ or $\HH^! \in \x^\downarrow$. In the former case, the miniswap applied at $\x^\downarrow$ is {\sf L1.1}, so $\HH$ appears in $\x^\downarrow$ in $T$. In the latter case, $\x^\downarrow$ is not in a ladder, so $\HH$ appears in $\x^\downarrow$ in $T$. 

Conversely suppose $\HH$ is not a gene in $U$. We must show it does not appear in $T$. If it appeared in $T$, it must be created by some miniswap. Clearly no miniswap but {\sf L1.1} could possibly introduce a new gene. But if we apply {\sf L1.1} at some box $\x$ of $U$, introducing $\HH \in \x$ in $T$, then $U$ has $\HH \in \x^\uparrow$ by definition, so $\HH$ was indeed a gene of $U$.
\end{proof}

We prove the following proposition in Appendix~\ref{sec:backwards_goodness_proof}.
\begin{proposition}[Reverse swaps preserve goodness]\label{prop:goodness_preservation_reverse}
If $U$ is $\GG^+$-good, each $T \in \revswap_{\GG^+}(U)$ is $\GG$-good.
\end{proposition}

\begin{lemma}
\label{lemma:veryusefulfact}
Let $T$ be a $\GG$-good tableau
and $U\in \swap_{\GG}(T)$. 
\begin{itemize}
\item[(I)] If $\lab_U(\x) = \GG$, then $\lab_T(\x) \in \{\bullet_{\GG}, \GG\}$.
\item[(II)] If $\lab_U(\x) = \bullet_{\GG^+}$, then $\lab_T(\x) \in \{ \bullet_{\GG}, \GG, \FF^!, \GG^+\}$.
\item[(III)] If $\lab_U(\x) = \GG^!$, then $\lab_T(\x) = \GG$.
\end{itemize}
\end{lemma}
\begin{proof}
By inspection of the miniswaps.
\end{proof}

\begin{lemma}
\label{lem:useful_fact_reverse}
Let $U$ be a $\GG^+$-good tableau and $T \in \revswap_{\GG^+}(U)$.
\begin{itemize}
\item[(I)] If $\lab_U(\x) = \GG^!$, then $\lab_T(\x) = \GG$.
\item[(II)] If $\lab_U(\x) = \GG$, then $\lab_T(\x) \in \{\GG, \bullet_\GG\}$.
\item[(III)] If $\lab_U(\x) = \bullet_{\GG^+}$, then $\lab_T(\x) \in \{\bullet_\GG, \GG, \GG^+, \FF^!\}$.
 If moreover $\lab_T(\x) = \GG^+$, then $\lab_T(\x^\leftarrow) = \bullet_\GG$, while if moreover $\lab_T(\x) = \FF^!$, then $N_\FF = N_\GG$, $\lab_T(\x^\leftarrow) = \bullet_\GG$ and either $\GG \in \underline{\x}$ or $\circled{\GG} \in \underline{\x}$.
\end{itemize}
\end{lemma}
\begin{proof}
By inspection of the reverse miniswaps.
\end{proof}

For a good tableau $T$ of shape $\nu/\lambda$, 
define a {\bf $T$-patch} of $\nu/\lambda$ as one of the following:
\begin{itemize}
\item[(Pat.1)] A row of a snake of $T$ (including both upper and lower edges of the row). 
\item[(Pat.2)] A box not in a snake (the box excludes the edges). 
\item[(Pat.3)] A horizontal edge not bounding a box of a snake.
\end{itemize}
Clearly, 
the set $\{P\}$ of $T$-patches covers $\nu/\lambda$. Given a tableau $W$ of shape $\nu/\lambda$, let
$W|_{P}$ be the tableau obtained by restricting $W$ to $P$.

\begin{proposition}
\label{prop:swap/revswap_inversion}
Let $T, U$ be good.
Then $U \in \swap_{\GG}(T)$ if and only if $T \in \revswap_{\GG^+}(U)$.
\end{proposition}
\begin{proof}
($\Rightarrow$)
Suppose $U \in \swap_\GG(T)$. We show $T \in \revswap_{\GG^+}(U)$.

\begin{claim}
\label{claim:patchb}
Every ladder row $r$ of $U$ is contained in a distinct $T$-patch. 
\end{claim}
\begin{proof}
Distinctness is clear. We now argue containment.
If $r$ has one box, containment is trivial.
Otherwise, $r$ has two boxes, and we are in case {\sf L4} of the ladder row classification of Lemma~\ref{lem:ladder_row_classification}. So, in $U$, each box of $r$ 
contains $\bullet_{\GG^+}$ or $\GG$. One considers all possibilities,
under Lemma~\ref{lemma:veryusefulfact}, for the entries in $T$ 
of the boxes of $r$. Since $T$ is good, 
these boxes of $T$ either form a row of a snake section or are
\ytableausetup{boxsize=1.3em}
$\!\! \ytableaushort{\GG {\GG^+}}$. We are done by (Pat.1) in the former case.
The latter case cannot occur, since 
by inspection of the miniswaps,
this cannot swap to {\sf L4}. 
\end{proof}

By the definitions, notice that $\revswap_{\GG^+}(U)\neq 0$. Moreover:

\begin{claim}
\label{claim:patcha}
For each $T$-patch $P$, there exists $W\in \revswap_{\GG^+}(U)$ such that $W|_P=T|_P$ (ignoring virtual labels). 
\end{claim}
\begin{proof}
If $P$ is type (Pat.2), then by definition $T|_P=U|_P$, since $P$ is not part of a snake. In particular
$U|_P$ does not contain $\GG$ or $\bullet_{\GG^+}$. So $U|_P$ is not part of a ladder of $U$. Hence for any $W\in \revswap_{\GG^+}(U)$, $W|_P=U|_P=T|_P$ as
desired.  

If $P$ is type (Pat.3), then 
$T|_P=U|_P$, since $P$ is not part of a snake. Moreover, by definition,
no box $\y$ bounded by the edge $P$ is part of a snake in $T$. 
Therefore, $\bullet_{\GG},\GG\not\in \y$ in $T$.
Hence $\bullet_{\GG^+},\GG\not\in \y$ in $U$. So $P$ does not bound a box of a ladder of $U$. Thus for any $W\in \revswap_{\GG^+}(U)$, $W|_P=U|_P=T|_P$.

Finally if $P$ is type (Pat.1), by inspection of the miniswaps,
combined with Claim~\ref{claim:patchb}, $U|_P$ contains at most one ladder row $r$,
and possibly a non-ladder box $\y$. Since $\revswap_{\GG^+}$ does not
affect $\y$, it suffices to indicate the reverse miniswap
on $r$ to give our desired $W|_P=T|_P$. We refer to the list of outputs described in Section~\ref{sec:swaps}.

\noindent
{\sf H1:} Use {\sf L2.2} or {\sf L3} respectively on the two $\m$ outputs.

\noindent
{\sf H2:}  Use {\sf L1.2} or {\sf L2.3} respectively on the two $\m$ outputs.

\noindent
{\sf H3:} Use {\sf L1.2}: By $T$'s (G.2) and (G.9) and Lemma~\ref{lemma:veryusefulfact}(I) applied to $T$, we have $\GG\notin \x^\uparrow$ in $U$. 

\noindent
{\sf H4:} This case does not arise, since here $U$ does not exist.

\noindent
{\sf H5.1:} Use {\sf L1.2}.

\noindent
{\sf H5.2:} For the first output, use {\sf L1.2}. For the second output, use {\sf L4.4} or {\sf L4.5}. We must show
in the latter cases that $Z=\emptyset$. Otherwise if $\EE\in Z$, then $\EE\in \overline{\x}$ in $T$. Since $N_\EE=N_\GG$
in both $T$ and $U$, this contradicts Lemma~\ref{lem:how_to_check_ballotness} for $T$.

\noindent
{\sf H5.3:} Use {\sf L4.4} or {\sf L4.5}. The argument that these apply is the same as for {\sf H5.2}.

\noindent
{\sf H6:}  Use {\sf L2.2} for the first output and {\sf L4.2} for the second. By Lemma~\ref{lemma:veryusefulfact}(II) and $T$'s (G.2) and (G.4), 
$\bullet_{\GG^+}\notin \x^\downarrow$;  by Lemma~\ref{lemma:veryusefulfact}(I) and $T$'s (G.2) or (G.4),
$\GG^!\not\in \x^\downarrow$; that the $\GG\in \x$ is westmost follows from $T$'s (G.7) and
Claim~\ref{lemma:child123} applied to $T$.

\noindent
{\sf H7:} Use {\sf L1.2} for the first output: By Lemma~\ref{lemma:veryusefulfact}(I) and $T$'s (G.2) and (G.9), $U$ has 
$\GG\notin \x^\uparrow$. Use {\sf L2.3} for the second output and {\sf L4.3} for the third: By Lemma~\ref{lemma:veryusefulfact}(II) and $T$'s (G.2) or (G.4),
$\bullet_{\GG^+}\notin \x^\downarrow$; by Lemma~\ref{lemma:veryusefulfact}(I) and $T$'s (G.2) or (G.4),
$\GG^!\notin \x^\downarrow$; that the $\GG\in \x$ is not westmost follows from $T$'s $\circled{\GG}\in \underline{\x}$.

\noindent
{\sf H8:} Use {\sf L1.2}: By $T$'s (G.2) and (G.9) and Lemma~\ref{lemma:veryusefulfact}(I), $U$ has $\GG\notin \x^\uparrow$. 

\noindent
{\sf H9:} Here $r$ does not exist.

\noindent
{\sf B1:} Use {\sf L2.2} or {\sf L2.3}: By Lemma~\ref{lemma:veryusefulfact}(II) and $T$'s 
(G.2) or (G.4), $\bullet_{\GG^+}\not\in \x^\downarrow$; by Lemma~\ref{lemma:veryusefulfact}(III) and $T$'s (G.4), $\GG^!\notin \x^\downarrow$. 

\noindent
{\sf B2:} Use {\sf L4.4} or {\sf L4.5}; applicability is as for {\sf H5.2}.

\noindent
{\sf B3:} If we are not in the bottom row, we may use {\sf L4.4} or {\sf L4.5} as for {\sf B2}.
Otherwise, use {\sf L1.1}.

\noindent
{\sf T1:} Use {\sf L2.1}: By Lemma~\ref{lem:piecesobservations}(IV, VII), $T$ has $\GG\in \x^\downarrow$, so the
hypothesis holds by inspection of the miniswaps. 

\noindent
{\sf T2:} Use {\sf L4.4} or {\sf L4.5}; applicability is as for {\sf H5.2}.

\noindent
{\sf T3:} Use {\sf L2.1} or {\sf L4.1}; applicability is as for {\sf T1}.

\noindent
{\sf T4.1:} Use {\sf L1.2}: By $T$'s (G.2) and (G.12) and Lemma~\ref{lemma:veryusefulfact}(I), $U$ has $\GG\notin \x^\uparrow$ 
and $\GG\notin \overline{\x}$.

\noindent
{\sf T4.2:} Use {\sf L1.2} on the first output; applicability is as for {\sf T4.1}. Use {\sf L4.4} on the second output; applicability is as for {\sf H5.2}.

\noindent
{\sf T4.3:} Use {\sf L4.4}; applicability is as for {\sf H5.2}.

\noindent
{\sf T5:} Use {\sf L4.5}; applicability is as for {\sf H5.2}.

\noindent
{\sf T6:} This case does not arise, since here $U$ does not exist.
\end{proof}

By definition, $\revswap_{\GG^+}(U)$ is obtained by acting on ladder rows of $U$ 
independently. By Claim~\ref{claim:patcha}, it follows that $\revswap_{\GG^+}(U)$
is also obtained by acting on the $T$-patches of $U$ independently. Thus
($\Leftarrow$) holds by Claim~\ref{claim:patchb}.

\noindent
$(\Leftarrow)$ Suppose $T\in \revswap_{\GG^+}(U)$. We 
show $U$ is in $\swap_\GG(T)$. 

Recall $\swap_{\GG}(T)$ is a formal sum, given by independently replacing each snake section in each prescribed way. 
Trvially, by (Pat.1), each snake section is a union of $T$-patches. 
Moreover, if a snake section $\sigma$ consists of more than one $T$-patch, then $\sigma$ is a $\body$ with at least two rows, and hence either 
{\sf B2} or {\sf B3}. Therefore $\m(\sigma)$ has a unique output in this case.
Since $\swap_\GG$ acts trivially on the $T$-patches of types (Pat.2) 
and (Pat.3), by Lemma~\ref{lem:miniswap_commutation}, 
it follows that 
$\swap_\GG(T)$ is also given by acting independently 
on the $T$-patches of $T$. It remains to show that locally at $P$, we may swap $T|_P$ to obtain $U|_P$.

To make these local verifications, we use:
\begin{claim}\label{claim:newclaim}
\gap
\begin{itemize}
\item[(I)] Every ladder row of $U$ sits in a distinct $T$-patch of type (Pat.1). 
\item[(II)] Every $T$-patch $P$ of type (Pat.1) not coming from an {\sf H9} snake section, contains a ladder row of $U$.
\end{itemize}
\end{claim}
\begin{proof}
(I): By Lemma~\ref{lem:useful_fact_reverse}, every ladder row of $U$ is contained in a $T$-patch of type (Pat.1). Consider a $T$-patch $P$ of type (Pat.1); $P$ consists of at most two boxes. If $P$ does not consist of two boxes, clearly at most one ladder row of $U$ can be contained in it. If $P$ consists of two boxes, they are joined by a vertical edge. Since distinct ladder rows do not share a vertical edge, it follows that distinct ladder rows of $U$ are contained in distinct $T$-patches.

(II): By inspection of the reverse miniswaps.
\end{proof}

If $P$ is type (Pat.2) or (Pat.3), 
then by Claim~\ref{claim:newclaim}, $P$ does not intersect any ladder row of $U$. Thus $T|_P=U|_P$. By definition, $P$ is not part of any snake in $T$. Hence for any $V\in \swap_\GG(T)$, $V|_P=T|_P=U|_P$ as desired.  

Finally suppose $P$ is a patch of type (Pat.1). If it comes from an {\sf H9} snake section, then $V\in \swap_\GG(T)$, $V|_P=T|_P=U|_P$. Otherwise, by Claim~\ref{claim:newclaim}, $P$ contains a unique ladder row in $U$. We consider each ladder row type 
in turn and indicate the miniswaps
on $T|_P$ that give our desired $V|_P=U|_P$. We refer to the list of outputs described at the beginning of Section~\ref{sec:reversal}. The following case analysis completes the proof of $(\Rightarrow)$.

\noindent
{\sf L1.1:} Use {\sf B3}: Since $\lab_U(\x^\uparrow)=\GG$,
we apply at $\x^\uparrow$ either {\sf L2.1}, {\sf L4.1}, {\sf L4.4} or {\sf L4.5}. In each case
$\lab_T(\x^\uparrow)=\bullet_{\GG^+}$. Hence $\x$ and $\x^\uparrow$ are part of
the southmost two rows of a snake of $T$. We claim $\x^\leftarrow$ is not part of this snake. Note that by assumption $\x^\leftarrow$ is not part
of any ladder of $U$. Thus $\lab_U(\x^\leftarrow)=\lab_T(\x^\leftarrow)$ and
$\lab_T(\x^\leftarrow)\notin \{\bullet_\GG,\GG\}$. If $\x^\leftarrow$ is part of $\x$'s snake in $T$, then $\lab_T(\x^\leftarrow)=\FF^!\prec \GG$ and southeast of
some $\bullet_\GG$. Hence in $U$, $\x^\leftarrow$ is southeast of
some $\bullet_{\GG^+}$; this contradicts $U$'s (G.2) in view of $U$'s $\bullet_\GG\in \x$. Thus $\x$ is the unique box of the southmost row of its snake and by Definition-Lemma~\ref{def-lem:snake_classification}, it is
the southmost row of a {\sf B3} snake section.

\noindent
{\sf L1.2:} Use {\sf H2}, {\sf H3}, {\sf H7}, {\sf H8}, {\sf T4.1} or {\sf T4.2}: 
Since $\lab_{T}(\x^\downarrow)=\GG$, $\lab_U(\x^\downarrow)\in \{\bullet_{\GG^+},\GG\}$.
Hence by Lemma~\ref{lemma:laddersareshort}, $\x^\downarrow$ is not in $\x$'s snake in $T$.
Since $\lab_T(\x)=\bullet_\GG$, $\x^\uparrow$ is not in $\x$'s snake in $T$. Hence $\x$ is in a one-row snake. Since {\sf L1.2} applies, $\lab_U(\x^\rightarrow) \neq \GG$, so $\lab_T(\x^\rightarrow)\neq \GG$. 
Thus $\x$'s snake in $T$ is type (ii), (iv) or (vi)
in Definition-Lemma~\ref{def-lem:snake_classification}(III).
Type (ii) uses {\sf H2} or {\sf H3}; type (iv) uses {\sf H7} or {\sf H8};
type (vi) uses {\sf T4.1} or {\sf T4.2}.

\noindent
{\sf L2.1:} Use {\sf T1} or {\sf T3}: By assumption, $\lab_U(\x^\downarrow) \in \{\bullet_{\GG^+}, \GG^!\}$. Hence by inspection of the reverse miniswaps, $\lab_T(\x^\downarrow)=\GG$. Since $\lab_T(\x) = \bullet_\GG$, $\x^\uparrow$ is not in $\x$'s snake. Hence by Definition-Lemma~\ref{def-lem:snake_classification}(I,II), $\x$ is in its snake's {\tt tail}. 
By $T$'s (G.3), $\lab_T(\x^\rightarrow) \succ \GG$, so $\lab_U(\x^\rightarrow)\neq \FF^!$. Thus either
{\sf T1} or {\sf T3} applies.

\noindent
{\sf L2.2:} Use {\sf B1} for the first output. By assumption and $U$'s (G.9), $U$ has no $\bullet_{\GG^+}$ adjacent to $\x$. Moreover by $U$'s (G.4), no box adjacent to $\x$ is in any ladder. Hence $T$ has no $\bullet_\GG$ adjacent to $\x$. If $\FF^! \in \x^\leftarrow$ in $T$, then (possibly marked) $\FF \in \x^\leftarrow$ in $U$. If $\lab_U(\x^\leftarrow) = \FF^!$, then we contradict unmarked $\GG \in \x$ in $U$. If $\lab_U(\x^\leftarrow)$ is unmarked, then $U$ has no $\bullet_{\GG^+}$ northwest of $\x^\leftarrow$. By $U$'s (G.3) and (G.4), $U$ has no $\GG$ northwest of $\x^\leftarrow$. But since $\FF^! \in \x^\leftarrow$ in $T$, $T$ has a $\bullet_{\GG^+}$ northwest of $\x^\leftarrow$. Hence by Lemma~\ref{lemma:child123}, $U$ has a $\bullet_{\GG^+}$ or $\GG$ northwest of $\x^\leftarrow$, a contradiction.

Use {\sf H1} or {\sf H6} for the second output. Since $\x^\rightarrow$ is not in any ladder of $U$, $\lab_U(\x^\rightarrow) = \lab_T(\x^\rightarrow)$. Moreover by $U$'s (G.3), $\lab_U(\x^\rightarrow) \succ \GG$, so $\lab_T(\x^\rightarrow) \succ \GG$. If $\lab_T(\x^\rightarrow) = \GG^+$, {\sf H6} applies. Otherwise, {\sf H1} applies. 

\noindent
{\sf L2.3:} Use {\sf B1} for the first output; applicability is as for the first output of {\sf L2.2}.
Use {\sf H2} or {\sf H7} for the second output. Since $\x^\rightarrow$ is not in any ladder of $U$, $\lab_U(\x^\rightarrow) = \lab_T(\x^\rightarrow)$. Moreover by $U$'s (G.3), $\lab_U(\x^\rightarrow) \succ \GG$, so $\lab_T(\x^\rightarrow) \succ \GG$. If $\lab_T(\x^\rightarrow) = \GG^+$, {\sf H7} applies. Otherwise, {\sf H2} applies. 

\noindent
{\sf L3:} Use {\sf H1}. By $U$'s (G.12), $\lab_U(\x^\rightarrow) \notin \{ \GG, \GG^+ \}$. Moreover by $U$'s (G.13) and (G.12), $\lab_U(\x^\rightarrow)$ is not marked, so $\lab_U(\x^\rightarrow) \succeq \GG^+$. Thus $\lab_U(\x^\rightarrow) \succ \GG^+$.  Since $\x^\rightarrow$ is not in any ladder of $U$, $\lab_T(\x^\rightarrow) \succ \GG^+$.

\noindent
{\sf L4.1:} Use {\sf T3}. By inspection of the reverse miniswaps, $T$ has $\GG \in \x^\downarrow$. Hence $\x$'s snake in $T$ has at least two rows. Hence $\x$ is part of its snake's $\tail$.

\noindent
{\sf L4.2:} Use {\sf H6}.

\noindent
{\sf L4.3:} Use {\sf H7}.

\noindent
{\sf L4.4:} If $Z \neq \emptyset$, use {\sf T4.2} or {\sf T4.3}. Otherwise use {\sf H5.3}, {\sf B2}, {\sf B3} or {\sf T2}. If $Z \neq \emptyset$, some {\sf T4} applies. If it is {\sf T4.1}, $T$ has $\HH \in \underline{\x^\rightarrow}$ with $\family(\HH) = \family(\GG)+1$ and $N_\HH = N_\GG$. Hence $U$ also has $\HH \in \underline{\x^\rightarrow}$, contradicting Lemma~\ref{lem:how_to_check_ballotness} for $U$. If $Z = \emptyset$, there is nothing to check.

\noindent
{\sf L4.5:} If $Z \neq \emptyset$, use {\sf T5}. Otherwise use {\sf H5.3}, {\sf B2}, {\sf B3} or {\sf T2}.
\end{proof}

The following proposition characterizes good tableaux in terms of forward
swapping.

\begin{proposition}
A tableau $U$ is $\GG$-good if and only if $U\in 
\swap_{\GG^-}\circ\cdots \circ \swap_{1_2}\circ \swap_{1_1}(T^{(1_1)})$
for some bundled tableau $T$ and choice of inner corners of $T$ to initially
place $\bullet_{1_1}$'s in.
\end{proposition}
\begin{proof}
($\Rightarrow$) Given a $\GG$-good tableau $U$, let $T^{(1_1)}$ be any 
tableau appearing in $\revswap_{1_2}\circ \cdots \circ
\revswap_{\GG^-}\circ \revswap_{\GG}(U)$. By Proposition~\ref{prop:goodness_preservation_reverse}, $T^{(1_1)}$ is a $1_1$-good tableau. By $T^{(1_1)}$'s (G.2)
and (G.9), the $\bullet_{1_1}$'s of $T^{(1_1)}$ are at inner corners and
there is no genetic label northwest of a $\bullet_{1_1}$. Let $T$ be obtained
by removing the $\bullet_{1_1}$'s of $T^{(1_1)}$. Then it is clear $T$ is a bundled tableau. Now  $U\in 
\swap_{\GG^-}\circ\cdots \circ \swap_{1_2}\circ \swap_{1_1}(T^{(1_1)})$
holds by Propositions~\ref{prop:goodness_preservation} and~\ref{prop:swap/revswap_inversion}.

($\Leftarrow$) Immediate from 
Lemma~\ref{lem:adding_bullets_ok} and
Proposition~\ref{prop:goodness_preservation}.
\end{proof}

\section{The reversal tree}\label{sec:reversal_tree}
\label{sec:walkways}

\subsection{Walkways}
An {\bf $i$-walkway} $W$ in an $(i+1)_1$-good tableau $T$ is an edge-connected component of the collection of boxes $\x$ in $T$ 
such that:
\begin{itemize}
\item[(W.1)] $\bullet_{(i+1)_1}\in \x$; or
\item[(W.2)] $i_k\in \x$ and $\x$ is not southeast of a $\bullet_{(i+1)_1}$ 
(equivalently, $i_k\in \x$ is not marked). 
\end{itemize}

\begin{lemma}[Structure of an $i$-walkway]\label{lem:walkway_structure}
Let $W$ be an $i$-walkway.
\begin{itemize}
\item[(I)] Each column $c$ of $W$ has at most two boxes; if $c$ has two boxes, the southern box contains $\bullet_{(i+1)_1}$.
\item[(II)] $W$ has no $2\times 2$ subsquare. 
\item[(III)] $W$ is an edge-connected skew shape. 
\item[(IV)] The $\bullet_{(i+1)_1}$'s are at outer corners of $W$. 
\item[(V)] The box and upper edge labels of family $i$
form a $\prec$-interval in the set of genes.
\end{itemize}
\end{lemma}

Therefore, each $i$-walkway looks 
like:
 \[\ytableausetup{boxsize=0.9em}
\ytableaushort{
\none \none \none \none \none \none \none \star \star { \star} { \star} { \blank} ,
\none \none \star \star { \star} {\star} \star \bullet, { \star} \star \bullet }.
\]
where each $\star$ is a genetic label and the blank box contains either $\bullet_{(i+1)_1}$ or a genetic label.

\noindent
\emph{Proof of Lemma~\ref{lem:walkway_structure}:}
(I): By (G.2), at most one box of $c$ comes from (W.1). 
By (G.4), at most one box of $c$ comes from (W.2). Thus the first
assertion of (I) holds. The second assertion holds by (W.2).

(II): Suppose $W$ contains a $2\times 2$ subsquare. Then the two southern boxes of the subsquare
contain $\bullet_{(i+1)_1}$'s by (I), contradicting (G.2).  

(III): $W$ is edge-connected by definition. In view of (II), it remains to show there are no two boxes $\y, \z$ of $W$
with $\y$ NorthWest of $\z$. Suppose otherwise. 
By (G.2), at least one of $\y, \z$ contains a genetic label.
If $\bullet_{(i+1)_1}\in\y$ and $i_k\in \z$, we violate (W.2). If $\bullet_{(i+1)_1} \in \z$ and $i_k \in \y$, consider the box $\bbb$ in $\y$'s column and $\z$'s row. By (G.2), $\bbb$ contains a genetic label. By (G.4), $\lab(\bbb) > i_k$. Since $\bullet_{(i+1)_1}\in \z$, this contradicts (G.9).  
Finally, if $i_k \in \y$ and $i_h \in \z$, then we contradict (G.12).

(IV): Immediate from (W.2) and (G.2). 

(V): By the edge-connectedness of $W$ we know that $W$ occupies
consecutive columns. Thus we are done by (G.4)--(G.6).\qed

\subsection{Walkway reversal}
Let $U \in B_{\lambda, \mu}^\alpha$ for some $\alpha \in \{\nu\} \cup \nu^-$. 
Obtain $U^{(0)}$ from $U$ by
placing $\bullet_{(\ell(\mu)+1)_1}$ in each box of $\nu/\alpha$. The root of the {\bf reversal tree} ${\mathfrak T}_U$ is $U^{(0)}$. The children $\{U^{(1)}\}$ of $U^{(0)}$ are the tableaux in the formal sum 
$\revswap_{\ell(\mu)_1^+} \circ \dots \circ \revswap_{(\ell(\mu)+1)_1}(U^{(0)})$. 
By Proposition~\ref{prop:goodness_preservation_reverse}, each $U^{(1)}$ is 
$\ell(\mu)_1$-good. We define the children $\{U^{(2)}\}$ of a $U^{(1)}$ by reverse swapping successively through labels of family $\ell(\mu)-1$, etc. Similarly, all tableaux thus obtained are also good. (A tableau may have a copy of itself as a child; this occurs only if $U^{(0)}$ has no $\bullet_{(\ell(\mu)+1)_1}$'s.)
After $\ell(\mu)-i$ steps, a descendant $U'=U^{(\ell(\mu)-i)}$ is an $(i+1)_1$-good tableau. 

\begin{lemma}
\label{lem:connection_walkways_to_ladders}
Let $U'$ be an $(i+1)_1$-good tableau. If $\ell$ is a box or edge label that is not in an $i$-walkway,
then $\ell$ appears in the same location in every $T\in\revswap_{i_1^+} \circ \dots \circ \revswap_{(i+1)_1}(U')$.   
\end{lemma}
\begin{proof}
The case analysis is as follows:

\noindent
{\sf Case 1: ($\ell\in \x$ is a box label in $U'$)}: 

\noindent
{\sf Subcase 1.1: ($\family(\ell)\neq i$)}: 
During the reversal process $\revswap_{i_1^+} \circ \dots \circ \revswap_{(i+1)_1}$, the label $\ell$ is never part of any ladder 
consisting of $\HH$ and $\bullet_{\HH^+}$ where $\HH\in \{i_1,\ldots, i_{\mu_i}\}$. Thus $\revswap_{\HH^+}$ does not move $\ell$. 

\noindent
{\sf Subcase 1.2: ($\family(\ell)= i$):} Since $\x$ is not part of an
$i$-walkway, by (W.2) it is southeast of a $\bullet_{(i+1)_1}$ in $U'$. By inspection of the reverse miniswaps, this remains true for each tableau $V$ 
appearing in the reversal process $\revswap_{i_1^+} \circ \dots \circ \revswap_{(i+1)_1}$. The box $\x$ is never part of a ladder during this process, for when we 
apply $\revswap_{\HH^+}$, where $\HH$ is $\ell$'s gene,
$\bullet_{\HH^+}$ is northwest of $\x$ and so $\ell^!\in \x$. 
The case then follows.

\noindent
{\sf Case 2: ($\ell$ is an edge label in $U'$)}: Let $\x$ and $\x^\downarrow$ be the boxes adjacent to the edge.

\noindent
{\sf Subcase 2.1: ($\x$ and $\x^\downarrow$ do not contain a label of family $i$ in $U'$):} 
As above, $\x$ and $\x^\downarrow$ are not part of a ladder
consisting of $\HH$ and $\bullet_{\HH^+}$, where $\HH\in \{i_1,\ldots, i_{\mu_i}\}$. Hence neither is the $\ell$ in question, and
so this $\ell$ remains fixed throughout the reversal process. 

\noindent
{\sf Subcase 2.2: ($\x$ or $\x^\downarrow$ contains a label $\HH$ of family $i$ in $U'$):} 
By (G.4), at most one of $\x$ or $\x^\downarrow$ contains such a label. Without loss of generality, 
suppose it is $\x$ (the argument in the other case is the same). Since $\ell \in \underline{\x}$ is not part of an $i$-walkway,
neither is $\x$. By the arguments of {\sf Subcase 1.2}, $\x$ is never part of a ladder, since $\HH^! \in \x$. Thus $\underline{\x}$ is unchanged.
\end{proof}

Consider an $i$-walkway $W$ of $U'$. 
By Lemma~\ref{lem:walkway_structure}(V), the genes of family $i$ in $W$ form an interval; let it be $(w_1, \dots, w_n)$ in increasing $\prec$-order. 

\begin{lemma}[Characterization of one-row walkway reversals]
\label{lemma:onerowreversalcharacterization}
Let $W$ be a $1$-row $i$-walkway in an $(i+1)_1$-good tableau $U'$. 
Let $\aaa$ and $\z$ be the westmost and eastmost boxes of $W$, respectively. Consider the region $\mathcal{R}$ occupied by $W$.
\begin{itemize}
\item[(I)] Suppose $U'$ has $\bullet_{(i+1)_1} \in \z$ and \emph{no} label of family $i$
in ${\overline \z}$. Then there exists a filling $R$ of ${\mathcal R}$ with  $\bullet_{i_1}\in \aaa$ and $w_1 \notin \underline{\aaa}$ such that for any $V
\in \revswap_{i_1^+} \circ \dots \circ \revswap_{(i+1)_1}(U')$, $V|_\mathcal{R}= R$.
\item[(II)] Suppose $U'$ has $\bullet_{(i+1)_1} \in {\sf z}$ and a label of family
$i$ in ${\overline \z}$. Then there exists a filling $R$ of ${\mathcal R}$ with $\bullet_{i_1}\in \aaa$ and 
either $w_1 \in \underline{\aaa}$ or $\circled{w_1} \in {\underline \aaa}$ such that for any $V
\in \revswap_{i_1^+} \circ \dots \circ \revswap_{(i+1)_1}(U')$, $V|_\mathcal{R}= R$.
\item[(III)] Suppose $U'$ has a label of family $i$ in $\z$.
Then there exist two fillings $R, R'$ of ${\mathcal R}$ such that
\begin{itemize}
\item[(i)]  $R$ has $w_1\in \aaa$;
\item[(ii)] $R'$ has $\bullet_{i_1}\in \aaa$ and either $w_1 \in \underline{\aaa}$ or $\circled{w_1} 
\in \underline\aaa$;
\item[(iii)] $R$ and $R'$ are otherwise identical; and
\item[(iv)] for any $V
\in \revswap_{i_1^+} \circ \dots \circ \revswap_{(i+1)_1}(U')$, $V|_\mathcal{R}\in\{R,R'\}$.
\end{itemize}
\end{itemize}
\end{lemma}
\begin{proof}
We argue (I)--(III) separately, by induction on the number of boxes of $W$. The base cases (where $W$ consists of a single box $\aaa=\z$) are clear by Lemma~\ref{lem:ladder_row_classification}
and inspection of the reverse miniswaps. 
Assume $W$ has at least two boxes and
let ${\overline W}$ be $W$ with $\aaa$ removed.

(I): By induction, ${\overline W}$ reverses uniquely to some ${\overline R}$,
which has a $\bullet_{w_2}\in\aaa^\rightarrow$ and $w_2 \notin \underline{\aaa^\rightarrow}$. (By a technical modification of the hypotheses, we may apply the inductive hypothesis to this partial walkway here and below.) This extends uniquely by {\sf L4.4} or {\sf L4.5} (followed by some number of
applications of {\sf L1.2}) to an $R$ with the claimed properties. 

(II): The unique reversal ${\overline R}$ of ${\overline W}$ has a $\bullet_{w_2}\in\aaa^\rightarrow$ and $w_2 \in \underline{\aaa^\rightarrow}$. (By (V.2), $\circled{w_2} \notin \underline{\aaa^\rightarrow}$.) We obtain the desired unique reversal $R$ by applying {\sf L4.2} or {\sf L4.3} to $\{\aaa, \aaa^\rightarrow\}$ in $\overline{R} \cup \{ \aaa \}$.  

(III): There are precisely two reversals of 
${\overline W}$: ${\overline R}$ and ${\overline R}'$.
The former reversal has $w_2\in \aaa^\rightarrow$, while the latter
has $\bullet_{w_2}\in \aaa^\rightarrow$ and $w_2\in \underline{\aaa^\rightarrow}$. (By (V.2),
$\circled{w_2}\notin \underline{\aaa^\rightarrow}$.)  Applying {\sf L4.2} or {\sf L4.3} (as appropriate) to
$\{\aaa, \aaa^\rightarrow\}$ in ${\overline R}'\cup \{\aaa\}$
returns $R'$ as described. Applying
{\sf L2.2} or {\sf L2.3} (as appropriate) to $\aaa$ in ${\overline R}\cup \{\aaa\}$ returns
precisely $R$ and $R'$. (We apply {\sf L4.2} to ${\overline R}'\cup \{\aaa\}$ exactly when we apply {\sf L2.2} to ${\overline R}\cup \{\aaa\}$.) 
\end{proof}

\begin{lemma}[Characterization of multirow walkway reversals]
\label{lemma:tworowreversalcharacterization}
Let $W$ be an $i$-walkway \emph{with at least two rows} in an $(i+1)_1$-good tableau $U'$. Let $\aaa$ and $\z$ be the westmost and eastmost boxes, respectively, in its southmost row. Thus $\bullet_{(i+1)_1}\in \z$.
Let $\mathcal{R}$ be the region occupied by 
$W$.
\begin{itemize}
\item[(I)] Suppose $\aaa = \z$. Then there exists a filling $R$ of ${\mathcal R}$ with $w_1 \in \aaa$ such that for any $V \in \revswap_{i_1^+} \circ \dots \circ \revswap_{(i+1)_1}(U')$, $V|_\mathcal{R} = R$.
\item[(II)] Suppose $\aaa \neq \z$ and $\lab_W(\z^\leftarrow) = \lab_W(\z^\uparrow)$. Then there exists a filling $R$ of $\mathcal{R}$ with $\bullet_{i_1}\in \aaa$ and no label of family $i$ on $\underline{\aaa}$ such that for any $V \in \revswap_{i_1^+} \circ \dots \circ \revswap_{(i+1)_1}(U')$, $V|_\mathcal{R} = R$. 
\item[(III)] Suppose $\aaa \neq \z$ and $\lab_W(\z^\leftarrow) \neq \lab_W(\z^\uparrow)$. Then there exist two fillings $R, R'$ of $\mathcal{R}$ such that 
\begin{itemize}
\item[(i)] $R$ has $w_1 \in \aaa$;
\item[(ii)] $R'$ has $\bullet_{i_1} \in \aaa$ and either $w_1 \in \underline{\aaa}$ or $\circled{w_1} \in \underline{\aaa}$;
\item[(iii)] $R$ and $R'$ are otherwise identical; and
\item[(iv)] for any $V \in \revswap_{i_1^+} \circ \dots \circ \revswap_{(i+1)_1}(U')$, $V|_\mathcal{R} \in \{ R, R' \}$.
\end{itemize}
\end{itemize} 
\end{lemma}
\begin{proof}
\noindent
(I): Let $\overline W$ be $W$ with the 
two boxes in the westmost column of $W$ removed. If $\overline{W} = \emptyset$, then $W=\{\bullet_{(i+1)_1} \in \z, w_1\in\z^\uparrow\}$;
here we obtain the desired result by use of {\sf L1.1}
and {\sf L2.1}. 
Hence assume $\overline{W} \neq \emptyset$. Clearly, 
\begin{equation}
\label{eqn:wewe1233}
\lab_W(\z^{\uparrow\rightarrow}) \in \{ w_2, \bullet_{(i+1)_1} \}.
\end{equation}
Depending on whether
$\overline W$ has multiple rows, by induction or by Lemma~\ref{lemma:onerowreversalcharacterization},
there are at most two reversals of ${\overline W}$.

\noindent
{\sf Case 1: ($\overline{W}$ has a unique reversal $\overline{R}$)}: By (\ref{eqn:wewe1233}) and induction/Lemma~\ref{lemma:onerowreversalcharacterization}, we
have two scenarios possible:

\noindent
{\sf Subcase 1.1: ($\overline{R}$ has $\bullet_{w_2} \in \z^{\uparrow\rightarrow}$ and no labels of family $i$ appear
on $\underline{\z^{\uparrow\rightarrow}}$):} Here we extend to a unique reversal of $W$ by applying {\sf L4.4} or {\sf L4.5} at $\z^\uparrow$ and {\sf L1.1} at $\z$. This results in 
$w_1\in \z=\aaa$. 

\noindent
{\sf Subcase 1.2: ($\overline{R}$ has
$w_2\in \z^{\uparrow\rightarrow}$):} We extend to a unique reversal of $W$ by applying {\sf L2.1} at $\z^\uparrow$ and {\sf L1.1} at $\z$. This results in 
$w_1\in \z=\aaa$, as desired.

\noindent
{\sf Case 2: ($\overline{W}$ has two reversals $\overline R$ and ${\overline R}'$):} By (\ref{eqn:wewe1233}) and 
induction/Lemma~\ref{lemma:onerowreversalcharacterization},
${\overline R}$ and ${\overline R}'$ differ only in $\z^{\uparrow\rightarrow}$: $\overline R$ has $w_2\in \z^{\uparrow\rightarrow}$ whereas $\overline R'$ has $\bullet_{w_2} \in \z^{\uparrow\rightarrow}$ and 
$w_2\in \underline{\z^{\uparrow\rightarrow}}$. 
By {\sf L2.1} and {\sf L1.1} in the $\overline{R}$ case and by {\sf L4.1} and {\sf L1.1} in the $\overline{R}'$ case, both extend to the same reversal $R$ of $W$; here $R$ has $w_1 \in \z = \aaa$, as claimed.

(II): We have some cases.

\noindent
{\sf Case 1: (The southmost row of $W$ has exactly two boxes $\{\aaa=\z^\leftarrow, \z\}$):}
Let ${\overline W}$ be $W$ with $\{\aaa, \z, \z^\uparrow\}$ removed. If $\overline W$ is empty, the result is clear, so we may assume otherwise. Thus (\ref{eqn:wewe1233}) still holds. Depending on whether
$\overline W$ has multiple rows or not, either by induction or by Lemma~\ref{lemma:onerowreversalcharacterization},
it follows there are at most two reversals of ${\overline W}$. 

\noindent
{\sf Subcase 1.1: ($\overline{W}$ has a unique reversal ${\overline R}$):} 
By (\ref{eqn:wewe1233}) 
and induction/Lemma~\ref{lemma:onerowreversalcharacterization}, two scenarios are possible:

\noindent
{\sf Subcase 1.1.1: ($\overline{R}$ has $\bullet_{w_2} \in \z^{\uparrow\rightarrow}$ and no label of family $i$
on $\underline{\z^{\uparrow\rightarrow}}$):} 
We extend to a unique reversal $R$ of $W$ by applying {\sf L4.5} at $\{\z^\uparrow, \z^{\uparrow\rightarrow} \}$ and either {\sf L4.4} or {\sf L4.5} (as required) at $\{\aaa, \z\}$; $R$ has  $\bullet_{w_1} \in \aaa$ and no label of family $i$ on $\underline{\aaa}$. 

\noindent
{\sf Subcase 1.1.2: ($w_2\in \z^{\uparrow\rightarrow}$):} We extend to a unique reversal $R$ of $W$ by applying {\sf L2.1} at $\z^\uparrow$ and either {\sf L4.4} or {\sf L4.5} (as required) at $\z$. This again results in 
$\bullet_{w_1} \in \aaa$ and no label of family $i$ on $\underline{\aaa}$.

\noindent
{\sf Subcase 1.2: ($\overline{W}$ has two reversals $\overline R$ and $\overline R'$):} By (\ref{eqn:wewe1233}) and 
induction/Lemma~\ref{lemma:onerowreversalcharacterization},
${\overline R}$ and ${\overline R}'$ differ only in $\z^{\uparrow\rightarrow}$: $\overline R$ has a $w_2\in \z^{\uparrow\rightarrow}$ whereas $\overline R'$ has a 
$\bullet_{w_2} \in \z^{\uparrow\rightarrow}$ and 
$w_2\in \underline{\z^{\uparrow\rightarrow}}$. 
 By {\sf L2.1} and {\sf L4.4} or {\sf L4.5} in the $\overline{R}$ case and by {\sf L4.1} and {\sf L4.4} or {\sf L4.5} in the $\overline{R}'$ case, both extend to the same reversal $R$ of $W$. $R$ has $\bullet_{w_1} \in \z = \aaa$.
 
In each of the Subcases above, we are done after applying
a sequence of {\sf L1.2}'s at $\aaa$. 

\noindent
{\sf Case 2: (The southmost row of $W$ contains at least three boxes):}
Let $\overline W$ be $W$ with $\aaa$ removed. By induction, ${\overline W}$ has a unique
reversal ${\overline R}$ with $\bullet_{w_2}$ in $\aaa^\rightarrow$ and no label of of family $i$ on $\underline{\aaa^\rightarrow}$. 
Now we uniquely extend ${\overline R}$ to a reversal $R$ of $W$ by applying {\sf L4.4} or {\sf L4.5} at $\{\aaa, \aaa^\rightarrow \}$; $R$ has $\bullet_{w_1} \in \aaa$ and no label of of family $i$ on $\underline{\aaa}$, and the result follows
after applying
a sequence of {\sf L1.2}'s at $\aaa$. 

(III): 
Let $\overline{W}$ be $W$ with the southmost row and $\z^\uparrow$ removed.
Recall $\lab_W(\z) = \bullet_{(i+1)_1}$ and suppose $W$ has $w_{q-1}\in \z^\leftarrow$
and $w_q\in \z^{\uparrow}$. If $\overline W$ is empty, we are done by
applying {\sf L2.1} at $\z^\uparrow$ and {\sf L1.1} at $\z$, followed by application of Lemma~\ref{lemma:onerowreversalcharacterization}(III)
to the southmost row.
Thus assume ${\overline W}$ is not empty. By 
induction or Lemma~\ref{lemma:onerowreversalcharacterization}, 
there are at most two reversals of ${\overline W}$: 

\noindent
{\sf Case 1: ($\overline{W}$ has a unique reversal $\overline{R}$):} Observe that exactly one of the following two cases holds.

\noindent
{\sf Subcase 1.1: 
($\overline{R}$ has $\bullet_{w_{q+1}} \in \z^{\uparrow\rightarrow}$ and no label of family $i$
on $\underline{\z^{\uparrow\rightarrow}}$):}
Apply {\sf L4.4} at $\z^\uparrow$ and {\sf L1.1} at~$\z$. 

\noindent
{\sf Subcase 1.2: (${\overline R}$ has $w_{q+1}\in \z^{\uparrow\rightarrow}$):} 
Apply {\sf L2.1} at $\z^\uparrow$ and {\sf L1.1} at $\z$.

\noindent
{\sf Case 2: ($\overline{W}$ has two reversals $\overline R$ and $\overline R'$):} By
induction/Lemma~\ref{lemma:onerowreversalcharacterization},
${\overline R}$ and ${\overline R}'$ differ only in 
$\z^{\uparrow\rightarrow}$: $\overline R$ has $w_{q+1}\in \z^{\uparrow\rightarrow}$ whereas $\overline R'$ has a 
$\bullet_{w_{q+1}} \in \z^{\uparrow\rightarrow}$ and 
$w_{q+1}\in \underline{\z^{\uparrow\rightarrow}}$. 
Apply {\sf L2.1} and {\sf L1.1} in the $\overline{R}$ case.
Apply {\sf L4.1} and {\sf L1.1} in the $\overline{R}'$ case.

In each of the cases above, the indicated reverse miniswaps leave us with
the southmost row having $w_1\in \aaa$ and $w_q\in \z$. We
complete the reversal using
 Lemma~\ref{lemma:onerowreversalcharacterization}(III), yielding the desired
conclusion.
\end{proof}

\begin{proposition}\label{cor:walkway_reversals}
The children of a node $U'$ in ${\mathfrak T}_U$ are obtained by replacing each walkway $W$ with $R$
or $R,R'$ (as defined in Lemmas~\ref{lemma:onerowreversalcharacterization}
and~\ref{lemma:tworowreversalcharacterization}) independently in all possible ways.
\end{proposition}
\begin{proof}
That nothing changes outside the walkways is Lemma~\ref{lem:connection_walkways_to_ladders}.
Independence follows from walkways being edge-disjoint. 
\end{proof}

\begin{proposition}\label{cor:reversal_tree_is_tree}
${\mathfrak T}_U$ is a tree.
\end{proposition}
\begin{proof}
Let $U'$ and $U''$ be distinct $i_1$-good nodes of ${\mathfrak T}_U$. By induction and Lemmas~\ref{lemma:onerowreversalcharacterization} and~\ref{lemma:tworowreversalcharacterization}, $U'$ and $U''$ differ in the placement of
a label of family strictly larger than $i$. This label is unaffected by later reverse swaps, so $U'$
and $U''$ cannot have the same child.
\end{proof}

\begin{proposition}[Characterization of reversal tree leaves]
\label{rev:claim.c}
\gap
\begin{itemize} 
\item[(I)] Let $L$ be a leaf of ${\mathfrak T}_U$. Then if we ignore the $\bullet_{1_1}$'s, either $L = U$ or $L\in \Lambda^+$ and has shape $\nu/\rho$ for some $\rho \in \lambda^+$. Moreover, $[U]{\tt slide}_{\rho/\lambda}(L)\neq 0$.
\item[(II)] If $M\in\Lambda^+$ has shape $\nu/\rho$ and $[U]{\tt slide}_{\rho/\lambda}(M)\neq 0$, then $M$ appears as a leaf of ${\mathfrak T}_U$. \qed
\end{itemize} 
\end{proposition}
\begin{proof}
(I): By Proposition~\ref{prop:goodness_preservation_reverse}, $L$ is $1_1$-good. By (G.9), there are no labels northwest of a $\bullet_{1_1}$. By (G.2), $\bullet_{1_1}$'s appear in distinct rows and columns. This proves the second sentence. The third sentence then follows from Proposition~\ref{prop:swap/revswap_inversion}.

(II): Immediate from Proposition~\ref{prop:swap/revswap_inversion}.
\end{proof}

\section{The recurrence coefficients}\label{sec:recurrence_proof}

Given $U \in B_{\lambda, \mu}^\alpha$, where $\alpha \in \{\nu\} \cup \nu^-$, let ${\tt leaf}({\mathfrak T}_U)$ be the collection of leaves of the tree $\mathfrak{T}_U$ defined in Section~\ref{sec:walkways}.

Let $W$ be an $i$-walkway of shape $\dnu / \dlambda$ with $\bullet_{(i+1)_1}$'s in $\dnu / \dalpha$. 
Let $S$ be a reversal of $W$, as defined by Lemmas~\ref{lemma:onerowreversalcharacterization} and~\ref{lemma:tworowreversalcharacterization}.
Let $\aaa$ be the southwestmost box of $W$, $\bbb$ be the northeastmost box of $W$ and $\z$ the eastmost box of $W$'s southmost row.
By Lemma~\ref{lem:walkway_structure}(V), the labels of family $i$ of $S$ form an interval 
$(w_1,\ldots,w_n)$ with respect to $\prec$.
Let $\upper{\dalpha}$ denote $\dalpha$ with its southmost row deleted, and set $\upper{\dlambda} := \dlambda \cap 
\upper{\dalpha}$.
Let $\Delta(S,W) :=(\text{$\#\bullet_{i_1}$'s in $S$})-(\text{$\#\bullet_{(i+1)_1}$'s in $W$})$. For a tableau $T$, 
let $\widetilde{T}$ denote $T$ excluding boxes containing $w_1$ and outer corners containing $\bullet_{w_1^+}$. 

\begin{claim}\label{claim:forward_swap_walkway_coeffs}
\gap
\begin{itemize}
\item[(I.i)] 
If $S$ has $w_1 \notin \aaa^\rightarrow$ and $w_1$ or $\circled{w_1} \in \underline{\aaa}$, while $W$ has either at least two rows or $w_n \in \bbb$, then $[W]{\tt slide}_{\drho/\dlambda}(S) = (-1)^{\Delta(S,W)-1} (1-\wt \dalpha / (\upper{\dalpha} \cup \dlambda)) 
\wt \upper{\dalpha}/\upper{\dlambda}$.
\item[(I.ii)] If $S$ has $w_1 \notin \aaa^\rightarrow$ and $w_1$ or $\circled{w_1} 
\in \underline{\aaa}$, while $W$ has exactly one row and $w_n \in \overline{\bbb}$, then $[W]{\tt slide}_{\drho/\dlambda}(S) = (-1)^{\Delta(S,W)} \wt \dalpha/\dlambda$.
\item[(II)] If $S$ has $\bullet_{i_1} \in \aaa$, $w_1 \in \aaa^\rightarrow$ and $w_1 \notin \underline{\aaa}$, 
then $[W]{\tt slide}_{\drho/\dlambda}(S) = (-1)^{\Delta(S,W)} \wt \dalpha/\dlambda$. 
\item[(III)] If $S$ has $w_1 \in \aaa$, then $[W]{\tt slide}_{\drho/\dlambda}(S) = (-1)^{\Delta(S,W)} 
\wt \upper{\dalpha}/\upper{\dlambda}$.
\end{itemize}
\end{claim}
\begin{proof} 
We simultaneously induct on the number of genes of family $i$ in $S$. (We gloss over some technical reindexing in the arguments below.)
We check the base case of one gene directly from the swapping rules of Section~\ref{sec:swaps}. Now let us assume that $S$ has at least two genes of 
family $i$ and the claims hold for 
situations with fewer genes of family $i$.

In the illustrative examples below that accompany the general analysis, we use for simplicity $1,2,\ldots$ to 
represent $w_1,w_2,\ldots$ respectively. Also, for simplicity, our examples assume $\aaa$ is the southwest
corner of $k\times (n-k)$, i.e., $\beta(\aaa)=1-\frac{t_1}{t_2}$.

\noindent
{\sf Case {\normalfont (I.i).1}: ($\aaa^\rightarrow \neq \z$)}:
Consider $
\begin{picture}(90,20)
\put(1,1){$S = \ytableaushort{\none \none \bullet 4 5, \bullet 2 3}$}
\put(28,-15){$1$}
\end{picture}
$
and $W = \ytableaushort{\none \none 3 4 5, 1 2 \bullet}$. Then
\[\ytableausetup{boxsize=0.9em}
\begin{picture}(300,20)
\put(0,5){$\swap_1 (S) =
(1-\frac{t_1}{t_2})\ytableaushort{\none \none \bullet 4 5, 1 2 3}+
\frac{t_1}{t_2}\ytableaushort{\none \none \bullet 4 5, 1 \bullet 3}
:=(1-\frac{t_1}{t_2})S'+\frac{t_1}{t_2}S''.$}
\put(197,-12){$2$}
\end{picture}
\]
Inductively by (III), $[W]{\tt slide}(\widetilde{S'}) = \frac{t_4}{t_7}$. 
Inductively by (I.i), $[W]{\tt slide}(\widetilde{S''}) = (1-\frac{t_2}{t_3})\frac{t_4}{t_7}$. 
Hence $[W]{\tt slide}(S) = \left(1-\frac{t_1}{t_2} \right)\frac{t_4}{t_7} + \frac{t_1}{t_2} \left(1-\frac{t_2}{t_3}\right)\frac{t_4}{t_7} = \left(1-\frac{t_1}{t_3}\right)\frac{t_4}{t_7}$,
as desired. In general, 
\begin{align*}
[W]{\tt slide}(S) \!\! &= \!\! \left(1-\hat{\beta}(\aaa) \right) \! (-1)^{\Delta(S,W)-1} \wt \upper{\dalpha}/\upper{\dlambda} + 
\hat{\beta}(\aaa) (-1)^{\Delta(S,W)-1} \! \left(1- \frac{\wt \dalpha / (\upper{\dalpha} \cup \dlambda)}
{\hat{\beta}(\aaa)}
\right) \! \wt \upper{\dalpha}/\upper{\dlambda} \\
&= \!\! (-1)^{\Delta(S,W)-1} (1-\wt \dalpha / (\upper{\dalpha} \cup \dlambda)) \wt \upper{\dalpha}/\upper{\dlambda}.
\end{align*}

\noindent
{\sf Case {\normalfont (I.i).2}: ($\aaa^\rightarrow = \z$)}: 
Let
$S=
\begin{picture}(48,23)
\put(1,13){$\ytableaushort{\none \none \none \bullet, \none \bullet 2 3,
 \bullet 2}$}
\put(4,-15){$1$}
\end{picture}
$
and $W = 
\begin{picture}(48,23)
\put(1,13){$\ytableaushort{\none \none \none 3, \none 2 3 \bullet,
1 \bullet}$}
\end{picture}$. 
Then
\[\ytableausetup{boxsize=0.9em}
\swap_1 (S) =
\left(1-\frac{t_1}{t_2}\right)
\begin{picture}(48,23)
\put(1,13){$\ytableaushort{\none \none \none \bullet, \none \bullet 2 3, 1 2}$}
\end{picture}
:=\left(1-\frac{t_1}{t_2}\right)S'.\] 
By (III), $[W]{\tt slide}(\widetilde{S'}) = \frac{t_3}{t_5}\frac{t_6}{t_7}$. Hence 
$[W]{\tt slide}(S) = (1-\frac{t_1}{t_2}) \frac{t_3}{t_5}\frac{t_6}{t_7}$, as desired. 
In general, 
\[
  [W]{\tt slide}(S)\!=\! \left(1-\hat{\beta}(\aaa) \right) (-1)^{\Delta(S,W)-1}\! \wt \upper{\dalpha}/\upper{\dlambda} 
\!  =\! (-1)^{\Delta(S,W)-1}\! (1-\wt \dalpha / (\upper{\dalpha} \cup \dlambda)) \wt \upper{\dalpha}/\upper{\dlambda}.\]

\noindent
{\sf Case} (I.ii):
Let
$S=
\begin{picture}(37,20)
\put(1,0){$\ytableaushort{
 \bullet 2 3}$}
\put(4,-6){$1$}
\end{picture}
$\ \ 
 and $W = \begin{picture}(37,20)
\put(1,0){$\ytableaushort{
 1 2 \bullet}$}
\put(26,8){$3$}
\end{picture}$. 
Then $\swap_1 (S) =
 (1-\frac{t_1}{t_2})
 \ytableaushort{ 1 2 3} + \frac{t_1}{t_2} \begin{picture}(37,20)
\put(1,0){$\ytableaushort{
 1 \bullet 3}$}
\put(15,-6){$2$}
\end{picture} :=(1-\frac{t_1}{t_2})S' + \frac{t_1}{t_2} S''$. By Lemma~\ref{lemma:onerowreversalcharacterization}, 
$[W]{\tt slide}(\widetilde{S'}) = 0$. By (I.ii), 
$[W]{\tt slide}(\widetilde{S''}) = \frac{t_2}{t_3}$. Hence $[W]{\tt slide}(S) 
= \frac{t_1}{t_2} \frac{t_2}{t_3} = \frac{t_1}{t_3}$, as desired. 
In general, 
\[
  [W]{\tt slide}(S)\!=\! \hat{\beta}(\aaa) (-1)^{\Delta(S,W)}\! \frac{1}{\hat{\beta}(\aaa)} \wt \dalpha/\dlambda
\!  =\! (-1)^{\Delta(S,W)} \wt \dalpha/\dlambda.\]

 \noindent
{\sf Case {\normalfont (II).1}: ($\aaa^\rightarrow \neq \z$)}:
Let $S = \ytableaushort{\none \none \none \bullet 4 5,
 \bullet 1 2 3}$ and $W = \ytableaushort{\none \none \none  3 4 5,
  1 2 3 \bullet}$. Then
$\swap_1 (S) =
 \frac{t_1}{t_2}
 \ytableaushort{\none \none \none \bullet 4 5,
 1 \bullet 2 3}:=\frac{t_1}{t_2}S'$. 
  By (II), $[W]{\tt slide}(\widetilde{S'}) = \frac{t_2}{t_4}\frac{t_5}{t_8}$. Hence 
$[W]{\tt slide}(S) = \frac{t_1}{t_2} \cdot \frac{t_2}{t_4}\frac{t_5}{t_8} = \frac{t_1}{t_4}\frac{t_5}{t_8}$, as desired. 
In general,
\[
[W]{\tt slide}(S) = \hat{\beta}(\aaa) \cdot (-1)^{\Delta(S,W)} \frac{1}{\hat{\beta}(\aaa)} \wt \dalpha/\dlambda 
= (-1)^{\Delta(S,W)} \wt \dalpha/\dlambda.
\]

\noindent
{\sf Case {\normalfont (II).2}: ($\aaa^\rightarrow = \z$ and the northmost $w_1 \in S$ is \emph{not} immediately below $\bullet_{i_1}$)}:
Let $S = \!\!\!\!\! \ytableaushort{\none \none \none \none \bullet 4 5,
 \none \bullet 1 2 3,
 \bullet 1}$ and $W = \!\!\!\!\! \ytableaushort{\none \none \none \none 3 4 5,
 \none 1 2 3 \bullet,
  1 \bullet}$. Then
 $\swap_1 (S) = \frac{t_1}{t_2}\frac{t_3}{t_4} \!\!\!
 \ytableaushort{\none \none \none \none \bullet 4 5,
 \none 1 \bullet 2 3,
 1 \bullet} := \frac{t_1}{t_2}\frac{t_3}{t_4} S' $.
 By (II), $[W]{\tt slide}(\widetilde{S'}) = -\frac{t_4}{t_6}\frac{t_7}{t_{10}}$. Hence $[W]{\tt slide}(S) 
= \frac{t_1}{t_2}\frac{t_3}{t_4} \cdot\left( -\frac{t_4}{t_6}\frac{t_7}{t_{10}}\right) 
= -\frac{t_1}{t_2}\frac{t_3}{t_6}\frac{t_7}{t_{10}}$, as desired. 
In general,
\[  [W]{\tt slide}(S) = \prod_{\x : \lab_W(\x) = 1} \hat{\beta}(\x) \cdot (-1)^{\Delta(S,W)} 
\prod_{\y : \lab_W(\y) > 1} \hat{\beta}(\y) \\
  = (-1)^{\Delta(S,W)} \wt \dalpha/\dlambda.\]

\noindent
{\sf Case {\normalfont (II).3}: ($\aaa^\rightarrow = \z$ and the northmost $w_1 \in S$ is immediately below $\bullet_{i_1}$)}:
Let $S= \!\!\!\!\! \ytableaushort{\none \none \none \none \bullet 5 6,
 \none \bullet 2 3 4,
 \bullet 1}$ and $W = \!\!\!\!\! \ytableaushort{\none \none \none \none 4 5 6,
 \none 1 2 3 \bullet,
  1 \bullet}$. Then
 \[
 \begin{picture}(400,28)
 \put(0,16){$\swap_1(S) = -\frac{t_1}{t_2}\frac{t_3}{t_4}
 \!\!\! \ytableaushort{\none \none \none \none \bullet 5 6,
 \none 1 2 3 4,
 1 \bullet}
 + \frac{t_1}{t_2}\frac{t_3}{t_4} \!\!\!
 \ytableaushort{\none \none \none \none \bullet 5 6,
 \none 1 \bullet 3 4,
 1 \bullet} :=  -\frac{t_1}{t_2}\frac{t_3}{t_4} S' +  \frac{t_1}{t_2}\frac{t_3}{t_4} S''.
 $}
 \put(217,0){$2$}
 \end{picture}
 \]
 By (III), $[W]{\tt slide}(\widetilde{S'}) = \frac{t_7}{t_{10}}$. 
By (I.i), $[W]{\tt slide}(\widetilde{S''}) = \left(1-\frac{t_4}{t_6} \right)\frac{t_7}{t_{10}}$.
 Hence $[W]{\tt slide}(S) = -\frac{t_1}{t_2}\frac{t_3}{t_4} \frac{t_7}{t_{10}}
+\frac{t_1}{t_2}\frac{t_3}{t_4} \left(1-\frac{t_4}{t_6}\right)\frac{t_7}{t_{10}}= 
-\frac{t_1}{t_2}\frac{t_3}{t_6}\frac{t_7}{t_{10}}$,
as desired.
Depending whether (I.i) or (I.ii) applies inductively, we have in general respectively  
\begin{align*}
[W]{\tt slide}(S) &=- \prod_{\x: \lab_{W}(\x)=1} \hat{\beta}(\x) \cdot Y (-1)^{\Delta(S,W)-1} + 
\prod_{\x : \lab_{W}(\x)=1} \hat{\beta}(\x) \cdot (1-Z)Y (-1)^{\Delta(S,W)-1} \\ &= (-1)^{\Delta(S,W)} 
YZ\prod_{\x:\lab_{W}(\x)=1} \hat{\beta}(\x) = (-1)^{\Delta(S,W)} \wt \dalpha/\dlambda
\end{align*} or
\begin{align*}
[W]{\tt slide}(S) &= \prod_{\x : \lab_{W}(\x)=1} \hat{\beta}(\x) \cdot Z (-1)^{\Delta(S,W)} \\
&= (-1)^{\Delta(S,W)} Z \prod_{\x : \lab_{W}(\x)=1} \hat{\beta}(\x) = (-1)^{\Delta(S,W)} \wt \dalpha/\dlambda,
\end{align*}
where $Y$ is the weight of the boxes of $W$ that contain genetic labels and are North of all $w_1$'s 
and $Z$ is the weight of the boxes of $W$ that contain genetic labels greater than $w_1$ and are not 
North of all $w_1$'s.

\noindent
{\sf Case {\normalfont (III).1}: ($\aaa \neq \z$)}: Let 
$S = \!\!\!\!\! \ytableaushort{\none \none \none \none \bullet 5 6,
 \none \bullet 2 3 4,
  1 2}$ and $W = \!\!\!\!\! \ytableaushort{\none \none \none \none 4 5 6,
 \none  2 3 4 \bullet,
  1 \bullet}$. Then $\swap_1(S) = \!\!\!\!\! \ytableaushort{\none \none \none \none \bullet 5 6, \none \bullet 2 3 4, 1 2}:=S'$.
  By (III), $[W]{\tt slide}(\widetilde{S'}) = \frac{t_3}{t_6}\frac{t_7}{t_{10}}$. 
Hence $[W]{\tt slide}(S) = \frac{t_3}{t_6}\frac{t_7}{t_{10}}$, as desired.
In general, $[W]{\tt slide}(S) = (-1)^{\Delta(S,W)} \wt \upper{\dalpha}/\upper{\dlambda}$.

\noindent
{\sf Case {\normalfont (III).2}: ($\aaa=\z$ and the northmost $w_1 \in S$ is \emph{not} immediately below $\bullet_{i_1}$)}: 
Let  $S = \ytableaushort{\none \none \none \bullet 4 5, \bullet 1 2 3, 1}$ and 
$W = \ytableaushort{\none \none \none 3 4 5, 1 2 3 \bullet, \bullet}$. Then
$\swap_1(S) = 
-\frac{t_2}{t_3} \ytableaushort{\none \none \none \bullet 4 5, 1 \bullet 2 3, \bullet} := -\frac{t_2}{t_3} S'$.
By (II), $[W]{\tt slide}(\tilde{S'}) = -\frac{t_3}{t_5}\frac{t_6}{t_9}$. Hence 
$[W]{\tt slide}(S) = - \frac{t_2}{t_3} \cdot \left( -\frac{t_3}{t_5}\frac{t_6}{t_9} \right) 
= \frac{t_2}{t_5}\frac{t_6}{t_9}$, as desired.
In general, 
\[[W]{\tt slide}(S) = -\prod_{\x: \lab_{W}(\x)=1} \hat{\beta}(\x) \cdot  (-1)^{\Delta(S,W)-1} 
\prod_{\y: \lab_{W}(\y)>1} \hat{\beta}(\y) = (-1)^{\Delta(S,W)} \wt \upper{\dalpha} / \upper{\dlambda}.\]

\noindent
{\sf {\normalfont Case (III).3}: ($\aaa=\z$ and the northmost $w_1\in S$ 
is immediately below $\bullet_{i_1}$)}:
Let $S = \ytableaushort{\none \none \none \bullet 4 5, \bullet 2 3 4, 1}$ and 
$W = \ytableaushort{\none \none \none 4 5 \bullet, 1 2 3 \bullet, \bullet}$. Then 
$\begin{picture}(350,23)
\put(0,0){$\swap_1(S) = \frac{t_2}{t_3}
\ytableaushort{\none \none \none \bullet 4 5, 1 2 3 4, \bullet}
-\frac{t_2}{t_3}
\ytableaushort{\none \none \none \bullet 4 5, 1 \bullet 3 4, \bullet} := \frac{t_2}{t_3}S' - \frac{t_2}{t_3}S''$
\put(-131,-16){$2$}
}
\end{picture}
$
By (III), $[W]{\tt slide}(\tilde{S'}) = -\frac{t_6}{t_8}$. By (I.i), 
$[W]{\tt slide}(\tilde{S''}) = -(1 - \frac{t_3}{t_5})\frac{t_6}{t_8}$. Hence 
$[W]{\tt slide}(S) =\frac{t_2}{t_3} \cdot \left( -\frac{t_6}{t_8} \right) - \frac{t_2}{t_3} \cdot 
\left( -(1 - \frac{t_3}{t_5})\frac{t_6}{t_8} \right) = -\frac{t_2}{t_5}\frac{t_6}{t_8} ,$ as desired.
Depending whether (I.i) or (I.ii) applies inductively, we have in general respectively
\begin{align*}
[W]{\tt slide}(S) &= \prod_{\x: \lab_{W}(\x)=1} \hat{\beta}(\x) \cdot  (-1)^{\Delta(S,W)} Y 
- \prod_{\x: \lab_{W}(\x)=1} \hat{\beta}(\x) \cdot (-1)^{\Delta(S,W)} (1-Z)Y \\
&= (-1)^{\Delta(S,W)} Y Z \prod_{\x: \lab_{W}(\x)=1} \hat{\beta}(\x) 
= (-1)^{\Delta(S,W)} \wt \upper{\dalpha} / \upper{\dlambda}
\end{align*} or
\begin{align*}
[W]{\tt slide}(S) &= - \prod_{\x: \lab_{W}(\x)=1} \hat{\beta}(\x) \cdot (-1)^{\Delta(S,W) - 1} Z \\
&= (-1)^{\Delta(S,W)} Z \prod_{\x: \lab_{W}(\x)=1} \hat{\beta}(\x) = (-1)^{\Delta(S,W)} \wt \upper{\dalpha} / \upper{\dlambda},
\end{align*}
where $Y$ is the weight of the boxes of $W$ containing genetic labels and are North of all 
$w_1$'s and $Z$ is the weight of the boxes of $W$ containing genetic labels greater than $w_1$ and are not North of all $w_1$'s.
\end{proof}

\begin{example}\label{ex:reversal_tree}
Let $\lambda=(1)$, $\nu=(3,2)$ and $\mu=(2,1)$. Consider $\ytableausetup{boxsize=1.1em}
U=\ytableaushort{{*(lightgray) \blank} {1_1} {1_2},{1_1} {2_1}}\in\Lambda
$. Below, we give the reversal tree $\mathfrak{T}_U$. 
 
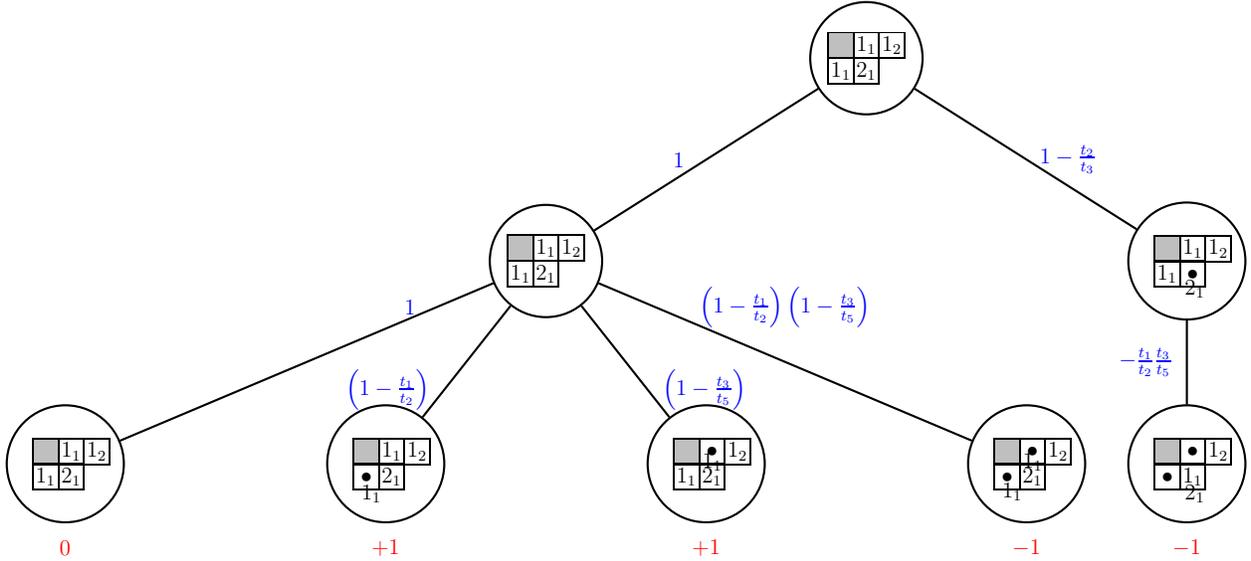
\begin{center}
\begin{tikzpicture}[
    scale = .71, transform shape, thick,
    every node/.style = {draw, circle, minimum size = 10mm},
    grow = down,  
    level 1/.style = {sibling distance=12cm},
    level 2/.style = {sibling distance=6cm}, 
    level distance = 3.8cm
  ]
 
  \node[head] (U)
          {$\ytableaushort{ {*(lightgray) \blank} {1_1} {1_2},{1_1} {2_1}}$}
   child {   node [head] (U') {\ytableaushort{{*(lightgray) \blank} {1_1} {1_2},{1_1} {2_1}}}
     child { node [head] (A) {$\begin{picture}(40,32)
\put(3,16){${\ytableaushort{{*(lightgray)\blank} {1_1} {1_2}, {1_1} {2_1}}}$}
\end{picture}
$}}
     child { node [head] (B) {$\begin{picture}(40,32)
\put(3,16){${\ytableaushort{{*(lightgray)\blank} {1_1} {1_2}, {\bullet} {2_1}}}$}
\put(7,-3){$1_1$}
\end{picture}
$}}
     child { node [head] (C) {$\begin{picture}(40,32)
\put(3,16){${\ytableaushort{{*(lightgray)\blank} {\bullet} {1_2}, {1_1} {2_1}}}$}
\put(18,14){$1_1$}
\end{picture}
$}}
     child { node [head] (D) {$\begin{picture}(40,32)
\put(3,16){${\ytableaushort{{*(lightgray)\blank} {\bullet} {1_2}, {\bullet} {2_1}}}$}
\put(7,-2){$1_1$}
\put(18,14){$1_1$}
\end{picture}
$}}
   }
   child {   node [head] (U'') {\begin{picture}(40,32)
\put(3,16){$\ytableaushort{{*(lightgray)\blank} {1_1} {1_2}, {1_1} \bullet}$}
\put(19,-2){$2_1$}
\end{picture}}
     child { node [head] (E) {$\begin{picture}(40,32)
\put(3,16){${\ytableaushort{{*(lightgray)\blank} {\bullet} {1_2}, {\bullet} {1_1}}}
$}
\put(19,-3){$2_1$}
\end{picture}
$}}
   };

  \begin{scope}[nodes = {draw = none}]
    \path (U) -- (U') node [midway, left]  {\color{blue} $1$};
    \path (U')--(A) node [pos=0.15, left] {\color{blue} $1$};
    \path (U')     -- (B) node [near end, left]  {\color{blue} $\left(1-\frac{t_1}{t_2}\right)$};
    \path (U')     -- (C) node [near end, right] {\color{blue} $\left(1-\frac{t_3}{t_5}\right)$};
    \path (U')--(D) node [pos=0.15, right]{\color{blue} $\ \ \ \ \ \ \left(1-\frac{t_1}{t_2}\right)\left(1-\frac{t_3}{t_5}\right)$};
    \path (U) -- (U'') node [midway, right] {\color{blue} $1 - \frac{t_2}{t_3}$};
    \path (U'')     -- (E) node [midway, left]  {\color{blue} $-\frac{t_1}{t_2}\frac{t_3}{t_5}$};
    \begin{scope}[nodes = {below = 30pt}]
      \node            at (A) {\color{red} $0$};
      \node [name = X] at (B) {\color{red} $+1$};
      \node            at (C) {\color{red} $+1$};
      \node            at (D) {\color{red} $-1$};
      \node [name = Y] at (E) {\color{red} $-1$};
    \end{scope}
  \end{scope}
\end{tikzpicture}
\end{center}
Each edge is labeled (in {\color{blue} blue}) by
$[U']{\tt swap}_{i_{\mu_i}}\circ \cdots\circ
{\tt swap}_{i_1}(V')$ where $U'$ is the parent of the $i_1$-good tableau $V'$. This label agrees with the
application of Claim~\ref{claim:forward_swap_walkway_coeffs} to each $i$-walkway of $V'$.
Below each leaf (in {\color{red} red}) is the coefficient in $\Lambda^+$ (i.e., $(-1)^{|\rho/\lambda|+1}$ if nonzero). 
\qed
\end{example}

\begin{lemma}
\label{rev:claim.e}
Suppose $U'$ is an $(i+1)_1$-good node of $\mathfrak{T}_U$. 
Let $\Gamma$ be the boxes of $U$ containing labels of family $i$. Then
$\sum_{V'}(-1)^{\text{$1+\#\bullet$'s in $V'$}}[U']\swap_{i_{\mu_i}} \circ \dots \circ \swap_{i_1} (V')=(-1)^{\text{$1+\#\bullet$'s in $U'$}} \wt \Gamma$,
where the sum is over all children $V'$ of $U'$ in ${\mathfrak T}_U$.
\end{lemma}
\begin{proof}
Consider boxes of $U'$ containing unmarked labels of family $i$ or $\bullet_{(i+1)_1}$. 
By (W.1) and (W.2), these boxes decompose into an edge-disjoint union of $i$-walkways $W_1,W_2,\ldots, W_t$. 
Let $\Gamma_j$ be the boxes of $W_j$ in $U$ containing labels of family $i$; thus  
$\Gamma=\Gamma_1 \sqcup \Gamma_2
\sqcup \cdots \sqcup \Gamma_t$.
Let $R_j$ and $R_j'$ (if it exists) be the  reversal(s) defined by Lemmas~\ref{lemma:onerowreversalcharacterization} and \ref{lemma:tworowreversalcharacterization} with respect to the walkway $W_j$. As computed by Claim~\ref{claim:forward_swap_walkway_coeffs}, let $a_j$ be the coefficient
of $W_j$ obtained by sliding $R_j$. Let
$b_j$ be the coefficient
of $W_j$ obtained by sliding $R_j'$ if it exists; set $b_j:=0$
if $R_j'$ does not exist.
We now assert that
\begin{equation}
\label{eqn:tya987}
(-1)^{\text{$\#\bullet$'s in $R_j$}}a_j
+(-1)^{\text{$\#\bullet$'s in $R_j'$}}b_j=
(-1)^{\text{$\#\bullet$'s in $W_j$}}\wt \Gamma_j.
\end{equation}
Suppose there is a unique reversal (i.e., $b_j=0$). This occurs under 
Lemma~\ref{lemma:onerowreversalcharacterization}(I,II) and Lemma~\ref{lemma:tworowreversalcharacterization}(I,II).
In these four cases, $R_j$ is respectively the $S$ from (II), (I.ii), (III) and (II) of Claim~\ref{claim:forward_swap_walkway_coeffs}. Hence in each of these cases, (\ref{eqn:tya987}) is immediate from the apposite case of Claim~\ref{claim:forward_swap_walkway_coeffs} (note that for Lemma~\ref{lemma:tworowreversalcharacterization}(I), the southmost row of $W_j$ has a single box and $\upper{\overline{\alpha}}/\upper{\overline{\lambda}} = \overline{\alpha}/\overline{\lambda} = \Gamma_j$). 
Suppose there are two reversals. This occurs under Lemma~\ref{lemma:onerowreversalcharacterization}(III) and Lemma~\ref{lemma:tworowreversalcharacterization}(III), which show that $R_j$ is the $S$ from Claim~\ref{claim:forward_swap_walkway_coeffs}(III)
and $R_j'$ is the $S$ from Claim~\ref{claim:forward_swap_walkway_coeffs}(I.i). 
Hence (\ref{eqn:tya987}) also follows in these cases, by adding the two apposite coefficients given by  Claim~\ref{claim:forward_swap_walkway_coeffs}. 

Since by Proposition~\ref{cor:walkway_reversals} all $V'$ are obtained
by independent replacements of $W_j$ by $R_j$ and $R_j'$
(if it exists),
\begin{align*}
\sum_{V'}(-1)^{\text{$1+\#\bullet$'s in $V'$}}[U']\swap_{i_{\mu_i}} \circ \dots \circ \swap_{i_1} (V') &=
-\prod_{j=1}^t \left( (-1)^{\text{$\#\bullet$'s in $R_j$}}a_j
+(-1)^{\text{$\#\bullet$'s in $R_j'$}}b_j \right) \\
 &=  -\prod_{j=1}^t (-1)^{\text{$\#\bullet$'s in $W_j$}}\wt \Gamma_j \\
&= (-1)^{\text{$1+\#\bullet$'s in $U'$}}\wt \Gamma.
\end{align*}
\end{proof}

\begin{lemma}
\label{lemma:Lambdaminusrecweight}
Let $U'$ be an $(i+1)_1$-good node of ${\mathfrak T}_U$.
Let $\Gamma^{(i)}$ be the set of boxes
$\{\x\in \alpha/\lambda:\family(\lab_U(\x))\leq i\}$. Then
\[
\sum_{T}(-1)^{1+\text{$\#\bullet$'s in $T$}}[U']\swap_{i_{\mu_i}} \circ
\swap_{i_{\mu_i}^{-}}\circ \cdots \circ
\swap_{1_1^+} \circ \swap_{1_1}(T)=\wt(\Gamma^{(i)})(-1)^{1+\text{$\#\bullet$'s in $U'$}},
\]
where the sum is over all $T\in {\tt leaf}({\mathfrak T}_U)$ that are
descendants of $U'$.
\end{lemma}
\begin{proof}
We induct on $i\geq 0$. In the base case $i=0$, $U'=T$ for 
$T\in {\tt leaf}({\mathfrak T}_U)$ and the lefthand side equals 
$(-1)^{1+\text{$\#\bullet$'s in $T$}}$. This equals the righthand
side since $\Gamma^{(0)}=\emptyset$ so $\wt \Gamma^{(0)}=1$.

Now let $i>0$. We have
$\sum_{T}(-1)^{1+\text{$\#\bullet$'s in $T$}}[U']\swap_{i_{\mu_i}}\circ
\swap_{i_{\mu_i}^{-}}\circ \cdots \circ
\swap_{1_1^+}\circ\swap_{1_1}(T)$
\begin{align*}
= & \sum_{\text{$V'$ a child of $U'$}} \ \ \sum_{T\in {\tt leaf}({\mathfrak T}_{U'}) \text{\ below $V'$}} (-1)^{\text{$1+\#\bullet$'s in $T$}}[U']\swap_{i_{\mu_i}}\circ
\cdots \circ \swap_{i_{1}}\circ \\ 
\ & \ \ \ \ \ \ \ \ \ \ \ \ \ \ \ \ \  \ \ \ \ \ \ \ \ \ \ \ \ \ \ \ \ \ \ \hspace{2in} \swap_{(i-1)_{\mu_{i-1}}}\circ
\cdots \circ \swap_{1_{1}}(T)\\
= & \sum_{\text{$V'$ a child of $U'$}} \ \ \sum_{T\in {\tt leaf}({\mathfrak T}_{U'}) \text{\ below $V'$}} (-1)^{\text{$1+\#\bullet$'s in $T$}}[U']\swap_{i_{\mu_i}}\circ
\cdots \circ \swap_{i_{1}}(V')\cdot \\ 
\ & \ \ \ \ \ \ \ \ \ \ \ \ \ \ \ \ \  \ \ \ \ \ \ \ \ \ \ \ \ \ \ \ \ \ \ \hspace{2in} [V']\swap_{(i-1)_{\mu_{i-1}}}\circ
\cdots \circ \swap_{1_{1}}(T).
\end{align*} 
The previous equality is since ${\mathfrak T}_{U}$ is a tree (Proposition~\ref{cor:reversal_tree_is_tree}) and
$V'$ is the unique child of $U'$ that is an ancestor of $T$. The previous
summation equals
\begin{align*}
\ &  \sum_{\text{$V'$ a child of $U'$}} \!\!\!\!\!\!
[U']{\tt swap}_{i_{\mu_i}}\circ
\cdots \circ {\tt swap}_{i_{1}}(V')\!\!\!\!\!\!\!\!\!\!\!\!
 \sum_{T\in {\tt leaf}({\mathfrak T}_{U'}) \text{\ below $V'$}} \!\!\!\!\!\!\!\!\! (-1)^{\text{$1+\#\bullet$'s in $T$}}  [V']{\tt swap}_{(i-1)_{\mu_{i-1}}}\!\!\!\!\!\!\circ
\cdots \circ {\tt swap}_{1_{1}}(T)\\
= & \sum_{\text{$V'$ a child of $U'$}} [U']{\tt swap}_{i_{\mu_i}}\circ
\cdots \circ {\tt swap}_{i_{1}}(V') \cdot {\tt wt}(\Gamma^{(i-1)})
(-1)^{\text{$1+\#\bullet$'s in $V'$}} \text{\ \ \ \ \ \ \ \ \ \ \ \ \ \ (by induction)}\\ 
= & \ {\tt wt}(\Gamma^{(i-1)} ) \sum_{\text{$V'$ a child of $U'$}} (-1)^{\text{$1+\#\bullet$'s in $V'$}} [U']{\tt swap}_{i_{\mu_i}}\circ
\cdots \circ {\tt swap}_{i_{1}}(V')\\
= & \ {\tt wt}(\Gamma^{(i-1)} ) \cdot (-1)^{\text{$1+\#\bullet$'s in $U'$}}{\wt}(\Gamma) \text{ \ \  \hspace{2.5in} (by Lemma~\ref{rev:claim.e})}\\
= &  (-1)^{\text{$1+\#\bullet$'s in $U'$}}  {\tt wt}(\Gamma^{(i)} ),
\end{align*} 
since by definition ${\wt}(\Gamma^{(i)})={\wt}(\Gamma) \cdot {\wt}(\Gamma^{(i-1)})$.
\end{proof}

\begin{proposition}
\label{prop:Lambdaminusrecweight}
For $U\in B_{\lambda,\mu}^\alpha$,
\begin{equation}
\label{eqn:lambda-eqn}
\sum_{T \in {\tt leaf}({\mathfrak T}_U)}(-1)^{|\rho(T)/\lambda|+1}[U]{\tt slide}_{\rho(T)/\lambda}(T)=\wt(\alpha/\lambda)(-1)^{|\nu/\alpha|+1}
\end{equation}
where $\rho(T)\in \{\lambda\}\cup \lambda^+$ is the ``inner shape'' of $T$, i.e., $T$ has shape $\nu/\rho(T)$.
\end{proposition}
\begin{proof}
Take $U'=U$ in Lemma~\ref{lemma:Lambdaminusrecweight}. 
\end{proof}

Now assume $U\in B_{\lambda,\mu}^\nu$. The root of ${\mathfrak T}_U$ contains no $\bullet_{(\ell(\mu) + 1)_1}$'s. One leaf of $\mathfrak{T}_U$ is $U$ itself. This is the unique leaf not in $\Lambda^+$. Let ${\tt leaf}^*(\mathfrak{T}_U)$ be the collection of all other leaves.

\begin{proposition}
\label{prop:Lambdarecweight}
For $U\in B_{\lambda,\mu}^{\nu}$,
\begin{equation}
\label{eqn:lambdaeqn}
\sum_{T \in {\tt leaf}^*(\mathfrak{T}_U)}(-1)^{|\rho(T)/\lambda|+1}[U]{\tt slide}_{\rho(T)/\lambda}(T)=1-{\tt wt}(\nu/\lambda)
\end{equation}
where $\rho(T) \in \lambda^+$ is the ``inner shape'' of $T$, i.e., $T$ has shape $\nu/\rho(T)$.
\end{proposition}
\begin{proof}
This is immediate from Proposition~\ref{prop:Lambdaminusrecweight}, since $\nu=\alpha$ and the contribution from the excluded leaf is $1$.
\end{proof}

\begin{example}
In Example~\ref{ex:reversal_tree}, summing the weights below the left child of $U$ gives $1 - \frac{t_1}{t_2}\frac{t_3}{t_5}$, in agreement with 
Lemma~\ref{rev:claim.e}. Proposition~\ref{prop:Lambdarecweight} asserts in this case that
\[1-{\tt wt}(\nu/\lambda)=1-\frac{t_1}{t_5}=\left(1-\frac{t_1}{t_2}\right)+
\left(1-\frac{t_3}{t_5}\right)-\left(1-\frac{t_1}{t_2}\right)\left(1-\frac{t_3}{t_5}\right)
+\frac{t_1}{t_2}\frac{t_3}{t_5}\cdot\left(1-\frac{t_2}{t_3}\right),\]
as the reader may verify. \qed
\end{example}

Recall
$\Lambda^+ = \sum_{\rho\in \lambda^+} (-1)^{|\rho/\lambda|+1} \sum_{T \in B_{\rho, \mu}^\nu} T$. For $T \in B_{\rho, \mu}^\nu$, write 
$T^{(1_1)}$ (cf.~Section~\ref{subsection:swapsandslides}) for $T$ with $\bullet_{1_1}$ in each box of $\rho / \lambda$.

Now set
\begin{equation}
\label{eqn:P_G}
P_{\GG} := \sum_{\rho\in \lambda^+} (-1)^{|\rho/\lambda|-1} \sum_{T \in B_{\rho, \mu}^\nu} \swap_{\GG^-} \circ \swap_{(\GG^{-})^-} \circ \dots \circ \swap_{1_1}(T^{(1_1)}).
\end{equation}
In particular, $P_{1_1}$ is $\Lambda^+$ where each $T$ is replaced by $T^{(1_1)}$. By Lemma~\ref{lem:adding_bullets_ok} and Proposition~\ref{prop:goodness_preservation}, each $P_\GG$ is a formal sum of $\GG$-good tableaux.

The main conclusion of this section is
\begin{proposition}
\label{prop:slideLambda+equals}
$P_{\GG_{\rm max}^+}$ with all $\bullet_{\GG_{\rm max}^+}$'s removed equals $\Lambda+\Lambda^-$.
\end{proposition}
\begin{proof}
By Corollary~\ref{cor:codomain} each tableau appearing in $P_{\GG_{\rm max}^+}$ (with $\bullet_{\GG_{\rm max}^+}$'s removed) is a tableau
in $\Lambda+\Lambda^-$. On the other hand, given any $U$ appearing in $\Lambda+\Lambda^-$, we constructed the tree ${\mathfrak T}_U$ in Section~\ref{sec:reversal_tree}. By Proposition~\ref{rev:claim.c}, the leaves of ${\mathfrak T}_U$ are exactly those tableaux $T\in \Lambda^+$ such that $U\in {\tt slide}_{\rho/\lambda}(T)$. It remains to show that $[U] P_{\GG_{\rm max}^+}=1-{\tt wt}(\nu/\lambda)$
if $U\in \Lambda$ and $[U] P_{\GG_{\rm max}^+}=(-1)^{|\nu/\delta|+1}{\tt wt}(\delta/\lambda)$ if $U\in \Lambda^-$ and the shape of
$U$ is $\delta/\lambda$. These are precisely the statements of Propositions~\ref{prop:Lambdarecweight} and~\ref{prop:Lambdaminusrecweight}, respectively. 
\end{proof}

\section{Weight preservation}\label{sec:weight_preservation}
\subsection{Fine tableaux and their weights}\label{sec:weights_for_bulleted_tableaux}
A tableau is {\bf fine} if it is good or can be obtained from a good tableau by swapping some subset of its snakes, i.e.\ it appears in the formal sum of tableaux resulting from this partial swap.

Let $T$ be fine and fix $\x\in T$.
Suppose $\ell\in\underline{\x}$. Define $\mathtt{edgefactor}(\ell)$ as in Section~\ref{sec:tableau_weights}; see (\ref{eqn:edgefactordef}).
The {\bf edge weight} $\mathtt{edgewt}(T) :=
\prod_\ell \mathtt{edgefactor}(\ell)$, where the product is over all (non-virtual) edge labels of $T$.

Suppose $T$ is obtained by swapping some of the snakes of the good tableau $S$ and $U$ is obtained 
from $T$ by swapping the remaining snakes. We define the 
positions in $T$ of a {\bf virtual label} $\circled\HH$ as follows. Consider a box $\x$ in column $c$.
If $c$ intersects a snake in $S$ that has been swapped in $T$, and that snake is not the upper 
snake described in Lemma~\ref{lem:snakes_arranged_SW-NE}(III), then 
$\circled{\HH}\in \underline{\x}$ (in $T$) if and only if
$\circled{\HH}\in \underline{\x}$ (in $U$). Otherwise, $\circled{\HH}\in \underline{\x}$ (in $T$) if and only if $\circled{\HH}\in \underline{\x}$ (in $S$). Observe that if $T$ is indeed good, this definition is clearly consistent with the definition of virtual labels in a good tableau.

Suppose $\circled{\HH}\in\underline{\x}$.
If $\lab_T(\x)$ is marked and each $\FF \in \underline{\x}$ with $\FF \prec \GG$ is marked, then 
\begin{equation}\label{eqn:funny_virtual_factor}
\mathtt{virtualfactor}_{\underline{\x}\in T}(\circled{\HH}) := -{\tt edgefactor}_{\underline{\x}\in T}(\HH) = \frac{t_{\sf Man(\x)}}{t_{r + N_{\HH} + 1-\family(\HH)+{\sf Man(\x)}}} - 1.
\end{equation}
Otherwise 
\begin{equation}\label{eqn:normal_virtual_factor}
\mathtt{virtualfactor}_{\underline{\x} \in T}(\circled{\HH}) := 1-{\tt edgefactor}_{\underline{\x} \in T}(\HH) = \frac{t_{\sf Man(\x)}}{t_{r + N_{\HH} + 1-\family(\HH)+{\sf Man(\x)}}} .
\end{equation}
The {\bf virtual weight} $\mathtt{virtualwt}(T)$ is
$\prod_{\circled{\ell}} \mathtt{virtualfactor}(\circled{\ell})$,
where the product is over all instances of virtual labels.

Call $\x \in T$ {\bf productive} if any of the following hold:
\begin{itemize}
\item[(P.1)] $\lab_T(\x) < \lab_T(\x^\rightarrow)$ or $\x^\rightarrow \notin T$;
\item[(P.2)] $\bullet_{i_{k+1}}\in \x$, $i_k\in\x^\leftarrow$, $i_{k+1}\in\underline{\x}$, and either $\family(\lab(\x^\rightarrow))\neq i$ or $\x^\rightarrow\not\in T$;
\item[(P.3)] $\HH\in\x$, $\bullet_\GG \in \x^\rightarrow$, and $\underline{\x^\rightarrow}$ does not contain a label of the same family as $\HH$; or
\item[(P.4)] $i_k\in\x$, $i_{k+1} \in \x^\rightarrow$ and $\bullet_{i_{k+1}} \in \x^{\rightarrow \uparrow}$, with $\x$ not SouthEast of a $\bullet_{i_{k+1}}$.
\end{itemize}

Define $\mathtt{boxfactor}(\x)$ and
{\bf box weight} $\mathtt{boxwt}(T)=\prod_\x \mathtt{boxfactor}(\x)$ as
in Section~\ref{sec:tableau_weights}, specifically (\ref{eqn:boxfactordef}), with the addendum that $\bullet_\HH \in \x$ is evaluated like $\HH\in \x$.

\pagebreak

\begin{example}
\gap
\begin{itemize} 
\item The right two boxes of $\ytableausetup{boxsize=0.9em}\ytableaushort{{1_1} {1_2} {2_1}}$ are productive by
(P.1). The left box is not productive. 
\item The left box of $\ytableausetup{boxsize=1.2em}
\begin{picture}(30,13)
\put(0,2){$\ytableaushort{{1_1} {\bullet_{1_2}}}$}
\put(17,-3){$1_2$}
\end{picture}$ is not productive. The right box is productive by (P.2).
\item  The first and third boxes of $\ytableaushort{{1_1} {\bullet_{1_2}} {2_1}}$ are productive by
(P.3) and (P.1) respectively. The middle box is not productive; a box with $\bullet_{1_2}$
is productive only if (P.2) holds.
\item  The right box of the second row in both
$\Scale[0.8]{\begin{picture}(35,22)
\ytableausetup{boxsize=1.2em}
\put(0,10){$\ytableaushort{{*(lightgray)\blank} {\bullet_{1_2}}, {1_1} {1_2}}$}
\end{picture}}$  and $\Scale[0.8]{\begin{picture}(35,25)
\ytableausetup{boxsize=1.2em}
\put(0,10){$\ytableaushort{{*(lightgray)\blank} {\bullet_{1_1}}, {1_1} {1_2}}$}
\end{picture}}$ is productive by (P.1). The left box in the
second row is productive only in the first case, by (P.4).\qed
\end{itemize}
\end{example}

Finally the {\bf weight} is
\[\wt(T) := (-1)^{d(T)}\mathtt{edgewt}(T) \cdot {\tt virtualwt}(T)\cdot {\tt boxwt}(T),\]
where $d(T)=\sum_{\GG} \left( |\GG|-1 \right)$, the sum is over genes $\GG$,
and $|\GG|$ is the (multiset) cardinality of $\GG$ (not including virtual labels). We will view $\wt$ as a
${\mathbb Z}[t_1^{\pm 1}, \ldots, t_n^{\pm n}]$-linear operator of formal sums of tableaux.

By Lemma~\ref{lem:bundled_tableaux_are_good}, bundled tableaux are good and hence also fine. Hence for a bundled tableau $B$, we have two {\it a priori} distinct notions of $\wt B$. The following lemma justifies our failure to distinguish these notationally: 
\begin{lemma}
For $B$ a bundled tableau, $\wt B$ as a fine tableau equals $\wt B$ as a bundled tableau.
\end{lemma}
\begin{proof}
By definition, the two notions of $\mathtt{edgewt}(B)$ coincide, as do the two notions of $d(B)$. Since $B$ has no $\bullet$'s, only (P.1) is available to effect productivity. Hence the two notions of productive boxes coincide, and thus, by definition, so too do the two notions of $\mathtt{boxwt}(B)$. As remarked above, the locations of virtual labels are the same, whether we think of $B$ as bundled or fine.
By Lemma~\ref{lem:wt=tildewt}, $\wt B$ as a bundled tableau is
\[
(-1)^{d(B)} \mathtt{edgewt}(B) \mathtt{boxwt}(B) \prod_{\circled{\ell}} \left( 1 - {\tt edgefactor}(\ell) \right),
\] where the product is over all instances of virtual labels and ${\tt edgefactor}(\ell)$ means the factor that would be given by $\ell$ in $\circled{\ell}$'s place. Since $B$ is bundled, it has no marked labels. Hence ${\tt virtualwt}(B)$ is calculated using only (\ref{eqn:normal_virtual_factor}), not (\ref{eqn:funny_virtual_factor}). Thus ${\tt virtualwt}(B) = \prod_{\circled{\ell}} \left( 1 - {\tt edgefactor}(\ell) \right)$, and the lemma follows.
\end{proof}

\subsection{Main claim about weight preservation}
\begin{proposition}\label{thm:weight_preservation}
\gap
\begin{itemize}
\item[(I)] $\wt P_{1_1} = \wt \Lambda^+$.
\item[(II)] For every $\GG$, $\wt P_{\GG} = \wt P_{1_1}$.
\item[(III)] $\wt P_{\GG_{\rm max}^+} = \wt \Lambda + \wt \Lambda^{-}$.
\end{itemize}
\end{proposition}
\begin{proof}
We will first prove the easier statements (I) and (III).

(I):  Suppose $T \in B_{\rho, \mu}^\nu$ for some $\rho \in \lambda^+$. It is enough to show $\wt T = \wt T^{(1_1)}$. Certainly ${\tt edgewt}(T) = {\tt edgewt}(T^{(1_1)})$ and $d(T) = d(T^{(1_1)})$. Adding $\bullet_{1_1}$'s preserves the virtual labels' locations, so ${\tt virtualwt}(T) = {\tt virtualwt}(T^{(1_1)})$.

A productive box in $T$ is also productive in $T^{(1_1)}$ and has the same ${\tt boxfactor}$. Suppose $\x$ is a productive box of $T^{(1_1)}$ that is not productive in $T$. It satisfies one of (P.1)--(P.4). If $\x$ satisfies (P.1), it is productive in $T$. 
If it satisfies (P.2), then $\bullet_{1_1}\in \x$ and $\x^\leftarrow$ contains a label,
contradicting $\x^\leftarrow\in \rho$.
If it satisfies (P.3), then $\x^\rightarrow \in \rho$, contradicting that $\x$ contains a label. Finally if $\x$ satisfies (P.4), then $\bullet_{i_{k+1}}\in\x^{\rightarrow \uparrow}$ and $i_k\in \x$. But every $\bullet$ in $T^{(1_1)}$ is $\bullet_{1_1}$. Hence $i_{k+1}=1_1$, which is impossible since $1_0$ is not a label in our alphabet. Thus the productive boxes of $T$ and $T^{(1_1)}$ are the same, and with the same
respective {\tt boxfactor}s. Therefore, $\wt T = \wt T^{(1_1)}$.

(III): Suppose $U\in P_{\GG_{\rm max}^+}$
and let ${\widetilde U}$ be given by deleting each $\bullet_{{\GG^{+}_{\rm max}}}$.
Proposition~\ref{prop:slideLambda+equals} states $P_{\GG_{\rm max}^+}$ with all $\bullet_{\GG_{\rm max}^+}$'s removed equals $\Lambda+\Lambda^-$.
Thus, it suffices to show $\wt U = \wt {\widetilde U}$. Clearly, ${\tt edgewt}(U) = {\tt edgewt}({\widetilde U})$ and $d(U) = d(\widetilde{U})$. 
One checks that the virtual labels of $U$ and the virtual labels of ${\widetilde U}$ 
appear in the same places. 
Hence ${\tt virtualwt}(U) = {\tt virtualwt}({\widetilde U})$.

Suppose $\x$ is productive in $U$. Then it satisfies one of
(P.1)--(P.4). If $\x$ satisfies (P.1) in $U$, then it satisfies (P.1) in ${\widetilde U}$. 
Now $\x$ cannot satisfy (P.2) in $U$, since if it did, 
$\bullet_{\GG^+_{\rm max}}\in\x$ and $\underline{\x}$ contains a label, contradicting Lemma~\ref{lem:bullets_migrate_out}.
If $\x$ satisfies (P.3) in $U$, then it satisfies (P.1) in ${\widetilde U}$. If $\x$ satisfies (P.4) in $U$, then $\bullet_{\GG^+_{\rm max}}\in\x^{\rightarrow \uparrow}$ but is not an outer corner, again contradicting Lemma~\ref{lem:bullets_migrate_out}.
Thus if $\x$ is productive in $U$, it is productive in ${\widetilde U}$. Conversely,
if $\x$ is productive in ${\widetilde U}$, it satisfies (P.1), since there are no $\bullet_{\GG^+_{\rm max}}$'s in ${\widetilde U}$. Hence $\x$ satisfies (P.1) or (P.3) in $U$. Thus the productive boxes of $U$ and ${\widetilde U}$ are the same. These boxes have the same
${\tt boxfactor}$s. Thus ${\tt boxwt}(U) = {\tt boxwt}({\widetilde U})$.

(II): We induct on $\GG$ with respect to $\prec$. The base case $\GG=1_1$ is trivial. The inductive hypothesis is that
$\wt P_\GG=\wt P_{1_1}$. Our inductive step is to show $\wt P_{\GG^+} = \wt P_{\GG}$.

Consider the set 
\[{\tt Snakes}_{\GG}=\{S \text{\ is a snake in $T$}\colon [T]P_{\GG}\neq 0\}.\]
We emphasize that each $S\in {\tt Snakes}_{\GG}$ refers to a particular 
instance of a snake in a specific tableau $T\in P_{\GG}$. In particular,
${\tt Snakes}_{\GG}$ is not a multiset.

For ${\mathcal B}\subseteq {\tt Snakes}_{\GG}$ 
define $\swapset_{{\mathcal B}}(T)$ to be the formal
sum of fine tableaux obtained by swapping each snake of ${\mathcal B}$
that appears in $T$ (done in any order, as permitted by
Lemma~\ref{lem:miniswap_commutation}).

We will construct $m$ subsets ${\mathcal B}_i$ such that (D.1) and (D.2) below hold: 
\begin{itemize}
\item[(D.1)] We have
a disjoint union
${\tt Snakes}_{\GG}=\bigsqcup_{1\leq i\leq m} {\mathcal B}_{i}$.
\item[(D.2)] For every $1 \leq i \leq m$ and $J \subseteq \{1, \dots, \hat{i}, \dots, m\}$, let $\mathcal{B}_J := \cup_{j \in J} \mathcal{B}_j$. Then 
\begin{equation}
\label{eqn:abc}
\sum_{T\in \Gamma_i} [T]P_{\GG}\cdot \wt (\swapset_{\mathcal{B}_J} (T)) = \sum_{T \in \Gamma_i} [T]P_{\GG} \cdot\wt (\swapset_{\mathcal{B}_i} \circ 
\swapset_{\mathcal{B}_J}(T)),
\end{equation}
where  $\Gamma_i:=\{T\in P_\GG\colon \text{$T$ contains a snake from ${\mathcal B}_i$}\}$.
\end{itemize}

\begin{claim}
The existence of $\{{\mathcal B}_i\}$  satisfying (D.1) and (D.2) implies
$\wt(P_{\GG^+})=\wt(P_{\GG})$.
\end{claim}
\begin{proof}
By Lemma~\ref{lem:miniswap_commutation}, snakes 
may be swapped in any order, so choose an arbitrary ordering of the blocks $\mathcal{B}_i$.
By (D.1),
$P_{\GG^+}:={\tt swap}_{\GG}(P_{\GG})={\tt swapset}_{{\mathcal B}_m}\circ\cdots\circ {\tt swapset}_{{\mathcal B}_1}(P_{\GG})$. Thus
\begin{align*}
{\tt wt}(P_{\GG^+})&= \wt ({\tt swap}_{\GG}(P_{\GG})) \\
&= \wt ( {\tt swapset}_{{\mathcal B}_m}\circ\cdots\circ {\tt swapset}_{{\mathcal B}_1}(P_{\GG})) \\
&= \wt ( {\tt swapset}_{{\mathcal B}_{m-1}}\circ\cdots\circ {\tt swapset}_{{\mathcal B}_1}(P_{\GG}))
\end{align*}
Here we have just used (\ref{eqn:abc}) from (D.2) 
together with linearity of
${\wt}$ and ${\swapset}_{{\mathcal B}_i}$ and the triviality
$\swapset_{\mathcal{B}_J} (T) = \swapset_{\mathcal{B}_i} \circ 
\swapset_{\mathcal{B}_J}(T)$ for $T\not\in \Gamma_i$.
Repeating this argument $m-1$ further times, we obtain the desired equality with
${\tt wt}(P_{\GG})$.
\end{proof}

In order to provide the desired decomposition, we need to first construct certain \emph{``pairing'' maps}. These are given in Section~\ref{sec:pairing}. Given these, the description of the
decomposition satisfying (D.1) and (D.2) is relatively straightforward and is found in Appendix~\ref{sec:blockdecomp}. 
\end{proof}

\subsection{Pairing maps}
\label{sec:pairing}
Let $G_{\lambda,\mu}^\nu(\GG)$ be the set of $\GG$-good tableaux of shape $\nu/\lambda$ and content $\mu$.
For $\QQ \prec \GG$ and $T \in G_{\lambda,\mu}^\nu(\GG)$, let $\mathfrak{R}_\QQ(T) := \{ V \in \revswap_{\QQ^+} \circ \dots \circ \revswap_\GG(T)\}$.

\begin{lemma}\label{lem:back_and_forth}
For any genes $\QQ \prec \GG$ and any tableau $T \in G_{\lambda,\mu}^\nu(\GG)$
\begin{equation*}
\mathfrak{R}_\QQ(T) = \{ W \in G_{\lambda,\mu}^\nu(\QQ) : T \in \swap_{\GG^-} \circ \dots \circ \swap_\QQ(W) \}.
\end{equation*}
\end{lemma}
\begin{proof}
This is immediate from Proposition~\ref{prop:swap/revswap_inversion}, noting that, by  Lemmas~\ref{lem:content_preservation} and~\ref{lemma:reversecontentpres} and Propositions~\ref{prop:goodness_preservation} and~\ref{prop:goodness_preservation_reverse}, both forward and reverse swaps preserve goodness and content.
\end{proof}

Let ${\mathcal S}_1$ be the subset of tableaux in
$G_{\lambda,\mu}^\nu(i_k)$ with a box $\x$ such that for some $\ell \geq k$,
$\bullet_{i_k} \in \x$,
$\bullet_{i_k} \in \x^{\rightarrow\uparrow}$,
$i_{\ell+1}  \in \x^\rightarrow$ and
$\circled{i_\ell} \in \underline{\x}$,
i.e.\ locally the tableau is
\ytableausetup{boxsize=1.6em}
${\mathcal C}_1=\Scale[0.8]{\begin{picture}(40,37)
\put(0,22){\ytableaushort{\none {\bullet_{i_k}}, {\bullet_{i_k}} {i_{\ell+1}}}}
\put(5,-3){$\Scale[.7]{\circled{i_\ell}}$}
\end{picture}}$ (with possibly additional edge labels), where $\x$ southwestmost depicted box.
Let ${\mathcal S}_1'$ be the subset of tableaux in
$G_{\lambda,\mu}^\nu(i_k)$ 
with a box $\x$ such that
$i_\ell \in \x$,
$i_{\ell+1} \in \x^\rightarrow$,
$\bullet_{i_k} \in \x^{\rightarrow\uparrow}$, $i_\ell$ appears outside of $\x$ and
no $\bullet_{i_k}$ appears West of $\x$ in $\x$'s row.
Locally the tableau is
${\mathcal C}_1'=\Scale[0.8]{\begin{picture}(40,35)
\put(0,17){\ytableaushort{\none {\bullet_{i_k}}, {i_\ell} {i_{\ell+1}}}}
\end{picture}}$ (with possibly additional edge labels).

\begin{lemma}
\label{lemma:C1unique}
If $T\in {\mathcal S}_1$ (respectively, ${\mathcal S}_1'$), there is a unique ${\mathcal C}_1$ (respectively, ${\mathcal C}_1'$) that it contains.
\end{lemma}
\begin{proof}
Let $\x$ be the lower-left box of any fixed choice of 
$\mathcal{C}_1$ in $T$.
Since $\circled{i_\ell} \in \underline{\x}$, the $i_{\ell+1} \in \x^\rightarrow$ is westmost in $T$ by (G.6). Hence this configuration is unique.
The argument for the other claim is the same, except we replace ``$\circled{i_\ell} \in \underline{\x}$'' 
with ``$i_\ell \in \x$''.
\end{proof}

For $T\in {\mathcal S}_1$, let $\phi_1(T)$ to be the same tableau with the unique ${\mathcal C}_1$ replaced by ${\mathcal C}_1'$. (By this we mean that we delete the labels specified in ${\mathcal C}_1$ and add the labels specified in ${\mathcal C}_1'$; any additional edge labels in $T$ are unchanged.)

\begin{lemma}\label{lem:phi1_welldefined}
$\phi_1:{\mathcal S}_1\to {\mathcal S}_1'$ is a bijection.
\end{lemma}
\begin{proof}
Let $\phi_1^{-1}:{\mathcal S}_1'\to {\mathcal S}_1$ be the putative inverse of $\phi_1$, defined by replacing ${\mathcal C}_1'$ in a $T\in {\mathcal S}_1'$ by ${\mathcal C}_1$. We are done once we show that $\phi_1$ and $\phi_1^{-1}$ are well-defined 
since the maps are clearly injective and are mutually inverse. 

Let $\x$ be the southwestmost box in the unique (by Lemma~\ref{lemma:C1unique}) $\mathcal{C}_1$ in $T$.

($\phi_1$ is well-defined):  Let $T\in {\mathcal S}_1$. We only need that $\phi_1(T)$ is good. 
Conditions (G.1) and (G.2) hold trivially in $\phi_1(T)$. (G.3) holds if $\x^\leftarrow$ is empty. Suppose $\FF\in\x^\leftarrow$. By $T$'s (G.9), $\FF \prec i_k$. Hence $\FF \prec i_\ell$, and (G.3) holds in $\phi_1(T)$. The $\circled{i_\ell} \in \underline{\x}$ in $T$ shows that $\phi_1(T)$ satisfies (G.4), (G.6) and (G.8). (G.5), (G.7), (G.9), (G.11) hold trivially. 
Since $i_{\ell+1}\in \x^\rightarrow$ is not marked, by  Lemma~\ref{lem:strong_form_of_G10}(II) there is no marked label in $T$ 
in $\x$'s row, so (G.10) and (G.13) hold for $\phi_1(T)$.
For (G.12), suppose $T$ has labels $\ell, \ell'$ that violate (G.12) in $\phi_1(T)$. Since $\ell$ must be northWest of $\x$, by $T$'s (G.9), $\ell \prec i_k$. Since $\ell'$ must be southeast of $\x$, by $T$'s (G.3), (G.4) and (G.11), $\ell' \succ i_{\ell+1}$. Hence $\family(\ell) = \family(\ell') = i$. If $\ell$ is North of $\x$, then by (G.4) the box of $\x$'s row directly below $\ell$ contains a label that violates $T$'s (G.9). By $T$'s (G.4), $\ell'$ is not South of $\x^\rightarrow$. Hence $\ell, \ell'$ are box labels 
in the row of $\x$, and no violation of $\phi_1(T)$'s (G.12) occurs.

($\phi_1^{-1}$ is well-defined): Let $T'\in {\mathcal S}_1'$. We must show that (G.7) and (G.13) hold in $\phi_1^{-1}(T')$ and that (G.1)--(G.6) and (G.8)--(G.12) hold even if the virtual label is replaced by a nonvirtual one (cf. (V.1)--(V.3)).
(G.1), (G.3)--(G.10), (G.12) and (G.13) are trivial to verify. To verify (G.2) for  $\phi_1^{-1}(T')$, it suffices to show $T'$ has no 
$\bullet_{i_k}$ South of $\x$ in the same column, or West of $\x$ in the same row. (G.9) for $T'$ rules out the 
possibility of $\bullet_{i_k}$ South of $\x$ in the same column of $T'$. By definition, there is no $\bullet_{i_k}$ West of $\x$ in the same row. To see (G.11) for $\phi_1^{-1}(T')$, we check there is no marked label $\FF^!$ in the column of $\x$. Such a label cannot appear North of $\x$ in $T'$ by Lemma~\ref{lem:strong_form_of_G10} and (G.2), considering the $\bullet_{i_k} \in \x^{\rightarrow \uparrow}$. By (G.4), it cannot appear South of $\x$ in $T'$ either.
\end{proof}

\begin{proposition}\label{prop:bundle_buddies:poofseesaw/empty}
For each $T \in \mathcal{S}_1$, $[T]P_{i_k} = -[\phi_1(T)]P_{i_k}$.
\end{proposition}
\begin{proof}
Let $T^\dagger:=\phi_1(T)$.

\noindent
{\sf Special case $k=1$:} 
Let $\widetilde{T}$ be the tableau obtained from $T$ by deleting:
\begin{itemize}
\item all labels of family $i$ and 
greater; 
\item all marked labels; and
\item all boxes containing a deleted box label. 
\end{itemize} 
Notice that any label SouthEast of a deleted label or a $\bullet_{i_1}$ will have been deleted.

As well we reindex the genes so that the subscripts of each family form an initial segment of ${\mathbb Z}_{>0}$. (This reindexing
is only possibly needed if $T$ contained a marked label.) We leave $\bullet_{i_1}$'s in place. In the same way, produce ${\widetilde{T^\dagger}}$ from $T^\dagger$. By definition of $\phi_1$,
$\widetilde{T}$ has one more $\bullet_{i_1}$ than $\widetilde{T^\dagger}$ and otherwise the two tableaux are exactly the same (the family $i$ labels of ${\mathcal C}_1$ and ${\mathcal C}_1'$ having been deleted). 

 Ignoring $\bullet_{i_1}$'s, $\widetilde{T}, 
\widetilde{T^\dagger}$ are of some common skew shape $\theta / \lambda$. If we include the $\bullet_{i_1}$'s, their respective total shapes are some $\omega/\lambda$ and $\omega^\dagger/\lambda$ where $\omega,\omega^\dagger\in\theta^+$.

\begin{claim}
\label{claim:content_squeezing}
${\widetilde T}\in G_{\lambda,\widetilde\mu}^{\omega}(i_1)$ and ${\widetilde {T^\dagger}}\in G_{\lambda,\widetilde\mu}^{\omega^\dagger}(i_1)$
where $\widetilde\mu$ is a partition (e.g., if $T$ has no marked labels then $\widetilde\mu:=(\mu_1,\mu_2,\ldots,\mu_{i-1})$).
Thus, ${\widetilde T}$ and ${\widetilde{T^\dagger}}$ (with $\bullet_{i_1}$'s removed) are in $B_{\lambda, \widetilde\mu}^{\theta}$.
\end{claim}
\begin{proof}
We prove the claim for ${\widetilde T}$; the proof for ${\widetilde {T^\dagger}}$ is essentially the same. 

 Clearly, (G.1)--(G.7), (G.9) and (G.12) for ${\widetilde T}$ are inherited from the assumption $T$ is good. (G.10), (G.11) and (G.13) are vacuous for ${\widetilde T}$. It remains to show (G.8) holds for ${\widetilde T}$ 
(which moreover implies $\widetilde\mu$ is a partition).

Suppose $\widetilde T$ fails (G.8). Then there is a least $q$ such that $\widetilde T$ has a ballotness violation between 
families $q$ and $q+1$. That is, in some genotype $G$ of
$\widetilde T$ there are more labels of family $q+1$ than of family $q$ in some initial segment of ${\tt word}(G)$.
Since we have deleted all labels of family $i$ and greater, $q < i - 1$. By failure of (G.8), either there exist 
labels $q_r$ and $(q+1)_s$ of $\widetilde{T}$ with $N_{q_r} = N_{(q+1)_s}$ such that $(q+1)_s$ appears before $q_r$ in ${\tt word}(G)$, or else there is a label $(q+1)_s$ of $\widetilde{T}$ with $N_{(q+1)_s} > N_{q_v}$ for all $v$. Let $q_{r'}$ (if $q_r$ exists) and $(q+1)_{s'}$ be the corresponding labels of $T$. We assert in the former case that $N_{q_{r'}} \leq N_{(q+1)_{s'}}$ in $T$.
In the latter case, we assert $N_{(q+1)_{s'}} > N_{q_{v'}}$ in $T$ for all $v'$. Either of these inequalities
contradicts $T$'s (G.8). 

To see these assertions, suppose that $q_h$ is a gene of $T$ that is entirely deleted in the construction of $\widetilde T$ (i.e.\ every instance of $q_h$ in $T$ is marked). 
Consider an instance of $q_h$ in $T$ in $\x$ or $\underline{\x}$. 
Since this $q_h$ is marked and $q < i-1$, by Lemma~\ref{lem:strong_form_of_G13} we know $T$ has some nonvirtual and 
marked $(q+1)^!_{z} \in \underline{\x}$ with $N_{q_h} = N_{(q+1)_{z}}$. 
By $T$'s (G.7), there is no $(q+1)_{z}$ West of $\x$ in $T$. 
By $T$'s Lemma~\ref{lem:how_to_check_ballotness}, there is no $(q+1)_{z}$ East of $\x$  in $T$. 
Hence the $(q+1)^!_{z} \in \underline{\x}$ is the only $(q+1)_{z}$ in $T$. 
Since it is marked, the gene $(q+1)_{z}$ is entirely deleted in $\widetilde T$. 
By this argument, if $q_{\hat h}$ is any other gene of $T$ that is entirely deleted in the construction of $\widetilde T$, there is 
a distinct $(q+1)_{\hat z}$ with $N_{q_{\hat h}}=N_{(q+1)_{\hat z}}$ that is also entirely deleted in $\widetilde T$. 
Hence $N_{q_{r'}} \geq N_{(q+1)_{s'}}$ or $N_{(q+1)_{s'}} > N_{q_{v'}}$ in $T$ for all $v'$, as asserted.

The last sentence of the claim follows from the first by Lemma~\ref{lem:deleting_outer_bullets_ok}, since no genetic label is southeast of a $\bullet_{i_1}$.
\end{proof}

In view of Claim~\ref{claim:content_squeezing}, it makes sense to speak of ${\mathfrak T}_{\widetilde{T}}$ and of 
${\mathfrak T}_{\widetilde{T^\dagger}}$. By Proposition~\ref{prop:Lambdaminusrecweight},
\begin{equation}
\label{eqn:klm912}
\sum_{L \in {\tt leaf}({\mathfrak T}_{\widetilde{T}})}(-1)^{|\rho(L)/\lambda|+1}[\widetilde{T}]{\tt slide}_{\rho(L)/\lambda}(L)
=(-1)^{1 + \text{\# of $\bullet$'s in $\widetilde{T}$}} \cdot \wt(\theta/\lambda).
\end{equation}
Similarly, 
\begin{equation}
\label{eqn:klm913}
\sum_{L \in {\tt leaf}({\mathfrak T}_{\widetilde{T^\dagger}})}(-1)^{|\rho(L)/\lambda|+1}[\phi_1(\widetilde{T})]{\tt slide}_{\rho(L)/\lambda}(L)
=(-1)^{1 + \text{\# of $\bullet$'s in $\widetilde{T^\dagger}$}} \cdot \wt(\theta/\lambda).
\end{equation}
In particular, these quantities differ by a factor of $-1$.

By inspection of the reverse miniswaps, $\revswap_{a_q}$ for $1\leq a\leq i-1$ does not affect
any labels of family $i$ or greater or any labels that are marked in $T$. Hence one sees that
$\revswap_{1_2} \circ \dots \circ  \revswap_{(i-1)_{\mu_{i-1}}}\circ \revswap_{i_1}(T)$ (respectively $T^\dagger$) is the same as
$\revswap_{1_2} \circ \dots \circ \revswap_{(i-1)_{\mu_{i-1}}}\circ\revswap_{i_1}(\widetilde{T})$  (respectively $\widetilde{T^\dagger}$) 
followed by  adding back the
labels of $T \setminus {\widetilde T}$ (respectively $T^\dagger \setminus {\widetilde{T^\dagger}}$).
Therefore,  by our comparison of (\ref{eqn:klm912}) and (\ref{eqn:klm913}) above,
\[[T]P_{i_1} = (-1)^{1 + \text{\# of $\bullet$'s in $\widetilde{T}$}} \cdot \wt(\theta/\lambda) = -[T^\dagger]P_{i_1},\]
as desired.

\noindent
{\sf Reduction to the $k=1$ case:}
In the calculation of $\revswap_{i_2}\circ \revswap_{i_3}\circ\cdots\circ \revswap_{i_k}(T)$ and $\revswap_{i_2}\circ \revswap_{i_3}\circ\cdots\circ \revswap_{i_k}(T^\dagger)$, it is straightforward by inspection that each reverse miniswap involving either $\bullet$ of ${\mathcal C}_1$ or 
the $\bullet$ of ${\mathcal C}_1'$ is {\sf L1.2}.
Therefore there exists an instance of $\mathcal{C}_1$ in each $W \in \mathfrak{R}_{i_1}(T)$ and an instance of $\mathcal{C}_1'$ in each $W' \in \mathfrak{R}_{i_1}(T^\dagger)$. By Lemma~\ref{lemma:C1unique}, these instances are unique. Extending $\phi_1$ linearly, since $T$ and $T^\dagger$ are the same outside of the regions $\mathcal{C}_1$ and $\mathcal{C}_1'$, it is easy to see inductively that for all $2 \leq q \leq k$, \[\phi_1(\revswap_{i_q} \circ \dots \circ \revswap_{i_k}(T)) = \revswap_{i_q} \circ \dots \circ \revswap_{i_k}(T^\dagger).\] In particular, $\phi_1$ bijects $\mathfrak{R}_{i_1}(T)$ with $\mathfrak{R}_{i_1}(T^\dagger)$. 

Let $V \in \mathfrak{R}_{i_1}(T)$. By the $k=1$ case above, $[V]P_{i_1} = -[\phi_1(V)]P_{i_1}.$ Moreover, when we apply $\swap_{i_k^-} \circ \dots \circ \swap_{i_1}$ to $V$ and $\phi_1(V)$, 
each miniswap involving a $\bullet$ of ${\mathcal C}_1$ or ${\mathcal C}_1'$ is {\sf H3}. Hence, 
$[T]\swap_{i_k^-} \circ \dots \circ \swap_{i_1}(V) = [T^\dagger]\swap_{i_k^-} \circ \dots \circ \swap_{i_1}(\phi_1(V))$. Thus by Lemma~\ref{lem:back_and_forth},
$[T]P_{i_k} = -[T^\dagger]P_{i_k}$.
\end{proof}

Let ${\mathcal S}_2$ be the subset of tableaux in $G_{\lambda,\mu}^{\nu}(i_k)$
with a box $\x$ such that $\bullet_{i_k} \in \x$, $\bullet_{i_k} \in \x^{\rightarrow\uparrow}$, $i_{k+1} \in \x^\rightarrow$ and $i_k \in \underline{\x}$, i.e.\ locally the tableau is
\ytableausetup{boxsize=1.6em}
$\mathcal{C}_2 = \Scale[0.8]{\begin{picture}(40,37)
\put(0,19){\ytableaushort{\none {\bullet_{i_k}}, {\bullet_{i_k}} {i_{k+1}}}}
\put(7,-4){$i_k$}
\end{picture}}$ (with possibly additional edge labels). Let ${\mathcal S}_2'$ be the subset of tableaux in
$G_{\lambda,\mu}^{\nu}(i_k)$ with a box $\x$ such that $i_k \in \x$, $i_{k+1} \in \x^\rightarrow$, $\bullet_{i_k} \in \x^{\rightarrow\uparrow}$, no $i_k$ appears outside of $\x$ and no $\bullet_{i_k}$ appears West of $\x$ in $\x$'s row. Locally the tableau is
${\mathcal C}_2'=\Scale[0.8]{\begin{picture}(40,37)
\put(0,17){\ytableaushort{\none {\bullet_{i_k}}, {i_k} {i_{k+1}}}}
\end{picture}}$ (with possibly additional edge labels). 

\begin{lemma}
If $T\in {\mathcal S}_2$ (respectively, ${\mathcal S}_2'$), there is a unique ${\mathcal C}_2$ (respectively, ${\mathcal C}_2'$) that it contains.
\end{lemma}
\begin{proof}
Let $\x$ be the southwestmost box of a $\mathcal{C}_2$ in $T$.
By (G.7), the $i_k \in \underline{\x}$ is the westmost $i_k$ in $T$; hence this configuration is unique.
The claim about ${\mathcal C}_2'$ is clear since the $i_k$ is unique.
\end{proof}

For $T\in {\mathcal S}_2$, let $\phi_2(T)$ be $T$ with the unique ${\mathcal C}_2$ replaced 
by ${\mathcal C}_2'$. 

\begin{lemma}\label{lem:phi2_welldefined}
$\phi_2:{\mathcal S}_2\to {\mathcal S}_2'$ is a bijection.
\end{lemma}
\begin{proof}
This may be proved almost exactly as Lemma~\ref{lem:phi1_welldefined}.
\end{proof}

\begin{proposition}\label{prop:bundle_buddies:expseesaw/empty}
For each $T \in \mathcal{S}_2$, $[T]P_{i_k} = -[\phi_2(T)]P_{i_k}$.
\end{proposition}
\begin{proof}
Let $T^\dagger:=\phi_2(T)$.
Let $\x$ be the southwestmost box of ${\mathcal C}_2$ in $T$. Then $\x$ is also the southwestmost box of
${\mathcal C}_1'$ in $T^\dagger$. 

\noindent
{\sf Special case $k=1$:} 
The proof is \emph{verbatim} the argument for the $k=1$ case of Proposition~\ref{prop:bundle_buddies:poofseesaw/empty}.

\noindent
{\sf Reduction to the $k=1$ case:}
Suppose $k>1$. Let $\mathcal{Z}$ be the set of boxes in an $i_k$-good tableau that either (1) contain $\bullet_{i_k}$ 
or (2) contain a label $\FF$ with
$i_1\preceq \FF\preceq i_{k-1}$ and are not southeast 
of a $\bullet_{i_k}$. Call an edge connected component of ${\mathcal Z}$ an {\bf $i_k$-walkway}.
We will now apply the development of $i$-walkways,
from Sections~\ref{sec:reversal_tree} and~\ref{sec:recurrence_proof}, 
in slightly modified form to the
{\bf $i_k$-walkways}.
To be more precise, Lemmas~\ref{lem:walkway_structure}, 
\ref{lem:connection_walkways_to_ladders}, \ref{lemma:onerowreversalcharacterization} and~\ref{lemma:tworowreversalcharacterization} are true
after replacing ``$(i+1)_1$'' with ``$i_k$'' and
``$i$-walkway'' with ``$i_k$-walkway''. In addition, Claim~\ref{claim:forward_swap_walkway_coeffs} holds \emph{verbatim}. 
The proofs are trivial modifications of those given. 

Let $W$ be the $i_k$-walkway of $T$ containing $\x$ 
($W$ includes all edges of boxes in $W$). Let $W^\dagger$ be the analogous $i_k$-walkway of $T^\dagger$. Note that $W$ and $W^\dagger$ have the same skew shape.

\begin{claim}
\label{claim:i_k_walkway_reversals}
Let ${\mathbb S}$, ${\mathbb S}'$ and ${\mathbb T}$ 
be respectively the set of reversals of $W$, $W^\dagger$ and $W^c$ (the complement of $W$)
under $\revswap_{i_2}\circ\cdots\circ \revswap_{i_k}$. Then:
\begin{itemize}
\item[(I)] ${\mathfrak R}_{i_1}(T)=\{V\in G_{\lambda,\mu}^\nu(i_1) : V|_{W}\in {\mathbb S}, V|_{W^c}\in {\mathbb T}\}$
\item[(II)] ${\mathfrak R}_{i_1}(T^\dagger)=\{V'\in G_{\lambda,\mu}^\nu(i_1) : V'|_{W^\dagger}\in {\mathbb S}', V'|_{(W^\dagger)^c}\in {\mathbb T}\}$
\end{itemize}
\end{claim}
\begin{proof}
We prove only (I), as the proof of (II) is similar (using $T|_{W^c}=T^\dagger|_{W^c}$).
Fix $2\leq h\leq k$ and let $L$ be a ladder of $A\in{\mathfrak R}_{i_h}(T)$. 
$L$ contains only $\bullet_{i_h}$ and unmarked $i_{h-1}$. 
Each of the boxes $\x$ of $L$ is in ${\mathcal Z}$:
This is clear if $h=k$ and follows for smaller $h$ by induction.  
Thus $L \subseteq {\mathcal Z}$.
Therefore, since $L$ is edge connected, it sits inside an
edge connected component of ${\mathcal Z}$. Thus, since
$W$ is one such component,
reverse swapping acts independently on $W$ and $W^c$. 
\end{proof}

\noindent
{\sf Case 1: ($W$ (and hence $W^\dagger$) has a single row):} By construction, $\x$ is the eastmost box of $W$
and $W^\dagger$.  By Lemma~\ref{lemma:onerowreversalcharacterization}(II), for every $V \in \mathfrak{R}_{i_1}(T)$, $V|_W = R'$. By 
Lemma~\ref{lemma:onerowreversalcharacterization}(III), for every $V' \in \mathfrak{R}_{i_1}(T^\dagger)$, $V'|_{W^\dagger} \in \{R, R'\}$ 
where this $R'$ is the same as in the previous sentence.  

Since $R'$ is the unique reversal of $W$ and is a reversal of $W^{\dagger}$, we have
${\mathfrak R}_{i_1}(T)\subseteq {\mathfrak R}_{i_1}(T^\dagger)$ by Claim~\ref{claim:i_k_walkway_reversals}.
Let $\iota : \mathfrak{R}_{i_1}(T) \to \mathfrak{R}_{i_1}(T^\dagger)$ be the inclusion map. Let $f : \mathfrak{R}_{i_1}(T) \to \mathfrak{R}_{i_1}(T^\dagger)$ 
be the map given by replacing the $R'$ occupying the region $W$ with $R$. 
Again appealing to Claim~\ref{claim:i_k_walkway_reversals} we see that these maps are well-defined, injective and  
$\mathfrak{R}_{i_1}(T^\dagger) = \im \iota \sqcup \im f$. 

By Claim~\ref{claim:forward_swap_walkway_coeffs}(III), forward swapping 
$R$ produces $W^\dagger$ with coefficient $1$. By Claim~\ref{claim:forward_swap_walkway_coeffs} (part (I.i) or (I.ii), as appropriate) 
forward swapping $R'$ produces $\beta W + (1 - \beta) W^\dagger$ for some $\beta$.  Moreover,
when applying $\swap_{i_{k-1}} \circ \dots \circ \swap_{i_1}$ to $V\in  \mathfrak{R}_{i_1}(T)$
or $V'\in  \mathfrak{R}_{i_1}(T^\dagger)$, every snake lies entirely inside some 
edge-connected component of $\mathcal{Z}$. $W$ is one of these components. 
Thus, for each $V\in \mathfrak{R}_{i_1}$, $[T] \swap_{i_{k-1}} \circ \dots \circ \swap_{i_1}(V)$ factors as a contribution from the region $W$ times a contribution from $\mathcal{Z} \setminus W$. That is, for the same $\alpha$,
\[ \sum_{V \in \mathfrak{R}_{i_1}(T)} [T] \swap_{i_{k-1}} \circ \dots \circ \swap_{i_1}(V) =  \alpha\beta, \ \sum_{V' \in \mathfrak{R}_{i_1}(T^\dagger)} [T^\dagger] \swap_{i_{k-1}} \circ \dots \circ \swap_{i_1}(V') =  \alpha\]
\[\sum_{V' \in \mathfrak{R}_{i_1}(T^\dagger)} [T^\dagger] \swap_{i_{k-1}} \circ \dots \circ \swap_{i_1}(V') =  \alpha(1-\beta).\]

Therefore,  
\begin{align*}
[T]P_{i_k}& = \sum_{V \in \mathfrak{R}_{i_1}(T)} [V]P_{i_1} \cdot [T] \swap_{i_{k-1}} \circ \dots \circ \swap_{i_1}(V) = [V]P_{i_1} \alpha \beta,
\end{align*}
while
\begin{align*}
 [T^\dagger]P_{i_k} &= \sum_{V' \in \mathfrak{R}_{i_1}(T^\dagger)} [V']P_{i_1} \cdot [T^\dagger] \swap_{i_{k-1}} \circ \dots \circ \swap_{i_1}(V') \\
&= \sum_{V \in \mathfrak{R}_{i_1}(T)} [\iota(V)]P_{i_1} \cdot [T^\dagger] \swap_{i_{k-1}} \circ \dots \circ \swap_{i_1}(\iota(V))\\
&+ \sum_{V \in \mathfrak{R}_{i_1}(T)} [f(V)]P_{i_1} \cdot [T^\dagger] \swap_{i_{k-1}} \circ \dots \circ \swap_{i_1}(f(V)) 
\end{align*}
\begin{align*}
&= \sum_{V \in \mathfrak{R}_{i_1}(T)} [V]P_{i_1} \cdot [T^\dagger] \swap_{i_{k-1}} \circ \dots \circ \swap_{i_1}(V) \\
&- \sum_{V \in \mathfrak{R}_{i_1}(T)} [V]P_{i_1} \swap_{i_{k-1}} \circ \dots \circ \swap_{i_1}(f(V)) \\
&= [V]P_{i_k} (\alpha(1-\beta) - \alpha).
\end{align*}
Now, $[T]P_{i_k} = -[T^\dagger]P_{i_k}$ follows.

\noindent
{\sf Case 2: ($W$ (and hence $W^\dagger$) has at least two rows):} There are three cases to consider, corresponding to the case of 
 Lemma~\ref{lemma:tworowreversalcharacterization}.

In Cases (I) and (II) of Lemma~\ref{lemma:tworowreversalcharacterization}, $W$ and $W^\dagger$ have a unique reversal $R$. By Claim~\ref{claim:forward_swap_walkway_coeffs}(III) or Claim~\ref{claim:forward_swap_walkway_coeffs}(II) respectively, 
forward swapping $R$ produces $\beta W - \beta W^\dagger$ for some $\beta$. 
In Case (III) of Lemma~\ref{lemma:tworowreversalcharacterization}, $W$ and $W^\dagger$ share the same pair of reversals $R, R'$. By Claim~\ref{claim:forward_swap_walkway_coeffs}(III) and (I.i), forward swapping $R$ produces $\beta W - \beta W^\dagger$ for some $\beta$, while forward swapping $R'$ produces $\beta' W - \beta' W^\dagger$ for some $\beta'$. Using these facts, one may argue similarly to {\sf Case 1} to deduce $[T]P_{i_k} = -[T^\dagger]P_{i_k}$.
\end{proof}

Let ${\mathcal S}_{3}$ be the subset of
tableaux in $G_{\lambda, \mu}^\nu(i_k)$ with a box $\x$ such that $\bullet_{i_k} \in \x$, $i_k \in \x^\rightarrow$ and $i_k \in \underline{\x}$, i.e. locally the tableau is
\ytableausetup{boxsize=1.6em}
${\mathcal C}_3=\Scale[0.8]{\begin{picture}(40,22)
\put(0,0){$\ytableaushort{{\bullet_{i_k}} {i_k}}$.}
\put(5,-3){$i_k$}
\end{picture}}$ (with possibly additional edge labels). 
Let ${\mathcal S}_{3}'$ be the subset of tableaux in $G_{\lambda,\mu}^\nu(i_k)$ with a box $\x$ such that $\bullet_{i_k} \in \x$, $i_k \in \x^\rightarrow$, no $i_k$ appears West of $\x^\rightarrow$, $i_{k-1}\notin\x^\leftarrow$, and
$(i+1)_h\notin\underline{\x^\rightarrow}$ where $N_{i_k}=N_{(i+1)_h}$. Locally the tableau is
${\mathcal C}_3'=\Scale[0.8]{\begin{picture}(42,22)
\put(0,0){$\ytableaushort{{\bullet_{i_k}} {i_k}}$}
\end{picture}}$ (with possibly additional edge labels).

\begin{lemma}
\label{lemma:C3unique}
If $T\in {\mathcal S}_3$ (respectively, ${\mathcal S}_3'$), there is a unique ${\mathcal C}_3$ (respectively, 
${\mathcal C}_3'$) that it contains.
\end{lemma}
\begin{proof}
If ${\mathcal C}_3$ occurs in a good tableau, it is unique since the edge $i_k$ is westmost in its gene by (G.7). Similarly
${\mathcal C}_3'$ is unique since the $i_k\in \x^\rightarrow$ is westmost by assumption.
\end{proof}

Define $\phi_3(T)$ to be $T$ with the unique ${\mathcal C}_3$ replaced by ${\mathcal C}_3'$.

\begin{lemma}\label{lem:phi3_welldefined}
$\phi_{3}:{\mathcal S}_{3}\to {\mathcal S}_{3}'$ is a bijection.
\end{lemma}
\begin{proof}
Define the (putative) inverse $\phi_3^{-1}$ by replacing 
${\mathcal C}_3'$ with ${\mathcal C}_3$.
Once we establish that  $\phi_3$ and $\phi_3^{-1}$ are well-defined, we are done, since
$\phi_3$ and $\phi_3^{-1}$ are clearly mutually inverse.

Let $T\in {\mathcal S}_3$. Trivially, each (G.$n$) holds for $\phi_{3}(T)$. By $T$'s (G.12), $i_{k-1}\notin \x^\leftarrow$. If $(i+1)_h \in \underline{\x^\rightarrow}$ in $\phi_3(T)$ with $N_{i_k} = N_{(i+1)_h}$, then $T$ would violate Lemma~\ref{lem:how_to_check_ballotness}. By $T$'s (G.4) and (G.7), the $i_k\in \x^\rightarrow$ is westmost in $\phi_3(T)$.

Now let $T\in {\mathcal S}_3'$. We check the goodness conditions for $\phi_3^{-1}(T)$.
\begin{claim}\label{claim:3bundlebundle}
No label of family $i$ appears in $\x$'s column in $T$.
\end{claim}
\begin{proof}
By $T$'s (G.12), there are no labels of family $i$ North of $\x$ and in its column.
By $T$'s (G.11), a label $\ell$ South of $\x$ and in its column
is not marked, i.e., $\ell\succeq i_k$. Since we assumed the $i_k\in \x^\rightarrow$ is westmost, $\ell\neq i_k$ .
By $T$'s (G.6), $\ell\neq i_l$ for $l>k$. Hence $i_k<\ell$.
\end{proof}

\noindent
(G.4) and (G.5):  By $T$'s (G.9), all labels North of $\x$ and in its column are of family at most $i$. By $T$'s (G.11), all labels South of $\x$ and in its column are of family at least $i$. Hence by Claim~\ref{claim:3bundlebundle}, $\phi_3^{-1}(T)$'s (G.4) and (G.5) follow.

\noindent
(G.8): If there is a genotype $G$ of $\phi_3^{-1}(T)$ that is not ballot, then it uses the $i_k \in \underline{\x}$. Furthermore, since $T$ is ballot, some $(i+1)_h$ with $N_{i_k} = N_{(i+1)_h}$ appears in ${\tt word}(G)$ before the $i_k \in \underline{\x}$. By Lemma~\ref{lem:how_to_check_ballotness} applied to $T$, this $(i+1)_h$ can only be South of $\x^\rightarrow$ and in $\x^\rightarrow$'s column or North of $\x$ and in $\x$'s column. By $T$'s (G.9), it cannot be North of $\x$ and in its column. Suppose it appears South of $\x^\rightarrow$ and in its column. By assumption, $(i+1)_h \notin \underline{\x^\rightarrow}$. Hence suppose it appears south of $\x^{\rightarrow\downarrow}$, and consider $\lab(\x^\downarrow)$. By (G.11) $\family(\lab(\x^\downarrow)) \geq i$. By Claim~\ref{claim:3bundlebundle}, $\family(\lab(\x^\downarrow)) \neq i$. By $T$'s (G.3) and (G.4), $\lab(\x^\downarrow) \prec (i+1)_h$. Hence $\family(\lab(\x^\downarrow)) = i+1$. But by Lemma~\ref{lemma:Gsoutheast}, $\lab(\x^\downarrow) \neq (i+1)_h$. Hence by $T$'s (G.6), $(i+1)_{h-1} \in \x^\downarrow$. This creates a (G.8) violation in $T$, as this label is read before any $i_{k-1}$.

\noindent
(G.12): Since $T$ is good, if $\phi_3^{-1}(T)$ violates (G.12), the violation involves the $i_k \in \underline{\x}$. Since by assumption $i_{k-1} \notin \x^\leftarrow$, the last sentence of (G.12) does not apply. Suppose $i_j$ is SouthEast of $\underline{\x}$, then it is also SouthEast of $i_k \in \x^\rightarrow$, which will lead to a violation of $T$'s (G.12). Suppose $i_j$ is NorthWest of $\underline{\x}$, then to avoid a violation of $T$'s (G.12) with the $i_k \in \x^\rightarrow$, $i_j$ must be either in $\x$'s row or in an upper edge of that row. Since we have $\bullet_{i_k} \in \x$, this avoids violating $\phi_3^{-1}(T)$'s (G.12).

All of the remaining (G.$n$)-conditions are trivial to verify.
\end{proof}

\begin{proposition}\label{prop:bundle_buddies:horizontal/death}
For $T\in {\mathcal S}_{3}$,  $[T]P_{i_k} = [\phi_{3}(T)]P_{i_k}$.
\end{proposition}
\begin{proof}
Let $T^\dagger:=\phi_{3}(T)$.
By inspection of the reverse miniswaps, and downward induction on $\QQ$, there is a bijection 
$f_\QQ:{\mathfrak R}_\QQ(T) \to {\mathfrak R}_\QQ(T^\dagger)$ 
given by deletion of the $i_k\in \underline{\x}$.
If $L\in {\mathfrak R}_{1_1}(T)$, then $L$ and $f_{1_1}(L)$ have the same number of $\bullet_{1_1}$'s.
Hence, $[L]P_{1_1} = [f_{1_1}(L)]P_{1_1}$; cf. (\ref{eqn:P_G}).

Extend $f_{\QQ}$ linearly. By inspection of the miniswaps,  
\[f_{i_k}(\swap_{i_k^{-}} \circ \dots \circ \swap_{1_1}(L))=
\swap_{i_k^{-}} \circ \dots \circ \swap_{1_1}(f_{1_1}(L)).\]
Hence by Lemma~\ref{lem:back_and_forth}, $[T]P_{i_k} = [T^\dagger]P_{i_k}$.
\end{proof}

Let $\mathcal{S}_4$ be the subset of tableaux in $G_{\lambda,\mu}^\nu(i_k)$ with a box $\x$ such that $\bullet_{i_k} \in \x$, $\FF^! \in \x^\rightarrow$, $i_k \in \underline{\x}$ and $\circled{i_k} \in \underline{\x^\rightarrow}$, i.e.\ locally the tableau is
\ytableausetup{boxsize=1.8em}
$\mathcal{C}_4 = \Scale[0.8]{\begin{picture}(45,22)
\put(0,0){$\ytableaushort{{\bullet_{i_k}} {\FF^!}}$}
\put(5,-4){$i_k$}
\put(26,-4){$\Scale[.8]{\circled{i_k}}$}
\end{picture}}$ (with possibly additional edge labels). 
Let $\mathcal{S}_4'$ be the subset of tableaux in $G_{\lambda,\mu}^\nu(i_k)$ with a box $\x$ such that $\bullet_{i_k} \in \x$, $\FF^! \in \x^\rightarrow$, $i_k \in \underline{\x^\rightarrow}$, $(i+1)_h\not\in \underline{\x^\rightarrow}$ if $N_{(i+1)_h} = N_{i_k}$, and $i_{k-1}\not\in \x^\leftarrow$.
Locally the tableau is
$\mathcal{C}_4'=\Scale[0.8]{\ytableausetup{boxsize=1.8em}
\begin{picture}(45,24)
\put(0,0){$\ytableaushort{{\bullet_{i_k}} {\FF^!}}$}
\put(28,-4){$i_k$}
\end{picture}}$ (with possibly additional edge labels).

\begin{lemma}
If $T\in {\mathcal S}_4$ (respectively, ${\mathcal S}_4'$), there is a unique ${\mathcal C}_4$ (respectively, ${\mathcal C}_4'$) that it contains.
\end{lemma}
\begin{proof}
This follows since by (G.7), 
$T$ contains at most one edge label $i_k$.
\end{proof}

Set $\phi_4 : \mathcal{S}_4 \to \mathcal{S}_4'$ by replacing $\mathcal{C}_4$ with $\mathcal{C}_4'$.

\begin{lemma}\label{lem:phi4_welldefined}
$\phi_4:{\mathcal S}_4\to {\mathcal S}_4'$ is a bijection.
\end{lemma}
\begin{proof}
Define a putative inverse $\phi^{-1}_4:{\mathcal S}_4'\to {\mathcal S}_4$ by replacing 
$\mathcal{C}_4'$ with $\mathcal{C}_4$. Clearly, $\phi_4$ and $\phi_4^{-1}$ are mutually inverse. It remains to check 
well-definedness. Indeed,
it is trivial to check each goodness condition holds for $\phi_4(T)$. By Lemma~\ref{lem:how_to_check_ballotness} for $T$, there is not $(i+1)_h \in \underline{\x^\rightarrow}$ with $N_{(i+1)_h} = N_{i_k}$. By $T$'s (G.12), $i_{k-1}\not\in \x^\leftarrow$. Thus $\phi_4$ is well-defined.

\begin{claim}\label{claim:4bundlebundle}
No label of family $i$ appears in $\x$'s column in $T$.
\end{claim}
\begin{proof}
By $T$'s (G.12), $i_\ell$ cannot appear North of $\x$ and in its column.
If $i_\ell$ is South of $\x$ and in its column, then by $T$'s (G.6) and (G.7),
$\ell < k$, so this $i_\ell$ is marked, contradicting $T$'s (G.11). 
\end{proof}

Now let $T\in {\mathcal S}_4'$. We check the goodness conditions for $\phi_4^{-1}(T)$:

(G.4) and (G.5):  By $T$'s (G.9), every label North of $\x$ and in its column has family at most $i$. 
By $T$'s (G.11), every label South of $\x$ and in its column has family at least $i$. Moreover, by Claim~\ref{claim:4bundlebundle}, no label of family $i$ appears in $\x$'s column in $T$. Hence (G.4) and (G.5) hold in $\phi_4^{-1}(T)$.

(G.8): Suppose $\phi_4^{-1}(T)$ has a nonballot genotype $G$. By $T$'s (G.8), $G$ must use the $i_k\in \underline\x$. 
Also by $T$'s (G.8), some $(i+1)_h$ with
$N_{(i+1)_h}=N_{i_k}$ appears in ${\tt word}(G)$ before this $i_k\in \underline\x$. By $T$'s (G.9) and (G.8), this $(i+1)_h$ appears South of $\x^\rightarrow$ and in $\x^\rightarrow$'s column. By $T$'s (G.4) and the first hypothesis on ${\mathcal S}_4'$,
in fact $(i+1)_h\in \x^{\rightarrow\downarrow}$. By $T$'s (G.3), $\family(\lab(\x^\downarrow)) \leq i+1$. By (G.11) and the $\bullet_{i_k} \in \x$, $\family(\lab(\x^\downarrow)) \geq i$.
By Claim~\ref{claim:4bundlebundle}, no
label of family $i$ appears in $\x$'s column in $T$. 
Thus $\family(\lab(\x^\downarrow)) = i+1$.
Then by $T$'s (G.3) and (G.6), $(i+1)_{h-1}\in\x^\downarrow$. Hence by Claim~\ref{claim:4bundlebundle}, this contradicts Lemma~\ref{lem:how_to_check_ballotness} for $T$.

(G.12): If there is an $i_\ell$ SouthEast of the $i_k\in \underline\x$ in $\phi_4^{-1}(T)$, then we either violate $T$'s (G.2),
(G.4) or (G.12). Now suppose there is an $i_\ell$ NorthWest of $i_k \in \underline\x$ in $\phi_4^{-1}(T)$. By $T$'s (G.12), this $i_\ell$ is West and either in $\x$'s row or on the upper edge of that row.
If $i_\ell\in \x^\leftarrow$, then $\ell=k-1$ by $T$'s (G.6). However then we contradict the last hypothesis 
on ${\mathcal S}_4'$. So the $i_\ell$ and $i_k$ satisfy (G.12).

The remaining goodness conditions are trivial to verify.
\end{proof}

\begin{proposition}\label{prop:bundle_buddies:T6andT4}
For each $T \in \mathcal{S}_4$, $[T]P_{i_k} = [\phi_4(T)]P_{i_k}$.
\end{proposition}
\begin{proof}
Let $T^\dagger=\phi_4(T)$. Let $f_\QQ:{\mathfrak R}_{\QQ}(T) \to {\mathfrak R}_{\QQ}(T^\dagger)$
be defined by deleting the $i_k\in {\underline\x}$ and replacing the $\circled{i_k}\in {\underline\x^\rightarrow}$ by
$i_k$. Now the proof proceeds exactly as that for Proposition~\ref{prop:bundle_buddies:horizontal/death}.
\end{proof}

\section{Proof of the conjectural $K_T$ rule from \cite{Thomas.Yong:H_T}}\label{sec:proof_of_conjecture}
We briefly recap the conjectural rule for $K_{\lambda,\mu}^{\nu}$ from \cite[Section~8]{Thomas.Yong:H_T}. 
An {\bf equivariant increasing tableau} is an edge-labeled filling of
$\nu/\lambda$ using the labels $1,2,\dots, |\mu|$ such that each label is strictly smaller than any label below it in its column and each \emph{box} label is strictly smaller than the box label immediately to its right.
Any subset of the boxes of $\nu/\lambda$
may be marked by $\star$'s, except that if $i$ and $i+1$ are box labels in the same row, then the box containing $i$ may \emph{not} be $\star$-ed.
Let ${\tt EqInc}(\nu/\lambda,|\mu|)$ denote
the set of all such equivariant increasing tableaux.

An {\bf alternating ribbon} $R$ is a filling of a short ribbon by two symbols such that
adjacent boxes are filled differently;
all edges except the southwestmost edge are empty; and if this edge is filled, it is filled with the other symbol than in the box above it.
Let ${\tt switch}(R)$ be the alternating ribbon of the same shape
where each box is instead filled with the other symbol. If the southwestmost
edge was filled by one of these symbols, that symbol is deleted.
If $R$ consists of a single box with only one
symbol used, then ${\tt switch}$ does nothing to it. Define
${\tt switch}$ to act on an edge-disjoint union of alternating ribbons, by acting on each
independently. 

\begin{example}
\ytableausetup{boxsize=1.2em}
Let $R =
\begin{picture}(53,22)
\put(0,10){$\ytableaushort{\none \heartsuit \spadesuit, \heartsuit \spadesuit}$.}
\put(3,-10){$\spadesuit$}
\end{picture}
$ Then ${\tt switch}(R) = \begin{picture}(60,20)
\put(0,10){$\ytableaushort{\none \spadesuit \heartsuit, \spadesuit \heartsuit}$.}
\end{picture}
$ \qed
\end{example}

Given $T\in {\tt EqInc}(\nu/\lambda,|\mu|)$ and an inner corner $\x\in\lambda$, label $\x$ with $\bullet$ and erase all $\star$'s. Call this tableau $V_1$. Consider
the alternating ribbons $\{ R_1\}$  made of $\bullet$ and $1$. $V_2$ is obtained by applying ${\tt switch}$ to each $R_1$. Now let $\{R_2\}$ be the collection of
ribbons consisting of $\bullet$ and $2$, and produce $V_3$ by applying ${\tt switch}$ to each $R_2$. Repeat until the $\bullet$'s have been switched past all the numerical
labels in $T$; the final placement of numerical labels gives ${\tt KEqjdt}_{\x}(T)$, the {\bf slide} of $T$ into $\x$. The sequence $V_1, V_2, \dots$ is the {\bf switch sequence} of $(T, \x)$. Finally, define ${\tt KEqrect}(T)$ by
successively applying ${\tt KEqjdt}_\x$ in {\bf column rectification order}, i.e., successively
pick $\x$ to be the eastmost inner corner.

\begin{lemma}\label{lem:increasing_goodness}
For $V_j$ in the switch sequence of $(T, \x)$: 
\begin{itemize}
\item[(I)] The numerical
box labels strictly increase along rows from left to right 
(ignoring $\bullet$'s).
\item[(II)] The numerical labels strictly increase down columns
(ignoring $\bullet$'s and reading labels of a given edge in increasing order).
\item[(III)] Every numerical label southeast of a $\bullet$ is at least $j$.
\item[(IV)] Every numerical label northwest of a $\bullet$ is strictly less than $j$. 
\end{itemize}
\end{lemma}
\begin{proof}
These are proved by simultaneous induction on $j$.
In the inductive step, one considers any $2\times 2$ local piece of $V_j$
and studies the possible cases that can arise as one transitions from $V_j\to V_{j+1}$; we leave the straightforward details to the reader.
\end{proof}

A set of labels is a {\bf horizontal strip} if they are
arranged in increasing order from southwest to northeast, with no two labels of the set in the same column.

\begin{lemma}\label{lem:horizontal_stripness_preservation}
Let $T \in {\tt EqInc}(\nu/\lambda, |\mu|)$ and $\x \in \lambda$ be an inner corner. Then $\{i, i+1, \dots, j\}$ forms a horizontal strip in $V_k$ of the switch sequence of $(T, \x)$ if and only it does so in $V_{k+1}$.
\end{lemma}
\begin{proof}
This quickly reduces to consideration of the possibilities in a $2\times 2$ local piece of $V_k$. Then we proceed by 
straightforward case analysis using Lemma~\ref{lem:increasing_goodness}. 
\end{proof}

A label $\mathfrak s \in T$ is {\bf special}
if it is an edge label or
lies in a $\star$-ed box. At most one ${\mathfrak s}$ appears in a column $c$. In column rectification order, each slide ${\tt KEqjdt}_\x$ for $\x \in c$ moves an ${\mathfrak s}$ in $c$ at most one step North
(and it remains in $c$). A
special label ${\mathfrak s}$ in $c$ {\bf passes
through} $\x$ if it occupies $\x$ at any point during $c$'s rectification \emph{and} initially $\mathfrak{s} \notin \x$.
Let $\x_1,\ldots,\x_s$ be the boxes $\mathfrak{s}$ passes through and
let $\y_1,\ldots,\y_t$ be the numerically labeled boxes East of $\x_s$ in the same
row. Set
${\tt factor}_{K}({\mathfrak s}):=1-\prod_{i=1}^{s} {\hat \beta}(\x_i)\prod_{j=1}^{t}
{\hat\beta}(\y_j)$.
If ${\mathfrak s}$ does not move during
the rectification of $c$, then
${\tt factor}_{K}({\mathfrak s}) :=0$.
Now set
${\tt wt}_{K}(T):=\prod_{{\mathfrak s}} {\tt factor}_{K}({\mathfrak s})$,
where the product is over all special labels.
Lastly, we define ${\tt sgn}(T) :=  (-1)^{|\mu| - \#\text{$\star$'s in $T$} - \#\text{labels in $T$}}$.

Let $\mu[1] = ( 1, 2, 3, \ldots, \mu_1)$, $\mu[2] = (\mu_1+1,\mu_1+2,\ldots,\mu_1+\mu_2)$, etc.
Let $T_{\mu}$ be the {\bf superstandard tableau} of shape $\mu$, i.e., row $i$ is filled by $\mu[i]$.
The following is the conjecture of \cite{Thomas.Yong:H_T}:

\begin{theorem}
\label{thm:oldconj}
$K_{\lambda,\mu}^\nu=\sum_{T} {\tt sgn}(T) \cdot {\tt wt}_{K}(T)$,
where the sum is over \[\mathcal{A}_{\lambda,\mu}^\nu := \{T\in {\tt EqInc}(\nu/\lambda,|\mu|) \colon {\tt KEqrect}(T)=T_{\mu} \}.\]
\end{theorem}

We will prove Theorem~\ref{thm:oldconj} (after some preparation) by connecting to Theorem~\ref{thm:main}.  

Let ${\mathcal B}_{\lambda,\mu}^\nu$ be the set of all $T \in {\tt BallotGen}(\nu/\lambda)$ that have content $\mu$.
We need a {\bf semistandardization map} $\Phi \colon {\mathcal A}_{\lambda,\mu}^\nu\to {\mathcal B}_{\lambda,\mu}^\nu$. Given $A\in {\mathcal A}_{\lambda,\mu}^\nu$, erase all $\star$'s and replace the labels $1,2,\ldots,\mu_1$  with
$1_1,1_2,\ldots,1_{\mu_1}$ respectively. Next, replace
$\mu_1+1,\mu_1+2,\ldots,\mu_1+\mu_2$ by
$2_1,2_2,\ldots,2_{\mu_2}$ respectively, etc. The result is
$\Phi(A)$. Note $\Phi$ is not bijective.
Define a {\bf standardization map} $\Psi \colon {\mathcal B}_{\lambda,\mu}^\nu\to {\mathcal A}_{\lambda,\mu}^\nu$ by reversing the above process in the obvious way; $\Psi(B)$ is $\star$-less.

\begin{lemma}
\label{lem:Asemistd}
For $B \in {\mathcal B}_{\lambda,\mu}^\nu$, $\Psi(B)\in {\tt EqInc}(\nu/\lambda,|\mu|)$.
\end{lemma}
\begin{proof}
That $\Psi(B)$ has the desired shape and content is clear.
Row strictness follows from (S.1), and column strictness from (S.2).
\end{proof}

\begin{lemma}
\label{lem:Ahorizstrip}
For $B \in {\mathcal B}_{\lambda,\mu}^\nu$ and for each $i$, $\mu[i]$ forms a horizontal strip in $\Psi(B)$ and also in each tableau of any switch sequence during the column rectification of $\Psi(B)$.
\end{lemma}
\begin{proof}
By (S.2--4), the labels $i_1, \dots, i_{\mu_i}$ form a horizontal strip of $B$. The claim for $\Psi(B)$ is then immediate by definition of $\Psi$. The claim for the tableaux of the switch sequences then follows 
by Lemma~\ref{lem:horizontal_stripness_preservation}. 
\end{proof}

\begin{lemma}\label{lem:no_edge_labels}
Let $B \in {\mathcal B}_{\lambda,\mu}^\nu$. Then 
\begin{itemize}
\item[(I)] after column rectifying the eastmost $j$ columns of $\Psi(B)$, there are no edge labels in these eastmost $j$ columns; and 
\item[(II)] while rectifying the next column, there is never an edge label north of a $\bullet$ and in the same column, in any tableau of any switch sequence.
\end{itemize}
\end{lemma}
\begin{proof}
(I): Suppose there were such an edge label $\ell \in \underline{\x}$ after rectifying the eastmost $j$ columns. Then $\ell \in \underline{\x}$ in $\Psi(B)$, since rectification never adds a label to any edge. Suppose $\x$ is in the $i$th row from the top of $\Psi(B)$. Then since no label of $B$ is too high, $\ell \in \mu[k]$ where $k \leq i$. Let the boxes North of $\underline{\x}$ and in the same column be $\x_1, \dots, \x_i = \x$ from north to south. By Lemma~\ref{lem:increasing_goodness}(II), we have for each $e$ that $\lab(\x_e) \in \mu[f(e)]$ for some $f(e) \leq k$. But then by Lemma~\ref{lem:Ahorizstrip}, $f : \{1, 2, \dots, i\} \to \{1, 2, \dots, k-1\}$ is injective, a contradiction.

(II): Let $c$ be the column currently being rectified. For the columns East of $c$, the claim follows from part (I), noting that rectification never adds a label to any edge. For column $c$ itself, the claim is vacuous if there is no $\bullet$ in $c$. If there is $\bullet \in c$, the claim follows from noting that every label of column $c$ North of this $\bullet$ must have participated in some ${\tt switch}$ and that ${\tt switch}$ never outputs any edge labels. 
\end{proof}

An equivariant
increasing tableau $T$ is {\bf ballot} if $\Phi(T)$ is ballot in the sense of Section~\ref{sec:ballot_property}. That is, for every ${\widetilde T}$ obtained by selecting one copy of each label in $T$, every initial segment of ${\widetilde T}$'s column reading word has, for each $i\geq 1$, at least as many
labels from $\mu[i]$ as from $\mu[i+1]$. We extend this definition to tableaux with $\bullet$'s by ignoring the $\bullet$'s.

\begin{lemma}
\label{lem:Aballot}
Let $B \in {\mathcal B}_{\lambda,\mu}^\nu$. Then $\Psi(B)$ is ballot, as is each tableau of any switch sequence during the column rectification of $\Psi(B)$.
\end{lemma}
\begin{proof}
Let $A = \Psi(B)$.
Since $B$ is ballot and $\Phi(A) = B$, $A$ is ballot by definition. 
Suppose that some $V_q$ is ballot, but $V_{q+1}$ is not. 
Then there exist $i$ and a $\widetilde{V_{q+1}}$ with a ballotness violation between $\mu[i]$ and $\mu[i+1]$.

If $q \notin \mu[i] \cup \mu[i+1]$, then the labels of $\mu[i]$ and $\mu[i+1]$ appear in the same locations in $V_q$ and $V_{q+1}$, contradicting that $V_q$ is ballot. 

If $q\in \mu[i+1]$, then
no $\mu[i]$-label moves.
For each $\ell\in \mu[i+1]$ appearing in $\widetilde{V_{q+1}}$, there is an $\ell$ east of that position in $V_q$. Hence we construct a nonballot $\widetilde{V_q}$ by choosing those corresponding $\ell$'s, the same labels from $\mu[i]$ as in $\widetilde{V_{q+1}}$, and all other labels arbitrarily. This contradicts that $V_q$ is ballot.  

Finally if $q\in \mu[i]$, then there is some $\x$ in column $c$ of $V_q$ with $\bullet \in \x$ and $q\in \x^\rightarrow$ such that the $q$ moving into
$\x$ violates ballotness in the columns East of
$c$. That is, locally the {\tt switch} is
\[\ytableausetup{boxsize=1.1em} V_q \supseteq \ytableaushort{a b, \bullet q, d e}\mapsto
\ytableaushort{a b, q \bullet , d e} \subseteq V_{q+1} \text{\, \, or \, \,}
V_q \supseteq \ytableaushort{a \bullet, \bullet q, d e}\mapsto
\ytableaushort{a q, q \bullet , d e} \subseteq V_{q+1},
\]
where the $\x$ is the left box of the second row. The $q \in \x^\rightarrow$ is Westmost in $V_q$, since otherwise the nonballotness of $V_{q+1}$ contradicts that $V_q$ is ballot. In particular, $q \neq d$. Hence by Lemma~\ref{lem:increasing_goodness}(III), $q<d$.

Since $V_q$ is ballot but $V_{q+1}$ is not, there is a $\bar{q} \in \mu[i+1]$
in $c^\rightarrow$ in $V_q$, and hence in $V_{q+1}$. By Lemma~\ref{lem:increasing_goodness}(II) applied to $V_q$, this $\bar{q}$ is below $q$ in $c^\rightarrow$. By Lemma~\ref{lem:no_edge_labels}(I), there are no edge labels East of column $c$. So in fact $e$ and hence $d$ both exist. Indeed by Lemma~\ref{lem:increasing_goodness}(II) and Lemma~\ref{lem:Ahorizstrip},
$e=\bar{q}$.  By Lemma~\ref{lem:Ahorizstrip}, $q$ is the only label of $\mu[i]$ that appears in $c$ in $V_{q+1}$. Hence $d\not\in \mu[i]$. Thus by Lemma~\ref{lem:increasing_goodness}(I) applied to $V_q$,
we conclude $d\in \mu[i+1]$. However this again contradicts that $V_q$ is ballot.
\end{proof}

For $A \in {\tt EqInc}(\nu/\lambda,|\mu|)$, let $A^{(k)}$ be the ``partial'' tableau that is the column rectification of the eastmost $k$ columns of $A$.
\begin{lemma}
\label{lem:A^k}
Let $B \in \mathcal{B}_{\lambda, \mu}^\nu$ and let $A = \Psi(B)$. For each $i$, the $i$th row of $A^{(k)}$ consists of a (possibly empty) final segment from $\mu[i]$.
\end{lemma}
\begin{proof}
By Lemma~\ref{lem:increasing_goodness}(I, II), $A^{(k)}$ has strictly increasing rows and columns.
By Lemma~\ref{lem:Ahorizstrip}, the labels $\mu[i]$ form a horizontal strip in $A^{(k)}$ for each $i$; moreover the labels of $\mu[i]$ appearing in $A^{(k)}$ are a final segment of $\mu[i]$.  By Lemma~\ref{lem:no_edge_labels}(I), there are no edge labels in $A^{(k)}$.
By Lemma~\ref{lem:Aballot}, $A^{(k)}$ is ballot. The lemma follows.
\end{proof}

\begin{corollary}
\label{cor:rects}
$A$ rectifies to $T_\mu$.
\end{corollary}
\begin{proof}
Immediate from Lemma~\ref{lem:A^k}.
\end{proof}

\begin{proposition}\label{prop:Psi_codomain}
For $B \in {\mathcal B}_{\lambda,\mu}^\nu$, $\Psi(B) \in{\mathcal A}_{\lambda,\mu}^\nu$.
\end{proposition}
\begin{proof}
By Lemma~\ref{lem:Asemistd}, $\Psi(B) \in {\tt EqInc}(\nu/\lambda, |\mu|)$. By Corollary~\ref{cor:rects}, $\Psi(B)$ rectifies to $T_\mu$.
\end{proof}

\begin{lemma}
\label{lem:horizstrip}
For $A \in {\mathcal A}_{\lambda,\mu}^\nu$, $\mu[i]$ forms a horizontal strip in $A$ and each $A'$ in the column rectification of $T$.
\end{lemma}
\begin{proof}
This is true for $T_\mu$, and hence true for $A$ and each $A'$ by Lemma~\ref{lem:horizontal_stripness_preservation}.
\end{proof}

\begin{lemma}
\label{lem:Bsemistd}
For $A \in {\mathcal A}_{\lambda,\mu}^\nu$, $\Phi(A)$ is semistandard.
\end{lemma}
\begin{proof}
Row-strictness of $A$ implies that $\Phi(A)$ satisfies (S.1).
Since by Lemma~\ref{lem:horizstrip}, $\mu[i]$ is a horizontal strip in $A$ for each $i$, (S.2)--(S.4) hold in $\Phi(A)$. 
\end{proof}

\begin{lemma}
\label{lem:Bballot}
For $A \in {\mathcal A}_{\lambda,\mu}^\nu$, $\Phi(A)$ is ballot.
\end{lemma}
\begin{proof}
Suppose $\Phi(A)$ is not ballot. Then by definition, $A$ is not ballot. We assert that every tableau in every switch sequence in the column rectification of $A$ is also \emph{not} ballot, implying $T_\mu$ is not ballot, a contradiction.

Suppose $V_\ell$ is not ballot, but $V_{\ell+1}$ is. We derive a contradiction. Since $V_\ell$ is not ballot, we pick a nonballot ${\widetilde V_\ell}$. Suppose this nonballotness can be blamed on
positions $\aaa_1, \ldots, \aaa_s$ containing labels of  $\mu[i]$ and positions $\bbb_1, \ldots, \bbb_{s+1}$ containing labels of $\mu[i+1]$ (for some $i$). Suppose $\aaa_1,\ldots, \aaa_s$ and $\bbb_1,\ldots, \bbb_{s+1}$ are left to right in $\widetilde{V_\ell}$; no two $\aaa_j$'s (respectively $\bbb_j$'s) are in the same column by Lemma~\ref{lem:horizstrip}. We may assume $\bbb_1$ is southwestmost among all these positions, say in column $c$ and that among all offending choices of $i$ and positions, we picked one so that $c$ is eastmost.

Since $V_{\ell+1}$ is supposed ballot, there
is a label $\ell\in \mu[i+1]$ in $\bbb_1$ of $V_\ell$ that moved to column $c^\leftarrow$. Locally, the {\tt switch} is
$\ytableaushort{x y, \bullet \ell} \mapsto \ytableaushort{x y, \ell \bullet}$.
By Lemma~\ref{lem:horizstrip}, $\mu[i+1]$ forms a horizontal strip in $V_\ell$.
Hence $x,y\notin \mu[i+1]$. Also, no label in column $c$ is in $\mu[i]$
since otherwise we contradict that $c$ is chosen eastmost. Now, there is some label $m\in \mu[i]$ above the $\bullet$ in column $c^\leftarrow$
of $V_\ell$ since $V_{\ell+1}$ is ballot.
Using Lemma\ref{lem:no_edge_labels}(II), it follows that $m=x$. Now, we have argued $y\notin \mu[i]\cup
\mu[i+1]$. However, by Lemma~\ref{lem:increasing_goodness}(I, II) applied to $V_\ell$, there are no other possibilities for $y$,
a contradiction.
\end{proof}

\begin{proposition}\label{prop:Phi_codomain}
For $A \in {\mathcal A}_{\lambda,\mu}^\nu$, $\Phi(A) \in{\mathcal B}_{\lambda,\mu}^\nu$. 
\end{proposition}
\begin{proof}
By construction, $\Phi(A)$ is an edge-labeled genomic tableau of shape $\nu / \lambda$ and content $\mu$. By Lemma~\ref{lem:Bsemistd}, $\Phi(A)$ is semistandard.  By Lemma~\ref{lem:Bballot}, $\Phi(A)$ is ballot. Since $A$ rectifies to $T_\mu$, no label of $\Phi(A)$ is too high.
\end{proof}

Given a label $\ell$ in $A\in {\mathcal A}_{\lambda,\mu}^\nu$, 
let $\Phi(\ell)$ be the corresponding label in $\Phi(A)\in {\mathcal B}_{\lambda,\mu}^\nu$.
Recall the definitions of Section~\ref{sec:tableau_weights}.

\begin{lemma}
\label{lemma:edge}
\gap
\begin{itemize}
\item[(I)] If $\ell$ is an edge label, then ${\tt factor}_K(\ell)={\tt edgefactor}(\Phi(\ell))$.
\item[(II)] If $\ell$ is in a $\star$-ed box, then ${\tt factor}_K(\ell)=1-{\tt boxfactor}(\Phi(\ell))$.
\end{itemize}
\end{lemma}
\begin{proof}
These follow from the definitions of the factors
combined with Lemma~\ref{lem:A^k}.
\end{proof}

\begin{lemma}
\label{lemma:starwt}
If $B\in {\mathcal B}_{\lambda,\mu}^\nu$, then
\[{\tt boxwt}(B) = \sum_{A\in \Phi^{-1}(B)}(-1)^{\text{$\#\star$'s in $A$}}\prod_{\text{\rm special box label $\ell$ of $A$} }{\tt factor}_K(\ell).\]
\end{lemma}
\begin{proof}
A box $\x$ is productive in $B$ if and only if it may be $\star$-ed in $\Psi(B)$.
We are done by Lemma~\ref{lemma:edge}(II) and the ``inclusion-exclusion'' identity
$\sum_{S\subset [N]} (-1)^{|S|}\prod_{s\in S}(1-z_s)=z_1 z_2\cdots z_N$.
\end{proof}

\begin{proof}[Proof of Theorem~\ref{thm:oldconj}]
Recall Theorem~\ref{thm:oldconj} asserts
$K_{\lambda,\mu}^\nu=\sum_{A\in {\mathcal A}_{\lambda,\mu}^\nu} {\tt sgn}(A){\tt wt}_K(A)$. To see this, 
observe that by Propositions~\ref{prop:Psi_codomain} and~\ref{prop:Phi_codomain},

\begin{align*}
\sum_{A\in {\mathcal A}_{\lambda,\mu}^\nu}& {\tt sgn}(A){\tt wt}_K(A)
= \sum_{B\in {\mathcal B}_{\lambda,\mu}^\nu}\sum_{A\in \Phi^{-1}(B)}
{\tt sgn}(A) {\tt wt}_K(A)\\
= & \sum_{B\in {\mathcal B}_{\lambda,\mu}^\nu}\sum_{A\in \Phi^{-1}(B)} \!\!
(-1)^{|\mu| - \#\text{$\star$'s in $A$} - \#\text{labels in $A$}} \!\!\!\!
\prod_{\text{edge label $\ell$ of $A$} }\!\!\!\!{\tt factor}_K(\ell)\!\!\!\!
\prod_{\text{special box label $\ell$ of $A$} }\!\!\!\!{\tt factor}_K(\ell)\\
= & \sum_{B\in {\mathcal B}_{\lambda,\mu}^\nu}\sum_{A\in \Phi^{-1}(B)} \!\!
(-1)^{|\mu| - \#\text{labels in $A$}}
\left(\prod_{\text{edge label $\ell$ of $A$} }\!\!\!\!\!\!\!{\tt factor}_K(\ell)\right)
(-1)^{\text{$\#\star$'s in $A$}} \!\!\!\!\!\!\!\!\!\!\!
\!\!\!\prod_{\text{special box label $\ell$ of $A$} }\!\!\!\!\!\!\!\!\!\!\!\!\!\!\!\!{\tt factor}_K(\ell).
\end{align*}
The number of labels of $A$ equals the number of labels of $B$ for any $A\in \Phi^{-1}(B)$.
Combining this with Lemma~\ref{lemma:edge}(I) shows the previous expression equals
\begin{align*}
= & \sum_{B\in {\mathcal B}_{\lambda,\mu}^\nu}(-1)^{|\mu| - \#\text{labels in $B$}}
\left(\prod_{\text{edge label $\ell$ of $B$}}\!\!\!\!\!\!\!\!{\tt edgefactor}(\ell)\right)
\sum_{A\in \Phi^{-1}(B)} \!\! (-1)^{\text{$\#\star$'s in $A$}}\!\!\!\!\!\!\!\!\!\!\!\!\!
\prod_{\text{special box label $\ell$ of $A$} }\!\!\!\!\!\!\!\!\!\!\!\!\!\!{\tt factor}_K(\ell).
\end{align*}
By Lemma~\ref{lemma:starwt}, this equals
\[=\sum_{B\in {\mathcal B}_{\lambda,\mu}^\nu}(-1)^{|\mu| - \#\text{labels in $B$}}
{\tt edgewt}(B)
{\tt boxwt}(B):=L_{\lambda,\mu}^\nu.\]
Since by Theorem~\ref{thm:main}, $L_{\lambda,\mu}^\nu=K_{\lambda,\mu}^\nu$, we are done.
\end{proof}

\appendix
\section{Proof of Proposition~\ref{prop:goodness_preservation}}
\label{sec:forward_goodness_proof}
We check (G.1)--(G.13) are preserved. Let $U \in \swap_\GG(T)$. 
We prove the conditions in the order: (G.1), (G.2), (G.4), (G.5), (G.6), (G.7), (G.3), (G.8), (G.9), (G.11), (G.10), (G.12), (G.13). In this way, each proof depends only on previously proved conditions.
It is also necessary to show that the prescribed virtual labels in {\sf H5.2} and {\sf T4.2} satisfy the rules (V.1)--(V.3) for virtual labels. This is done in Lemma~\ref{lemma:virtuallabelsexist} as part of the discussion of (G.13).

\smallskip
\noindent
{\bf (G.1):} 
By $T$'s (G.1), no label of $T$ is too high. Hence if some label $\mathcal{P}$ of $U$ is too high, it must be placed in such a location by some miniswap. We therefore consider all miniswaps that might place a label $\mathcal{P}$ on $\x$ or $\underline{\x}$ in $U$, when there is no $\mathcal{P}$ or $\circled{\mathcal{P}}$ north of $\underline{\x}$ in $T$.
By inspection, the miniswaps that can do so are {\sf H1}, {\sf T1}, {\sf T3}, {\sf T4} and {\sf T5}. 

{\sf H1:} The first output of {\sf H1} is not problematic. For if the $\GG \in \x$ in the first output were too high, the $\GG \in \underline{\x}$ in $T$ would also have been to high, in violation of $T$'s (G.1). If the second output of {\sf H1} creates a label that is too high, then by definition of $\gamma$, that output is produced with coefficient $0$. Thus {\sf H1} does not create a tableau violating (G.1).

{\sf T1:} Suppose $T$ has $\bullet_\GG \in \x$ and that, after applying {\sf T1}, $U$ has $\GG \in \x$, which is too high. By Lemma~\ref{lem:piecesobservations}(IV,VII), $T$ has $\GG \in \x^\downarrow$. By $T$'s (G.1), this $\GG \in \x^\downarrow$ is not too high. Since $\GG \in \x^\downarrow$ is \emph{not} too high, but $\GG \in \x$ \emph{is} too high, $\x$ is in row $i-1$ where $i = \family(\GG)$. In particular, $i > 1$.

By $T$'s (G.2), $\bullet_\GG \notin \x^\rightarrow$ in $T$, so if $\x^\rightarrow$ is a box of $T$, then it contains a genetic label. Consider $\lab_T(\x^\rightarrow)$. If $\lab_T(\x^\rightarrow) \prec \GG$, then $\lab_T(\x^\rightarrow)$ is marked. Hence by Lemma~\ref{lem:same999}, $T$ has $\circled{\GG} \in \underline{\x^\rightarrow}$. This contradicts that we are applying {\sf T1}, for then $\x^\rightarrow$ is adjoined to the snake containing $\x$ by (R.3). If $\lab_T(\x^\rightarrow) = \GG$, this again contradicts that we are applying {\sf T1}. If $\lab_T(\x^\rightarrow) \succ \GG$, then $\family(\GG) \leq \family(\lab_T(\x^\rightarrow))$. Thus if $\GG \in \x$ were too high in $U$, then $\lab_T(\x^\rightarrow)$ would already be too high in $T$, violating $T$'s (G.1). Thus $\x^\rightarrow$ must not be a box of $T$.

Let $c$ be the column of $\x$. Consider the labels of $T$ of family $i-1$.  
Suppose such a label appears in column $c$ of $T$. It cannot be South of $\x$, for then it would be marked and violate $T$'s (G.11).  It cannot be in $\x$, since $\bullet_\GG \in \x$. It cannot be North of $\x$, for then it would be too high and violate $T$'s (G.1). Thus there is no label of family $i-1$ in column $c$. 

Since $\x^\rightarrow$ is not a box of $T$, every column East of $c$ contains at most $i-2$ boxes. Therefore any label of family $i-1$ East of column $c$ would be too high, contradicting $T$'s (G.1). Consider a genotype of $T$ involving the $\GG \in \x^\downarrow$. No label of family $i-1$ is read before the $\GG \in \x^\downarrow$. This contradicts $T$'s (G.8). Thus this miniswap cannot create a label that is too high.

{\sf T3:} If the $\GG \in \x$ in either output of {\sf T3} were too high, the $\GG^+ \in \x^\rightarrow$ in $T$ with $\family(\GG^+) = \family(\GG)$ would violate $T$'s (G.1). 

{\sf T4 and T5:} Suppose that either {\sf T4} or {\sf T5} created a label $\mathcal{P}$ that was too high in $U$. Then in the notation of those miniswaps, $U$ has $\mathcal{P}$ in the edge $\overline{\x}$ and we must have $\mathcal{P} \in \{ \FF \} \cup Z$. 
By $T$'s (G.13), there is a label $\QQ \in \underline{\x^\rightarrow}$
(possibly virtual) in $T$ with $\family{\QQ} = \family(\mathcal{P}) + 1$. This $\QQ \in \underline{\x^\rightarrow}$ is then too high in $T$, contradicting $T$'s (G.1).

\smallskip
\noindent
{\bf (G.2):} Consider a snake $S$ in $T$. Since $S$ is a short ribbon (Lemma~\ref{lem:snakes_are_short_ribbons}),
in the region of $U$ defined by $S$, (G.2) can only be violated by having two
$\bullet_{\GG^+}$'s in the same row or column. By inspection of the miniswaps, no two $\bullet_{\GG^+}'s$ can appear in the same row.
If two $\bullet_{\GG^+}$'s appear in the same column, the top $\bullet_{\GG^+}$ arose from a {\sf T4.1} or
{\sf T4.2} miniswap. However in those cases the edge label $\GG\in \tail(S)$ implies $\tail(S)=S$ by $T$'s (G.7).

Thus we check $U$'s (G.2) for pairs of 
snakes $S, S'$. By Lemma~\ref{lem:snakes_arranged_SW-NE}, say $S$ is southwest of $S'$.
If $S$ is entirely SouthWest of $S'$, (G.2) preservation is clear. It remains to consider
the situations where $S$ and $S'$ share a row or a column.

Suppose the snakes are in the configuration of Lemma~\ref{lem:snakes_arranged_SW-NE}(II). 
Here $S=\{\x\}=\ytableaushort{{\bullet_{\GG}}}$ ($\GG,\circled{\GG} \not\in\underline{\x}$). So $\x$ takes part in a trivial {\sf H3} miniswap.  
By $T$'s (G.2), the southmost row $r'$ of $S'$ (assumed to be in $\x$'s row) does not contain $\bullet_{\GG}$. 
Thus $r'$ takes part in a {\sf H9}, {\sf B1} or {\sf B3} miniswap. 
{\sf H9} and {\sf B1} do not introduce a $\bullet_{\GG^+}$, so (G.2) holds here.
We claim {\sf B3} is not possible.
If $r'$ participates in a {\sf B3} miniswap, then by definition $r'=\{\y\}=\ytableaushort{\GG}$. It cannot be that $\y^\leftarrow = \x$, for then $S$ and $S'$ would be the same snake. Let $\FF = \lab(\y^\leftarrow)$. By (G.3), $\FF \prec \GG$. Hence the $\bullet_\GG \in \x$ means that the $\FF \in \y^\leftarrow$ is marked. But then $\y^\leftarrow$ was adjoined to $S'$ by (R.3), i.e.,
$r'=\{\y^\leftarrow,\y\}=\ytableaushort{{\FF^!} \GG}$, a contradiction.

Finally suppose the snakes are in the configuration of Lemma~\ref{lem:snakes_arranged_SW-NE}(III). 
The two adjacent rows of $S$ and $S'$ are $\ytableausetup{boxsize=1.2em}\ytableaushort{\none {\bullet_{\GG}}, {\bullet_{\GG}} {\GG^+}}$ or
$\ytableaushort{\none {\bullet_{\GG}} {\GG^+}, {\bullet_{\GG}} {\GG^+}}$. 
Hence $S'$ takes part in a {\sf H3} or {\sf H8} trivial miniswap. Let $\x$ be the east box of the northmost row of $S$.
The box $\x$ takes part in miniswap {\sf H6}, {\sf H7}, {\sf H8}, or {\sf T3}. 
If it is miniswap {\sf H8}, $\bullet_{\GG^+}\not\in \x$ in $U$.
In the other cases, by definition of $\alpha$, the tableaux produced with $\bullet_{\GG^+} \in \x$ appears with coefficient $0$.

\smallskip
\noindent
{\bf (G.4):}
We show that $U$ does not violate (G.4) in a given column $c$. 

\noindent
{\sf Case 1: ($\bullet_{\GG} \notin c$ in $T$):} By inspection of the
miniswaps,
$c$ either has 
labels removed or else a box label of $c$ is pushed onto a lower edge
of the same box (and a $\bullet_{\GG^+}$ comes into $c$). 

\noindent
{\sf Subcase 1.1: ($c$ is strictly increasing in $T$):} By the above, $c$ is strictly increasing in $U$.

\noindent
{\sf Subcase 1.2: ($c$ is not strictly increasing in $T$):}
Therefore $c$ contains $\{\x, \x^\uparrow \} = \ytableaushort{\FF, {\FF^!}}$ in $T$. 
By the above observation, it suffices to show $\{\x, \x^\uparrow \}$ is not
$\ytableaushort{{\FF^!}, {\FF^!}}$ or $\ytableaushort{\FF, \FF}$ in $U$. Since $\FF^!$ appears in $T$, $\FF \prec \GG$. 
Since $\FF \in \x^\uparrow$ in $T$, there is no $\bullet_\GG$ in $T$ northwest of $\x^\uparrow$. Thus in $U$ no $\bullet_{\GG^+}$ can appear northwest of $\x^\uparrow$. Hence $\FF^! \notin \x^\uparrow$ in $U$. This rules out the first scenario.

We now rule out the second scenario.
By Lemma~\ref{lem:strong_form_of_G10}, in $T$ there is a $\bullet_\GG$ in some box $\y$ West of $\x$ in the same row, 
and furthermore $\EE^! \in \y^\rightarrow$ in $T$.
By Lemma~\ref{lem:strong_form_of_G13}, $T$ has 
$\GG' \in \underline{\y^\rightarrow}$ (possibly marked), $\circled{\GG'} \in \underline{\y^\rightarrow}$
or $(\GG')^!\in \y^\rightarrow$, where $\family(\GG') = \family(\GG)$. Let $S$ be the snake containing $\y$. 

\noindent
{\sf Subcase 1.2.1: ($\y^\rightarrow \in S$):} Since $\EE^!\in\y^\rightarrow$ we have $\{ \y, \y^\rightarrow \} = \tail(S)$. Thus, 
$U$ has $\bullet_{\GG^+} \in \y$ or $\bullet_{\GG^+} \in \y^\rightarrow$. Hence $\FF \notin \x$ in $U$. 

\noindent
{\sf Subcase 1.2.2: ($\y^\rightarrow \notin S$ and neither $\GG$ nor $\circled{\GG}$ appears in $\y$'s column):} Then $S=\{\y\}$ undergoes {\sf H3} and 
$\lab_{U}(\y)=\bullet_{\GG^+}$. Hence $\FF^! \in \x$ in $U$. 

\noindent
{\sf Subcase 1.2.3: ($\y^\rightarrow \notin S$ and either $\GG$ or $\circled{\GG}$ appears in $\y$'s column):}
By Lemma~\ref{lem:same999} applied to $T$, $\GG'=\GG$. Hence $\y^\rightarrow \in S$, violating the assumption of {\sf Subcase 1.2.3}.

\noindent
{\sf Case 2: $(\bullet_\GG \in \x$, where $\x$ is a box of $c$ in $T$):} Let $S$ be the snake containing $\x$.

\noindent
{\sf Subcase 2.1: ($\x \in \head(S)$):} Clearly, there is no (G.4) violation except possibly if we apply {\sf H5.2} or {\sf H5.3}, where $\lab(\x^\rightarrow)=\GG$; thus we assume we are using one of these miniswaps.
Let $\FF$ be the $\prec$-greatest label appearing in $\x^\uparrow$ or $\underline{\x^\uparrow}$. Let $\HH$ be the $\prec$-least label appearing in $\x^\downarrow$ or $\overline{\x^\downarrow}$. After the miniswap, $\GG$ appears in $\x$. We show $\FF < \GG < \HH$.  Since in $T$, $\FF$ is northwest of $\bullet_\GG$, $\FF \prec \GG$ by $T$'s (G.9). If $\family(\FF) = \family(\GG)$, then the $\FF \in \x^\uparrow$ or $\underline{\x^\uparrow}$ and the $\GG \in \x^\rightarrow$ violate $T$'s (G.12). Hence $\FF < \GG$.
If $\HH \prec \GG$, then the $\HH \in \x^\downarrow$ or $\overline{\x^\downarrow}$ is marked in $T$, violating (G.11). If $\HH=\GG$, then 
since we are using {\sf H5.2} or {\sf H5.3}, $\HH=\GG\in \x^\downarrow$ so $\x^\downarrow \in S$, contradicting $\x \in \head(S)$. 
Hence $\GG \prec \HH$. So by $T$'s (G.6), $\GG < \HH$.

\noindent
{\sf Subcase 2.2: ($\x \in \tail(S)$):} 

\noindent
{\sf Subcase 2.2.1: $(\tail(S)$ is {\sf T1}, {\sf T2} or {\sf T3}):} 
By Lemma~\ref{lem:piecesobservations}(IV,V,VII), $S$
has at least two rows and $\lab_T(\x^\downarrow)=\GG$. Suppose there were a label $\QQ$ on $\overline{\x^\downarrow}$ in $T$. By $T$'s (G.4), 
$\QQ < \GG$. But then this $\QQ \in \overline{\x^\downarrow}$ is marked, violating $T$'s (G.11). Hence $\overline{\x^\downarrow}$ is empty.
Let $\FF$ be the $\prec$-greatest label appearing in $\x^\uparrow$ or $\underline{\x^\uparrow}$. Let $\HH$ be the $\prec$-least label appearing in $\x^{\downarrow\downarrow}$ or $\underline{\x^\downarrow}$. Since there is $\GG \in \x^\downarrow$ in $T$, by $T$'s (G.4) we have $\FF < \GG < \HH$.

Each of {\sf T1}, {\sf T2}, or {\sf T3} puts $\GG \in \x$. The swap does not affect $\FF$ nor $\HH$ in column $c$. Now, if the swap 
puts $\bullet_{\GG^+}\in \x^\downarrow$ in $U$, we are done since $\FF < \GG < \HH$. So assume otherwise. Then $\x^\downarrow$ takes part in a miniswap {\sf H5.1}, {\sf H5.2} (choosing the first output), or {\sf H9}. 
In these cases, $U$ has $\GG \in \x$ and $\GG \in \x^\downarrow$. 
Since $\FF < \GG < \HH$, to show that (G.4) holds, we need $\GG \in \x^\downarrow$ in $U$ to be marked.
If the miniswap was {\sf H5.1} or {\sf H5.2}, then
$\bullet_{\GG^+} \in \x^{\downarrow\leftarrow}$ in $U$, so $U$ has $\GG^! \in \x^\downarrow$. If the miniswap is {\sf H9}, then there is some marked label $\EE^! \in \x^{\downarrow\leftarrow}$. By $T$'s (G.10), there is some $\bullet_\GG$ West of $\x^{\downarrow\leftarrow}$ and in its row of $T$. By Lemma~\ref{lem:snakes_arranged_SW-NE}, this $\bullet_\GG$ is part of a single-box snake, which undergoes miniswap {\sf H3}, the $\bullet_\GG$ becoming $\bullet_{\GG^+}$ in the same position. Hence $U$ has $\GG^! \in \x^\downarrow$.

\noindent
{\sf Subcase 2.2.2: ($\tail(S)$ is {\sf T4} or {\sf T5}):} The output of {\sf T4.1} and the first output of {\sf T4.2} leave $c$ unaffected, so 
since no other box of $c$ is part of a snake,
we are done.  The three remaining possibilities (second output of {\sf T4.2},
{\sf T4.3} and {\sf T5}) are similar, so we argue them together.
In these cases, notice $\x^\uparrow$ is not part of a snake. Let $\FF$ and $Z$ be as in the description of these miniswaps. Each places $\GG \in \x$ and $\FF \cup Z \in \overline{\x}$. Let $\HH$ be the $\prec$-least label on $\x^\downarrow$ or $\overline{\x^\downarrow}$. By $T$'s (G.11) 
and since $\lab_T(\x)=\bullet_\GG$, we have $\GG \preceq \HH$. If $\GG < \HH$, this $\HH$ is not moved by the swap, and column $c$ of $U$ satisfies (G.4) at least south of $\x$. Otherwise $\family(\GG) = \family(\HH)$. Since each of these miniswaps says $\GG$ or $\circled{\GG} \in \underline{\x^\rightarrow}$, if $\GG \prec \HH$, then by $T$'s (G.6), $\GG = \HH$. If $\GG=\HH \in \overline{\x^\downarrow}$, we are in {\sf T6}, a contradiction. Hence $\GG=\HH \in \x^\downarrow$ (and $\overline{\x^\downarrow}$ is empty). Consider what happens to $\x^\downarrow$ during the swap. The analysis
to show (G.4) is satisfied south of $\x$
is essentially the same as in {\sf Subcase 2.2.1}, so we omit the details.

Finally, we show (G.4) for $U$ north of $\x$. Let $\EE$ be the $\prec$-greatest label appearing in $T$ in $\x^\uparrow$ or $\underline{\x^\uparrow}$. By $T$'s (G.9), $\EE \prec \GG$. Each of the miniswaps of interest asserts $\GG$ or $\circled{\GG} \in \underline{\x^\rightarrow}$. 
Hence by $T$'s (G.12), $\EE < \GG$. Indeed by $T$'s (G.12), if $\family(\EE) \neq \family(\FF)$ and $\family(\EE) \neq \family(\mathcal{Z})$ for any $\mathcal{Z} \in Z$. Hence by Lemma~\ref{lem:strong_form_of_G13} applied to $\x^\rightarrow$, $\EE < \FF$. Therefore (G.4) holds in $U$ in $c$ north of $\x$.

\noindent
{\sf Subcase 2.2.3: ($\tail(S)$ is {\sf T6}):} No tableau is produced.

\noindent
{\sf Subcase 2.3: ($\x \in \body(S)$):} By Definition-Lemma~\ref{def-lem:snake_classification} and
 Lemma~\ref{lem:piecesobservations}(III), $T$ has $\GG \in \x^\downarrow$ and $\GG \in \x^\rightarrow$. The swap places $\GG \in \x$, and either replaces the $\GG \in \x^\downarrow$ with a $\bullet_{\GG^+}$ or else leaves the $\GG$ in place there. For the rest of the analysis, one proceeds exactly in the manner given
in {\sf Subcase 2.2.1}.

\smallskip
\noindent
{\bf (G.5):} If $U$ violates (G.5), the violation occurs on a horizontal edge $e$ bounding a box of a snake $S$ in $T$ (here $e$ may possibly be
a northern boundary edge of $S$, although only the edge labels of the southern boundary edges are defined as part of $S$).
First assume that we are not in the case of Lemma~\ref{lem:snakes_arranged_SW-NE}(III), so $S$ does not share a column with any other snake.

We break our analysis based on where $e$ appears in relation to $S$.

\noindent
{\sf Case 1: ($e$ bounds a box of $\body(S)$ but not a box of $\head(S)$ or $\tail(S)$)}:
There is no change of labels on $e$ between $T$ and $U$. Hence there is no (G.5) violation on $e$ in $U$.

\noindent
{\sf Case 2: ($e = \underline{\x}$ for some $\x \in \head(S)$)}: The only head miniswaps that could introduce new edge labels
onto $e$ are {\sf H6} and {\sf H7}. In these cases, $T$ has a $\GG^+ \in \x$ that moves to $e = \underline{\x}$ in $U$. If $\GG' \in e$ in $T$ with $\family(\GG') = \family(\GG)$, we violate $T$'s (G.4). Hence the $\GG^+ \in e$ in $U$ is the only label of its family on $e$, as desired.

\noindent
{\sf Case 3: ($e = \overline{\x}$ for some $\x \in \head(S)$)}: If $\x$ is the only box of $\head(S)$, we used {\sf H1} to move a $\GG$ from $\underline{\x}$ to $\overline{\x}$. If there is a label $\GG' \in \overline{\x}$ in $T$ with $\family(\GG') = \family(\GG)$, we
violate $T$'s (G.4). If $|\head(S)|=2$, no miniswap introduces edges onto a northern edge.

\noindent
{\sf Case 4: ($e = \underline{\x}$ for some $\x \in \tail(S)$)}: If $|\tail(S)|=1$, no new edge labels occur during any miniswap (namely
{\sf T1}), so we are done. So assume $|\tail(S)|=2$. New edge labels on $e$ can only occur when using {\sf T3} (second output).
Here, $T$ has $\GG^+ \in \x$, while $U$ has $\GG^+ \in \underline{\x}$. If $U$ violates (G.5), there is $\GG' \in \underline{\x}$ in $T$ with $\family(\GG') = \family(\GG)$, but this contradicts $T$'s (G.4).

\noindent
{\sf Case 5: ($e = \overline{\x}$ for some $\x \in \tail(S)$)}: The miniswaps that could introduce edge labels onto $e$ are
{\sf T4.2}, {\sf T4.3} and {\sf T5}. In the notation of those miniswaps, $T$ has $\bullet_\GG \in \x$, an $\FF^! \in \x^\rightarrow$, and a set $Z$ of labels $\ell$ on $\underline{\x^\rightarrow}$ such that $\FF \prec \ell \prec \GG$. In $U$, all of these labels $\{\FF\} \cup Z$ may have moved to $\overline{\x}$. $U$ violates (G.5) only if it has a label $\QQ \in \overline{\x}$ with $\family(\QQ) = \family(\mathcal{Z})$ for some $\mathcal{Z} \in \{\FF\} \cup Z$. However this $\QQ$ and $\mathcal{Z}$ would violate $T$'s (G.12).

Finally, suppose we are in the case of Lemma~\ref{lem:snakes_arranged_SW-NE}(III). In $T$ the adjacent rows of the snakes are
$\ytableausetup{boxsize=1.2em}\ytableaushort{\none {\bullet_{\GG}}, {\bullet_{\GG}} {\GG^+}}$ \ or
$\ytableaushort{\none {\bullet_{\GG}} {\GG^+}, {\bullet_{\GG}} {\GG^+}}$. Let $\x$ be the east box of the south row in either case.
If $i_j,k_\ell\in \overline{\x}$ in $U$, then $i_j,k_\ell\in \overline{\x}$ in $T$, since 
no miniswap affects $\overline{\x}$. By $T$'s (G.5), $i\neq k$.

\smallskip
\noindent
{\bf (G.6):}
Consider $\HH, \HH'$ in $T$ with $\HH \prec \HH'$ and $\family(\HH) = \family(\HH')$. Say the eastmost $\HH$ in $T$ appears in column $c$, while the westmost $\HH'$ appears in column $d$. By $T$'s (G.4) and (G.6), $c$ is West of $d$. By the swaps' construction, in any $U \in \swap_\GG(T)$, the westmost $\HH'$ in $U$ appears at most one column west of $d$, while the eastmost $\HH$ in $U$ is no further east than column $c$. In any case, no $\HH'$ can be West of an $\HH$ in $U$.

\smallskip
\noindent
{\bf (G.7):} Let $e$ be an edge with $\HH \in e$ in $U$.
We must show there is no $\HH$ West of $e$ in $U$.

\begin{claim}
\label{claim:G7forwardA}
Let ${\mathcal R}$ be the region consisting of the leftmost $c-1$ columns of $T$ 
(equivalently $U$). If $U$ has an $\HH$ in ${\mathcal R}$, then $T$  
has an $\HH$ either in ${\mathcal R}$ or in column $c$. 
\end{claim}
\begin{proof}
By inspection of the miniswaps, if there is an $\HH$ in column $d$ of $U$, then there was an $\HH$ or $\circled{\HH}$ in either $d$ or $d^\rightarrow$ in $T$. By definition of virtual labels, the existence of $\circled{\HH}$ implies the existence of $\HH$ further West. The claim follows.
\end{proof}

\noindent
{\sf Case 1: ($\HH \notin e$ in $T$):} We list the miniswaps that put $\HH\in e$ in $U$: {\sf H1}, {\sf H6}, {\sf H7}, {\sf T3}, {\sf T4.2}, {\sf T4.3}, {\sf T5}. In what follows, $\x$ refers to the notation of the miniswap discussed.
For {\sf H1}, locally 
we have
\ytableausetup{boxsize=1em} 
$\begin{picture}(93,18)
\put(0,0){$T=\ytableaushort{\bullet}\mapsto 
\ytableaushort{\bullet}=U$}
\put(25,-5){$\HH$}
\put(58,10){$\HH$}
\end{picture}
$ (in fact, $\HH=\GG$). 
By $T$'s (G.7),
the $\HH\in {\underline \x}$ is westmost in $T$. If the $\HH \in \overline{\x}(=e)$ is not westmost in $U$, then by Claim~\ref{claim:G7forwardA} there is some $\HH$ in $e$'s column in $T$ that takes part in a miniswap leading to an $\HH$ West of $e$ in $U$. Clearly, this $\HH$ is not the $\HH\in e$. Thus there are two $\HH$'s in $e$'s column, violating $T$'s (G.4). For {\sf H6}, {\sf H7} and {\sf T3}, we have $\HH=\GG^+$ and $\GG\in \x$ in $U$. Thus by $U$'s (G.4), (G.5) and (G.6), there is not $\HH=\GG^+$ West of $e=\underline{\x^\rightarrow}$ in $U$ . For the 
remaining cases, $\HH\in \{\mathcal F\}\cup Z$ 
(in the notation of the miniswaps). These labels in $\x^\rightarrow$ and 
$\underline{\x^\rightarrow}$  of $T$ are marked. Hence by 
Lemma~\ref{lemma:markedHiswestmost}, they are all westmost in their respective genes in $T$. Therefore
the same labels of $e$ in $U$ are westmost by Claim~\ref{claim:G7forwardA}.

\noindent
{\sf Case 2: ($\HH \in e$ in $T$):} No miniswap involving $\HH\in e$ both keeps an $\HH\in e$
and puts an $\HH$ West of $e$. Thus, 
if there 
is an $\HH$ West of $e$ in $U$, then by $T$'s (G.7) combined with Claim~\ref{claim:G7forwardA} there is
an $\HH$ in the column of $e$ in $T$ other than the $\HH\in e$. This contradicts $T$'s (G.4).

\smallskip
\noindent
{\bf (G.3):}
Consider a row $r$ of $T$.

\noindent
{\sf Case 1: (The labels of $r$ strictly $\prec$-increase from left to right, ignoring $\bullet_{\GG}$'s):} If there is no $\bullet_{\GG}$ in $r$,
then only {\sf H9}, {\sf B1} and {\sf B3} miniswaps could involve labels of $r$. 
Therefore either $r$ is unchanged by $\swap_{\GG}$ (or only some labels
became marked), or
a $\GG$ in $r$ is replaced by $\bullet_{\GG^+}$. Thus, $U$'s (G.3) holds for $r$ in this situation.

Otherwise, $T$ has $\bullet_{\GG}$ in $r$, say in box $\x$. By assumption,
the exceptional configuration of (G.3) does not occur here. We consider all miniswaps involving a $\bullet_\GG$. Note that $T$ and $U$
are identical in $r$ both West of $\x$ and East of $\x^\rightarrow$. Hence it suffices to study the affect of a miniswap locally
at $\{\x^\leftarrow,\x,\x^\rightarrow\}$.

If {\sf H1} or {\sf H2} applies at $\x$, 
then locally at $\x$, $T$ looks like $\ytableausetup{boxsize=1.6em}\Scale[0.8]{\begin{picture}(60,20)
\put(0,0){$\ytableaushort{\FF {\bullet_{\GG}} \HH}$}
\put(26,-5){$\GG$}
\end{picture}}
$
or
$\Scale[0.8]{\begin{picture}(60,20)
\put(0,0){$\ytableaushort{\FF {\bullet_{\GG}} \HH}$}
\put(25,-4){\Scale[.7]{$\circled{\GG}$}}
\end{picture}
}$, where $\x$ is the center box. (If $\FF$ or $\HH$ does not exist, the argument is simplified.) It remains to show $\FF \prec \GG \prec \HH$. By $T$'s (G.9), $\FF \prec \GG$. We cannot have $\GG = \HH$, for then we would apply {\sf H4} or {\sf H5.3}, instead of {\sf H1} or {\sf H2}. Suppose $\HH \prec \GG$. Then the $\HH \in \x^\rightarrow$ is marked, and by Lemma~\ref{lem:same999}, $\circled{\GG} \in \underline{\x^\rightarrow}$. Hence we would apply {\sf T5} or {\sf T6}, instead of {\sf H1} or {\sf H2}. 

Applying miniswaps {\sf H3}--{\sf H5}, {\sf H8}, {\sf B2}, {\sf B3}, {\sf T2} or {\sf T6} at $\x$ clearly preserves (G.3) for $r$.

Suppose {\sf H6}, {\sf H7} or {\sf T3} applies at $\x$ in $T$. Then locally at $\x$, $T$ looks like
$\Scale[0.8]{\begin{picture}(60,20)
\put(0,0){$\ytableaushort{\FF {\bullet_{\GG}} {\GG^+}}$}
\put(25,-4){$\GG$}
\end{picture}}
$,
$\Scale[0.8]{\begin{picture}(60,20)
\put(0,0){$\ytableaushort{\FF {\bullet_{\GG}} {\GG^+}}$}
\put(24,-5){\Scale[.7]{$\circled{\GG}$}}
\end{picture}
}$ 
or
$\Scale[0.8]{\begin{picture}(60,20)
\put(0,0){$\ytableaushort{\FF {\bullet_{\GG}} {\GG^+}}$}
\end{picture}}
$ respectively. (If there is no $\FF$, there is nothing to show.) 
By $T$'s (G.9), $\FF \prec \GG$. Hence (G.3) holds for $r$ in $U$.

Suppose {\sf T1} applies at $\x$. Then at $\x$, $T$ locally looks like $\Scale[0.8]{$\ytableaushort{\FF {\bullet_{\GG}} \HH, \none \GG}$}$ (where the $\GG\in \x^\downarrow$ is guaranteed by Lemma~\ref{lem:piecesobservations}(IV,VII);
again, if $\FF$ or $\HH$ does not exist, the argument is only easier). We must show $\FF \prec \GG \prec \HH$. By $T$'s (G.9), $\FF \prec \GG$. Since we are applying {\sf T1}, $\HH \neq \GG$. Now repeat the above argument for {\sf H1, H2} above \emph{verbatim}.

Suppose {\sf T4} or {\sf T5} applies at $\x$.
Locally at $\x$, $T$ looks like $\Scale[0.8]{\begin{picture}(60,20)
\put(0,0){$\ytableaushort{{\bullet_{\GG}} {\mathcal{E}^!} {\FF}}$}
\put(23,-4){$\GG$}
\end{picture}}
$ or $\Scale[0.8]{\begin{picture}(60,20)
\put(0,0){$\ytableaushort{{\bullet_{\GG}} {\mathcal{E}^!} {\FF}}$}
\put(23,-4){\Scale[0.7]{$\circled{\GG}$}}
\end{picture}}$, while
$U$ looks like
$\Scale[0.8]{\begin{picture}(60,20)
\put(0,0){$\ytableaushort{\GG {\bullet_{\GG^+}} {\FF}}$}
\end{picture}}$. By assumption $\lab_T(\x^\leftarrow)\prec \EE$. Thus,
if $\GG \prec \FF$ (or there is no $\FF$), $U$'s (G.3) is satisfied.
Otherwise $\FF \preceq \GG$. Then $\FF\in \x^{\rightarrow\rightarrow}$ is marked in $U$. By Lemma~\ref{lem:strong_form_of_G13}, $N_\GG = N_\EE$, so by $T$'s Lemma~\ref{lem:how_to_check_ballotness}, $\FF \neq \GG$. Therefore $\FF \prec \GG$. By $T$'s (G.6), we have $\FF < \GG$. This three box configuration $\{ \x, \x^\rightarrow, \x^{\rightarrow\rightarrow}\}$ of $U$ is the exceptional configuration of (G.3).

\noindent
{\sf Case 2: (The labels of $r$ do not strictly increase):} Thus $r$ contains the local configuration
$\Scale[0.8]{\begin{picture}(60,20)
\put(0,0){$\ytableaushort{\HH {\bullet_\GG} {\FF^!}}$}
\end{picture}
}$, where $\HH > \FF$. Call the middle box $\x$. By Lemma~\ref{lem:strong_form_of_G13}, there is $\GG'$ or $\circled{\GG'} \in \underline{\x^\rightarrow}$ with $\family(\GG')=\family(\GG)$ and $N_{\GG'} = N_\FF$.

If $\GG' \neq \GG$, then $\x^\rightarrow$ is not part of the snake containing $\x$. Further by Lemma~\ref{lem:same999}, there is no $\GG$ or $\circled{\GG}$ in $\x^\downarrow$ or $\underline{\x}$ in $T$. Hence $\x$ takes part in a {\sf H3} miniswap, $r$ is unchanged (except for the subscript on the $\bullet$) and (G.3) is preserved. 

If $\GG' = \GG$, we apply {\sf T4}, {\sf T5} or {\sf T6}. Recall {\sf T6} produces no tableau. In the case of {\sf T4.1} and the first output of {\sf T4.2}, we make no local changes in row $r$, so $U$'s (G.3) follows from $T$'s. The remaining considerations are the second output of {\sf T4.2} and the outputs of {\sf T4.3} and {\sf T5}. By $T$'s (G.9), $\HH \prec \GG$. Let $\EE := \lab_T(\x^{\rightarrow \rightarrow})$ (if $\EE$ does not exist, the argument is trivialized). Since $N_\EE = N_{\GG'} = N_\FF$, by $T$'s Lemma~\ref{lem:how_to_check_ballotness}, $\EE \neq \GG$. If $\GG \prec \EE$, then $U$'s (G.3) holds. Otherwise $\EE \prec \GG$, and the $\EE \in \x^{\rightarrow\rightarrow}$ in $T$ is marked. Given $T$'s $\GG$ or $\circled{\GG} \in \underline{\x^\rightarrow}$, it follows by $T$'s (G.6) that $\EE < \GG$. Therefore $U$ has the exceptional (G.3) configuration in $r$.

\smallskip
\noindent
{\bf (G.8):}  Consider any two genes $\EE$ and $\FF$ with $\family(\FF) = \family(\EE) + 1$ and $N_\EE = N_\FF$. It suffices to show that in $U$, every $\EE$ is read before every $\FF$, that is:

\begin{claim}
\label{claimG8}
Given (nonvirtual) instances $e,f$ of $\EE$ and $\FF$ respectively in $U$, either $e$ is East of $f$ or else $e$ is north of $f$ in the same column.
\end{claim}
\begin{proof}
Suppose the claim fails for some fixed choice of $e$ and $f$.  Thus in $U$ either $f$ is North of $e$ in its column
or else $f$ is East of $e$. The first scenario contradicts $U$'s (G.4), so assume the second occurs.

If $U$ has a label $\QQ$ in column $c$, then $T$ has $\QQ$ or $\circled{\QQ}$ in $c$ or $c^\rightarrow$.
Thus since $U$ has $f$ East of $e$, by $T$'s Lemma~\ref{lem:how_to_check_ballotness},
$T$ has $e$ and $f$ in the same column. By $T$'s (G.4), $e$ is north of $f$ in $T$.

\noindent
{\sf Case 1: ($\family(\EE) < \family(\GG)$):}
We may assume a miniswap moves $e$ West. Since $\family(\EE) < \family(\GG)$, the only such miniswaps are {\sf T4} and {\sf T5}. Hence $T$ has a box $\x$ with $\GG \in \underline{\x}$ and either $e \in \underline{\x}$ or $e \in \x$. By $T$'s (G.4), $f \in \underline{\x}$. These miniswaps may move $e$ to $\overline{\x^\leftarrow}$, but then by definition, they will also move $f$ to $\x^\leftarrow$ or $\overline{\x^\leftarrow}$.

\noindent
{\sf Case 2: ($\family(\EE) = \family(\GG)$):}
We are done unless $\swap_\GG$ moves $e$ West. The possible miniswaps are {\sf H5}, {\sf B2}, {\sf B3}, {\sf T2}, and {\sf T4}. 
By inspection of the 
miniswaps, $\EE = \GG$.

\noindent
{\sf Case 2.1: ($f \in \x$):}
By $T$'s (G.4), either $e \in \overline{\x}$ or else $e \in \x^\uparrow$ with $\overline{\x}$ empty. In all the miniswaps of interest, $\bullet_\EE \in \x^{\uparrow\leftarrow}$.
So by $T$'s (G.2), $\x^\leftarrow$ contains a genetic label $h$ of some gene $\HH$. The local picture is either
$\Scale[0.8]{\begin{picture}(38,30)
\put(0,14){$\ytableausetup{boxsize=1.5em}\ytableaushort{{\bullet_\EE} e, h f}$}
\end{picture}}$
or 
$\Scale[0.8]{\begin{picture}(42,30)
\put(0,14){$\ytableausetup{boxsize=1.5em}\ytableaushort{{\bullet_\EE} \star, h f}$}
\put(25,12){$e$}
\end{picture}}$ (where $\star$ is some genetic label).
By $T$'s (G.11), $h$ is not marked; hence $\HH \succeq \EE$. 
Thus  by $T$'s (G.6), either $\HH > \EE$ or $\HH = \EE$. If $\HH = \EE$, $h$ and $f$ violate Lemma~\ref{lem:how_to_check_ballotness} for $T$. Thus $\EE < \HH$. By $T$'s (G.3), $\HH \prec \FF$. Since $\family(\FF) = \family(\EE) + 1$, we have 
$\family(\HH) = \family(\FF)$. Indeed by $T$'s (G.6), $\FF = \HH^+$. Now since $N_\EE = N_\FF$, 
we violate $T$'s (G.8) unless there is an $\EE'$ above $h$ in its column with $\family(\EE') = \family(\EE)$. 
If this $\EE' \in \overline{\x^\leftarrow}$, then by (G.11) and (G.6), $\EE' = \EE$. 
But then $f$ and this $\EE \in \overline{\x^\leftarrow}$ violate Lemma~\ref{lem:how_to_check_ballotness} for $T$. 
Otherwise, this $\EE'$ is North of $\x^{\uparrow\leftarrow}$. But then this violates $T$'s (G.12) together with the $e \in \x^\uparrow$ or $\overline{\x}$.

\noindent
{\sf Case 2.2: ($f \in \underline{\x}$):}
By $T$'s (G.4), either $e \in \underline{\x}$ or $e \in \x$. In the relevant miniswaps, $\bullet_\GG \in \x^\leftarrow$. If {\sf B2}, {\sf B3} or {\sf T2} applies, then $T$ has $\EE \in \x^{\leftarrow\downarrow}$; together with $f$, this violates
Lemma~\ref{lem:how_to_check_ballotness} for $T$. Hence {\sf H5} or {\sf T4} applies. Since $f \in \underline{\x}$, any {\sf H5} miniswap used is {\sf H5.1}, while any {\sf T4} miniswap used is {\sf T4.1}; both of these fix $e$, a contradiction.

\noindent
{\sf Case 3: ($\family(\EE) > \family(\GG)$):}
Neither $e$ nor $f$ is affected by $\swap_\GG$, so the Claim holds. 
\end{proof} 

\smallskip
\noindent
{\bf (G.9):} Suppose $\FF$ is northwest of $\bullet_{\GG^+} \in \y$ in $U$ and $\GG^+ \preceq \FF$; we seek a contradiction. We may suppose such $\FF$ is in $\x$ or $\underline{\x}$ in $U$.

\noindent
{\sf Case 1: ($\FF$ appears in the same position in $T$):} By inspection of the miniswaps, $T$ has $\bullet_\GG \in \y, \y^\leftarrow,$ or $\y^\uparrow$. By $T$'s (G.9), the $\FF \in \x$ or $\underline{\x}$ is not northwest of this box. Hence one of the following subcases occurs:

\noindent
{\sf Subcase 1.1: ($T$ has $\bullet_\GG \in \y^\leftarrow$ and $\y$ is South of $\x$ in its column):} Since $\GG \prec \FF$, $T$ 
contains no $\FF^!$, so by $T$'s (G.4), $\FF < \lab_T(\y)$ implying $\GG^+ < \lab_T(\y)$. Hence $\y$ is not part of any snake in $T$, contradicting $\bullet_{\GG^+} \in \y$ in $U$. 

\noindent
{\sf Subcase 1.2: ($T$ has $\bullet_\GG \in \y^\uparrow$ and $\y$ is East of $\x$ in its row):} 
By $T$'s (G.3) and (G.9), $\FF \prec \lab_T(\y)$. Then $\GG^+ \prec \lab_T(\y)$, so $\y$ is not part of any snake in $T$, contradicting $\bullet_{\GG^+} \in \y$ in $U$. 

\noindent
{\sf Case 2: ($\FF$ does not appear in the same position in $T$):}
By inspection of the miniswaps, no label $\HH \succ \GG^+$ is affected by $\swap_{\GG}$. Furthermore labels $\GG^+$ can only be affected if $\family(\GG^+) = \family(\GG)$. Since $\GG^+ \preceq \FF$, this implies $\FF = \GG^+$ with $\family(\FF) = \family(\GG)$. 
By inspection of the miniswaps that affect $\GG^+$, $T$ has $\FF = \GG^+ \in \x$, while $U$ has $\bullet_{\GG^+} \in \x$ and $\FF = \GG^+ \in \underline{\x}$.
Since $\FF \in \underline{\x}$ is Southeast of the $\bullet_{\GG^+}\in \x$ in $U$, by $U$'s (G.2) it cannot also be northwest of a $\bullet_{\GG^+}$, contradicting our assumption.

\smallskip
\noindent
{\bf (G.11):}
Consider a label $\FF^! \in \x$ or $\underline{\x}$ in $U$. By $T$'s (G.2) and 
inspection of the miniswaps, $T$ has either $\FF$ or $\FF^!$ in the same position.

\noindent
{\sf Case 1: (This $\FF$ is marked in $T$):}
By definition, there is a $\bullet_{\GG^+}$ northwest of $\FF^!$ in $U$. If $U$ has a $\bullet_{\GG^+}$ South of $\x$ and in its column, this contradicts $U$'s (G.2). 

We now show $U$ has no $\bullet_{\GG^+}$ North of $\x$ and in its column:
By inspection of the miniswaps, if $\bullet_{\GG^+} \in \y$ in $U$, then $T$ has a $\bullet_\GG$ northwest of $\y$. 
By $T$'s (G.11), $T$ has no $\bullet_\GG$ in $\x$'s column. Since $\FF \prec \GG$, by Lemma~\ref{lemma:Gsoutheastofbullet} $T$ has no $\bullet_\GG$ NorthWest of $\x$. 

\noindent
{\sf Case 2: (This $\FF$ is unmarked in $T$):} 
By definition, $U$ has a $\bullet_{\GG^+}$ northwest of $\FF^!$. By inspection of the miniswaps, if $\bullet_{\GG^+} \in \y$ in $U$, then $T$ has a $\bullet_\GG$ northwest of $\y$. Therefore $T$ has a $\bullet_\GG$ northwest of said $\FF$. Since the $\FF$ is unmarked, it must be that $\GG \preceq \FF$. But since $\FF^!$ appears in $U$, $\FF \prec \GG^+$. Hence $\FF = \GG$. 
By Lemma~\ref{lemma:Gsoutheastofbullet}, $T$ has no $\bullet_\GG$ NorthWest of $\x$. So $T$ has $\bullet_\GG$ either West of $\x$ and in its row or else North of $\x$ and in its column. Since $\FF=\GG$, only the latter case is a concern. In that case, by $T$'s (G.4) and (G.11), in fact $\FF =\GG \in \x$ and $\bullet_\GG \in \x^\uparrow$. Hence by inspection of the miniswaps, $U$ has $\FF = \GG \in \x^\uparrow$. (Note {\sf T4} does not apply at $\x^\uparrow$ by $T$'s (G.7).) Thus $U$ has no $\bullet_{\GG^+}$ in $\x$'s column.

\smallskip
\noindent
{\bf (G.10):}
Consider a label $\FF^! \in \x$ or $\underline{\x}$ in $U$. By $T$'s (G.2) and inspection of the miniswaps, $T$ has either $\FF$ or $\FF^!$ in the same position.

\noindent
{\sf Case 1: (This $\FF$ is marked in $T$):}
By inspection of the miniswaps, if $\bullet_{\GG^+} \in \y$ in $U$, then  $T$ has a $\bullet_\GG$ northwest of $\y$. 
By $T$'s (G.11), $T$ has no $\bullet_\GG$ in $\x$'s column. Since $\FF \prec \GG$, by Lemma~\ref{lemma:Gsoutheastofbullet}, $T$ has no $\bullet_\GG$ NorthWest of $\x$. Hence $U$ has no $\bullet_{\GG^+}$ Northwest of $\x$. But by definition, $U$ has a $\bullet_{\GG^+}$ northwest of $\FF^!$, so it must be in $\x$'s row.

\noindent
{\sf Case 2: (This $\FF$ is unmarked in $T$):}
Since $\FF^!$ appears in $U$, $\FF \preceq \GG$. By definition, $U$ has $\bullet_{\GG^+}$ northwest of said $\FF^!$. By inspection of the miniswaps, if $\bullet_{\GG^+} \in \y$ in $U$, then $\bullet_\GG$ appeared northwest of $\y$ in $T$. Hence $T$ has $\bullet_\GG$ northwest of this $\FF$. Since this $\FF$ is unmarked in $T$, $\FF \succeq \GG$. Thus $\FF = \GG$.

By Lemma~\ref{lemma:Gsoutheastofbullet}, $T$ has no $\bullet_\GG$ NorthWest of $\x$. Therefore $U$ has no $\bullet_{\GG^+}$ NorthWest of $\x$. But $U$ has some $\bullet_{\GG^+}$ northwest of $\x$, so it is either West of $\x$ in $\x$'s row or North of $\x$ in $\x$'s column. In the former case, we are done; in the latter case, we contradict $U$'s (G.11).

\smallskip
\noindent
{\bf (G.12):} Define the {\bf neighborhood} of a box $\uuu$ to be ${\tt Neigh}(\uuu):=\{\uuu, \uuu^\leftarrow, \uuu^\uparrow, {\underline \uuu},
{\overline {\uuu^\leftarrow}}\}$. For a lower edge $\underline{\uuu}$, let ${\tt Neigh}(\underline{\uuu}):=\{\underline{\uuu},\uuu, \uuu^\leftarrow, {\overline \uuu}, \overline{\uuu^\leftarrow}\}$. Given a (possibly virtual) instance $q \in \uuu$ or $\underline{\uuu}$ in $T$ of the gene $\QQ$, let
the {\bf children} of $q$ be all (nonvirtual) $\QQ$'s in $U$ in ${\tt Neigh}(\uuu)$ or ${\tt Neigh}(\underline{\uuu})$, respectively. 
Finally define the children of a $\bullet_{\GG}\in \uuu$ in $T$ to be those $\bullet_{\GG^+}$ in $\uuu, \uuu^\rightarrow,\uuu^\downarrow$ in $U$. Clearly,

\begin{lemma}
\label{lemma:child123}
Every $q$ in $U\in \swap_{\GG}(T)$ is a child of at least one (possibly virtual) $q$ in $T$.
Also, every $\bullet_{\GG^+}$ in $U\in \swap_{\GG}(T)$ is a child of at least one $\bullet_{\GG}$ in $T$. \qed
\end{lemma}

Suppose $\HH$ and $\HH'$ are instances in $T$ of genes of the same family. By Lemma~\ref{lemma:child123},
it suffices to confirm $U$'s (G.12) for $\ell$ a child of $\HH$ and $\ell'$ a child of $\HH'$. To do this, we break into cases depending
on the relative position of $\HH$ and $\HH'$. By relabeling, we may assume $\HH$ west of $\HH'$.  Specifically,
{\sf Cases 1--3} below concern the situation $\HH$ NorthWest of $\HH'$.
{\sf Cases 4--7} consider the case $\HH$ southwest of $\HH'$.

For the first three cases, let $\x, \y$ be boxes in the same row $r$ of $T$ with $\x$ West of $\y$.
By $T$'s (G.12), we may assume $\HH \in \overline{\x}$ or $\HH \in \x$, as well as $\HH' \in \y$ or $\HH' \in \underline{\y}$. 
By $T$'s (G.12), there is a $\bullet_\GG$ in some box $\z$ of row $r$ appearing East of $\x$ and west of $\y$.

\noindent
{\sf Case 1: (In $T$, we have $\HH$ or $\circled{\HH} \in \overline{\x}$ and $\HH'$ or $\circled{\HH'} \in \underline{\y}$):}
By $T$'s (G.4), $\HH < \lab_T(\x)$. By $T$'s (G.9), $\lab_T(\x) \prec \GG$. Hence $\HH < \GG$. Since $\family(\HH) = \family(\HH')$, also $\HH' < \GG$. Therefore the $\HH' \in \underline{\y}$ is marked (and is not virtual). By $T$'s (G.11), this forces $\z \neq \y$. For convenience, assume $\HH \in \overline{\x}$. (The argument where this label is virtual is strictly easier.)

By Lemma~\ref{lem:strong_form_of_G10}, the box labels in $\y$ and in every box strictly between $\y$ and $\z$ are also marked. Let $\lab_T(\z^\rightarrow) := \EE^!$ and note $\EE < \HH' < \GG$ (the first inequality by a combination of
$T$'s (G.3) and (G.4)). 
Summarizing, $T$ locally looks like one of the following at $r$:
\[\begin{picture}(90,20)
\ytableausetup{boxsize=1.2em}
\put(0,9){$\ytableaushort{\star \cdots {\bullet_{\GG}} {\EE^!} \cdots \star}$}
\put(4,19){$\HH$}
\put(75,4){$\Scale[0.85]{{\HH'}^!}$}
\put(5,-3){$\x$}
\put(35,-3){$\z$}
\put(77,-3){$\y$}
\end{picture} \text{\ \ \ \ or \ \ \ \  }
\begin{picture}(63,20)
\ytableausetup{boxsize=1.2em}
\put(0,9){$\ytableaushort{\star \cdots {\bullet_{\GG}} {\EE^!}}$}
\put(4,19){$\HH$}
\put(46,4){$\Scale[0.85]{{\HH'}^!}$}
\put(5,-3){$\x$}
\put(35,-3){$\z$}
\put(52,-3){$\y$}
\end{picture}.
\]
By $T$'s (G.2), $\bullet_{\GG}\not\in \x^\uparrow, \x^{\uparrow\leftarrow}$. Therefore
${\overline \x}$ is not part of a snake and so the only child of $\HH \in \overline{\x}$ in $U$ is in $\overline{\x}$. 

Suppose $\z$ is the only box in its snake section, i.e., we apply {\sf H1}, {\sf H2}, {\sf H3} or {\sf T1}. 
If the miniswap is {\sf H1}, {\sf H2} or {\sf T1}, then $\GG$ or $\circled{\GG}$ appears in $\z^\downarrow$ or $\underline{\z}$. Hence by Lemma~\ref{lem:same999}, $\circled{\GG} \in \underline{\z^\rightarrow}$, so $\z^{\rightarrow}$ is adjoined by (R.3), contradicting $\z$ the only box in its snake section. Thus the miniswap is {\sf H3} and the unique child of $\bullet_{\GG}$ is in $\z$.
Moreover $\y$ is not part of a snake, or takes part in a trivial {\sf H9} miniswap. Hence the unique child of $\HH' \in \underline{\y}$ is at $\underline{\y}$ in $U$.
Thus $U$'s (G.12) holds in this scenario.

Otherwise the miniswap at $\z$ involves $\z$ and $\z^\rightarrow$. Then the miniswap is {\sf T4}, {\sf T5} or {\sf T6}. In these cases, the child of $\bullet_{\GG}$ is in either $\z$ or $\z^\rightarrow$ in $U$.  If the child of
$\HH' \in \underline{\y}$ is in $\underline{\y}$, we are done.
If not, $\y = \z^\rightarrow$ and $\HH' \in Z$ (in the notation of the miniswaps). Then the child of
$\HH'$ is at $\overline{\z}$ in $U$, so the child of $\HH$ is not North of the child of $\HH'$ and (G.12) holds
vacuously.

\noindent
{\sf Case 2: (In $T$, we have $\HH \in \x$ and $\HH'$ or $\circled{\HH'} \in \underline{\y}$):}
By $T$'s (G.9), $\HH \prec \GG$.

\noindent
{\sf Subcase 2.1: ($\z = \y$):} By $T$'s (G.11), the $\HH'$ or $\circled{\HH'} \in \underline{\y}$ is not marked;
hence $\GG \preceq \HH'$. Thus $\HH\prec\GG\preceq \HH'$, so $\family(\GG) = \family(\HH)$. 
Combined with $T$'s (G.2), this implies the unique child of $\HH \in \x$ is in $\x$.
By $T$'s (G.3) and (G.9), $T$ has $\GG' \in \z^\leftarrow$ with $\family(\GG') = \family(\GG) = \family(\HH)$ and $\GG' \prec \GG$. Hence by $T$'s (G.6), $\GG' = \GG^-$ and $\HH' = \GG$; moreover by $T$'s (V.2), this
$\HH' = \GG$ is not virtual, since it is westmost. Hence, locally at $r$, $T$ is
\[\begin{picture}(90,20)
\ytableausetup{boxsize=1.2em}
\put(0,9){$\ytableaushort{{\HH} {\cdots} {\GG'} {\bullet_{\GG}}}$}
\put(45,2){$\HH'$}
\put(5,-6){$\x$}
\put(44,-6){$\y=\z$}
\end{picture} \text{ \ \ (where $\HH'=\GG, \GG'=\GG^-$).} \]
Thus the miniswap involving $\z$ is one of {\sf H1}, {\sf H4}, {\sf H6} and {\sf T6}. Now {\sf H4} and {\sf T6}
produce no output. If {\sf H1} or {\sf H6} applies, the child of $\HH'\in \underline{\y}=\underline{\z}$ is 
northEast of the child of $\HH\in \x$, so (G.12) is confirmed vacuously.

\noindent
{\sf Subcase 2.2: ($\z \neq \y$):}
By $T$'s (G.4), $\lab_T(\y) < \HH'$. Hence by $\family(\HH) \leq \family(\GG)$, $\lab_T(\y)$ is marked. By Lemma~\ref{lem:strong_form_of_G10}, some $\EE^! \in \z^\rightarrow$. 
The remainder of this case is argued exactly as {\sf Case 1}.

\noindent
{\sf Case 3: (In $T$ we have $\HH$ or $\circled{\HH} \in \overline{\x}$ and $\HH' \in \y$):}
Since $\HH' \in \y$, $\bullet_\GG \notin \y$, so $\z \neq \y$. By $T$'s (G.4) and (G.9), $\HH < \lab_T(\x) \prec \GG$. Therefore also $\HH' < \GG$ and so $\HH' \in \y$ is marked. 
Since $\overline{\x}$ does not participate in the swap, if $\circled{\HH} \in \overline{\x}$, 
this $\circled{\HH}$ has no children, so the (G.12) confirmation is vacuous here. So assume 
$\HH \in \overline{\x}$ in $T$; since $\overline{\x}$ does not participate in the swap, its only child
is in the same position in $U$.
In summary, locally at $r$, $T$ is
\[\begin{picture}(90,16)
\ytableausetup{boxsize=1.3em}
\put(0,4){$\ytableaushort{\star \cdots {\bullet_\GG} \cdots {{\HH'}^!}}.$}
\put(3,14){$\HH$}
\put(5,-8){$\x$}
\put(35,-8){$\z$}
\put(74,-8){$\y$}
\end{picture}\]

\noindent
{\sf Subcase 3.1: ($\z = \y^\leftarrow$):}
Consider the miniswap involving $\z$. First suppose $\z$ is not the only box in its snake section. By 
Definition-Lemma~\ref{def-lem:snake_classification}(I,II), it involves $\y$ and is {\sf T4}, {\sf T5} or {\sf T6}. The last miniswap produces no output.
For the first two miniswaps, one possibility is that the child of ${\HH'}^!\in \y$ is at $\overline{\z}$ in $U$.
Here the (G.12) confirmation is vacuous. Otherwise $U$'s unique child of ${\HH'}^!\in \y$ is in $\y$; here the unique child of $\bullet_\GG \in \z$ is in $\z$ and so $U$'s (G.12) holds.

Otherwise, $\z$ is the only box in its snake section. Thus {\sf H1}, {\sf H2}, {\sf H3} or {\sf T1} applies. 
If {\sf H1}, {\sf H2} or {\sf T1} applies, then by definition or Lemma~\ref{lem:piecesobservations}(IV,VII),
$T$ has $\GG$ or $\circled{\GG}$ in $\z^\downarrow$ or $\underline{\z}$. Hence by Lemma~\ref{lem:same999}, $\circled{\GG} \in \underline{\y}$, so $\y$ is adjoined by (R.3), contradicting $\z$ the only box in its snake section.
Thus it is {\sf H3}, and $U$'s unique children of $\HH' \in \y$ and $\bullet_\GG \in \z$ are in $\y, \z$ respectively; hence
(G.12) is confirmed here.  

\noindent
{\sf Subcase 3.2: ($\z \neq \y^\leftarrow$):}
In this case, $\HH' \in \y$ is not part of a snake in $T$ or takes part in a trivial {\sf H9} miniswap; thus its unique child is in $\y$ in $U$.
Hence it suffices to check that $U$ has a $\bullet_{\GG^+}$ between $\x$ and $\y$.
By Lemma~\ref{lem:strong_form_of_G10}(II), since $\lab_T(\y)$ is marked, $\lab_T(\z^\rightarrow)$ is also marked.
Consider the miniswap involving $\z$. If $T$ has no $\GG$ or $\circled{\GG}$ in $\z^\downarrow$ or $\underline{\z}$, then the miniswap must be one of {\sf H3}, {\sf T4} or {\sf T5} (it cannot
be {\sf T1} by Lemma~\ref{lem:piecesobservations}(IV,VII)). For each of these, a child of
$\bullet_\GG\in \z$ appears in $\z$ or $\z^\rightarrow$ in $U$.
If $T$ has $\GG$ or $\circled{\GG}$ in $\z^\downarrow$ or $\underline{\z}$, then by Lemma~\ref{lem:same999} $\circled{\GG} \in \underline{\z^\rightarrow}$. Hence the miniswap is {\sf T5} or {\sf T6}. In the former case, a child of
$\bullet_{\GG}\in \z$ appears in $\z^\rightarrow$ in $U$. In the latter case, $U$ does not exist.

\noindent
{\sf Case 4: (In $T$, $\HH \in \aaa$ is southwest of $\HH' \in \bbb$):} We will use:

\begin{claim}
In $U$, each child of $\HH\in\aaa$ is west of each child of $\HH'\in\bbb$.
\end{claim}
\begin{proof}
If $\aaa$ is West of $\bbb$, then the claim holds by the definition of children.
So assume $\aaa$ and $\bbb$ are in the same column. By $T$'s (G.4), $\HH = \HH'$ and $\HH \in \aaa$ is marked. Hence $\HH \prec \GG$. By Lemma~\ref{lem:strong_form_of_G10}(II), $T$ has a $\bullet_\GG$ in some box 
$\z$ West of $\aaa$ and in its row. By Lemma~\ref{lem:strong_form_of_G10}(II), every box label strictly between 
$\z$ and $\aaa$ is also marked. Thus by $T$'s (G.11), $T$ has no $\bullet_\GG$ in any column
East of $\z$ and west of $\aaa$. Furthermore by $T$'s (G.2), $T$ has no $\bullet_\GG$ Northwest of 
$\z$. Summarizing, $T$ is locally:
$\begin{picture}(73,24)
\ytableausetup{boxsize=1.3em}
\put(10,10){$\ytableaushort{\none \none {\HH}, {\bullet_\GG} {\cdots} {\HH^!}}$}
\put(0,0){$\z$}
\put(63,0){$\aaa$}
\put(63,14){$\bbb$}
\end{picture}$.
Hence 
$\bbb$ is not part of any snake, and so $U$'s unique child of $\HH'\in \bbb$ 
is in $\bbb$.
\end{proof}

No child of $\HH\in \aaa$ is North of $\aaa^\uparrow$ and 
no child of $\HH'\in \bbb$ is South of $\underline{\bbb}$. 
Hence if $\aaa$ is at least two rows below $\bbb$, the (G.12) confirmation is vacuous.

\noindent
{\sf Subcase 4.1: ($\aaa$ is exactly one row south of $\bbb$):}
By inspection of the miniswaps, a child of $\HH\in \aaa$
can only appear North of a child of $\HH'\in\bbb$ if $\bullet_{\GG}\in \aaa^\uparrow, \bbb^\leftarrow$. Then by $T$'s (G.2), $\aaa^\uparrow=\bbb^\leftarrow$. Here, 
the $\HH' \in \bbb$ has a child South of $\bbb$ only if $T$ is locally
$\ytableaushort{{\bullet_{\GG}} {\HH'}, {\HH}}=\ytableaushort{{\bullet_{\GG}} {\GG^+}, {\GG}}$
and $\HH'=\GG^+\in \bbb$ is part of a {\sf T3} miniswap, i.e.,
$\begin{picture}(90,20)\put(0,12){$\ytableaushort{{\bullet_{\HH}} {\HH'}, {\HH}}\mapsto 
\ytableaushort{{\HH} {\bullet_{\HH'}}, {\star}}$}
\put(72,7){$\Scale[0.8]{\HH'}$}
\end{picture}$ (here $\star \in \{\HH^!, \bullet_{\HH'} \}$; the uncertainty is irrelevant). By $T$'s (G.3), it follows $\lab_U(\bbb^\rightarrow)$
is not marked and $\lab_U(\bbb^\rightarrow) \neq \HH'$. Thus we confirm (G.12).

\noindent
{\sf Subcase 4.2: ($\aaa$ is in the same row $r$ as $\bbb$):}
Suppose $\HH' \in \bbb$ has a child South of $\bbb$. Then $\bbb$ is part of a {\sf H6}, {\sf H7} or {\sf T3} miniswap. So $U$'s unique
child of $\HH' \in \bbb$ is at $\underline{\bbb}$ and $U$ has $\bullet_\GG \in \bbb$. 
By $T$'s (G.2), $T$ has no $\bullet_\GG \in \aaa^\uparrow$, so no child of $\HH\in \aaa$ is in $\aaa^\uparrow$. Thus $U$'s (G.12) is confirmed.

Otherwise no child of $\HH' \in \bbb$ is South of $\bbb$. 
If a child of $\HH\in \aaa$ is in $\overline{\aaa^\leftarrow}$, we used
{\sf T4} or {\sf T5} at $\aaa$; thus $U$ has $\bullet_{\GG^+} \in \aaa$ and (G.12) is confirmed. Thus it remains to consider the scenario that
a child of $\HH\in \aaa$ is at $\aaa^\uparrow$ in $U$. This scenario is impossible: By inspection of the miniswaps,
$\GG = \HH$ and $T$ has $\bullet_{\GG} \in \aaa^\uparrow$. Let $\lab_T(\bbb^\uparrow) := \EE$ (by $T$'s (G.2), $\bullet_{\GG}\not\in
\bbb^\uparrow$). Locally $T$ is
$\begin{picture}(52,40)
\put(0,22){$\ytableaushort{{\bullet_\HH} \cdots \EE, \HH \cdots {\HH'}}$}
\put(5,-4){$\aaa$}
\put(38,-4){$\bbb$}
\end{picture}$ (where $\HH=\GG$).
By $T$'s (G.4), either $\EE < \HH'$ or $\EE = \HH'$. 
If $\EE = \HH'$, then by $T$'s (G.4), the $\HH' \in \bbb$ is marked; hence $\HH' \prec \HH$, contradicting $T$'s (G.6).
Thus $\EE < \HH'$. Then $\EE < \HH$, so $\EE \in \bbb^\uparrow$ is marked. By Lemma~\ref{lem:strong_form_of_G13}, $T$ has a label 
$\HH''$ or $\circled{\HH''} \in \underline{\bbb^\uparrow}$ with $\family(\HH'')=\family(\HH)$. 
With $T$'s $\HH' \in \bbb$, this violates $T$'s (G.4).

\noindent
{\sf Case 5: (In $T$, $\HH \in \underline{\aaa}$ is southwest of $\HH' \in \bbb$):} By $T$'s (G.4), $\aaa$ is West of $\bbb$.

\noindent
{\sf Subcase 5.1: ($\aaa$ South of $\bbb$):} By the definition of children, every child of 
 $\HH \in \underline{\aaa}$ is southwest of every child of $\HH' \in \bbb$, so there is nothing to confirm here.

\noindent
{\sf Subcase 5.2: ($\aaa$ and $\bbb$ are in the same row):} By the definition of children, every child of
$\HH\in \underline{\aaa}$ is west of every child of $\HH'\in\bbb$. Every child of $\HH' \in \bbb$ is north of $\underline{\bbb}$.
Thus we are only concerned with the cases that a child of $\HH \in \underline{\aaa}$ is north of $\aaa$, so we assume this. Moreover, by inspection of the miniswaps,
we may assume $T$ has $\bullet_\GG \in\aaa$ or $\bullet_\GG \in \aaa^\leftarrow$.

If $\bullet_\GG \in \aaa^\leftarrow$, then $\aaa$ is part of a {\sf T4} or {\sf T5} miniswap. 
Hence the unique child of $\HH\in \underline{\aaa}$ is in $\aaa^\leftarrow$ or $\overline{\aaa^\leftarrow}$; the unique child of $\HH' \in \bbb$ is in $\bbb$;
and $U$ has $\bullet_{\GG^+} \in \aaa$. So (G.12) is confirmed.

Otherwise, $T$ has $\bullet_\GG \in \aaa$ and moreover $\HH = \GG$. Let $\EE := \lab_T(\aaa^\rightarrow)$. Locally between $\aaa$ and $\bbb$, $T$ is:
$\begin{picture}(88,20)
\put(10,2){$\ytableaushort{{\bullet_\HH} \EE \cdots {\HH'}}$}
\put(13,-4){$\HH$}
\put(0,4){$\aaa$}
\put(80,4){$\bbb$}
\end{picture}$. 
If $\EE^! \in \aaa^\rightarrow$, then by Lemma~\ref{lem:same999}, $\circled{\HH} \in \underline{\aaa^\rightarrow}$, and so $\aaa$ swaps by {\sf T6} and there is no tableau $U$.
Thus $\EE \in \aaa^\rightarrow$ is not marked, and by $T$'s (G.3), $\family(\EE) = \family(\HH)$. By $T$'s (G.6), either $\EE = \HH$ or $\EE = \HH^+$. In the former case,
{\sf H4} applies at $\aaa$, producing no $U$. In the latter case, {\sf H6} applies and one confirms (G.12) by inspection.
(If $\bbb = \aaa^\rightarrow$, one checks, as was done in {\sf Case 4}, that 
the configuration of the final sentence of (G.12) does not occur.) 

\noindent
{\sf Case 6: (In $T$, $\HH \in \aaa$ is southwest of $\HH' \in \underline{\bbb}$):}
By $T$'s (G.4), $\aaa$ is SouthWest of $\bbb$. 
We may assume $\aaa$ is one row South of $\bbb$ and $U$ has a child of $\HH\in \aaa$ in $\aaa^\uparrow$ (otherwise the (G.12) check is vacuous). 
The unique child of $\HH'\in \underline{\bbb}$ is in $\underline{\bbb}$.
Hence, $\HH = \GG$ and $T$ has $\bullet_{\GG} \in \aaa^\uparrow$. Consider the miniswap involving $\aaa^\uparrow$. It is {\sf B2}, {\sf B3}, {\sf T1}, {\sf T2}, {\sf T3}, {\sf T4} or {\sf T5}. If it is {\sf B2}, {\sf B3}, {\sf T2} or {\sf T3}, then locally $T$ is
\[\begin{picture}(90,15)
\put(10,3){$\ytableaushort{{\bullet_\HH} \HH \cdots \star, \HH}$}
\put(61,-3){$\HH'$}
\put(0,-8){$\aaa$}
\put(80,6){$\bbb$}
\end{picture} \text{ \ \ or \ \ \  }
\begin{picture}(90,15)
\put(10,3){$\ytableaushort{{\bullet_\HH} {\HH^+} \cdots \star, \HH}$}
\put(61,-3){$\HH'$}
\put(0,-8){$\aaa$}
\put(80,6){$\bbb$}
\end{picture}
.\] 
Since by $T$'s (G.2), the $\bullet_\HH \in \aaa^\uparrow$ is the only $\bullet_\HH$ in its row, this contradicts $T$'s (G.12).
Thus it is {\sf T1}, {\sf T4} or {\sf T5}. Let $\EE := \lab_T(\aaa^{\uparrow\rightarrow})$. By $T$'s (G.3) and (G.4), $\EE < \HH'$. Hence $\EE < \HH$, and so $\EE^! \in 
\aaa^{\uparrow\rightarrow}$. By Lemma~\ref{lem:same999}, $T$ has $\circled{\HH} \in \underline{\aaa^{\uparrow\rightarrow}}$. Therefore the miniswap is 
{\sf T5} and one confirms (G.12) directly.

\noindent
{\sf Case 7: ($T$ has $\HH \in \underline{\aaa}$ southwest of $\HH' \in \underline{\bbb}$):}
By $T$'s (G.4) and (G.5), $\aaa$ is West of $\bbb$. Hence by the definition of children, $U$ has every child of $\HH \in \underline{\aaa}$ west of every child of $\HH' \in \underline{\bbb}$. 
We may assume that $\aaa$ and $\bbb$ are in the same row and that some child of $\HH \in \underline{\aaa}$ is north of $\aaa$, for otherwise the (G.12) confirmation is vacuous. Then $T$ has $\bullet_\GG \in \aaa$ or $\bullet_\GG \in \aaa^\leftarrow$.

\noindent
{\sf Subcase 7.1: ($\bullet_\GG \in \aaa$):}
Here $\GG = \HH$. By $T$'s (G.4), $\lab_T(\bbb) < \HH'$, whence $\lab_T(\bbb) < \HH$. Therefore, $\lab_T(\bbb)$ is marked. By Lemma~\ref{lem:strong_form_of_G10}(II), $\lab_T(\aaa^\rightarrow)$ is also marked, and so by Lemma~\ref{lem:same999}, $\circled{\HH} \in \underline{\aaa^\rightarrow}$. Thus the miniswap involving $\aaa$ is {\sf T6}, and $U$ does not exist. 

\noindent
{\sf Subcase 7.2: ($\bullet_\GG \in \aaa^\leftarrow$):}
The miniswap involving $\aaa^\leftarrow$ is either {\sf T4} or {\sf T5}. $U$ has $\bullet_{\GG^+} \in \aaa$ and the unique child of $\HH \in \underline{\aaa}$ is in  $\aaa^\leftarrow$ or $\overline{\aaa^\leftarrow}$. The (G.12) confirmation is therefore clear.

\smallskip
\noindent
{\bf (G.13):} By inspection of the miniswaps, if $\EE^! \in \x$ or $\underline{\x}$ in $U$,
then $T$ has $\EE$ or $\EE^!$ in the same location. Thus there are two cases:

\noindent
{\sf Case 1: (This $\EE$ is marked in $T$):} By $T$'s (G.13), there is an $\FF$ or $\circled{\FF}$ in $\underline{\x}$ with $N_\EE=N_\FF$
and $\family(\FF)=\family(\EE)+1$. If $\FF\in \underline{\x}$ is nonvirtual, then since it appears in the same place in $U$, $U$'s (G.13) holds. 
Thus suppose $T$ has $\circled{\FF}\in \underline{\x}$. We check the conditions for $\circled{\FF}$ to appear in $\underline{\x}$ in $U$. Let $U^\star$ be $U$ with $\FF$ added in $\underline{\x}$.

((V.1) holds, i.e., $\FF\in \underline{\x}$ is not marked in $U^\star$): By Lemma~\ref{lem:virtualG13label}, $T$ has $\bullet_\GG \in \x^\leftarrow$. Note $\FF\succeq \GG$, since otherwise 
$\circled{\FF}\in \underline{\x}$ would be marked in $T$, a contradiction. If $\FF=\GG$, then {\sf T5} or {\sf T6} would apply
at $\{ \x^\leftarrow,\x\}$, 
contradicting $\EE^!\in \x$ or $\EE^! \in \underline{\x}$ in $U$. 
Thus $\FF\succ \GG$, as desired.

((V.2) holds, i.e. $U$ has an $\FF$ West of $\underline{\x}$:) By $T$'s (V.2), $T$ has an $\FF$ West of $\x$. This remains true for $U$ since no swap removes a nonvirtual genetic label without putting one further west.

((G.1) holds for $U^\star$:) Immediate from $T$'s $\circled{\FF} \in \underline{\x}$ and $U$'s (G.1).

((G.4) holds for $U^\star$:) We have $\EE<\FF<\lab_T(\x^\downarrow)$. Since
$\FF\succ \GG$, $\lab_T(\x^\downarrow)=\lab_U(\x^\downarrow)$. Hence $U^\star$ does not violate (G.4) locally, so by $U$'s (G.4), we are done.

((G.5) holds for $U^\star$:) Since $\EE^!\in \x$ or $\underline{\x}$ in $U$, by inspection of the miniswaps, $T$ and $U$ have the same set of nonvirtual labels on $\underline{\x}$.

((G.6) holds for $U^\star$:) If there is a ${\widetilde \FF}$ West of $\x$'s column in $U$, there is one west of $\x$'s column in $T$. Hence we are done by $T$'s (V.3).

((G.8) holds for $U^\star$:) 
It suffices to show that in the column reading word of $U^\star$ (with $\FF$ placed in $\underline{\x}$), no $\EE$ is read after the $\FF \in \underline{\x}$. By $T$'s (V.3), this is true in $T$.  
By inspection of the miniswaps, $\EE$ does not appear West of $\x$ in $U$. Hence by $U^\star$'s (G.4), we are done.

((G.9) holds for $U^\star$:) Since $U$ has $\EE^!\in \x$ or $\underline{\x}$, $U$ has a $\bullet_{\GG^+}$
northwest of $\x$. So by $U$'s (G.2), there is no a $\bullet_{\GG^+}$ southeast of $\underline{\x}$.

((G.12) holds for $U^\star$:) This follows from the following two claims:
\begin{claim} If $\family(\FF)=\family({\widetilde \FF})$, then $U$ has no 
${\widetilde \FF}$ SouthEast of $\underline{\x}$. 
\end{claim}
\begin{proof}
By $T$'s (G.11) and $T$'s $\EE^!$ in $\x$'s column, $T$ has a $\bullet_\GG$ northWest of $\x$. Hence by $T$'s (G.2), $T$ has no $\bullet_\GG$ SouthEast of $\x$. Thus by the (G.12) condition of $T$'s (V.3), $T$ has no ${\widetilde \FF}$ SouthEast of $\circled{\FF} \in \underline{\x}$.
A child of ${\widetilde \FF}$ will only be South of its parent if $T$ has ${\widetilde \FF}\in \y$
and $U$ has ${\widetilde \FF}\in \underline{\y}$. Thus $U$ has no ${\widetilde \FF}$ SouthEast of $\underline{\x}$ that is a child of a nonvirtual ${\widetilde \FF}$ in $T$.

The remaining concern is ${\circled {\widetilde \FF}}\in {\underline\y}$ in 
$T$ with a child SouthEast of $\x$ in $U$. By inspection of the miniswaps, 
this ${\circled {\widetilde \FF}}$ cannot have a child South or East of $\underline{\y}$. Therefore $\y$ must be SouthEast of $\x$ in $T$. Again by inspection of the miniswaps, $T$ has a $\bullet_\GG\in \y^\leftarrow$ or a $\bullet_\GG\in \y$. This is impossible by $T$'s (G.2), recalling $T$'s $\bullet_\GG$ northWest of $\x$. 
\end{proof}

\begin{claim}
If $U$ has an ${\widetilde \FF}$ NorthWest of 
$\underline{\x}$ with $\family(\FF)=\family({\widetilde \FF})$, then this ${\widetilde \FF}$ and the
$\FF\in \underline{\x}$ satisfy $U^*$'s (G.12). 
\end{claim}
\begin{proof}
By assumption, $\EE^!$ remains in $\x$ or $\underline{\x}$ in $U$. 
By Lemma~\ref{lem:virtualG13label}, $T$ has $\bullet_\GG\in \x^\leftarrow$ and $\family(\FF)=\family(\GG)$. By $T$'s (V.1), $\GG\preceq \FF$. If $\GG=\FF$, then $\{\x^\leftarrow, \x\}$ is a {\tt tail} of type {\sf T5} or {\sf T6}; however, these miniswaps do not leave
$\EE^!$ in place, a contradiction. Hence $\GG\prec \FF$. Therefore by $T$'s (G.5) and Lemma~\ref{lem:same999}, $T$ has no $\GG$ or $\circled{\GG}$ in $\x^\leftarrow$'s column.
Hence {\sf H3} applies at $\x^\leftarrow$, and $U$ has $\bullet_{\GG^+}\in \x^\leftarrow$.

By $T$'s (G.12), there is no ${\widetilde \FF}$ NorthWest of ${\overline \x}$. Since no miniswap moves a label more than one box north, it suffices to consider a ${\widetilde \FF}$ in $T$ that is in either $\y$, $\underline\y$, ${\overline \y}$ or $\y^\downarrow$, where $\y$ is in $\x$'s row. By $T$'s (G.2) and $\bullet_\GG\in \x^\leftarrow$, $T$ has no $\bullet_\GG$ strictly northwest of $\x^\leftarrow$. Hence no such ${\widetilde \FF}$ in these four positions can move North. Therefore any $\widetilde{\FF}$ NorthWest of $\FF\in \underline{\x}$ in $U$ satisfies (G.12), in view of the $\bullet_{\GG^+}\in \x^\leftarrow$. Note that the forbidden configuration from the final
sentence of (G.12) cannot occur since $\FF\neq \GG$, whereas the forbidden configuration forces
$\FF=\GG$.  
\end{proof}

\noindent
{\sf Case 2: (This $\EE$ is unmarked in $T$):} By Lemma~\ref{lemma:child123}, since there is a $\bullet_{\GG^+}$ northwest of this $\EE$ in $U$, there is a $\bullet_\GG$ northwest of $\x$ in $T$. Moreover, by $U$'s (G.11), $\x^\leftarrow$ is a box of $U$, and hence of $T$. Since this $\EE$ is unmarked in $T$, $\EE \succeq \GG$. Since this $\EE$ is marked in $U$, $\EE \prec \GG^+$. Thus, $\EE = \GG$. Let $\HH$ be the gene (if it exists) with $\family(\HH) = \family(\GG) + 1$ and $N_\HH = N_\GG$. If $T$ has $\HH \in \underline{\x}$, then since it appears in the same place in $U$, $U$'s (G.13) holds. 
Thus assume $\HH \notin \underline{\x}$ in $T$.

\noindent
{\sf Subcase 2.1: (In $T$, $\bullet_\GG \in \x^\leftarrow$)}:
The miniswap applied at $\x^\leftarrow$ is {\sf H5.2} or {\sf T4.2}.
We are done by the following lemma:

\begin{lemma}
\label{lemma:virtuallabelsexist}
If $T$ is a $\GG$-good tableau where {\sf H5.2} or {\sf T4.2} applies and $U \in \swap_\GG(T)$,
then all prescribed $\circled{\HH}$'s from the outputs of {\sf H5.2} and {\sf T4.2} are valid 
virtual labels in the sense of {\normalfont (V.1)--(V.3)}.
\end{lemma}
\begin{proof}
Consider such a miniswap. We may assume that $U$ contains an output with a prescribed $\circled{\HH}$, say in $\underline{\x}$.
Let $U^\star$ be $U$ with $\HH$ added in $\underline{\x}$.

((V.1) holds, i.e., $\HH\in \underline{\x}$ is not marked in $U^\star$):
By assumption, $\GG<\HH$ and so $\GG^+\preceq \HH$.

((V.2) holds, i.e., some $\HH$ appears in $U$ West of $\x$): Since $\circled{\HH} \in \underline{\x}$ in $T$, by $T$'s (V.2), some $\HH$ appears in $T$ West of $\x$. By inspection of the miniswaps, this $\HH$ has a child, which is West of $\x$ in $U$.

((G.1) holds in $U^\star$): Immediate from the (G.1) condition of $T$'s (V.3).

((G.4) holds in $U^\star$): By $U$'s (G.4), 
the only concern is an $\widetilde{\HH}$ in $\x$'s column of $U$ with $\family({\widetilde \HH})=\family(\HH)$. By Lemma~\ref{lemma:child123}, this $\widetilde{\HH}$ is a child of an $\widetilde{\HH}$ in $T$. Since $\GG<\widetilde{\HH}$, this $\widetilde{\HH}$ in $T$ is in the same location as its unique child in $U$, contradicting the (G.4) condition of $T$'s (V.3).

((G.5) holds in $U^\star$:) Neither miniswap in question affects the nonvirtual
labels on $\underline{\x}$, so we are done by the (G.5) condition of $T$'s (V.3).

((G.6) holds in $U^\star$:) All labels of this family
appear in the same places in $T$ and $U$, so we are done by the (G.6) condition of $T$'s (V.3).

((G.8) holds in $U^\star$:) First suppose there there is a nonballot 
genotype $G_{U^\star}$ of $U^\star$, with the $\HH\in \underline{\x}$ taken. Since no labels of $\family(\HH)+k$ (for
$k\geq 0$) are moved by $\swap_\GG$, by $T$'s (G.8) there is no violation among those families. By $U$'s (G.8), it suffices to consider the possibility that in ${\tt word}(G_{U^\star})$ the selected $\GG$ appears after the $\HH$. However, if such a $G_{U^\star}$ exists, then by $T$'s (G.4) and inspection of the miniswaps,
it follows there is a $\GG$ West of $\x$ in $T$; this contradicts $\circled{\HH}\in \underline{\x}$ in $T$.

((G.9) holds in $U^\star$:) By $U$'s (G.2), $U$ has no $\bullet_{\GG^+}$ southeast of $\x^\leftarrow$.

((G.12) holds in $U^\star$):
By $T$'s (G.2) and (G.12), $T$ has no $\widetilde{\HH}$ SouthEast of $\underline{\x}$
with $\family(\widetilde{\HH})=\family(\HH)$. The same is true in $U$, as $\swap_\GG$ does not affect $\widetilde{\HH}$.
It remains to consider such ${\widetilde\HH}$ in $U$ that are NorthWest of $\underline{\x}$. Since $\swap_\GG$ does not affect $\widetilde{\HH}$, by $T$'s (G.12), no ${\widetilde\HH}$ is North of $\overline{\x}$ in either $T$ or $U$. Thus assume ${\widetilde\HH}$ is West of $\x$ and either in its row or on a top edge of its row. By assumption, $T$'s $\bullet_\GG\in \x^\leftarrow$ becomes the desired $\bullet_{\GG^+}\in \x^\leftarrow$ in $U$.
\end{proof}

\noindent
{\sf Subcase 2.2: (In $T$, $\DD^! \in \x^\leftarrow$)}:
If $\GG \in \underline{\x}$, this $\GG$ is not in a snake of $T$. Otherwise $\GG \in \x$ and {\sf H9} applies at $\x$.
By Lemma~\ref{lem:strong_form_of_G13}, $T$ has some $\widetilde{\GG} \in \underline{\x^\leftarrow}$ or $\circled{\widetilde{\GG}} \in \underline{\x^\leftarrow}$ with $\family(\widetilde{\GG}) = \family(\GG)$ and $N_{\widetilde{\GG}} = N_\DD$.
By Lemma~\ref{lem:how_to_check_ballotness}, $\widetilde{\GG} \neq \GG$. Therefore by $T$'s (G.6), $\widetilde{\GG} = \GG^-$.
By $T$'s (G.11), $T$ has a $\bullet_\GG$ northWest of $\x^\leftarrow$. Hence this $\widetilde{\GG} \in \underline{\x^\leftarrow}$ is nonvirtual and marked. Thus by $T$'s (G.13), there is some $\widetilde{\HH} \in \underline{\x^\leftarrow}$ or $\circled{\widetilde{\HH}} \in \underline{\x^\leftarrow}$ with $\family(\widetilde{\HH}) = \family(\widetilde{\GG}) + 1$ and $N_{\widetilde{\HH}} = N_{\widetilde{\GG}}$.  We claim there is an $\HH$ South of $\x$ and in its same column, with $\family(\HH)=\family(\GG)+1 = \family(\widetilde{\HH})$ and $N_\HH = N_\GG$. Certainly by $N_{\widetilde{\HH}}= N_{\widetilde{\GG}} = N_\DD$, there is such an $\HH$ somewhere in $T$. By Lemma~\ref{lem:how_to_check_ballotness}, it is located as described. Now if $\HH \in \underline{\x}$, we are done. Otherwise by $T$'s (G.4), $\HH \in \x^\downarrow$. But now $\lab_T(\x^{\leftarrow \downarrow}) \prec \HH$ by $T$'s (G.3) and $\lab_T(\x^{\leftarrow\downarrow}) > \widetilde{\HH}$ by $T$'s (G.4), a contradiction.
\qed

\section{Proof of Proposition~\ref{prop:goodness_preservation_reverse}}
\label{sec:backwards_goodness_proof}
We check that (G.1)--(G.13) are preserved. Let $T \in \revswap_{\GG^+}(U)$, where $U$ is $\GG^+$-good.
Below, the proof of property (G.$j$) only possibly depends
on earlier properties (G.$i$). We also show that the virtual labels prescribed by reverse miniswaps {\sf L2.3}, {\sf L4.3} and {\sf L4.5} are valid virtual labels in the sense of (V.1)--(V.3). This appears as Lemma~\ref{lem:virtual_label_validity_B} in the section `{\sf Consistency of the prescribed virtual labels},' located between the arguments for (G.12) and (G.13).

\smallskip
\noindent
{\bf (G.1):} 
Suppose $T$ has $\QQ \in \x$ or $\QQ \in \underline{\x}$ that is too high. By $U$'s (G.1), the label $\QQ$ does not appear in the same place in $U$. Hence $\QQ$ is placed in $\x$ or $\underline{\x}$ in $T$ by some reverse miniswap. We consider which reverse miniswap this might be. By $U$'s (G.1), $\QQ$ does not appear anywhere in $U$ north of $\underline{\x}$. Hence by visual inspection, the only miniswap to consider is {\sf L1.1}. However to apply {\sf L1.1}, we have by assumption $\QQ \in \x^\uparrow$ in $U$. Since this is impossible by $U$'s (G.1), $T$ cannot have any label too high.

\smallskip
\noindent
{\bf (G.2):} 
By Lemma~\ref{lem:ladders.sw.ne}, ladders lie in distinct rows and columns; hence $T$ has no $\bullet_\GG$ NorthWest of another.
Since $\revswap$ is defined by its action on rows, $T$ has at most one $\bullet_\GG$ in any row.

Suppose $T$ has $\bullet_\GG \in\x$ North of $\bullet_\GG \in y$ and in the same column. These boxes are in ladders of $U$.
By Lemma~\ref{lem:ladders.sw.ne}, $\x$ and $\y$ are in the same ladder of $U$.
By Lemma~\ref{lemma:laddersareshort}, $\x=\y^\uparrow$ and the two boxes are  $\ytableausetup{boxsize=1.3em}\ytableaushort{\GG, {\bullet_{\GG^+}}}$ in $U$. Hence, in order to have $\bullet_\GG \in \y$ in $T$, we must
apply {\sf L1.2} or {\sf L3} to $\y$. By definition, $U$'s $\GG\in \x$ means that 
{\sf L1.2} does not apply. {\sf L3} requires $\GG\in \overline{\y}$, contradicting $U$'s (G.4).

\smallskip
\noindent
{\bf (G.3):}
It is enough to confirm (G.3) for an arbitrary fixed row $\mathcal{R}$ of $T$. If $\mathcal{R}$ does not intersect any ladder of $U$, then $T$'s (G.3) is confirmed in $\mathcal{R}$ by $U$'s (G.3).

Otherwise by Lemma~\ref{lem:ladders.sw.ne}, $\mathcal{R}$ intersects a single ladder $L$. Let $r := \mathcal{R} \cap L$. Let $\x$ be the westmost box of $r$.

If $r$ is {\sf L1}, confirmation is trivial unless $\GG \in \x^\uparrow$.
In that case, locally at $\x$, $\revswap_{\GG^+}$ results in 
$\ytableausetup{boxsize=1.4em}\ytableaushort{\none \GG \none, \FF {\bullet_{\GG^+}} \HH} \mapsto \ytableaushort{\none \star \none, \FF \GG \HH}$ (here $\star = \bullet_{\GG}$, but this is not important to us). (If either $\FF$ or $\HH$ does not exist, the argument is simpler.) We need $\FF \prec \GG \prec \HH$. By $U$'s (G.9), $\FF \preceq \GG$. Since 
$r$ is {\sf L1}, $\FF \neq \GG$, so $\FF \prec \GG$. If $\QQ \in \x^{\uparrow\rightarrow}$ in $U$, by $U$'s (G.3) and (G.4) then $\GG \prec \QQ \preceq \HH$. Otherwise $\bullet_{\GG^+} \in  \x^{\uparrow\rightarrow}$ in $U$. By $U$'s 
(G.11), $\HH \in \x^\rightarrow$ is not marked in $U$. Thus $\GG^+ \preceq \HH$ and $\GG \prec \HH$, as desired.

Suppose $r$ is {\sf L2}. Then $\lab_U(\x) = \GG$, while $\lab_T(\x)\in \{\GG,\bullet_\GG \}$. Hence, $T$'s (G.3) is confirmed from $U$'s (G.3).

If $r$ is {\sf L3}, $r$ does not change and we are done.

Suppose $r$ is {\sf L4.1}, {\sf L4.2} or {\sf L4.3}. Locally at $r$, the reverse miniswap is
\[
\begin{picture}(190,15)
\put(10,0){$\ytableaushort{\FF \GG \bullet \HH} \mapsto \ytableaushort{\FF \bullet {\GG^+} \HH}$}
\put(47,-5){$\GG^+$}
\end{picture}
\text{(the labels of $\underline{\x}$ in $T$ are not displayed)}.\]
We need to show $\FF \prec \GG^+ \prec \HH$. (If $\FF$ or $\HH$ do not exist, the argument is simpler.) By $U$'s (G.3), $\FF \prec \GG$, so $\FF \prec \GG^+$. By $U$'s (G.12) (final sentence), $\HH\neq \GG^+$ and $\HH\in \x^{\rightarrow\rightarrow}$ is not marked in $U$. Thus $\GG^+ \prec \HH$. 

Lastly, suppose $r$ is {\sf L4.4} or {\sf L4.5}. Locally at $r$,
$\ytableaushort{\EE \GG \bullet \HH} \mapsto \ytableaushort{\EE \bullet \FF \HH}$, where $\FF$ is as in the definitions of {\sf L4.4} and {\sf L4.5}.

First we check $T$ has no (G.3) violation between $\EE$ and $\FF$. If $\EE < \FF$, this is obvious. If $\EE > \FF$, $\EE$ and $\FF$ form the exceptional configuration of (G.3) in $T$. Suppose $\family(\EE) = \family(\FF)$. By $U$'s (G.3) and (G.7), $\EE \neq \FF$. Hence by $U$'s (G.6), $\EE \prec \FF$, as desired. 
 
Now we check $\FF \prec \HH$.

\noindent
{\sf Case 1: ($\GG \prec \HH$):} Since $\FF \preceq \GG$, $\FF \prec \HH$ follows. 

\noindent
{\sf Case 2: ($\GG \succeq \HH$):}
By $U$'s (G.3), $\GG > \HH$ and $\HH^! \in \x^{\rightarrow\rightarrow}$ in $U$. 

\noindent
{\sf Subcase 2.1: ($\family(\FF) < \family(\HH)$):} Then $\FF \prec \HH$, and we are done.

\noindent
{\sf Subcase 2.2:  ($\family(\FF) = \family(\HH)$):} 
By Lemma~\ref{lemma:Gsoutheast} applied to $U$, $\FF \neq \HH$. Hence by $U$'s (G.6), $\FF \prec \HH$.

\noindent
{\sf Subcase 2.3: ($\family(\FF) > \family(\HH)$):} We derive a contradiction. By Lemma~\ref{lem:strong_form_of_G13}, $U$ has $\GG' \in \underline{\x^{\rightarrow\rightarrow}}$ or $\circled{\GG'} \in \underline{\x^{\rightarrow\rightarrow}}$ with $ N_\HH = N_{\GG'}$ and $\family(\GG) = \family(\GG')$. If $\GG = \GG'$, then by $U$'s (G.7), $\circled{\GG} \in \underline{\x^{\rightarrow\rightarrow}}$ in $U$. However $\underline{\x^{\rightarrow\rightarrow}}$ is southeast of a $\bullet_{\GG^+}$ in $\x^\rightarrow$ in $U$, contradicting $U$'s (V.1). Hence $\GG \neq \GG'$, so by $U$'s (G.6) $\GG' \succeq \GG^+$. If $\GG'\succ\GG^+$, then by $U$'s (G.6)
there must be a $\GG^+$ in the column of $\x^\rightarrow$. This $\GG^+$ is not
South of $\x^\rightarrow$ by $U$'s (G.12). It is also not north of
$\x^\rightarrow$ by $U$'s (G.9). Thus, $\GG'=\GG^+$.

Since $N_{\GG^+}=N_\HH$, $N_\GG = N_{\HH^-}$. By $U$'s (G.8), since $U$ has $\GG\in \x$, every $\HH^-$ 
appears 
before $\x$ in column reading order.  By $U$'s (G.4) and (G.6), $\HH^-$ does not appear east of $\x^{\rightarrow\rightarrow}$'s column.
By $U$'s (G.12), $\HH^-$ does not appear North of $\x^\rightarrow$ and in its column. 
Since $\bullet_{\GG^+}\in \x^\rightarrow$, $\HH^-$ is not in $\x^\rightarrow$.
By $U$'s (G.11), $\HH^-$ is not South of $\x^\rightarrow$ and in its column. 
Thus $\HH^-$ appears only in $\x$'s column and is north of $\x$. By $U$'s (G.12), it then follows that $\HH^-\in \overline{\x}$ in $U$. 
But then since $N_{\HH^-}=N_{\GG}$ and $\HH^-<\FF$, this contradicts the definition of $\FF$.

\smallskip
\noindent
{\bf (G.4):}
Consider an arbitrary column $c$ of $U$; we show that (G.4) holds for $c$ in $T$.

\noindent
{\sf Case 1: ($U$ has no $\bullet_{\GG^+}$ in $c$):}

\noindent
{\sf Subcase 1.1: ($c$ is $<$-increasing in $U$):} By inspection
of the reverse miniswaps, it is clear that
$c$ is $<$-increasing in $T$.

\noindent
{\sf Subcase 1.2: ($c$ contains $\ytableausetup{boxsize=1.2em}\ytableaushort{{\FF},{\FF^!}}$):}
Suppose the depicted $\FF$ is in $\x$ in $U$. Since $\FF^!\in \x^\downarrow$ in $U$, $\FF \preceq \GG$. So there are two possibilities:

\noindent
{\sf Subcase 1.2.1: ($\FF\prec \GG$):} Since $\FF\prec\GG$, $\x$ and $\x^\downarrow$ are not in ladders of $U$. Therefore, $\FF\in \x$ and $\FF\in \x^\downarrow$ in $T$; however;
we do not know {\it a priori} which of these $\FF$'s are marked. We must show
$T$ has unmarked $\FF\in \x$ and $\FF^!\in \x^\downarrow$. 
We will need the following
\begin{lemma}
\label{lem:keybulletreversalfact}
Let $U$ be $\GG^+$-good and $T\in \revswap_{\GG^+}(U)$. Suppose $\bullet_\GG\in \y$ in $T$. Then in $U$, either $\GG\in \y$ or $\bullet_{\GG^+}\in \y$. 
\end{lemma}
\begin{proof}
Since $\bullet_\GG\in \y$ in $T$, $\y$ is in a ladder of $U$. Thus the lemma is immediate from the definition of ladders.
\end{proof}

Since unmarked $\FF\in \x$ in $U$, $U$ has no $\bullet_{\GG^+}$ northwest of $\x$. By $U$'s (G.3) or (G.4), $U$ has no $\GG$ northwest of $\x$.
Hence by Lemma~\ref{lem:keybulletreversalfact}, $T$ has no $\bullet_\GG$  has northwest of $\x$; hence unmarked $\FF\in \x$ in $T$. 

By definition of marked labels, $U$ has $\bullet_{\GG^+}$ 
in some box $\z$ northwest of $\x^\downarrow$ in $U$. By inspection of the miniswaps, there is a $\bullet_\GG$ northwest of $\z$ in $T$. Hence
$\FF^!\in \x^\downarrow$ in $T$.

\noindent
{\sf Subcase 1.2.2: ($\FF = \GG$):} By inspection of the miniswaps, $\bullet_\GG\in \x$ in $T$ and no other box or edge of the column was changed.
Hence we are done by $U$'s (G.4).

\noindent
{\sf Case 2: ($U$ has a $\bullet_{\GG^+}$ in $c$):}
By $U$'s (G.11), $U$ has no marked labels in $c$. Hence by $U$'s (G.4), the labels of $U$ strictly $<$-increase down $c$, ignoring the $\bullet_{\GG^+}$. 
Suppose the $\bullet_{\GG^+}$ in $c$ is in $\x$.
By inspection, a (G.4) violation in $c$ can only occur from {\sf L1.1}, {\sf L4.4} or {\sf L4.5} applied at $\x$. If we apply {\sf L1.1} at $\x$, then locally at $c$, $\ytableaushort{\GG, \bullet} \mapsto \ytableaushort{\bullet, \GG}$, while the rest of the column is unchanged, so $c$ satisfies (G.4) in $T$.

Suppose we apply {\sf L4.4} or {\sf L4.5} at $\x$. 
 By definition $U$ has $\GG \in \x^\leftarrow$ and there is no $\GG^+ \not\in \underline{\x}$ with $\family(\GG^+) = \family(\GG)$. 
Let $\EE$ be the $\prec$-greatest gene in $c$ north of $\overline{\x}$ and let $\HH$ be the $\prec$-least gene in $c$ south of $\underline{\x}$. (If either of these fails to exist, the argument is simplified.) Let $\FF$ be the $\prec$-least element of $Z^\#$ in the notation of {\sf L4.4} and {\sf L4.5}. It remains to show

\begin{claim}
\gap
\begin{itemize}
\item[(I)] $\GG < \HH$, and 
\item[(II)] either $\EE < \FF$, or $\EE = \FF$ with $\EE$ an unmarked label in $\x^\uparrow$ and $\FF < \GG$.
\end{itemize}
\end{claim}
\begin{proof}
(I): By $U$'s (G.11), the $\HH$ in $c$ is not marked in $U$, so $\GG \prec \HH$. If $\family(\HH) = \family(\GG)$, then by $U$'s (G.12) $\HH \in \underline{\x}$ rather than $\x^\downarrow$. Then by $U$'s (G.7), $\HH \neq \GG$. But then by $U$'s (G.6), $\HH = \GG^+ \in \underline{\x}$, contradicting our assumption. Thus $\GG < \HH$.

(II): By $U$'s (G.9), $\EE \preceq \GG$. Hence by $U$'s (G.6), either $\EE = \GG$ or else $\EE < \GG$.

\noindent
{\sf Case A:  ($\EE = \GG$):}
By Lemma~\ref{lem:how_to_check_ballotness}, $Z = \emptyset$ (in the notation of {\sf L4.4} and {\sf L4.5}). Further by $U$'s (G.7), $\EE \in \x^\uparrow$ (rather than $\overline{\x}$). Hence locally $c$ changes as $\ytableaushort{\GG, \bullet} \mapsto \ytableaushort{\bullet, \GG}$, while the rest of $c$ is unchanged. Then $c$ satisfies (G.4) in $T$.

\noindent
{\sf Case B: ($\EE < \GG$):} 
If $\FF = \GG$, then $\EE < \FF$, as desired. So by $U$'s (G.4), we may assume $\FF < \GG$ and furthermore that $\family(\EE) \geq \family(\FF)$. Then by Lemma~\ref{lem:draggable_labels_in_complete_sets}, there is an $\widetilde{\EE} \in Z$ with $\family(\widetilde{\EE}) = \family(\EE)$.

If $\widetilde{\EE} \neq \EE$, then by $U$'s (G.6) $\EE = \widetilde{\EE}^+$.  Since $N_{\widetilde \EE} = N_\GG \geq 1$, $\family(\GG^+) = \family(\GG)$ and $N_\EE = N_{\GG^+}$. Hence by $U$'s (G.6) and (G.8), $U$ has $\GG^+$ in $c$ south of $\overline{\x}$. By $U$'s (G.9), it is south of $\underline{\x}$. By $U$'s (G.12), it is not South of $\underline{\x}$. Hence $\GG^+ \in \underline{\x}$ in $U$, contradicting the assumptions of {\sf L4.4} and {\sf L4.5}.

Thus $\widetilde{\EE} = \EE$. By Lemma~\ref{lem:how_to_check_ballotness}, $\widetilde{\EE} = \min Z$, so $\widetilde{\EE} = \FF$. By $U$'s (G.7), the $\EE$ in $c$ in $U$ is in $\x^\uparrow$. Since in $T$, $\bullet_\GG \in \x^\leftarrow$ and $\EE < \GG$, $\EE^! \in \x$ in $T$. 

We must show that unmarked $\EE \in \x^\uparrow$ in $T$. By $U$'s (G.2), $U$ has no $\bullet_{\GG^+}$ northwest of $\x$. By $U$'s (G.3) and (G.4), $U$ has no $\GG$ northwest of $\x^\uparrow$. Hence by Lemma~\ref{lem:keybulletreversalfact}, $T$ has no $\bullet_\GG$ northwest of $\x^\uparrow$, so $T$ has unmarked $\EE \in \x^\uparrow$.
\end{proof}

\smallskip
\noindent
{\bf (G.5):} We are only concerned with reverse miniswaps that produce or relocate edge labels. That is {\sf L2.2}, {\sf L3}, {\sf L4.2}, {\sf L4.4} and {\sf L4.5}. In {\sf L2.2}, we create an edge label $\GG$ on $\underline{\x}$, while $U$ has $\GG \in \x$. Hence by $U$'s (G.4), there is no other label of $\GG$'s family on $\underline{\x}$. This verifies $T$'s (G.5) in this situation. A similar argument applies for {\sf L3} and {\sf L4.2}.

The arguments for {\sf L4.4} and {\sf L4.5} are similar. Consider any $\QQ \in A''$ or ${\widetilde{A''}}$. If (G.5) fails in $T$, we may assume $U$ has $\QQ' \in \underline{\x^\rightarrow}$ with $\family(\QQ) = \family(\QQ')$. If $\QQ \neq \GG$, then by $U$'s (G.4), $\QQ < \GG$, so $\QQ' \in \underline{\x^\rightarrow}$ is marked in $U$, violating $U$'s (G.11). 
Thus $\QQ=\GG$. By $U$'s (G.7), $\QQ' \neq \GG$. Hence by $U$'s (G.6), $\QQ' = \GG^+$, contradicting that {\sf L4.4} or {\sf L4.5} applies.

\smallskip
\noindent
{\bf (G.6):}  We will use
\begin{lemma}
\label{lem:keyreversalfact}
Let $U$ be $\GG^+$-good.
If $\HH$ appears in column $c$ of $T\in \revswap_{\GG^+}(U)$, then $U$ has $\HH$ in column $c$ or column
$c^\leftarrow$. 
\end{lemma}
\begin{proof}
By inspection of the reverse miniswaps.
\end{proof}
Suppose $i_b$ appears West of $i_a$ in $T$. Then by Lemma~\ref{lem:keyreversalfact},
$i_b$ appears west of $i_a$ in $U$. Therefore by $U$'s (G.4), either $b=a$ or $i_b$ is West of $i_a$ in $U$. Thus by $U$'s (G.6), $b \leq a$.

\smallskip
\noindent
{\bf (G.7):} Let $\QQ \in \underline{\x}$ in column $c$ of $T$. First suppose $U$ has $\QQ \in \underline{\x}$. By $U$'s (G.7), $U$ has no $\QQ$ West of column $c$.
Hence by Lemma~\ref{lem:keyreversalfact}, $T$ has no $\QQ$ West of column $c$, as desired.

So suppose $U$ has $\QQ \notin \underline{\x}$. Thus $T$'s $\QQ\in \underline{\x}$ was created by one of the reverse miniswaps {\sf L2.2}, {\sf L3}, {\sf L4.2}, {\sf L4.4} or {\sf L4.5}. In each
case, we are done by Lemma~\ref{lem:keyreversalfact}, provided that we know $U$ has no $\QQ$ West of $c$. This is by assumption in {\sf L2.2} and {\sf L4.2}. This holds for {\sf L3} by $U$'s (G.7). For {\sf L4.4}, $\QQ\in A''$ (in the terminology of {\sf L4.4}). If $\QQ=\GG$, we are done by assumption; otherwise we are done by $U$'s (G.7),
since $\QQ\in A$ (in the terminology of {\sf L4.4}). For {\sf L4.5}, 
$\QQ\in \widetilde{A''}$ (in the terminology of {\sf L4.5}), and we are done by $U$'s (G.7), since $\QQ\in A$ (in the terminology of {\sf L4.5}).

\smallskip
\noindent
{\bf (G.8):}
Suppose $N_\EE = N_\FF$ and $\family(\FF) = \family(\EE) + 1$. By $T$'s (G.4), to show $T$'s (G.8), it suffices to show no $\FF$ is East of any $\EE$ in $T$.

Let $e$ be the westmost instance of $\EE$ and $f$ the eastmost instance of $\FF$ in $U$. By $U$'s (G.8), $e$ is east of $f$ in $U$. Swapping does not move $e$ West and moves $f$ at most one column East. We may therefore assume $e$ and $f$ are in the same column $c$ of $U$. We may also assume the swap moves $f$ East, i.e., the swap involving $f$ is {\sf L4.4} or {\sf L4.5} (say at $\{ \x, \x^\rightarrow \}$). 
If $e \in \overline{\x}$, then $T$'s westmost $\EE$ is in $\x^\rightarrow$ or $\underline{\x^\rightarrow}$, and there is no (G.8) violation.
Hence by $U$'s (G.4), we may assume $e \in \x^\uparrow$.
Note that $N_\FF = N_\GG$, $U$ has $\GG \in \x$, and either $\FF = \GG$ or $\FF \in \overline{\x}$.

By $U$'s (G.3), $\EE \prec \QQ := \lab_U(\x^{\uparrow\rightarrow})$. By $U$'s (G.9), $\QQ \preceq \GG$. By Lemma~\ref{lem:how_to_check_ballotness}, $\QQ \neq \GG$. Hence $\QQ \prec \GG$, so by $U$'s (G.6), $\QQ < \GG$.
By $U$'s (G.4) and Lemma~\ref{lem:draggable_labels_in_complete_sets}, for every $\family(\FF) < i < \family(\GG)$, there is a label $\HH^i \in \overline{\x}$ with $\family(\HH^i) = i$ and $N_{\HH^i} = N_\EE = N_\FF = N_\GG$. By Lemma~\ref{lem:how_to_check_ballotness}, $\QQ\neq \FF$ and $\QQ \neq \HH^i$. Hence by $U$'s (G.6), $\QQ$ is one of $\EE^+, \FF^+, (\HH^i)^+$. Hence, $U$ has a $\GG^+$ in $c^\rightarrow$ with $\family(\GG^+) = \family(\GG)$. By $U$'s (G.9) and (G.12), it follows that $\GG^+ \in \underline{\x^\rightarrow}$; this contradicts the assumptions of {\sf L4.4} or {\sf L4.5}.

\smallskip
\noindent
{\bf (G.9):}  Let $\FF \succeq \GG$. Suppose $\FF\in \x$ (or $\FF\in \underline{\x}$) 
is northwest of $\bullet_\GG \in \y$ in $T$. By 
inspection of the reverse miniswaps, $U$ has an $\FF$ northwest of 
$\underline{\x}$. 

Suppose $\FF\succ \GG$. 
By $U$'s (G.9), $U$ has no $\bullet_{\GG^+}$ southeast
of $\x$. Hence by Lemma~\ref{lem:keybulletreversalfact},
$\GG\in \y$ in $U$. Thus by $U$'s (G.3) and (G.4),
$\FF\preceq \GG$, contradicting $\FF\succ \GG$.

Thus suppose $\FF=\GG$. 
By Lemma~\ref{lem:keybulletreversalfact}, in $U$
either $\bullet_{\GG^+} \in \y$ or $\GG\in \y$. Suppose
$\GG\in \y$. By Lemma~\ref{lemma:Gsoutheast}, $\y$ is not southEast of $\x$,
i.e., $\x$ and $\y$ are in the same column. Hence by $U$'s (G.4), $U$ has $\GG^!\in \y$. Hence $\y$ is not in a ladder of $U$, contradicting $T$'s $\bullet_\GG \in \y$. Thus, $U$ has $\bullet_{\GG^+} \in\y$. By Lemma~\ref{lem:GandbulletG+}, it follows that $\x$ and $\y$ are
in the same row or column. By $U$'s (G.9), in fact $\y=\x^\downarrow$
or $\y=\x^\rightarrow$. Now, by inspection of the reverse miniswaps, we conclude that
$\FF\not\in \x$ (respectively, $\FF\not\in \underline{\x}$) 
or $\bullet_{\GG}\not\in \y$ in $T$, contradicting our initial assumptions.

\smallskip
\noindent
{\bf (G.10):}
Consider $\FF^!\in \x$ or $\FF^!\in \underline{\x}$ in $T$.

\noindent
{\sf Case 1: (This $\FF$ is not in the same location in $U$):} By assumption, 
$\FF\prec\GG$. By inspection
of the reverse miniswaps, the $\FF^!$ in question appears in $T$ as a result of
{\sf L4.4} or {\sf L4.5}. By definition $\bullet_{\GG}\in\x^\leftarrow$
in $T$ and we are done.

\noindent
{\sf Case 2: (This $\FF$ is in the same location in $U$):} 
By $U$'s (G.3) and (G.4), $U$ has no $\GG$ northwest of $\x$. However, $T$ has a 
$\bullet_\GG\in \z$ northwest of $\x$. Hence by
Lemma~\ref{lem:keybulletreversalfact}, $U$ has a $\bullet_{\GG^+}$ northwest of $\x$. Since $\FF\prec\GG\prec \GG^+$, this means $U$'s $\FF \in \x$ or $\underline{\x}$ is marked

By Lemma~\ref{lem:strong_form_of_G10}, $U$ has a $\bullet_{\GG^+}\in {\sf w}$ West of $\x$ and in the same row. If $T$ has $\bullet_\GG \in {\sf w}$ or ${\sf w}^\leftarrow$, we are done.
Otherwise, {\sf L1.1} applies at ${\sf w}$. Then $U$ has $\GG\in {\sf w}^\uparrow$, and by Lemma~\ref{lem:strong_form_of_G13} $U$ has a $\widetilde{\GG}\in \underline{\x}$ (possibly virtual)
in $U$ such that $\family(\GG)=\family({\widetilde \GG})$; this contradicts
$U$'s (G.12). 

\smallskip
\noindent
{\bf (G.11):}
Consider $\FF^!\in \x$ or $\FF^!\in \underline{\x}$ in $T$.

\noindent
{\sf Case 1: ($\FF$ is not in the same location in $U$):} By assumption, 
$\FF\prec\GG$. By inspection
of the miniswaps, the $\FF^!$ in question appears in $T$ as a result of
{\sf L4.4} or {\sf L4.5}. Therefore we have $\FF < \GG$ and $N_\FF = N_\GG$.

Suppose $\bullet_\GG \in \y$ in $T$, where $\y$ is in $\x$'s column. By $U$'s (G.2), $\bullet_{\GG^+} \notin \y$ in $U$. Hence by Lemma~\ref{lem:keybulletreversalfact}, $U$ has $\GG \in \y$. This contradicts Lemma~\ref{lem:how_to_check_ballotness} (applied to $U$), since $U$ has $\GG \in \y$ and $\FF \in \overline{\x^\leftarrow}$, where $\FF < \GG$ and $N_\FF = N_\GG$. 

\noindent
{\sf Case 2: ($\FF$ is in the same location in $U$):} 
By the first paragraph of the argument of (G.10) {\sf Case 2}, this $\FF$ is marked in $U$. 
Suppose $T$ has $\bullet_\GG \in \y$, where $\y$ is in $\x$'s column. By $U$'s (G.11), $U$ has no $\bullet_{\GG^+}$ in $\x$'s column. Hence by Lemma~\ref{lem:keybulletreversalfact}, $U$ has $\GG \in \y$. By $U$'s (G.4), since $\FF \prec \GG$, $\y$ is South of $\x$ and in its column.
By assumption, $T$ has some $\bullet_\GG$ northwest of $\x$. With $\bullet_\GG \in \y$, this contradicts $T$'s (G.2).

\smallskip
\noindent
{\bf (G.12):}
Let $\x, \z$ be boxes in row $r$ with $\x$ West of $\z$. Suppose $\family(\FF)= \family(\FF')$. Cases 1--3 consider the case $U$ has $\FF$ NorthWest of $\FF'$.
By $U$'s (G.12), we may assume $U$ has $\FF \in \overline{\x}$ or $\FF \in \x$ as well as $\FF' \in \z$ or $\FF' \in \underline{\z}$. By $U$'s (G.12), $U$ has a $\bullet_{\GG^+}$ in some box $\y$ of row $r$ that is East of $\x$ and west of $\z$. Cases 4--7 consider the case $U$ has $\FF$ southwest of $\FF'$. 

\noindent
{\sf Case 1: ($U$ has $\FF$ or $\circled{\FF} \in \overline{\x}$ and $\FF'$ or $\circled{\FF'} \in \underline{\z}$):} 
By $U$'s (G.4), $\FF < \lab_U(\x)$. By $U$'s (G.9), $\lab_U(\x) \prec \GG^+$. Hence $\FF < \GG^+$. Since $\family(\FF) = \family(\FF')$, also $\FF' < \GG^+$. Therefore $U$'s $\FF' \in \underline{\z}$ is marked (and is not virtual). By $U$'s (G.11), it follows $\y \neq \z$. Moreover by $U$'s Lemma~\ref{lem:strong_form_of_G10}, $\lab_U(\y^\rightarrow)$ is marked.
By $U$'s (G.4), $\GG \notin \z$. Hence $\z$ is not in a ladder, and so $T$ has $\FF' \in \underline{\z}$.  
For convenience, assume $\FF \in \overline{\x}$. (The argument where this label is virtual is strictly easier).

\noindent
{\sf Case 1.1: ($\x$ and $\y$ are \emph{not} in the same ladder of $U$):}
By Lemma~\ref{lem:ladders.sw.ne}, $\x$ is not in any ladder.
Hence $T$ has $\FF \in \overline{\x}$. 
By inspection, unless the reverse miniswap applied at $\y$ is {\sf L1.1}, $T$ has $\bullet_\GG \in \y$ or $\bullet_\GG \in \y^\leftarrow \neq \x$. Hence although $T$ has $\FF$ NorthWest of $\FF'$, there is no (G.12) violation. If the reverse miniswap is {\sf L1.1}, then $U$ has $\GG \in \y^\uparrow$.  
By Lemma~\ref{lem:same999}, $U$ has $\circled{\GG} \in \underline{\y^\rightarrow}$; with the $\GG \in \y^\uparrow$, this violates the (G.12) condition of $U$'s (V.3).

\noindent
{\sf Case 1.2: ($\x$ and $\y$ are in the same ladder of $U$):}
Then $\y = \x^\rightarrow$ and the ladder row is type {\sf L4}. 
If it is {\sf L4.1}, {\sf L4.2} or {\sf L4.3}, then $U$ has $\GG \in \y^\leftarrow$, $\bullet_{\GG^+} \in \y$ and $\GG^+ \in \underline{\y}$ with $\family(\GG^+) = \family(\GG)$. Since $\lab_U(\y^\rightarrow)$ is marked, this contradicts $U$'s (G.12).

Hence it is {\sf L4.4} or {\sf L4.5}. By Lemma~\ref{lem:strong_form_of_G13}, $\underline{\y^\rightarrow}$ contains a label $\FF''$ with $\family(\FF'') = \family(\FF)$ and a (possibly virtual) label $\GG''$ with $\family(\GG'') = \family(\GG)$. By Lemma~\ref{lem:strong_form_of_G13}, $N_{\FF''} = N_{\GG''}$. By $U$'s (G.12), there is no label of $\FF$'s family North of $\y$ and in $\y$'s column. By $U$'s (G.11), there is no label of that family South of $\y$ and in $\y$'s column. Hence there is no label of $\FF$'s family in $\y$'s column. Hence by $U$'s (G.6) and (G.7), $\FF'' = \FF^+$. By $U$'s (G.12), there is no label in $\y$'s column of the same family as $\GG$. Hence by $U$'s (G.6) and Lemma~\ref{lemma:Gsoutheast}, $\GG'' = \GG^+$. So $N_\FF = N_\GG$ and $\FF \in Z$ (in the notation of {\sf L4.4}/{\sf L4.5}). Thus $T$ has $\FF \in \underline{\y}$ or $\FF \in \y$. Therefore by $T$'s (G.3) and (G.4), $T$ has $\FF \in \underline{\y}$, and so has $\FF$ southWest of $\FF'$, in agreement with (G.12).

\noindent
{\sf Case 2: ($U$ has $\FF \in \x$ and $\FF'$ or $\circled{\FF'} \in \underline{\z}$):}
By $U$'s (G.9), $\FF \prec \GG^+$. 

\noindent
{\sf Subcase 2.1: ($\y=\z$):} By $U$'s (G.11), $(\FF')^! \notin \underline{\z}$; hence $\GG^+ \preceq \FF'$. Thus $\family(\FF) = \family(\GG^+) = \family(\GG)$. Hence by $U$'s (G.9), $\family(\lab_U(\y^\leftarrow)) = \family(\FF)$. So by $U$'s (G.6) and (G.9), $\lab_U(\y^\leftarrow) = (\FF')^-$ and $\FF' = \GG^+$. Hence by $U$'s (G.6) and (G.7), $U$ has (nonvirtual) $\GG^+ \in \underline{\y}$. The applicable reverse miniswap is then {\sf L4.1}, {\sf L4.2} or {\sf L4.3}. Hence in $T$, $\FF$ is not NorthWest of $\FF'$, and (G.12) holds.

\noindent
{\sf Subcase 2.2: ($\y \neq \z$):}
By $U$'s (G.4), $\lab_U(\z) < \FF'$. Therefore $\lab_U(\z)$ is marked; by Lemma~\ref{lem:strong_form_of_G10}, $\lab_U(\y^\rightarrow)$ is marked. 
The reverse swap does not affect $\underline{\z}$. Hence we may assume $T$ has $\FF' \in \underline{\z}$.

If the reverse swap moves the $\FF \in \x$ South, then $T$ has $\FF$ southWest of $\FF'$, in accordance with (G.12). No reverse swap can move the $\FF \in \x$ North. Hence we may assume $T$ has $\FF \in \x$ or $\FF \in \x^\rightarrow$.
If $\lab_T(\x^\rightarrow) = \FF$, then $\y = \x^\rightarrow$ and $\FF = \GG$, so $T$ has $\y \ni \GG \succeq \lab_T(\y^\rightarrow)$, contradicting $T$'s (G.3).
Thus $\lab_T(\x) = \FF$.  By Lemma~\ref{lem:strong_form_of_G13}, $U$ has a label of the same family as $\GG$ on $\underline{\y^\rightarrow}$. Hence the reverse miniswap involving $\y$ is not {\sf L1.1}, since $\GG \in \y^\uparrow$ would violate $U$'s (G.12). Hence by inspection of the reverse miniswaps, $T$ has $\bullet_\GG \in \y$ or $\bullet_\GG \in \y^\leftarrow$, in accordance with $T$'s (G.12). 

\noindent
{\sf Case 3: ($U$ has $\FF$ or $\circled{\FF} \in \overline{\x}$ and $\FF' \in \z$):}
Since $\FF' \in \z$, $\bullet_{\GG^+} \notin \z$, so $\y \neq \z$. By $U$'s (G.4), $\FF < \lab_U(\x)$ whereas
by $U$'s (G.9), $\lab_U(\x) \prec \GG^+$; hence $\FF < \GG^+$. Therefore also $\FF' < \GG$ and so $\FF' \in \z$ is marked in $U$. The box $\z$ is not part of a ladder, so $T$ has $\FF' \in \z$. No reverse swap can move move the $\FF \in \overline{\x}$ North. If it moves South, it will be southWest of $\z$ in $T$, so no (G.12) violation ensues. Hence we may assume $T$ has $\FF \in \overline{\x}$. 
By Lemma~\ref{lem:strong_form_of_G10}, $\y^\rightarrow$ contains a marked label $\EE^!$, and so by Lemma~\ref{lem:strong_form_of_G13}, $U$ has a (possibly virtual) $\GG' \in \underline{\y^\rightarrow}$ with $\family(\GG') = \family(\GG)$.

The reverse miniswap involving $\y$ is not {\sf L1.1}, for $\GG \in \y^\uparrow$ violates $U$'s (G.12), together with $\GG' \in \underline{\y^\rightarrow}$. Thus, $T$ has $\bullet_\GG \in \y$ or $\bullet_\GG \in \y^\leftarrow$. Unless $T$ has $\bullet_\GG \in \y^\leftarrow$ and $\x = \y^\leftarrow$, this does not violate $T$'s (G.12).
Suppose $\x = \y^\leftarrow$ and $T$ has $\bullet_\GG \in \x$. The reverse miniswap involving $\y$ is {\sf L4}. Hence $U$ has $\GG \in \x$. By Lemma~\ref{lem:strong_form_of_G13}, either $\family(\EE) = \family(\FF)$ or $U$ has some $\FF'' \in \underline{\y^\rightarrow}$ with $\family(\FF'') = \family(\FF)$. By Lemma~\ref{lem:strong_form_of_G13}, $N_{\GG'} = N_\EE = N_{\FF''}$. 
By $U$'s (G.11) and (G.12), $\y$'s column does not contain a label of $\FF$'s family. Hence by $U$'s (G.6) and (G.7), $\FF'' = \FF^+$.
By $U$'s (G.9) and (G.12), if $\y$'s column contains a label of the same family as $\GG$, it is on $\underline{\y}$. By $U$'s (G.6) and (G.7), it can only be $\GG^+$; this contradicts $U$'s (G.12) (last sentence).
Thus $U$ has $\GG^+ \notin \underline{\y}$, and so $\GG' = \GG^+$. The reverse miniswap is {\sf L4.4} or {\sf L4.5}. We have $N_\FF = N_\GG$. Thus $\FF \in Z$ (in the notation of {\sf L4.4} and {\sf L4.5}), contradicting that $T$ has $\FF \in \overline{\x}$.

\noindent
{\sf Case 4: ($U$ has $\FF \in \aaa$ southwest of $\FF' \in \bbb$):} Say $\bbb$ is in row $r$. No reverse miniswap can move $\FF$ North or move $\FF'$ further South than $\bbb^\downarrow$. Hence unless $\aaa$ is in row $r$, $T$ has $\FF$ southwest of $\FF'$ and no (G.12) violation ensues. So assume $\aaa$ is in row $r$.

If the reverse swap moves $\FF$, then it cannot also move $\FF'$; hence $T$ has $\FF$ southwest of $\FF'$ and no (G.12) violation ensues. Thus assume $T$ has $\FF \in \aaa$.
To violate (G.12), the reverse swap must move $\FF'$ South. So in $T$, $\FF' \in \underline{\bbb}$, $\FF' \in \underline{\bbb^\rightarrow}$ or $\FF' \in \bbb^\downarrow$.

\noindent
{\sf Subcase 4.1: ($T$ has $\FF' \in \underline{\bbb}$):}
Here $\FF' = \GG$ and the reverse miniswap involving $\bbb$ is {\sf L2.2} or {\sf L4.2}. Although $T$ has $\FF$ NorthWest of $\FF'$, $T$ has $\bullet_\GG \in \bbb$ to avoid violating (G.12). (We avoid violating the last sentence of $T$'s (G.12) by $U$'s (G.3) in the {\sf L2.2} case and by $T$'s $\bullet_\GG \in \bbb$ in the {\sf L4.2} case.)

\noindent
{\sf Subcase 4.2: ($T$ has $\FF' \in \underline{\bbb^\rightarrow}$):}
Here $U$ has $\bullet_{\GG^+} \in \bbb^\rightarrow$ and the reverse miniswap involving $\bbb^\rightarrow$ is {\sf L4.4} or {\sf L4.5}. Then $T$ has $\bullet_\GG \in \bbb$, so although $T$ has $\FF$ NorthWest of $\FF'$, they do not violate (G.12).

\noindent
{\sf Subcase 4.3: ($T$ has $\FF' \in \bbb^\downarrow$):}
Here $U$ has $\bullet_{\GG^+} \in \bbb^\downarrow$ and $\FF' = \GG$. Hence by $U$'s (G.2), $\lab_U(\aaa^\downarrow)$ is a genetic label. Since $\aaa^\downarrow$ is northWest of a $\bullet_{\GG^+}$, $\lab_U(\aaa^\downarrow)$ is not marked. Hence by $U$'s (G.4), $\FF < \lab_U(\aaa^\downarrow)$. But by $U$'s (G.9), $\lab_U(\aaa^\downarrow) \prec \GG^+$, a contradiction.

\noindent
{\sf Case 5: ($U$ has $\FF \in \underline{\aaa}$ southwest of $\FF' \in \bbb$):} Say $\bbb$ is in row $r$. No reverse miniswap can move $\FF$ further North than $\aaa$ or $\FF'$ further South than $\bbb^\downarrow$. Therefore unless $\aaa$ is in $r$, $T$ has $\FF$ southwest of $\FF'$ and no (G.12) violation ensues. Hence assume $\aaa$ is in $r$.

\noindent
{\sf Subcase 5.1: (The reverse swap moves $\FF$ North):}
Here $U$ has $\bullet_{\GG^+} \in \aaa$; the reverse miniswap involving $\aaa$ is {\sf L4.1}, {\sf L4.2} or {\sf L4.3}; and $T$ has $\FF \in \aaa$. By $U$'s (G.2), $\FF'$ takes part in no reverse miniswap, so $T$ has $\FF' \in \bbb$. Hence $T$ has $\FF$ southwest of $\FF'$ and no (G.12) violation ensues.

\noindent
{\sf Subcase 5.2: ($T$ has $\FF \in \underline{\aaa}$):}
To have $\FF$ NorthWest of $\FF'$ in $T$, the reverse swap must move $\FF'$ to $\bbb^\downarrow$. Hence $U$ has $\bullet_{\GG^+} \in \bbb^\downarrow$ and $\FF' = \GG$. Since $U$ has $\FF \in \underline{\aaa}$ and $\bullet_{\GG^+} \in \bbb^\downarrow$, it has a label in $\aaa^\downarrow$. By $U$'s (G.2), it is a genetic label $\HH$. By $U$'s (G.4), $\FF < \HH$. Hence since $\family(\FF) = \family(\GG)$, $\GG < \HH$. But by $U$'s (G.9), $\HH \prec \GG^+$, a contradiction.

\noindent
{\sf Case 6: ($U$ has $\FF \in \aaa$ southwest of $\FF' \in \underline{\bbb}$):}
Say $\aaa$ is in row $r$.
No swap can move $\FF$ North. No swap can move $\FF'$ further South than $\underline{\bbb^\downarrow}$. Hence unless $\bbb \in r^\uparrow$, $T$ has $\FF$ southwest of $\FF'$ and no (G.12) violation ensues.
Thus, assume $\bbb \in r^\uparrow$. 
To obtain a (G.12) violation, $\FF'$ must move South to $\underline{\bbb^\downarrow}$ or $\underline{\bbb^{\downarrow\rightarrow}}$.

\noindent
{\sf Subcase 6.1: ($T$ has $\FF' \in \underline{\bbb^\downarrow}$):}
Here $U$ has $\bullet_{\GG^+} \in \bbb^\downarrow$, $T$ has $\FF \in \aaa$ and $\FF' = \GG$. So $T$ has $\bullet_\GG \in \bbb^\downarrow$. It remains to show we do not violate the last sentence of (G.12). By $U$'s (G.12), $\family(\lab_U(\bbb^{\downarrow\rightarrow})) \neq \family(\GG)$; hence the same in true in $T$. Hence if $\lab_T(\bbb^{\downarrow\rightarrow})$ is marked, then $\lab_U(\bbb^{\downarrow\rightarrow})$ is marked. Then by Lemma~\ref{lem:strong_form_of_G13}, $U$ has a label on $\underline{\bbb^{\downarrow\rightarrow}}$ of the same family as $\GG$, contradicting $U$'s (G.12). Thus (G.12) is confirmed in $T$.

\noindent
{\sf Subcase 6.2: ($T$ has $\FF' \in \underline{\bbb^{\downarrow\rightarrow}}$):}
Here $U$ has $\bullet_{\GG^+} \in \bbb^{\downarrow\rightarrow}$ and $\GG \in \bbb^\downarrow$, while $T$ has $\FF \in \aaa$ and $\bullet_\GG \in \bbb^\downarrow$. Thus, although $T$ has $\FF$ NorthWest of $\FF'$, they do not violate (G.12).

\noindent
{\sf Case 7: ($U$ has $\FF \in \underline{\aaa}$ southwest of $\FF' \in \underline{\bbb}$):}
Say $\bbb$ is in row $r$. No swap can move $\FF$ further North than $\aaa$. No swap can move $\FF'$ further South than $\underline{\bbb^\downarrow}$. Hence if $\aaa$ is South of row $r^\downarrow$, then $T$ has $\FF$ southwest of $\FF'$ and no (G.12) violation ensues.

\noindent
{\sf Case 7.1: ($\aaa$ is in row $r^\downarrow$):}
$T$ has $\FF$ southwest of $\FF'$, unless $T$ has both $\FF \in \aaa$ and $\FF'$ in $\underline{\bbb^\downarrow}$ or $\FF' \in \underline{\bbb^{\downarrow\rightarrow}}$. Suppose these both occur. Then $U$ has $\bullet_{\GG^+} \in \aaa$ and either $\bullet_{\GG^+} \in \bbb^\downarrow$ or $\bullet_{\GG^+} \in \bbb^{\downarrow\rightarrow}$. Since by $U$'s (G.4), $\aaa$ is West of $\bbb$, this contradicts $U$'s (G.2).

\noindent
{\sf Case 7.2: ($\aaa$ is in row $r$):}
Suppose $T$ has both $\FF \in \aaa$ and $\FF'$ in $\underline{\bbb^\downarrow}$, $\FF' \in \bbb^{\downarrow\rightarrow}$ or $\FF' \in \underline{\bbb^{\downarrow\rightarrow}}$. Then $U$ has $\bullet_{\GG^+} \in \aaa$ and either $\bullet_{\GG^+} \in \bbb^\downarrow$ or $\bullet_{\GG^+} \in \bbb^{\downarrow\rightarrow}$, contradicting $U$'s (G.2). Hence the reverse swap cannot both move $\FF$ North and move $\FF'$ South.

Suppose $T$ has $\FF \in \aaa$. Then $T$ has $\FF' \in \underline{\bbb}$ and $U$ has $\bullet_{\GG^+} \in \aaa$. Then $\FF = \GG^+$ and $\family(\FF) = \family(\GG) = \family(\GG^+)$. The reverse miniswap involving $\aaa$ is {\sf L4.1}, {\sf L4.2} or {\sf L4.3}. Hence $U$ has $\GG \in \aaa^\leftarrow$. By $U$'s (G.4), $\lab_U(\bbb) < \FF'$, so it is marked. By Lemma~\ref{lem:strong_form_of_G10}, $\lab_U(\aaa^\rightarrow)$ is marked. This contradicts the last sentence of $U$'s (G.12).

Suppose $T$ has $\FF'$ in $\underline{\bbb^\downarrow}$, $\bbb^{\downarrow\rightarrow}$ or $\underline{\bbb^{\downarrow\rightarrow}}$. Then $T$ has $\FF \in \underline{\aaa}$  and $\bullet_\GG \in \bbb^\downarrow$, while $U$ has $\bullet_{\GG^+} \in \bbb^\downarrow$ or $\bbb^{\downarrow\rightarrow}$. Hence although $T$ has $\FF$ NorthWest of $\FF'$, they do not violate (G.12).

\smallskip
\noindent
{\bf Consistency of the prescribed virtual labels:}
\begin{lemma}\label{lem:virtual_label_validity_B}
Let $U$ be a $\GG^+$-good tableau in which we apply {\sf L2.3}, {\sf L4.3} or {\sf L4.5}, and let $T \in \revswap_{\GG^+}(U)$. Then all prescribed $\circled{\GG}$'s from the outputs of {\sf L2.3}, {\sf L4.3} and {\sf L4.5} are valid virtual labels in the sense of {\normalfont (V.1)--(V.3)}.
\end{lemma}
\begin{proof}
Consider such a reverse miniswap. We may assume $T$ contains an output with a prescribed $\circled{\GG}$, say in $\underline{\x}$. Let $T^\star$ be $T$ with $\GG$ added in $\underline{\x}$.

((V.1) holds, i.e., $\GG \in \underline{\x}$ is not marked in $T^\star$): Every $\bullet$ in $T^\star$ is $\bullet_\GG$, so by definition no $\GG$ in $T^\star$ is marked.

((V.2) holds, i.e., $T$ has a $\GG$ West of $\x$): By assumption of the reverse miniswaps, $U$ has some $\GG$ West of $\x$. Hence by Lemma~\ref{lem:keyreversalfact}, $T$ also has a $\GG$ West of $\x$. 

((G.1) holds in $T^\star$): Immediate from $U$'s (G.1).

((G.4) holds in $T^\star$): Let $\HH = \lab_T(\x^\downarrow)$. We must show $\GG < \HH$.  By Lemma~\ref{lem:ladders.sw.ne}, $\x^\downarrow$ is not in a ladder of $U$, so $\HH = \lab_U(\x^\downarrow)$. Hence by $U$'s (G.4), we get $\GG < \HH$ in the case of {\sf L2.3} and {\sf L4.3}. In the case of {\sf L4.5}, by $U$'s (G.11), $\family(\HH) \geq \family(\GG)$. But by $U$'s (G.12), $\family(\HH) \neq \family(\GG)$, so $\GG < \HH$.

Let $\FF = \lab_T(\x^\uparrow)$. We must show $\FF < \GG$ in the {\sf L2.3} and {\sf L4.3} cases. (In the {\sf L4.5} case, there is nothing to confirm here.) Since $\x^\uparrow$ is not part of a ladder in $U$, $\FF = \lab_U(\x^\uparrow)$ as well. Hence $\FF < \GG$ follows from $U$'s (G.4).

((G.5) holds in $T^\star$): Suppose not. Then there is some $\GG' \in \underline{\x}$ in $T$ with $\family(\GG') = \family(\GG)$. This $\GG'$ is not the result of any reverse miniswap, so $\GG' \in \underline{\x}$ in $U$. In the case of {\sf L2.3} or {\sf L4.3}, this contradicts $U$'s (G.4). In the case of {\sf L4.5}, note that by $U$'s (G.6) and (G.7), $\GG' = \GG^+$; this contradicts the assumption of {\sf L4.5}. 

((G.6) holds in $T^\star$): This follows from Lemma~\ref{lem:keyreversalfact}, as in the proof of $T$'s (G.6) above.

((G.8) holds in $T^\star$): In the case of {\sf L4.5}, this is immediate from the assumption that $N_\GG = N_\EE$ for every $\EE \in Z$. In the case of {\sf L2.3}, consider the tableau $\widetilde{T}$ differing from $T$ only in the box $\x$, where we choose the other output of {\sf L2.3}. By $\widetilde{T}$'s (G.8), $\widetilde{T}$ is ballot. But $\widetilde{T}$ and $T^\star$ have the same column reading words. Hence $T^\star$ is ballot.

Finally consider the case of {\sf L4.3}. By $T$'s (G.8), $T$ is ballot. Hence if there is a genotype $G$ of $T^\star$ that is not ballot, $G$ uses $\GG \in \underline{\x}$. Let $\FF$ be the gene with $\family(\FF) = \family(\GG) - 1$ and $N_\FF = N_\GG$. Since $G$ is not ballot, $\FF$ appears after $\GG$ in ${\tt word}(G)$. Say $\FF$ appears in column $c$ in $G$. By inspection of the reverse miniswaps, $\FF$ appears in $U$ either in column $c$ or column $c^\leftarrow$. Thus considering this $\FF$ and $U$'s $\GG \in \x$, we contradict Lemma~\ref{lem:how_to_check_ballotness} for $U$.

((G.9) holds in $T^\star$): $\underline{\x}$ is southeast of a $\bullet_\GG$ in $T$. Hence by $T$'s (G.2), $\underline{\x}$ is not northwest of a $\bullet_\GG$ in $T$, and this (G.9) verification is vacuous.

((G.12) holds in $T^\star$): For {\sf L2.3}, consider the tableau $\widetilde{T}$ differing from $T$ only in the box $\x$, where we choose the other output of {\sf L2.3}. By $\widetilde{T}$'s (G.12), any label $\GG'$ of $\widetilde{T}$ with $\family(\GG') = \family(\GG)$ that is West of $\x$ must be no further North than the upper edge of $\x$'s row. Since $T^\star$ and $\widetilde{T}$ are identical outside of $\x$, this is also true of $T^\star$. Since $T^\star$ has $\bullet_\GG \in \x$ and $\GG \notin \x^\leftarrow$, $T^\star$'s (G.12) then follows.

For {\sf L4.3}, observe that in light of the $\GG^+ \in \x^\rightarrow$ in $T$, it follows from $T$'s (G.12) that any label $\GG'$ of $T$ with $\family(\GG') = \family(\GG)$ that is West of $\x$ must be no further North than the upper edge of $\x$'s row. Since $T^\star$ has $\bullet_\GG \in \x$ and $\GG \notin \x^\leftarrow$, $T^\star$'s (G.12) then follows.

The {\sf L4.5} case is proved by repeating the arguments above for $T$'s (G.12) ({\sf Cases 3--5}).
\end{proof}

\smallskip
\noindent
{\bf (G.13):} For every marked label $\EE^!$ in $T$, $U$ has $\EE$ or $\EE^!$ in the same position. Thus our analysis splits into two cases:

\noindent
{\sf Case 1: ($U$ has $\EE^! \in \bbb$ or $\underline{\bbb}$):}
Let $\ell$ be this instance of $\EE^!$ and let $\bbb$ be in column $c$. $U$ also has some $\FF$ or $\circled{\FF} \in \underline{\bbb}$ with $N_\EE = N_\FF$ and $\family(\FF) = \family(\EE) + 1$. Since $\ell$ is marked, $\EE \preceq \GG$.

\noindent
{\sf Subcase 1.1: ($\family(\EE) \leq \family(\GG) - 2$):}
Such a marked label is not moved by any reverse miniswap. Hence, $\ell$ is in the same position in $U$ and $T$.
By Lemma~\ref{lem:strong_form_of_G13}, $U$ has $\FF \in \underline{\bbb}$ (rather than $\circled{\FF}$). This $\FF$ is also not  moved by any reverse miniswap, so $T$ has $\FF \in \underline{\bbb}$ as well, and (G.13) is satisfied.

\noindent
{\sf Subcase 1.2: ($\family(\EE) = \family(\GG) - 1$):} Here, $\family(\FF) = \family(\GG)$. As in {\sf Subcase~1.1}, $\ell$ is in the same position in $U$ and $T$. Suppose $U$ has $\FF \in \underline{\bbb}$ (rather than $\circled{\FF}$). By $U$'s (G.11) $\bullet_{\GG^+} \notin c$ in $U$. Hence there is no reverse miniswap that affects $\FF \in \underline{\bbb}$. So $T$ has $\FF \in \underline{\bbb}$ and (G.13) holds.

Hence, assume $\circled{\FF} \in \underline{\bbb}$ in $U$. By Lemma~\ref{lem:virtualG13label}, $\family(\FF) = \family(\GG^+)$ and $U$ has $\bullet_{\GG^+} \in \bbb^\leftarrow$. Let $T^\star$ be $T$ with $\FF \in \bbb$. By $U$'s (V.1), $\GG \prec \FF$. Hence $T^\star$ satisfies (V.1). By the proof of Lemma~\ref{lemma:reversecontentpres}, it satisfies (V.2).  It remains to show that $T^\star$ satisfies (G.1), (G.4)--(G.6), (G.8), (G.9) and (G.12).

(G.1): Since $U$ has $\circled{\FF} \in \underline{\bbb}$, this follows from $T$'s (G.1).

(G.4): By $T$'s (G.4), since $\EE < \FF$, $T$ has $\QQ < \FF$ for any $\QQ$ north of $\bbb$ in $c$. By $U$'s (G.4), since no reverse miniswap affects labels of family greater than $\family(\FF) = \family(\GG)$, we also have $\QQ > \FF$ for any $\QQ$ appearing south of $\bbb$ in $c$ in $T$. Hence (G.4) holds in $T^\star$.

(G.5): By $U$'s (G.5), $U$ has no label of family $\family(\FF)$ on $\underline{\bbb}$. Indeed by $U$'s (G.4), $U$ has no label of that family in $c$. By $U$'s (G.11), $U$ has $\bullet_{\GG^+} \notin c$. Hence no ladder of $U$ intersects $c$ and $T$ has no label of family $\family(\FF)$ on $\underline{\bbb}$. Hence (G.5) holds in $T^\star$.

(G.6): As no ladder of $U$ intersects $c$, the genes West of $c$ in $U$ are exactly the genes West of $c$ in $T$, and the genes East of $c$ in $T$ are a subset of those East of $c$ in $U$. Hence (G.6) holds in $T^\star$.

(G.8): As no ladder of $U$ intersects $c$, $c$ is the same in $U$ and $T$. Since $N_\EE = N_\FF$ and $\family(\FF) = \family(\EE) + 1$, it suffices to check no $\EE$ is read in $T^\star$ after $\FF \in \underline{\bbb}$.
By $U$'s (G.4), there is no $\EE$ in $U$ South of $\underline{\bbb}$ in $c$. Thus, this is also true in $T$. By Lemma~\ref{lemma:markedHiswestmost}, $\ell$ is the westmost $\EE$ in $U$. But the genes that appear West of $c$ in $U$ are exactly the genes that appear West of $c$ in $T$. Hence no $\EE$ appears West of $\bbb$ in $T$, no $\EE$ in read in $T^\star$ after $\underline{\bbb}$ and $T^\star$'s (G.8) holds.

(G.9): Since $U$ has $\bullet_{\GG^+} \in \bbb^\leftarrow$, by $U$'s (G.2), $\bbb$ is not northwest of a $\bullet_{\GG^+}$. Hence by inspection of the reverse miniswaps, $\bbb$ is not northwest of a $\bullet_\GG$ in $T$, so (G.9) holds in $T^\star$.

(G.12): Suppose $i$ is a label of $\FF$'s family appearing in $\aaa$ or $\underline{\aaa}$ NorthWest of $\underline{\bbb}$ in $T$. First suppose $i$ does not appear in the same position in $U$. Then $i$ was involved in a reverse miniswap. If it was {\sf L1.1}, then $\GG \in \aaa^\uparrow$ in $U$, contradicting $U$'s (G.12). It obviously was not {\sf L1.2}. If it was an {\sf L2} miniswap, $i \in \aaa$ in $U$, contradicting $U$'s (G.12). The same holds for {\sf L3}. For {\sf L4.1--3} to apply, $\aaa = \bbb^\leftarrow$ by $U$'s (G.2), but this contradicts $U$'s (G.12) (last sentence). In {\sf L4.4} and {\sf L4.5}, $Z \neq \emptyset$ and the only labels of concern are the $\GG$'s. Since they satisfy (G.12) in $U$, they do in $T^\star$.

Now, suppose that $i$ is in the same position in $U$ and $T$. Hence $i$ is NorthWest of $\underline{\bbb}$ in $U$. Let $\bbb$ be in row $r$. By $U$'s (G.12), either $U$ has $i \in \aaa$ and $\aaa \in r$, or else $U$ has $i \in \underline{\aaa}$ and $\aaa \in r^\uparrow$. Also $\aaa$ is West of $\bbb^\leftarrow$. Since the labels in question do not move, it remains to check that $\bullet_\GG$ appears in row $r$ in $T$ East of $\aaa$ and west of $\bbb$. We are clearly only concerned when the reverse miniswap involving $\bbb^\leftarrow$ is {\sf L1.1}, or is {\sf L4} with $\bbb^\leftarrow = \aaa^\rightarrow$. If it is {\sf L1.1}, $U$ has $\GG \in \bbb^{\leftarrow\uparrow}$, contradicting $U$'s (G.12). If it is {\sf L4} with $\bbb^\leftarrow = \aaa^\rightarrow$, then $i$ would not appear in the same position in $T$, contradicting our assumption.

Now, suppose $i$ is a label of $\FF$'s family appearing in $\aaa$ or $\underline{\aaa}$ SouthEast of $\underline{\bbb}$ in $T$. First, suppose that $i$ does not appear in the same position in $U$. Then $i$ was involved in a reverse miniswap. By $U$'s (G.2), this can only be a {\sf L2} reverse miniswap. It is obviously not {\sf L2.1}. If it is {\sf L2.2} or {\sf L2.3}, then by $U$'s (G.12), the $\GG$ or $\circled{\GG} \in \underline{\aaa}$ is not a (G.12) violation in $T$ because of $T$'s $\bullet_\GG \in \aaa$.

Otherwise $i$ is in the same position in $U$ and $T$. Hence $i$ is SouthEast of $\underline{\bbb}$ in $U$. By $U$'s (G.12), $U$ has a $\bullet_{\GG^+}$ SouthEast of $\bbb$. Given $U$'s $\bullet_{\GG^+} \in \bbb^\leftarrow$, this contradicts $U$'s (G.2).

We conclude that the desired $\circled{\FF}$ appears on $\underline{\bbb}$ in $T$.

\noindent
{\sf Subcase 1.3: ($\family(\EE) = \family(\GG)$):}
Suppose $\ell$ is moved by the reverse swap. Recall $\EE \prec \GG^+$. By inspection, no reverse swap will move such a $\EE \prec \GG^+$ with $\family(\EE) = \family(\GG)$ unless $\EE = \GG$. Since the $\bullet$'s in $T$ are $\bullet_\GG$'s, no $\EE$ will be marked in $T$, so $T$'s (G.13) check is vacuous.

Hence assume that $\ell$ is unaffected by $\revswap_{\GG^+}$ and indeed that $\EE \prec \GG$. Since $\family(\FF) = \family(\GG) + 1$, no reverse miniswap affects any instance of $\FF$. In particular, if $U$ has $\FF \in \underline{\bbb}$ (instead of $\circled{\FF}$), then $T$ will also have $\FF \in \underline{\bbb}$ and satisfy (G.13).

Hence further assume that $U$ has $\circled{\FF} \in \underline{\bbb}$. We need $\circled{\FF} \in \underline{\bbb}$ in $T$. By Lemma~\ref{lem:virtualG13label}, $U$ has $\bullet_{\GG^+} \in \bbb^\leftarrow$ and $\family(\FF) = \family(\GG^+)$. Let $c$ be $\bbb$'s column. By $U$'s (G.11), $U$ has $\bullet_{\GG^+} \notin c$.  By $U$'s (G.4), since $\EE \prec \GG$, $U$ has $\GG \notin c$. Hence no ladder of $U$ intersects $c$ and so column $c$ in $T$ is identical to column $c$ in $U$. 

Let $T^\star, U^\star$ be $T, U$ respectively with $\FF \in \underline{\bbb}$ added. We show $T^\star$ satisfies (V.1)--(V.3).

(V.1): Since $\FF > \GG$, this is obvious.

(V.2): By $U$'s (V.2), $U$ has an $\FF$ West of $c$. Hence there is an $\FF$ West of $c$ in $T$, as needed.

(V.3): We show $T^\star$ satisfies (G.1), (G.4), (G.5), (G.6), (G.8), (G.9) and (G.12).
We know $T$ satisfies these. 

(G.1): Immediate from $T$'s (G.1), given $U$'s $\circled{\FF} \in \underline{\bbb}$.

(G.4) and (G.5): These hold in $T^\star$ since they hold in $U^\star$, and column $c$ of $T$ is identical to column $c$ of $U$. 

(G.6): The genes West of $c$ in $U$ are exactly the genes West of $c$ in $T$, and the genes East of $c$ in $T$ are a subset of those East of $c$ in $U$. Now $T^\star$'s (G.6) follows.

(G.8): 
Immediate from $U$'s $\circled{\FF} \in \underline{\bbb}$ and the facts that 
\begin{itemize}
\item the genes West of $c$ in $T$ are exactly the genes West of $c$ in $U$,
\item the genes East of $c$ in $T$ are a subset of those East of $c$ in $U$, and 
\item column $c$ in $T$ is identical to column $c$ in $U$.
\end{itemize}

(G.9): Since $\ell$ is marked in $T$, $T$ has a $\bullet_\GG$ northwest of $\underline{\bbb}$. Hence by $T$'s (G.2), $\underline{\bbb}$ is not northwest of a $\bullet_\GG$ in $T^\star$, so this condition is vacuous.

(G.12): Take $\FF'$ with $\family(\FF') = \family(\FF)$. Suppose $\FF'$ is NorthWest of $\underline{\bbb}$ in $T^\star$. Since $\family(\FF') > \family(\GG)$, $\FF'$ appears in the same positions in both $T^\star$ and $U^\star$. Hence $\FF'$ is NorthWest of $\underline{\bbb}$ in $U'$. But then this $\FF'$ is northwest of $U^\star$'s $\bullet_{\GG^+} \in \bbb^\leftarrow$, contradicting $U^\star$'s (G.9). Thus $T^\star$ has no such $\FF'$ NorthWest of $\underline{\bbb}$. Similarly $T^\star$ has no $\FF'$ SouthEast of $\underline{\bbb}$.

\noindent
{\sf Case 2: ($\ell$ is a marked label in $T$ that is not marked in $U$):}
Suppose $\ell$ is an instance of $\EE^!$ on $\bbb$ or $\underline{\bbb}$ in $T$.  Since $\ell$ is marked and every bullet in $T$ is $\bullet_\GG$, $\EE \prec \GG$. Hence $\EE \prec \GG^+$, and any instance of $\EE$ southeast of a $\bullet_{\GG^+}$ is marked in $U$.

\noindent
{\sf Case 2.1: ($\family(\EE) = \family(\GG)$):}
No reverse miniswap affects any instance of $\EE$. Hence $\ell$ is in the same position in $U$ as in $T$. Since $\ell$ is unmarked in $U$, $U$ has no $\bullet_{\GG^+}$ northwest of $\underline{\bbb}$. Hence, since $U$ has $\EE \in \bbb$ or $\underline{\bbb}$, by $U$'s (G.3) and (G.4), $U$ has no $\GG$ northwest of $\underline{\bbb}$. But there is $\bullet_\GG$ northwest of $\underline{\bbb}$ in $T$; this contradicts Lemma~\ref{lem:keybulletreversalfact}.

\noindent
{\sf Case 2.2: ($\family(\EE) < \family(\GG)$):}
If $\ell$ is not moved by $\revswap_{\GG^+}$, we obtain a contradiction exactly as in {\sf Case 2.1}.
Otherwise, it is moved by a {\sf L4.4} or {\sf L4.5} reverse miniswap. Then by definition $N_\EE = N_\GG$. 
Let $\FF$ be the gene (which must exist) with $N_\FF = N_\EE = N_\GG$ and $\family(\FF) = \family(\EE) + 1$. By Lemma~\ref{lem:draggable_labels_in_complete_sets}, $\FF$ appears in $\bbb^\leftarrow$'s column in $U$. Hence by $U$'s (G.4), $U$ has $\FF \in \bbb^\leftarrow$ or $\overline{\bbb^\leftarrow}$. It is then in the set $A''$ or $A''' \cup \{\GG\}$ (in the notation of {\sf L4.4}/{\sf L4.5}) and appears on $\underline{\bbb}$ in $T$ (possibly virtual), as desired.
\qed

\section{Block decomposition; completion of proof of Proposition~\ref{thm:weight_preservation}(II)}
\label{sec:blockdecomp}

Below, we define, for each
$S\in {\tt Snakes}_\GG$, a set ${\mathcal B}_S$ containing $S$. Clearly
$\cup_{S\in {\tt Snakes}_\GG} {\mathcal B}_S={\tt Snakes}_\GG$. Along the way, we will argue that 
if $S,S'\in {\tt Snakes}_\GG$ and $S'\in {\mathcal B}_S$, then
${\mathcal B}_S={\mathcal B}_{S'}$. This proves (D.1). 

Recall $\Gamma_i:=\{T\in P_\GG\colon \text{$T$ contains a snake from ${\mathcal B}_i$}\}$.
From the construction, the following two additional conditions will also be essentially clear:
\begin{itemize}
\item[(D.3)] Suppose  $A,B\in \Gamma_i$. For each snake in ${\mathcal B}_i$ 
and $A$, there is a snake of ${\mathcal B}_i$ in $B$ and 
in the exact same location.
\item[(D.4)] $A, B \in \Gamma_i$ are identical outside of the snakes in $\mathcal{B}_i$.
\end{itemize}

The bulk of the work is to establish (D.2). This will be done simultaneously with the description of each
${\mathcal B}_S$. To establish (D.2), we must verify (\ref{eqn:abc}) by considering the {\tt boxfactor}s, {\tt edgefactor}s
and {\tt virtualfactor}s from every box and edge of the common shape $\nu/\lambda$.
Except where otherwise noted, by inspection, these factors do not change for boxes/edges not in ${\mathcal B}_S$. 
Thus, the majority of our discussion concerns
the region defined by ${\mathcal B}_S$. For simplicity, we assume $\mathcal{B}_J = \emptyset$. The modifications for the general case are straightforward, using (D.3) and (D.4).

Assume $\GG := i_k$. Let $S\in {\tt Snakes}_\GG$ be in the tableau $U$. We break into cases according to the type of $\head(S)$. We write $T_1, T_2, \dots$ for the fine tableau in $\swap_S(U)$, in the order illustrated in each case. We write $U_j$ for $T_j$ together with its coefficient in $\swap_S(U)$.

\noindent
{\sf Case 0: ($\head(S) = \emptyset$):} By Definition-Lemma~\ref{def-lem:snake_classification}, either the southmost row of $S$ contains a single box, or else it consists of two boxes $\x, \x^\rightarrow$ with $\bullet \in \x$ and a marked label in $\x^\rightarrow$. Thus, $\body(S)$ is either empty, or it falls under case {\sf B1} or {\sf B3}.

\noindent
{\sf Subcase 0.1: ($\body(S)$ is {\sf B1}):} Let ${\mathcal B}_S=\{S\}$.
Since $S$ contains no $\bullet_\GG$, ${\tt swapset}_{\mathcal{B}_S}(U) = U$. Thus $\wt({\tt swapset}_{\mathcal{B}_s}(U)) = \wt(U)$, which implies (\ref{eqn:abc}).

\noindent
{\sf Subcase 0.2: ($\body(S)$ is {\sf B3} or $\body(S) = \emptyset$):}
By  Definition-Lemma~\ref{def-lem:snake_classification}(III) either $S$ has at least two rows, or else $S = \tail(S)$. In either case, $\tail(S) \neq \emptyset$.

\noindent
{\sf Subcase 0.2.1: ($\tail(S)$ is {\sf T1}):}
Let ${\mathcal B}_{S}=\{S\}$. Locally at the snake $S$, this swap looks like
\ytableausetup{boxsize=1em}
$\ytableaushort{ \none \bullet,  \bullet \GG,  \GG}\mapsto \prod_{\x \ni \bullet_\GG} \hat{\beta}(\x) \cdot\ytableaushort{ \none \GG,  \GG \bullet, \bullet}$. Note that $\body(S)$ is nonempty in this case.
This swap does not affect the locations or weights of edge labels or virtual labels in $U$. Hence ${\tt edgewt}(U) = {\tt edgewt}(T_1)$ and ${\tt virtualwt}(U) = {\tt virtualwt}(T_1)$. One checks that a box outside $S$ is productive in $U$ if and only if it is productive in $T_1$. (The critical checks are for the box immediately east of the northmost box of $S$ and the box immediately west of the southmost box of $S$.) Also, each such productive box has the same ${\tt boxfactor}$ in $U$ and $T_1$. The productive boxes of 
$S_1$ are the boxes $\{\x\}$ containing $\GG$, while in $S$ they are the boxes $\{\x^\downarrow \}$ containing $\GG$.  
For each productive box $\x$ of $S_1$ with ${\tt boxfactor}(\x) := w_\x$, 
there is a corresponding productive box $\x^\downarrow$ in $S$ with ${\tt boxfactor}(\x^\downarrow) = \hat{\beta}(\x)  w_\x $. 

Thus $\wt U = \wt U_1$ follows from 
\[ (-1)^{d(U)} \!\!\!\! \prod_{\x : \lab_{U}(\x) = \bullet_{\GG}} \!\!\!\!  \hat{\beta}(\x)  w_\x  =   \left( \prod_{\x : \lab_{U}(\x) = \bullet_{\GG}} \!\!\!\!  \hat{\beta}(\x) \right) \cdot (-1)^{d(U)}  \!\!\!\! \prod_{\x : \lab_{U}(\x) = \bullet_{\GG}} \!\!\!\!   w_\x.\] 

\noindent
{\sf Subcase 0.2.2: ($\tail(S)$ is {\sf T2}):}
Let ${\mathcal B}_{S}=\{S\}$.
This case is similar to {\sf Subcase 0.2.1}; we have something like
\ytableausetup{boxsize=1em}
\[\ytableaushort{ \none \bullet \GG,  \bullet \GG,  \GG}\mapsto  - \prod_{\x \ni \bullet} \hat{\beta}(\x) \cdot\ytableaushort{ \none \GG \bullet,  \GG \bullet, \bullet}.\]
This swap preserves locations and weights of edge labels. Hence ${\tt edgewt}(U) = {\tt edgewt}(T_1)$. 
A box outside $S$ is productive in $U$ if and only if it is productive in $T_1$. Also, each such 
productive box has the same ${\tt boxfactor}$ in $U$ and $T_1$. Let $\y$ be the box containing the 
Northmost $\GG$ in $S$. In $S_1$, the productive boxes are the boxes $\{ \x\}$ containing $\GG$. 
The productive boxes of $S$ are the boxes $\{ \x^\downarrow \}$ containing $\GG$ in all but the northmost row of $S$, and $\y$ if $\GG^+ \not\in \y^\rightarrow$ with $\family(\GG) = \family(\GG^+)$. 
For each productive box $\x$ in $S_1$ with ${\tt boxfactor}(\x) := w_\x$, there is a corresponding productive $\x^\downarrow$ of $S$ with ${\tt boxfactor}(\x^\downarrow) = \hat{\beta}(\x) w_\x$. 

The box $\y$ is productive in $S$ with ${\tt boxfactor}(\y) := y$ if and only if $\circled{\GG} \in \overline{\y}$ in $S_1$ with ${\tt virtualfactor}_\GG(\overline{\y}) = y$. The swap does not otherwise affect the location or weight contribution of virtual labels. Finally note $(-1)^{d(U)} = (-1)^{d(T_1) - 1}$. If $\y$ is productive in $S$, then $\wt U = \wt U_1$ follows from
\[ (-1)^{d(U)} \cdot y \!\!\!\! \prod_{\x : \lab_{U}(\x) = \bullet_{\GG}} \!\!\!\!  \hat{\beta}(\x)  w_\x  =  \left( -\prod_{\x : \lab_{U}(\x) = \bullet_{\GG}} \!\!\!\!  \hat{\beta}(\x) \right) \cdot  (-1)^{d(U) - 1} \cdot y \cdot \!\!\!\! \prod_{\x : \lab_{U}(\x) = \bullet_{\GG}} \!\!\!\!   w_\x.\] If $\y$ is not productive in $S$, we use the same identity without $y$.

\noindent
{\sf Subcase 0.2.3: ($\tail(S)$ is {\sf T3}):}
This is again similar to {\sf Subcase 0.2.1}. Locally at $S$, we have something like
\ytableausetup{boxsize=1.3em}
\[\begin{picture}(350,43)
\put(0,30){$\ytableaushort{\none \bullet {\GG^+}, \bullet \GG, \GG} \mapsto \prod_{\x \ni \bullet} \hat{\beta}(\x) \cdot\ytableaushort{\none \GG {\GG^+}, \GG \bullet, \bullet} - \alpha \cdot \ytableaushort{\none \GG \bullet, \GG \bullet, \bullet}.$}
\put(238,25){$\GG^+$}
\end{picture}
\]
By Lemma~\ref{lem:piecesobservations}(V), $S$ has at least two rows. Let $\y$ be the box containing $\GG^+$ in $S$. Here $\alpha := \prod_{\x \ni \bullet} \hat{\beta}(\x)$ if $\bullet \notin \y^\uparrow$ and $\alpha := 0$ otherwise.

The locations and weights of virtual labels are unaffected by the swap. Therefore, 
${\tt virtualwt}(U) = {\tt virtualwt}(T_1) = {\tt virtualwt}(T_2)$. Furthermore the 
{\tt edgefactor}s and {\tt boxfactor}s from labels outside of $S$ are the same in each of 
$U, T_1, T_2$, so we restrict attention to the {\tt boxfactor}s and {\tt edgefactor}s from labels inside $S$.

\noindent
{\sf Subcase 0.2.3.1: ($\bullet_\GG \not\in\y^\uparrow$)}:
Let ${\mathcal B}_{S}=\{S\}$.
The productive boxes of $S_1$ and $S_2$ are the boxes $\{ \x\}$ containing $\GG$ and not in the northmost row, and possibly also $\y$ (depending on what label, if any, appears in $\y^\rightarrow$). 
The productive boxes of $S$ are those boxes $\{ \x^\downarrow \}$ containing $\GG$ not in the second row from the top, the box $\z$ containing $\GG$ in the second row from the top, and possibly also $\y$.
One sees $\y$ is productive in any of $S, S_1, S_2$ if and only if it is productive in all of them. 
Further, if it is productive, then it has the same {\tt boxfactor}, say $a$, in each one.

For each productive box $\x$ in each of $S_1, S_2$, with ${\tt boxfactor}(\x) := w_\x$, there is a corresponding productive box $\x^\downarrow$ in $S$, with ${\tt boxfactor}(\x^\downarrow) = \hat{\beta}(\x) w_\x$. 
Let ${\tt edgefactor}_{\underline{\z^\rightarrow} \in T_2}(\GG^+) := 1 - b$. Then ${\tt boxfactor}_U(\z) = \hat{\beta}(\z^\uparrow) b$.

Now (\ref{eqn:abc}) is the statement 
$\wt U = \wt( U_1 + U_2)$. If $\y$ is productive in $U$, this follows from the identity
\[ a \hat{\beta}(\z^\uparrow) b \prod_{\x} \hat{\beta}(\x) w_\x = \left( \hat{\beta}(\z^\uparrow)  \prod_{\x}  \hat{\beta}(\x) \right) \cdot \left[  a \prod_{\x}    w_\x - (1-b) \cdot a  \prod_{\x}    w_\x  \right] \] 
Otherwise we use the same identity without $a$.

\noindent
{\sf Subcase 0.2.3.2: ($\bullet_\GG \in \y^\uparrow$)}: Here $\alpha=0$, so we ignore $T_2$.
Let $S'$ be the snake containing $\y^\uparrow$ and let ${\mathcal B}_{S}=\{S, S'\}$. By Lemma~\ref{lem:snakes_arranged_SW-NE}, $S' = \{\y^\uparrow\}$ participates in a trivial {\sf H3} or {\sf H8} miniswap.

The productive boxes $\{\x\}$ of $S_1$ are those containing $\GG$ (even the box $\y^\leftarrow$) and possibly also $\y$ (depending on what label, if any, appears in $\y^\rightarrow$). The productive boxes of $S$ are the boxes $\{ \x^\downarrow \}$ containing $\GG$ and possibly also $\y$. One checks that $\y$ is productive in $S$ if and only if it is productive in $S_1$. If productive, 
${\tt boxfactor}_U(\y)={\tt boxfactor}_{T_1}(\y) := y$. In $S'$, $\y^{\uparrow\rightarrow}$ is productive if and only if it is in $S'_1$; if it is productive, it contributes the same ${\tt boxfactor}(\y^{\uparrow\rightarrow}) := q$ to both.

For each productive $\x$ in $S_1$ with ${\tt boxfactor}(\x) := w_\x$, there is a corresponding productive $\x^\downarrow$ in $S$ with ${\tt boxfactor}(\x^\downarrow)= \hat{\beta}(\x) w_\x$. If $\y$ and $\y^{\uparrow\rightarrow}$ are productive in $U$, (\ref{eqn:abc}) follows from 
\[   qy \!\!\!\! \prod_{\x : \lab_{U}(\x) = \bullet_{\GG}} \!\!\!\!  \hat{\beta}(\x)  w_\x  =  q\left( \prod_{\x : \lab_{U}(\x) = \bullet_{\GG}} \!\!\!\!  \hat{\beta}(\x) \right) \cdot  y \cdot \!\!\!\! \prod_{\x : \lab_{U}(\x) = \bullet_{\GG}} \!\!\!\!   w_\x.\]
Otherwise we use the same identity without $q$, $y$ or both.

\noindent
{\sf Subcase 0.2.4: ($\tail(S)$ is {\sf T4}):}
Let $\tail(S) = \{ \x, \y:= \x^\rightarrow \}$. 
By (G.7), the $\GG \in \underline{\y}$ is westmost in its gene, so $\body(S) = \emptyset$.

\noindent
{\sf Subcase 0.2.4.1: ($\tail(S)$ is {\sf T4.1}):}
Set ${\mathcal B}_S = \{ S \}$. Locally at $S$,
$\begin{picture}(85,13)
\put(0,0){$\ytableaushort{\bullet {\mathcal{F}^!}} \mapsto 
\ytableaushort{\bullet {\mathcal{F}^!}},$}
\put(17,-4){\Scale[.7]{\GG, \HH}}
\put(68,-4){\Scale[.6]{\GG^!, \HH}}
\end{picture}$
where we have $\family(\mathcal{H}) = \family(\GG) + 1$ and $N_\mathcal{H}  = N_\GG$. Virtual labels appear in the same places in $U$ and $T_1$. In particular, 
neither 
$U$ nor $T_1$ can have $\circled{\GG} \in \underline{\x}$, since it would be West of every $\GG$. Further no labels move. As no weights change, trivially $\wt U = \wt U_1$, which implies (\ref{eqn:abc}). 

\noindent
{\sf Subcase 0.2.4.2: ($\tail(S)$ is {\sf T4.2}):}

\noindent
{\sf Subcase 0.2.4.2.1: ($\GG^- \in \x^\leftarrow$ with $\family(\GG^-) = \family(\GG)$):}
Set ${\mathcal B}_S = \{ S \}$. Locally at $S$, the swap is 
$\begin{picture}(213,16)
\put(0,0){$\ytableaushort{{\GG^-} \bullet {\FF^!}} \mapsto \ytableaushort{{\GG^-} \bullet {\FF^!}} + \hat{\beta}(\x) \cdot \ytableaushort{{\GG^-} \GG \bullet}.$}
\put(33,-3){\Scale[.5]{\GG, \circled{\HH}}}
\put(100,-3){\Scale[.5]{\GG^!, \circled{\HH}}}
\put(177,14){\Scale[.5]{\FF \cup Z}}
\end{picture}$
Here (\ref{eqn:abc}) is equivalent to $\wt U = \wt T_1 + \hat{\beta}(\x) \wt T_2$.

The $\FF^! \in \y$ is productive in $U$ and $T_1$ if and only if $\circled{\FF} \in \overline{\y}$ in $T_2$. 
If these boxes are productive,
\[{\tt boxfactor}_U(\y)={\tt boxfactor}_{T_1}(\y)={\tt virtualfactor}_{\overline{\y} \in T_2}(\circled{\FF}) :=a.\]
The box $\x$ is not productive in $U$ or $T_1$, but it is in $T_2$. Let ${\tt boxfactor}_{T_2}(\x) := u$. Then ${\tt virtualfactor}_{\underline{\y}\in U}(\circled{\HH}) = u$ and ${\tt virtualfactor}_{\underline{\y} \in T_1}(\circled{\HH})= u-1$. 
Let $w := {\tt edgefactor}_{\underline{\y} \in T}(\GG)$. Then  we have ${\tt edgefactor}_{\underline{\y} \in T_1}(\GG^!)=w$ and ${\tt edgefactor}_{\overline{\x} \in T_2}(\FF)=w$.  The box $\x^\leftarrow$ is productive in $U$ and $T_1$, but not in $T_2$. We have ${\tt boxfactor}_U(\x^\leftarrow) = {\tt boxfactor}_{T_1}(\x^\leftarrow)  = \hat{\beta}(\x) u$.

If $\y$ is productive in $U$, 
$\wt U  = \wt T_1 + \hat{\beta}(\x) \wt T_2$ 
follows from the identity on ${\mathcal B}_S$-contributions:
$\hat{\beta}(\x)au^2w =  \hat{\beta}(\x)au(u-1)w + \hat{\beta}(\x) a u w$.
If $\y$ is not productive in $U$, it follows from the same identity after cancelling $a$'s.

\noindent
{\sf Subcase 0.2.4.2.2: ($\GG^- \notin \x^\leftarrow$ or $\GG^- < \GG$):}
Let $\overline{T} := \phi_4^{-1}(T) \in P_{\GG}$.   
Let $\overline{S}$ be the snake of $\overline{T}$ containing $\x$.
Set $\mathcal{B}_S = \{ S, \overline{S} \}$; thus $\Gamma=\{T,\overline{T}\}$.

Locally at $S$, the swap is
$\begin{picture}(166,16)
\put(0,0){$\ytableaushort{\bullet {\FF^!}} \mapsto \ytableaushort{\bullet {\FF^!}} + \hat{\beta}(\x) \cdot \ytableaushort{\GG \bullet}.$}
\put(17,-3){\Scale[.5]{\GG, \circled{\HH}}}
\put(68,-3){\Scale[.5]{\GG^!, \circled{\HH}}}
\put(129,14){\Scale[.5]{\FF \cup Z}}
\end{picture}$ Locally at ${\overline S}$, 
$\begin{picture}(60,16)
\put(0,0){$\ytableaushort{\bullet {\FF^!}} \mapsto 0.$}
\put(21,-3){\Scale[.5]{\circled{\GG}}}
\put(5,-3){\Scale[.7]{\GG}}
\end{picture}$
By Proposition~\ref{prop:bundle_buddies:T6andT4}, $[T]P_\GG = [\overline{T}]P_\GG$. Hence (\ref{eqn:abc}) is equivalent to
$\wt U + \wt \overline{U} = \wt U_1 + \wt U_2$. Here $U(\Gamma)=\{U,\overline U\}$.

The $\FF^! \in \y$ is productive in $U$, $T_1$ and $\overline{U}$
if and only if $\circled{\FF} \in \overline{\y}$ in $T_2$. 
If these boxes are productive,
\[{\tt boxfactor}_U(\y)={\tt boxfactor}_{T_1}(\y)={\tt boxfactor}_{\overline U}(\y)
={\tt virtualfactor}_{\overline{\y} \in T_2}(\circled{\FF}):=a.\]
The box $\x$ is not productive in $U, T_1$ or $\overline U$.
Let $w = {\tt edgefactor}_{\underline{\y} \in T}(\GG)$. Then  we have
${\tt virtualfactor}_{\underline{\y} \in \overline{T}}(\circled{\GG}) = -w$, ${\tt edgefactor}_{\underline{\y} \in T_1}(\GG^!)=w$ and ${\tt edgefactor}_{\overline{\x} \in T_2}(\FF)=w$. 
Let $u = {\tt virtualfactor}_{\underline{\y}\in U}(\circled{\HH})$. Then ${\tt virtualfactor}_{\underline{\y} \in T_1}(\circled{\HH})= u-1$. Let $1-v = {\tt edgefactor}_{\underline{\x} \in \overline{U}}(\GG)$. Then ${\tt boxfactor}_{T_2}(\x) = \frac{v}{\hat{\beta}(\x)}$. 

If $\y$ is productive in $U$, 
$\wt U + \wt \overline{U} = \wt U_1 + \wt U_2$ 
follows from the identity on ${\mathcal B}_S$-contributions:
$auw - a(1-v)w =  a(u-1)w + a\hat{\beta}(\x) \frac{v}{\hat{\beta}(\x)}w$.
The same is true if $\y$ is not productive in $U$ except that $a$ does not
appear.

\noindent
{\sf Subcase 0.2.4.3: ($\tail(S)$ is {\sf T4.3}):}

\noindent
{\sf Subcase 0.2.4.3.1: ($\GG^- \in \x^\leftarrow$ with $\family(\GG^-) = \family(\GG)$):}
Set $\mathcal{B}_S = \{ S\}$. 
 Locally at $S$, the swap is 
$\begin{picture}(150,16)
\put(0,0){$\ytableaushort{{\GG^-} \bullet {\FF^!}} \mapsto \hat{\beta}(\x) \cdot \ytableaushort{{\GG^-} \GG \bullet}.$}
\put(33,-3){\Scale[.5]{Z \cup \GG}}
\put(114,14){\Scale[.5]{\FF \cup Z}}
\end{picture}$
Here (\ref{eqn:abc}) is equivalent to $\wt U = \hat{\beta}(\x) \wt T_1$.

For every label $\ell \in Z \cup \GG$ in $U$, there is a unique label $\ell_1 \in \FF \cup Z$ with 
$\family(\ell_1) = \family(\ell) - 1$.  Also, 
${\tt edgefactor}_U(\ell)={\tt edgefactor}_{T_1}(\ell_1):=a_\ell$. 

In $U$, $\y$ is productive if and only if $\y^\rightarrow$ does not contain a label of the same family as $\FF$. Hence $\y$ is productive in $U$ if and only if $\circled{\FF} \in \overline{\y}$ in $T_1$.
If $\y$ is productive in $U$, then ${\tt boxfactor}_U(\y)={\tt boxfactor}_{\overline{U}}(\y) = {\tt virtualfactor}_{\overline{\y} \in T_1}(\circled{\FF}) := b$.

The box $\x$ is productive in $T_1$, but not $U$. Let $w:={\tt boxfactor}_{T_1}(\x)$. The box $\x^\leftarrow$ is productive in $U$, but not in $T_1$. We have ${\tt boxfactor}_U(\x^\leftarrow) = \hat{\beta}(\x) w$. 

Hence if $\y$ is productive in $U$, $\wt U = \hat{\beta}(\x) \wt T_1$ follows from 
\[
\hat{\beta}(\x) w  b \prod_{\ell \in Z \cup \GG} a_\ell  = \hat{\beta}(\x) b  w \prod_{\ell_1 \in \FF \cup Z} a_\ell.
\] Otherwise, we use the same identity after canceling $b$.

\noindent
{\sf Subcase 0.2.4.3.2: ($\GG^- \notin \x^\leftarrow$ or $\GG^- < \GG$):}
Let $\overline{T} := \phi_4^{-1}(T) \in P_{\GG}$.
Let $\overline{S}$ be the snake of $\overline{T}$ containing $\x$ and set $\mathcal{B}_S = \{ S, \overline{S} \}$. 

Locally at $S$, the swap is
\ytableausetup{boxsize=1.6em}
$\begin{picture}(136,18)
\put(0,0){$\ytableaushort{\bullet {\mathcal{F}^!}} \mapsto \hat{\beta}(\x) \cdot \ytableaushort{\GG \bullet}.$}
\put(22,-2){\Scale[.5]{Z \cup \GG}}
\put(90,18){\Scale[.5]{\mathcal{F} \cup Z}}
\end{picture}
$
Locally at ${\overline S}$, 
$\begin{picture}(66 ,16)
\put(0,0){$\ytableaushort{\bullet {\FF^!}} \mapsto 0.$}
\put(21,-3){\Scale[.5]{Z \cup \circled{\GG}}}
\put(5,-3){\Scale[.7]{\GG}}
\end{picture}$
By Proposition~\ref{prop:bundle_buddies:T6andT4}, $[T]P_\GG = [\overline{T}]P_\GG$. Hence (\ref{eqn:abc}) is equivalent to $\wt U + \wt \overline{U} = \wt U_1$.

For every label $\ell \in Z \cup \GG$ in $U$, there is a unique label $\ell_1 \in \FF \cup Z$ with 
$\family(\ell_1) = \family(\ell) - 1$.  Also, 
${\tt edgefactor}_U(\ell)={\tt edgefactor}_{T_1}(\ell_1):=a_\ell$. 
If $\ell \in Z$ in $U$, then there is unique $\overline{\ell} \in Z$ in $\overline{U}$ with $\family(\ell) = \family(\overline{\ell})$, and ${\tt edgefactor}_{\overline{U}}(\overline{\ell}) = a_\ell$. Further ${\tt virtualfactor}_{\underline{\y} \in \overline{U}}(\circled{\GG})= -a_\GG$.

In $U$, $\y$ is productive if and only if $\y^\rightarrow$ does not contain a label of the same family as $\FF$. Hence $\y$ is productive in $U$ if and only if it is productive in $\overline{U}$ and further if and only if $\circled{\FF} \in \overline{\y}$ in $T_1$.
If $\y$ is productive in $U$, then ${\tt boxfactor}_U(\y)={\tt boxfactor}_{\overline{U}}(\y) = {\tt virtualfactor}_{\overline{\y} \in T_1}(\circled{\FF}) := b$.

The box $\x$ is productive in $T_1$. Let $w:={\tt boxfactor}_{T_1}(\x)$.
  Observe ${\tt edgefactor}_{\underline{\x} \in \overline{U}}(\GG) =1- \hat{\beta}(\x) w$. 
Hence $\wt U + \wt \overline{U} = \wt U_1$ follows from 
\[
b \prod_{\ell \in Z \cup \GG} a_\ell + b (1- \hat{\beta}(\x) w) (-a_\GG) \prod_{\ell \in Z} a_\ell = \hat{\beta}(\x) b  w \prod_{\ell \in Z \cup \GG} a_\ell
\] if $\y$ is productive in $U$. If it is not productive, we use the same identity without $b$.

\noindent
{\sf Subcase 0.2.5: ($\tail(S)$ is {\sf T5}):} Either $\body(S) = \emptyset$ and $\tail(S) = S$, or else $\body(S) \neq \emptyset$.
Set $B_S = \{ S \}$. Let $\tail(S) = \{ \x, \y:= \x^\rightarrow \}$.

\noindent
{\sf Subcase 0.2.5.1: ($\body(S) = \emptyset$):}
Locally at $S$,
$\begin{picture}(130,18)
\put(0,0){$\ytableaushort{\bullet {\FF^!}} \mapsto \hat{\beta}(\x) \cdot \ytableaushort{\GG \bullet}.$}
\put(21,-2){\Scale[.5]{\circled{\GG} \cup Z}}
\put(91,18){\Scale[.5]{\FF \cup Z}}
\end{picture}$
In $U$, $\x$ is not productive, while $\y$ is productive if and only if $\y^\rightarrow$ does not contain a label of the same family as $\FF$. In $T_1$, $\x$ is productive, but $\y$ is not. $\y$ is productive in $U$ if and only if $\circled{\FF} \in \overline{\y}$ in $T_1$. If $\y$ is productive in $U$, then ${\tt boxfactor}_U(\y) = {\tt virtualfactor}_{\overline{\y} \in T_1}(\circled{\FF}) := a$.

For every edge label $\ell \in Z$ in $U$, there is a unique $\ell_1 \in \FF \cup Z$ in $T_1$ with $\family(\ell_1) = \family(\ell)-1$. Furthermore ${\tt edgefactor}_U(\ell) = {\tt edgefactor}_{T_1}(\ell_1) := b_\ell$. Let $\ell_1^M$ be greatest label of $Z$ in $T_1$, and let ${\tt edgefactor}_{T_1}(\ell_1^M) := w$. Then ${\tt virtualfactor}_{\underline{\y} \in U}(\circled{\GG}) = -w$.

Let ${\tt boxfactor}_{T_1}(\x) := c$. Then ${\tt virtualfactor}_{\underline{\x} \in U}(\circled{\GG}) = \hat{\beta}(\x) c$. Since $T_1$ has one more $\GG$ than $U$, $d(U) = d(T_1) - 1$.

In this case (\ref{eqn:abc}) is equivalent to $\wt U = \hat{\beta}(\x) \wt T_1$. This follows from 
\[
(-1)^{d(U)} \cdot \left(\prod_{\ell \in Z} b_\ell \right) \cdot (-w)\hat{\beta}(\x) c \cdot a   =\hat{\beta}(\x)\cdot (-1)^{d(U)+1} \cdot w \left( \prod_{\ell \in Z} b _\ell \right) \cdot a \cdot c,
\]
if $\y$ is productive in $U$. If it is not productive, we use the same identity without $a$.

\noindent
{\sf Subcase 0.2.5.2: ($\body(S) \neq \emptyset$):}
Locally at $S$, the swap looks like
\[\begin{picture}(250,42)
\put(0,28){$\ytableaushort{\none \bullet {\FF^!}, \bullet \GG, \GG} \mapsto -\prod_{\z \ni \bullet} \hat{\beta}(\z) \cdot \ytableaushort{\none \GG \bullet, \GG \bullet, \bullet}.$}
\put(40,27){\Scale[.5]{\circled{\GG} \cup Z}}
\put(169,46){\Scale[.5]{\FF \cup Z}}
\end{picture}
\]

Let $A$ be the set of boxes of $S_1$ containing $\GG$. The productive boxes of $S$ are $\{\z^\downarrow : \z \in A\}$, as well as perhaps $\y$, which is productive if and only if $\y^\rightarrow$ contains a label of the same family as $\FF$ in $U$. The productive boxes of $S_1$ are $A$. For each $\z \in A$, let $a_\z := {\tt boxfactor}_{T_1}(\z)$. Then ${\tt boxfactor}_U(\z^\downarrow) = \hat{\beta}(\z)  a_\z$. If $\y$ is productive in $U$, let $b := {\tt boxfactor}_U(\y)$.
Observe $\y$ is productive in $U$ if and only if $\circled{\FF} \in \overline{\y}$ in $T_1$. Furthermore if $\y$ is productive in $U$, then ${\tt virtualfactor}_{\overline{\y} \in T_1}(\circled{\FF}) = b$.

For every edge label $\ell \in Z$ in $U$, there is a unique $\ell_1 \in \FF \cup Z$ in $T_1$ with $\family(\ell_1) = \family(\ell)-1$. Furthermore ${\tt edgefactor}_U(\ell) = {\tt edgefactor}_{T_1}(\ell_1) := c_\ell$. Let $\ell_1^M$ be greatest label of $Z$ in $T_1$, and let ${\tt edgefactor}_{T_1}(\ell_1^M) := w$. Then ${\tt virtualfactor}_{\underline{\y} \in U}(\circled{\GG}) = -w$.

In this case (\ref{eqn:abc}) is equivalent to $\wt U =  -\prod_{\z \in A} \hat{\beta}(\z) \wt T_1$. This follows from
\[
\left( \prod_{\ell \in Z} c_\ell \right) \cdot (-w) \cdot b \prod_{\z \in A} (\hat{\beta}(\z)a_\z) = \left (- \prod_{\z \in A} \hat{\beta}(\z) \right) \cdot w \left( \prod_{\ell \in Z} c_\ell \right) \cdot b \cdot \prod_{\z \in A} a_\z.
\]

\noindent
{\sf Subcase 0.2.6: ($\tail(S)$ is {\sf T6}):}
This case is covered by {\sf Subcases~0.2.4.2} and~{\sf 0.2.4.3}.

\noindent
{\sf Case 1: ($\head(S)$ is {\sf H1}):}
Set $B_S = \{S\}$. Let $\x$ be the unique box of $S$. Locally at $S$, the swap is
\ytableausetup{boxsize=1em}
$\begin{picture}(124,14)
\put(0,0){$\ytableaushort{\bullet}\mapsto \beta(\x) \cdot \ytableaushort{{\GG}} + \gamma \cdot \ytableaushort{\bullet},$}
\put(3,-4){$\Scale[.8]{\GG}$}
\put(109,10){$\Scale[.8]{\GG}$}
\end{picture}
$ where $\gamma := 0$ if $\x$ is in row $i$ and $\gamma := 1$ otherwise.

Let ${\tt boxfactor}_{T_1}(\x) := a$. The box $\x$ is productive in $U$ if and only if $\family(\lab(\x^\leftarrow)) = \family(\GG)$, in which case ${\tt boxfactor}_U(\x) =a$. Observe ${\tt edgefactor}_{{\underline\x}\in U}(\GG) = 1-\hat{\beta}(\x)a$,
${\tt edgefactor}_{\overline{\x}\in T_2}(\GG) =  1 - a$, 
and $\x$ is not productive in $T_2$. Notice further that $\gamma = 0$ if and only if $a = 1$.

If $\x^\leftarrow$ is empty, then $\x^\leftarrow$ is not productive in any of $U, T_1, T_2$. If $\x^\leftarrow$ is nonempty and $\lab(\x^\leftarrow) < \GG$, then $\x^\leftarrow$ is productive in all three tableaux, and ${\tt boxfactor}_U(\x^\leftarrow) = {\tt boxfactor}_{T_1}(\x^\leftarrow) = {\tt boxfactor}_{T_2}(\x^\leftarrow) := b$. If $\x^\leftarrow$ is nonempty and $\family(\lab(\x^\leftarrow)) = \family(\GG)$, then $\x^\leftarrow$ is productive in $T_2$, but not in $U$ or $T_1$. Moreover by (G.6), $\GG^- \in \x^\leftarrow$, so ${\tt boxfactor}_{T_2}(\x^\leftarrow) = \hat{\beta}(\x)a$.

In this case (\ref{eqn:abc}) is equivalent to $\wt U = \beta(\x) \wt T_1 + \gamma \wt T_2$. Since $\gamma = 1$ whenever $\wt T_2 \neq 0$, it suffices to show $\wt U = \beta(\x) \wt T_1 + \wt T_2$. If $\x^\leftarrow$ is nonempty and $\lab(\x^\leftarrow) < \GG$, this follows from
\[ 
(1 - \hat{\beta}(\x) a) \cdot b = \beta(\x) \cdot ab + (1-a) \cdot b.
\]
If $\x^\leftarrow$ is empty, we use the same identity without $b$.
If $\family(\lab(\x^\leftarrow)) = \family(\GG)$, we use the identity
\[
(1 - \hat{\beta}(\x) a) \cdot a = \beta(\x) \cdot a + (1-a) \cdot \hat{\beta}(\x) a.
\]

\noindent
{\sf Case 2: ($\head(S)$ is {\sf H2}):}
Set $B_S = \{S\}$. Let $\x$ be the unique box of $S$. Locally at $S$,
\ytableausetup{boxsize=1em}
$\begin{picture}(108,12)
\put(0,0){$\ytableaushort{\bullet}\mapsto \ytableaushort{\bullet} + \beta(\x) \cdot \ytableaushort{{\GG}}.$}
\put(1,-4){\Scale[.6]{\circled{\GG}}}
\end{picture}$

Let ${\tt boxfactor}_{T_2}(\x) := a$. Then ${\tt virtualfactor}_{\underline{\x}\in U}(\circled{\GG}) := a \hat{\beta}(\x)$. In $T_1$, $\circled{\GG} \in \overline{\x}$ with ${\tt virtualfactor}_{\underline{\x}\in T_1}(\circled{\GG}) = a$. Due to $T_2$'s extra $\GG \in \x$, $d(T_2) = d(U) + 1 = d(T_1) + 1$.

In this case (\ref{eqn:abc}) is equivalent to $\wt U = \wt T_1 + \beta(\x) \wt T_2$. This follows from
\[
(-1)^{d(U)} \cdot a \hat{\beta}(\x) = (-1)^{d(U)} \cdot a + \beta(\x) \cdot (-1)^{d(U)+1} \cdot a.
\]

\noindent
{\sf Case 3: ($\head(S)$ is {\sf H3}):}  
Here $S=\{\x\}$. Let $\y = \x^{\downarrow\leftarrow}$. Locally at $S$,
\ytableausetup{boxsize=1.8em}
$\begin{picture}(155,12)
\put(0,0){$\Scale[0.8]{\ytableaushort{{\bullet_{i_k}}}}\mapsto \Scale[0.8]{\ytableaushort{{\bullet_{i_{k+1}}}}}.$}
\end{picture}
$

\noindent
{\sf Subcase 3.1: ($i_{k+1} \in \x^\downarrow  = \y^\rightarrow$, $i_k \in \y$, no $\bullet$ West of $\y$ in the same row)}:
In $U$ the $\y$ is not productive, whereas in $T_1$, $\y$ is productive. Let $a:={\tt boxfactor}_{T_1}(\y)$.

\noindent
{\sf Subcase 3.1.1: ($\y$ contains the only $i_k$ in $T$)}:
Let $\overline{T} := \phi_2^{-1}(T)$. 
Let $S'$ be the snake in $T$ containing $\y$, 
$\overline{S}$ be the snake in $\overline{T}$ 
containing $\x$ and $\overline{S'}$ the snake in $\overline{T}$ 
containing $\y$. Set $B_S = \{S, S', \overline{S}, \overline{S'} \}$. 
Locally at $S\cup S'$ and at $\overline{S}\cup \overline{S'}$ the
swaps are respectively,
\[\Scale[0.7]{\begin{picture}(46,27)
\put(0,17){\ytableaushort{\none \bullet, {i_k} {i_{k+1}}}}
\end{picture}}\  \mapsto
\Scale[0.7]{\begin{picture}(50,27)
\put(0,17){\ytableaushort{\none \bullet, {i_k} {i_{k+1}}}}
\end{picture}} \text{ \ \ and \ \ }
\Scale[0.7]{\begin{picture}(46,27)
\put(0,17){\ytableaushort{\none \bullet, \bullet {i_{k+1}}}}
\put(7,-6){$i_k$}
\end{picture}}\  \ \mapsto \beta(\y)
\Scale[0.7]{\begin{picture}(50,27)
\put(0,17){\ytableaushort{\none \bullet, {i_k} {i_{k+1}}}}
\end{picture}}
,\] 
where $\bullet=\bullet_{i_k}$ before the swap and $\bullet=\bullet_{i_{k+1}}$
after the swap.

By Proposition~\ref{prop:bundle_buddies:expseesaw/empty}, $[T]P_{{i_k}} = -[\overline{T}]P_{{i_k}}$. Hence, (\ref{eqn:abc}) is equivalent to 
\begin{claim}
$\wt (U-\overline{U}) = \wt (T_1 - \beta(\y) T_1)$.
\end{claim}
\begin{proof}
First, $\x$ is not productive in $U,\overline{U}$ or $T_1$.
Second, $\x^\downarrow$ is productive in $U,\overline{U}$ and $T_1$ and moreover
\[{\tt boxfactor}_U(\x^\downarrow)={\tt boxfactor}_{\overline U}(\x^\downarrow)={\tt boxfactor}_{T_1}(\x^\downarrow):=b.\]
Third, $\y$ is not productive in $\overline{U}$.

Next, ${\tt edgefactor}_{\underline{\y}\in \overline{U}}(i_k)=1-a\hat{\beta}(\y)$.
Hence the claim follows from
\[b-(1-a\hat{\beta}(\y))\cdot b = \hat{\beta}(\y)\cdot ab = (1-\beta(\y)){\tt wt}(T_1).\qedhere \]
\end{proof}

\noindent
{\sf Subcase 3.1.2: (Subcase 3.1.1 does not apply)}:
Let $\overline{T} := \phi_1^{-1}(T)$. Let $S'$ be the snake in $T$ containing
$\y$, $\overline{S}$ be the snake in $\overline{T}$ containing 
$\x$ and $\overline{S'}$ the snake in $\overline{T}$ containing $\y$. 
Set $B_S = \{S, S', \overline{S}, \overline{S'}\}$. 
Locally at $S\cup S'$ and at $\overline{S}\cup \overline{S'}$ the
swaps are respectively,
\[\Scale[0.7]{\begin{picture}(46,27)
\put(0,17){\ytableaushort{\none \bullet, {i_k} {i_{k+1}}}}
\end{picture}}\  \mapsto
\Scale[0.7]{\begin{picture}(50,27)
\put(0,17){\ytableaushort{\none \bullet, {i_k} {i_{k+1}}}}
\end{picture}} \text{ \ \ and \ \ }
\Scale[0.7]{\begin{picture}(46,27)
\put(0,17){\ytableaushort{\none \bullet, \bullet {i_{k+1}}}}
\put(5,-7){$\Scale[0.8]{\circled{i_k}}$}
\end{picture}}\  \ \mapsto 
\Scale[0.7]{\begin{picture}(50,31)
\put(0,17){\ytableaushort{\none \bullet, {\bullet} {i_{k+1}}}}
\end{picture}}
+\beta(\y)
\Scale[0.7]{\begin{picture}(50,27)
\put(0,17){\ytableaushort{\none \bullet, {i_k} {i_{k+1}}}}
\end{picture}}
,\] 
where $\bullet=\bullet_{i_k}$ before the swap and $\bullet=\bullet_{i_{k+1}}$
after the swap. Note that $T_1=\overline{T_2}$.

By Proposition~\ref{prop:bundle_buddies:poofseesaw/empty}, $[T]P_{{i_k}} = -[\overline{T}]P_{{i_k}}$. 
Thus (\ref{eqn:abc}) is equivalent to 
\begin{claim}
$\wt U - \wt \overline{U} = \wt T_1 - \wt \overline{T_1} - \beta(\y) \wt T_1$.
\end{claim}
\begin{proof}
Firstly, $\x$ is not productive in $U$, $\overline{U}$, $T_1$ or $\overline{T_1}$. Secondly, $\x^\downarrow$ is productive in $U$, $\overline{U}$, $T_1$ and $\overline{T_1}$. Moreover, 
\[
{\tt boxfactor}_U(\x^\downarrow) = {\tt boxfactor}_{\overline{U}}(\x^\downarrow) = {\tt boxfactor}_{T_1}(\x^\downarrow) = {\tt boxfactor}_{\overline{T_1}}(\x^\downarrow) := b.
\] Note $\y$ is not productive in $U$, $\overline{U}$ or $\overline{T_1}$.
Observe ${\tt virtualfactor}_{\underline{\y}\in \overline{U}}(\circled{i_k}) = a \hat{\beta}(\y)$.
 Finally, $d(U) = d(T_1) = d(\overline{U})+1 = d(\overline{T_1}) +1$. The claim then follows from
\[
(-1)^{d(U)} \cdot b - (-1)^{d(U) -1} \cdot a \hat{\beta}(\y) \cdot b = (-1)^{d(U)} \cdot ab - (-1)^{d(U) -1} \cdot b - \beta(\y) \cdot (-1)^{d(U)} \cdot ab. \qedhere
\]
\end{proof}

\noindent
{\sf Subcase 3.2: (Subcase 3.1 does not apply)}:

\noindent
{\sf Subcase 3.2.1: ($\x^\downarrow$ is part of a {\sf T3} $\tail$)}:
Let $S'$ be the snake containing $\x^\downarrow$, and let $\mathcal{B}_S = \mathcal{B}_{S'}$. The remaining discussion of this case is found with the discussion of $S'$; see e.g., {\sf Subcase 0.2.3}.

\noindent
{\sf Subcase 3.2.2: (Subcase 3.2.1 does not apply)}:
Set $\mathcal{B}_S = \{S\}$. Recall that locally at $S$,
\ytableausetup{boxsize=1.8em}
$\begin{picture}(55,16)
\put(0,0){$\Scale[0.8]{\ytableaushort{{\bullet_{i_k}}}}\mapsto \Scale[0.8]{\ytableaushort{{\bullet_{i_{k+1}}}}}.$}
\end{picture}
$  The swap affects no weight factors.

\noindent
{\sf Case 4: ($\head(S)$ is {\sf H4}):}
We argue the case $S = \head(S)$. When $S$ is a multirow, weight preservation follows by combining the present argument with that of {\sf Case 0.2}.

Assume $S = \{\x, \x^\rightarrow\}$. Let $\overline{T}=\phi_{3,\{\x,\x^\rightarrow\}}(T)$, and let $\overline{S}$ be the snake containing $\x$ in $\overline{T}$. Set $B_S = \{S, \overline{S}\}$. 
Locally at $S, \overline{S}$ respectively the swaps are
\[\ytableausetup{boxsize=1em}
\begin{picture}(53,10)
\put(0,0){$\ytableaushort{\bullet \GG} \mapsto 0$}
\put(3,-4){$\Scale[.8]{\GG}$}
\end{picture} 
\text{ and }
\begin{picture}(53,10)
\put(0,0){$\ytableaushort{\bullet \GG} \mapsto \hat{\beta}(\x) \ytableaushort{\GG \bullet}.$}
\end{picture}
\]
Since by Proposition~\ref{prop:bundle_buddies:horizontal/death}, $[T]P_{\GG} = [\overline{T}]P_{\GG}$, (\ref{eqn:abc}) is equivalent to:
\begin{claim}
$\wt U + \wt \overline{U} = \wt \overline{T_1}$.
\end{claim}
\begin{proof}
In $U$ and $\overline{U}$, $\x$ is not productive. However $\x$ is productive in $\overline{T_1}$. Let ${\tt boxfactor}_{\overline{T_1}}(\x) := a$. In $\overline{T_1}$, $\x^\rightarrow$ is not productive. In $U$, $\x^\rightarrow$ is productive if and only if $\x^{\rightarrow\rightarrow}$ does not contain a label of family $i$. Further, $\x^\rightarrow$ is productive in $U$ if and only if it is productive in $\overline{U}$ and if and only if $\circled{\GG} \in \overline{\x^\rightarrow}$ in $\overline{T_1}$. In this case, ${\tt boxfactor}_{U}(\x^\rightarrow) = {\tt boxfactor}_{\overline{U}}(\x^\rightarrow) = {\tt virtualfactor}_{\overline{\x^\rightarrow} \in \overline{T_1}}(\circled{\GG}) := b$.

In $T$, ${\tt edgefactor}(\underline{\x}) = 1 - \hat{\beta}(\x) a$. Finally $d(U) = d(\overline{U}) + 
1 = d(\overline{T_1}) + 1$. If $\x^\rightarrow$ is productive in $U$, then the claim follows from
\[
(-1)^{d(U)} \cdot (1 - \hat{\beta}(\x) a) \cdot b  + (-1)^{d(U)-1} \cdot b = \hat{\beta}(\x)(-1)^{d(U) - 1} \cdot ab.
\]
Otherwise we use the same identity with $b$'s removed.
\end{proof}

\noindent
{\sf Case 5: ($\head(S)$ is {\sf H5}):}
We only explicitly argue the case $S = \head(S)$.  The case $S$ is a multirow ribbon follows by combining the present arguments with those from {\sf Case 0.2}. Hence $S= \head(S)=\{\x, \y := \x^\rightarrow\}$.

\noindent
{\sf Subcase 5.1: ($\head(S)$ is {\sf H5.1}):} Let ${\mathcal B}_S=\{S\}$. Locally at $S$ the swap is
\ytableausetup{boxsize=1.4em}
$\begin{picture}(92,18)
\put(0,0){$\ytableaushort{\bullet \GG} \mapsto \ytableaushort{\bullet {\GG^!}}$}
\put(21,-5){$\HH$}
\put(74,-5){$\HH$}
\end{picture}$,
where  $\HH\in \underline{\y}$, $\family(\mathcal{H}) = \family(\GG) + 1$ and 
$N_\mathcal{H}  = N_\GG$. Since no labels move, (\ref{eqn:abc}) follows trivially.

\noindent
{\sf Subcase 5.2: ($\head(S)$ is {\sf H5.2}):} Locally at $S$ the swap is
\ytableausetup{boxsize=1.4em}
$\begin{picture}(170,12)
\put(0,0){$\ytableaushort{\bullet \GG} \mapsto \ytableaushort{\bullet {\GG^!}} + \hat{\beta}(\x) \cdot \ytableaushort{\GG \bullet}$.}
\put(20,-5){\Scale[.8]{$\circled{\HH}$}}
\put(73,-5){\Scale[.8]{$\circled{\HH}$}}
\end{picture}
$
By Lemma~\ref{lem:how_to_check_ballotness}, 
$\y$ contains the westmost instance of $\GG$ and that hence $\circled{\GG} \notin \underline{\x}$.

\noindent
{\sf Subcase 5.2.1: ($\GG^- \in \x^\leftarrow$ with $\family(\GG^-) =i$):} Let ${\mathcal B}_S=\{S\}$.
Now, $\y$ is productive in $U$ if and only if it is productive in $T_1$ and if and only if 
$\circled{\GG} \in \overline{\y}$ in $T_2$. In this case, 
\[{\tt boxfactor}_U(\y)={\tt boxfactor}_{T_1}(\y)={\tt virtualfactor}_{\overline{\y}\in T_2}(\circled{\GG}):=a.\]
Let ${\tt virtualfactor}_{\underline{\y} \in U}(\circled{\HH}) := b$. The box $\x$ is productive only in $T_2$, with ${\tt boxfactor}_{T_2}(\x) = b$. The box $\x^\leftarrow$ is productive in $U$ and $T_1$; $\x^\leftarrow$ is not productive in $T_2$. Furthermore, ${\tt boxfactor}_U(\x^\leftarrow) = {\tt boxfactor}_{T_1}(\x^\leftarrow) = \hat{\beta}(\x) b$. Observe ${\tt virtualfactor}_{\underline{\y} \in T_1}(\circled{\HH}) = b-1$.  
In this case, (\ref{eqn:abc}) is equivalent to $\wt U = \wt T_1 + \hat{\beta}(\x) \wt T_2$. If $\y$ is productive in $U$, this follows from the identity
\[
b \cdot a \hat{\beta}(\x) b = (b-1) \cdot a \hat{\beta}(\x) b  + \hat{\beta}(\x) \cdot a \cdot b. 
\]
If $\y$ is not productive in $U$, we use the same identity without $a$.

\noindent
{\sf Subcase 5.2.2: ($\GG^- \notin \x^\leftarrow$ or $\family(\GG^-) \neq i$):}
Observe $T \in \mathcal{S}_3'$. Let $\overline{T} = \phi_3^{-1}(T)$. Note that $\circled{\HH} \notin \underline{\y}$ in $\overline{T}$.
Let $\overline{S}$ be the snake of $\overline{T}$ containing $\x$.
Set $\mathcal{B}_S = \{S, \overline{S}\}$. The swap at $S$ is illustrated at the start of {\sf Case 5.2} above. Locally at $\overline{S}$, the swap is 
$\begin{picture}(170,20)
\put(0,0){$\ytableaushort{\bullet \GG} \mapsto 0$.}
\put(5,-3){$\Scale[.9]{\GG}$}
\end{picture}
$

Observe that $\y$ is productive in $U$ if and only if it is productive in $\overline{U}$ and if and only if it is productive in $T_1$; $\y$ is not productive in $T_2$. Also if $\y$ is productive in $U$, then ${\tt boxfactor}_U(\y) = {\tt boxfactor}_{\overline{U}}(\y) = {\tt boxfactor}_{T_1}(\y) := a$. There is $\circled{\GG} \in \overline{\y}$ in $T_2$ if and only if $\y$ is productive in $U$. In this case, ${\tt virtualfactor}_{\overline{\y} \in T_2}(\circled{\GG}) = a$.

Let $b := {\tt virtualfactor}_{\underline{\y} \in U}(\circled{\HH})$ and $1-c := {\tt edgefactor}_{\underline{\x} \in \overline{U}}(\GG)$. Consequently we have ${\tt virtualfactor}_{\underline{\y} \in T_1}(\circled{\HH}) = b - 1$, while ${\tt boxfactor}_{T_2}(\x) = c / \hat{\beta}(\x)$. Observe $d(U) = d(T_1) = d(T_2) = d(\overline{U}) - 1$.  By Proposition~\ref{prop:bundle_buddies:horizontal/death}, $[T]P_{\GG} = [\overline{T}]P_\GG$, to (\ref{eqn:abc}) is equivalent to $\wt U + \wt \overline{U} = \wt T_1 + \hat{\beta}(\x) \wt T_2$. If $\y$ is productive in $U$, this follows from
\[
(-1)^{d(U)} \cdot b \cdot a + (-1)^{d(U)-1} \cdot (1-c) \cdot a = (-1)^{d(U)} \cdot (b-1) \cdot a + \hat{\beta}(\x) \cdot (-1)^{d(U)} \cdot a \cdot \frac{c}{\hat{\beta}(\x)}.
\]
Otherwise it follows from the same identity without $a$'s.

\noindent
{\sf Case 5.3: ($\head(S)$ is {\sf H5.3}):}

\noindent
{\sf Subcase 5.3.1: ($\circled{\GG} \in \underline{\x}$):}
Set $B_S = \{S\}$. Let $S = \{ \x, \y := \x^\rightarrow\}$. Locally at $S$, the swap is $\ytableaushort{ \bullet \GG} \mapsto \hat{\beta}(\x) \ytableaushort{\GG \bullet}$.
Let $a := {\tt virtualfactor}_{\underline{\x} \in U}(\circled{\GG})$.
In $U$, $\x$ is not productive, while $\y$ is productive if and only if $\y^\rightarrow$ does not contain a label of family $i$. Further $\y$ is productive in $U$ if and only if $\circled{\GG} \in \overline{\y}$ in $T_1$. If $\y$ is productive in $U$, then ${\tt boxfactor}_U(\y) = {\tt virtualfactor}_{\overline{\y} \in T_1}(\circled{\GG} := b$. In $T_1$, $\x$ is productive, but $\y$ is not; ${\tt boxfactor}_{T_1}(\x) = \frac{a}{\hat{\beta}(\x)}$. Here (\ref{eqn:abc}) is equivalent to $\wt U = \hat{\beta}(\x) \wt T_1$. If $\y$ is productive in $U$, this follows from
$a \cdot b = \hat{\beta}(\x) \cdot b \cdot \frac{a}{\hat{\beta}(\x)}$.
Otherwise we use the same identity without $b$.

\noindent
{\sf Subcase 5.3.2: ($\GG^- \in \x^\leftarrow$ with $\family(\GG^-) = i$):}
By (G.12) and Lemma~\ref{lemma:Gsoutheast}, no label of family $i$ appears in $\x$'s column.
Hence the $\GG \in \y$ is the Westmost $\GG$. In particular, $\circled{\GG} \notin \underline{\x}$, so this case is disjoint from {\sf Subcase 5.3.1.} Set $B_S = \{S\}$. Locally at $S$, the swap is $\ytableaushort{ \bullet \GG} \mapsto \hat{\beta}(\x)\ytableaushort{\GG \bullet}$.

Let $a := {\tt boxfactor}_U(\x^\leftarrow)$. In $U$, $\x$ is not productive, while $\y$ is productive if and only if $\y^\rightarrow$ does not contain a label of family $i$ if and only if $\circled{\GG} \in \overline{\y}$ in $T_1$.
In this case ${\tt boxfactor}_U(\y) = {\tt virtualfactor}_{\overline{\y} \in T_1}(\circled{\GG}) := b$. In $T_1$, $\x^\leftarrow$ and $\y$ are not productive, while $\x$ is productive and ${\tt boxfactor}_{T_1}(\x) = a/\hat{\beta}(\x)$. Here (\ref{eqn:abc}) is equivalent to
${\tt wt}(U)=\hat{\beta}(\x){\tt wt}(T_1)$. If $\y\in U$ is productive then this follows from 
$ab=\hat{\beta}(x)\cdot b\cdot a/\hat{\beta}(\x)$; otherwise the same is true after removing the $b$'s.

\noindent
{\sf Subcase 5.3.3: (Subcases 5.3.1 and 5.3.2 do not apply):}
There is no $\GG^{-}\in \x^\leftarrow$ because we are not in {\sf Subcase 5.3.2}. 
Since we are not in {\sf Case 5.1}
there is not $\HH\in \underline{\y}$ with ${\tt family}(\HH)=i+1$ and
$N_{\HH}=N_{\GG}$. By the assumption that we are in {\sf Case 5}, there can be no label of $\family(\GG)$ in the column of $\x$. Thus if the $\GG\in \y$ were not Westmost then $\circled{\GG}\in \underline{\x}$ and we would be
in {\sf Subcase 5.3.1}, a contradiction. Therefore we conclude $T\in {\mathcal S}_3'$. Let ${\overline T}:=\phi_3^{-1}(T)$. 
Let ${\overline S}$ be the snake in ${\overline T}$ containing $\x$. Then set ${\mathcal B}_S=\{S,\overline{S}\}$.

Locally at $S$ we have the swap is $\ytableaushort{ \bullet \GG} \mapsto \hat{\beta}(\x) \ytableaushort{\GG \bullet}$.
Locally at $\overline{S}$, the swap is 
$\begin{picture}(66,20)
\put(0,0){$\ytableaushort{\bullet \GG} \mapsto 0$.}
\put(5,-3){$\Scale[.9]{\GG}$}
\end{picture}
$
By Proposition~\ref{prop:bundle_buddies:horizontal/death}, $[T]P_{k} = [\overline{T}]P_{k}$.
Therefore (\ref{eqn:abc}) is equivalent to ${\tt wt}(U)+{\tt wt}(\overline{U})=\hat{\beta}(\x){\tt wt}(T_1)$. 
This is exactly 
proved (up to renaming of tableaux) in {\sf Case~4}.

\noindent
{\sf Case 6: ($\head(S)$ is {\sf H6}):}
Here $S = \{\x, \x^\rightarrow \}$.

\noindent
{\sf Subcase 6.1: ($\bullet_{\GG} \notin \x^{\rightarrow \uparrow})$}:
Let ${\mathcal B}_{S}=\{S\}$. Locally at $S$,
\ytableausetup{boxsize=1.4em}
$\begin{picture}(320,23)
\put(0,0){$\ytableaushort{{\bullet} {\GG^+}} \mapsto \beta(\x) \cdot\ytableaushort{{\GG} {\GG^+}} 
+ \hat{\beta}(\x) \cdot \ytableaushort{{\GG} {\bullet}}$.}
\put(5,-3){$\GG$}
\put(185,-3){$\GG^+$}
\end{picture}
$ 

In $U, T_1$ and $T_2$, $\x$ is not productive. 
Now, $\x^\rightarrow$ is productive in $T_1$ if and only if it is productive in $T_2$ if and
only if it is productive in $U$; in case of productivity,
\[a:={\tt boxfactor}_U(\x^\rightarrow)={\tt boxfactor}_{T_1}(\x^\rightarrow)={\tt boxfactor}_{T_2}(\x^\rightarrow).\]
Next, let ${\tt edgefactor}_{{\underline\x}\in U}(\GG):= 1 - b$. Thus
${\tt edgefactor}_{\underline{\x^\rightarrow}\in T_2}(\GG^{+}) = 1-b/\hat{\beta}(\x)$.

Finally, (\ref{eqn:abc}) is equivalent to ${\tt wt}(U)=\beta(\x){\tt wt}(T_1)+\hat{\beta}(\x)(T_2)$. If
$\x^\rightarrow$ is productive in $U$, then this follows from
\[(1-b)\cdot a = \beta(\x)\cdot a+\hat{\beta}(\x)\cdot (1-b/\hat{\beta}(\x))\cdot a;\]
otherwise we are done by the same expression without the $a$'s.

\noindent
{\sf Subcase 6.2: ($\bullet_{\GG} \in \x^{\rightarrow \uparrow}$)}: Thus $T\in {\mathcal S}_2$. 
Let $\overline{T}:=\phi_2(T)$. Let $\overline{S}$ be the snake of $\overline{T}$ 
containing $\x$. Let $S'$ and $\overline{S'}$ be the snakes of $T$ and $\overline{T}$
containing $\x^{\rightarrow\uparrow}$, respectively.
Set ${\mathcal B}_S = \{ S, S', \overline{S}, \overline{S'} \}$. 

Notice $S'$ falls into {\sf Case 3} and in fact the ${\mathcal B}_{S'}$ defined there equals
the current ${\mathcal B}_S$. Hence (\ref{eqn:abc}) holds by {\sf Case 3}.

\noindent
{\sf Case 7: ($\head(S)$ is {\sf H7}):}
Here $S = \{ \x, \x^\rightarrow\}$.

\noindent
{\sf Subcase 7.1: ($\bullet_{\GG} \notin \x^{\rightarrow \uparrow}$)}: 
Let ${\mathcal B}_S=\{S\}$. Locally at $S$,
\ytableausetup{boxsize=1.4em}
\[\begin{picture}(300,18)
\put(0,0){${\ytableaushort{{\bullet}{\GG^+}}} \mapsto 
{\ytableaushort{{\bullet} {\GG^+}}} + \beta(\x) \cdot {\ytableaushort{{\GG} {{\GG^+}}}} + \hat{\beta}(\x) \cdot 
{\ytableaushort{{\GG} {\bullet}}}$}
\put(3,-5){\Scale[.8]{$\circled{\GG}$}}
\put(234,-3){\Scale[.9]{\GG^+}}
\end{picture}
\] 
First, $\x$ is not productive in $U,T_1,T_2$ or $T_3$ whereas $\x^\rightarrow$ is productive in each of $U,T_1,T_2$ and $T_3$
or otherwise not productive in any of these tableaux. If $\x^\rightarrow$ is productive then
\[a:={\tt boxfactor}_U(\x^\rightarrow)={\tt boxfactor}_{T_i}(\x^\rightarrow) \text{\ for $i=1,2,3$.}\]
Second,  let $b:={\tt virtualfactor}_{\underline{\x}\in U}(\circled{\GG})$.
 Third, ${\tt edgefactor}_{\underline{\x^\rightarrow}\in T_3}(\GG^+)=1-b/\hat{\beta}(\x)$.
Fourth, $d(T_1)=d(U)$, $d(T_2)=d(T_3)=d(U)+1$. 

Here (\ref{eqn:abc}) is equivalent to ${\tt wt}(U)={\tt wt}(T_1)+\beta(\x){\tt wt}(T_2)+\hat{\beta}(\x){\tt wt}(T_3)$.
If  $\x^\rightarrow$ is productive in $U$, then this follows from
\[(-1)^{d(U)}b\cdot a=(-1)^{d(U)}a+\beta(\x)(-1)^{d(U)+1}a+\hat{\beta}(\x)(-1)^{d(U)+1}\left(1-\frac{b}{\hat{\beta}(\x)}\right)\cdot a;\]
otherwise we are done by the same identity without the $a$'s.

\noindent
{\sf Subcase 7.2:  ($\bullet_{\GG} \in \x^{\rightarrow \uparrow}$)}: Thus $T\in {\mathcal S}_1$. 
Let $\overline{T}:=\phi_1(T)$. Let $\overline{S}$ be the snake of $\overline{T}$ 
containing $\x$. Let $S'$ and $\overline{S'}$ be the snakes of $T$ and $\overline{T}$
containing $\x^{\rightarrow\uparrow}$, respectively.
Set ${\mathcal B}_S = \{ S, S', \overline{S}, \overline{S'} \}$. 

Notice $S'$ falls into {\sf Case 3} and in fact the ${\mathcal B}_{S'}$ defined there equals
the current ${\mathcal B}_S$. Hence (\ref{eqn:abc}) holds by {\sf Case 3}.

\noindent
{\sf Case 8: ($\head(S)$ is {\sf H8}):} Here $\head(S)=S$. The definitions of $\mathcal{B}_S$ and subsequent analysis are exactly the same as in {\sf Case 3}.

\noindent
{\sf Case 9: ($\head(S)$ is {\sf H9}):} Let ${\mathcal B}_S=\{S\}$. If $\head(S)=S$, then since swapping at $S$ does nothing (including no change to any
$\bullet$ indices), (\ref{eqn:abc}) is trivially true. If $\head(S)\neq S$, then we use an argument similar to {\sf Subcase 0.2}. 

\section*{Acknowledgments}
AY thanks Hugh Thomas for his collaboration during \cite{Thomas.Yong:H_T}. Both authors thank him for his crucial role during the early stages of this project.
OP was supported by an Illinois Distinguished Fellowship, an NSF Graduate Research Fellowship and NSF MCTP
grant DMS 0838434.
AY was supported by NSF grants and a Helen Corley Petit fellowship at UIUC.

\end{document}